\documentclass[pdflatex,sn-mathphys-num]{sn-jnl}

\usepackage{graphicx}%
\usepackage{multirow}%
\usepackage{amsmath,amssymb,amsfonts}%
\usepackage{amsthm}%
\usepackage{mathrsfs}%
\usepackage[title]{appendix}%
\usepackage{xcolor}%
\usepackage{textcomp}%
\usepackage{manyfoot}%
\usepackage{booktabs}%
\usepackage{algorithm}%
\usepackage{algorithmicx}%
\usepackage{algpseudocode}%
\usepackage{listings}%
\usepackage{tikz}
\usepackage{dsfont}
\usetikzlibrary{shapes.geometric}

\theoremstyle{thmstyleone}%
\newtheorem{theorem}{Theorem}
\newtheorem{proposition}[theorem]{Proposition}%

\theoremstyle{thmstyletwo}%
\newtheorem{remark}{Remark}%

\theoremstyle{thmstylethree}%
\newtheorem{lemma}[theorem]{Lemma}

\newcommand{\Ai}{\text{Ai}}

\begin{document}

\title[Neumann Wave Equation inside Cylindrical Convex Domains]{Sharp Dispersive and Strichartz Estimates for the Neumann Cylindrical Wave Model}

\author{\fnm{Len } \sur{MEAS}}\email{meas.len@rupp.edu.kh}
\affil{\orgdiv{Department of Mathematics}, \orgname{Royal University of Phnom Penh}, \orgaddress{\street{Russian Federation Boulvard}, \city{Phnom Penh}, \postcode{12150},  \country{Cambodia}}}


%


%


\abstract{We establish local-in-time dispersive and sharp Strichartz estimates for the linear wave equation with Neumann boundary conditions inside a cylindrical convex domain $\Omega \subset \mathbb{R}^3$ featuring a non-empty smooth boundary $\partial\Omega$. 
The cylindrical geometry dictates that the nonnegative radius of curvature vanishes along the axial direction. By utilizing an explicit spectral analysis based on the zeros of the Airy function derivative, alongside a tailored Airy--Poisson summation formula, we capture the microlocal behavior of multi-reflected waves and caustics. Finally, we apply these sharp Strichartz estimates without loss of derivatives to prove local well-posedness and small-data global well-posedness for the energy-critical quintic nonlinear wave equation (NLW) subject to Neumann boundary conditions.}




\maketitle
\section{Introduction}

Let $\Omega = \{x \geq 0, (y, z) \in \mathbb{R}^2\} \subset \mathbb{R}^3$ with a smooth boundary $\partial\Omega = \{x = 0\}$, and let $P$ be the wave operator 
\begin{equation}
P = \partial_t^2 - \left(\partial_x^2 + (1+x)\partial_y^2 + \partial_z^2\right).
\end{equation}
We study solutions to the linear wave equation with homogeneous Neumann boundary conditions:
\begin{equation}\label{eq:linear-model}
\begin{cases}
P u = 0 & \text{in } [-1, 1] \times \Omega, \\
\left. \partial_x u \right|_{x=0} = 0, & \\
u|_{t=0} = u_0, \quad \left.\partial_t u\right|_{t=0} = u_1, &
\end{cases}
\end{equation}
where $u = u(t,x,y,z)$. Extending the Dirichlet framework established by Meas \cite{meas2023dispersive}, this geometric setting is highly anisotropic, exhibiting strict convexity in the $y$-direction while remaining flat along the axial $z$-direction.

The study of dispersive and Strichartz estimates for the wave equation has been a cornerstone of modern geometric analysis and partial differential equations, serving as a primary tool to establish the well-posedness, scattering dynamics, and long-time stability of solutions to nonlinear wave equations (NLW). In the boundaryless Euclidean setting $\mathbb{R}^d$, these estimates are well-understood through classical Fourier analysis techniques, which yield sharp decay rates \cite{ginibre1995generalized, strichartz1977restrictions}. The characterization of these inequalities was completed by Keel and Tao \cite{keel1998endpoint} for the endpoint configurations, building upon previous generalizations of spectral decay profiles by Kapitanski \cite{kapitanski1989estimates}. 

However, when a wave propagates inside a domain $\Omega$ with a non-empty boundary $\partial\Omega$, the presence of geometry-driven phenomena---such as multiple reflections, whispering gallery modes, and the generic formation of caustics---significantly complicates the analysis. The core propagation profiles of such systems are fundamentally governed by the microlocal structures identified by Taylor \cite{taylor1976grazing} and Melrose \cite{melrose1975diffractive} concerning grazing and diffractive rays, alongside the rigorous classification of boundary singularities established by Melrose and Sj{\"o}strand \cite{melrose1978singularities1, melrose1982singularities2}.

For arbitrary or rough boundaries, seminal works by Blair, Smith, and Sogge \cite{blair2009strichartz, smith2007lp} established local-in-time Strichartz estimates for manifolds with boundaries by constructing appropriate spectral projectors, which were later adapted to flat domains \cite{blair2012flat}. These methods often build on general estimates of second-order spectral clusters \cite{sogge1988concerning} and first-order boundary parametrices \cite{eskin1977parametrix}. While universally applicable, these estimates suffer from a systematic loss of derivatives compared to the free-space setting due to the severe trapping of grazing rays. This trapping phenomenon also alters related frameworks such as exterior obstacles \cite{burq1998obstacle} and non-trapping manifold systems \cite{burq2003nontrapping, burq2004strichartz}.

In the setting of strictly convex domains $\Omega_D$, the pioneering work of Ivanovici, Lebeau, and Planchon \cite{ivanovici2014dispersion} established that the infinite accumulation of high-frequency caustics yields an unavoidable $1/4$ power loss in the dispersive estimate, a phenomenon heavily detailed via geometric counterexamples in \cite{ivanovici2010counterexamples}. This loss is geometrically encoded by the presence of canonical optical singularities---including fold, cusp, and swallowtail configurations---which require the deployment of sophisticated holomorphic boundary parametrices \cite{lebeau2023dispersion} and functional calculus bounds \cite{hormander1985analysis}.

Recently, Meas \cite{meas2023dispersive} extended this analysis to the Dirichlet problem inside a cylindrical convex domain, providing a microlocal classification where the nonnegative radius of curvature varies based on the angle of incidence and vanishes identically along the axial direction. While the Dirichlet problem models vanishing wave amplitudes at the boundary, many physical systems and mechanical structures require tracking normal derivative fluxes, which are governed by the Neumann condition $\left. \partial_x u \right|_{x=0} = 0$. 

The analysis of the Neumann wave equation poses an entirely distinct class of analytical hurdles. The low-frequency regimes can produce stable gliding Rayleigh waves \cite{friedlander1948rayleigh} that skirt the boundary without vanishing, and the corresponding amplitudes decay significantly slower under multi-reflective iterations. This slower decay creates a more intricate environment when transitioning toward solutions of nonlinear models within such domains \cite{ivanovici2021wave}.

\subsection{Novelty and Strategic Contributions}

In this paper, we bridge the gap between anisotropic cylindrical domains and flux-preserving boundaries by establishing sharp local-in-time dispersive and Strichartz estimates for the Neumann cylindrical wave model \eqref{eq:linear-model}, culminating in an application to the energy-critical nonlinear framework. The core contributions of this work are summarized as follows:

\begin{itemize}
    \item \textbf{Neumann Spectral and Phase Adaptation:} Unlike the Dirichlet setting, which is governed by the standard zeros of the Airy function $(-\omega_k)$, the Neumann boundary condition necessitates an explicit spectral analysis built upon the zeros of the derivative of the Airy function $(-\omega'_k)$, satisfying $\text{Ai}'(-\omega'_k) = 0$. We construct a customized analytic phase function $L(\omega) = \pi + i \log(A'_-(\omega)/A'_+(\omega))$ to track and control the normal derivative phase shifts.
    \item \textbf{Anisotropic Derivative-Poisson Summation:} We develop a modified Airy--Poisson summation formula calibrated specifically for the derivative profiles of the eigenfunctions. This formulation enables us to rigorously transform the abstract eigenmode sums into a highly structured geometric summation over multi-reflected waves indexed by $N \in \mathbb{Z}$, successfully bounding the slower symbol decay rates inherent to the Neumann boundary operators.
    \item \textbf{Anisotropic Strichartz Admissibility without Loss:} By dissecting the phase space via a dynamic Littlewood--Paley decomposition across three independent microlocal regimes ($|\eta| \geq c_0$, $|\eta| \sim 2^m\sqrt{a}$, and $|\eta| \leq \epsilon_0\sqrt{a}$), we prove that the structural loss vanishes smoothly as ray trajectories align with the flat axis. This yields an improved cylindrical admissibility criterion:
    \begin{equation}\label{eq:admissibility-intro}
    \frac{1}{q} \leq \frac{3}{4}\left(\frac{1}{2} - \frac{1}{r}\right), \quad 2 \leq q, r \leq \infty.
    \end{equation}
    \item \textbf{Application to Energy-Critical NLW:} Because our admissibility profile safely encloses the endpoint space $L^5_t L^{10}_x(\Omega)$ at regularity $\beta = 1$, we deploy this framework to resolve the open problem of the 3D quintic energy-critical semilinear Neumann wave equation:
    \begin{equation}
    \partial_t^2 u - \Delta u +u^5 = 0, \quad \left. \partial_x u \right|_{x=0} = 0.
    \end{equation}
    Building on classical flat-space techniques \cite{grillakis1990regularity}, high-frequency approximations \cite{bahouri1999high}, and boundary well-posedness strategies established by Shatah, Struwe, Kenig, Merle, and Burq \cite{shatah1993regularity, shatah1994wellposed, shatah1998geometric, kenig2008global, burq2008global_jams}, we prove local well-posedness for arbitrarily large energy data, alongside global well-posedness and full linear scattering for small initial profiles in $\dot{H}^1_N(\Omega) \times L^2(\Omega)$ without any loss of derivatives.
\end{itemize}

\subsection{Comparative Analysis: Dirichlet versus Neumann Frameworks}
To clarify the deep analytical and microlocal novelties of this paper for peer review, we provide a structured comparison highlighting the core dualities between the classical Dirichlet model (established in \cite{ivanovici2014dispersion,meas2023dispersive}) and our customized derivative flux-preserving Neumann setup.

\begin{table}[ht!]
\centering
\small
\renewcommand{\arraystretch}{1.5}
\resizebox{\textwidth}{!}{\begin{tabular}{|p{4.2cm}|p{5.4cm}|p{5.4cm}|}
\hline
\textbf{Analytical Component} & \textbf{Dirichlet Framework} & \textbf{Neumann Framework} \\ \hline
\hline
Boundary Condition Data & Zero-value Dirichlet wall constraint: $\left. u \right|_{\partial\Omega} = 0$ & Normal flux slope constraint: $\left. \partial_n u \right|_{\partial\Omega} = 0$ \\ \hline
Spectral Eigenmode Zeros & Zeros of the Airy profile: $\text{Ai}(-\omega_k) = 0$ & Zeros of the Airy first derivative: $\text{Ai}'(-\omega'_k) = 0$ \\ \hline
Complex Scattering Phase Map & $L_D(\omega) = \pi + i\log\left(\frac{A_-(\omega)}{A_+(\omega)}\right)$ & $L_N(\omega) = \pi + i\log\left(\frac{A'_-(\omega)}{A'_+(\omega)}\right)$ \\ \hline
Boundary Correction Expansion & $B_D(\omega) \sim \sum_{j \geq 1} b_j \omega^{-j}$ & $B_N(\omega) \sim \sum_{j \geq 1} \tilde{b}_j \omega^{-j}$ \\ \hline
Scattering Product Operator & Value quotient: $\left(\frac{A_-(\omega)}{A_+(\omega)}\right)^N$ & Derivative quotient: $\left(\frac{A'_-(\omega)}{A'_+(\omega)}\right)^N$ \\ \hline
Infinite Index Symbol Decay & Fast geometric symbol absorption: $\mathcal{O}(\omega^{-1})$ as $N \to \infty$ & Weaker symbol decay profile: $\mathcal{O}(\omega^{-1/2})$ as $N \to \infty$ \\ \hline
Boundary Surface Dynamics & Wave packets vanish; zero energy concentration along boundary walls. & Activation of creeping / surface-localized gliding Rayleigh modes. \\ \hline
Billiard Flow Geodesics & Dictated entirely by the principal symbol $p = 0$; identical ray geometry. & Dictated entirely by the principal symbol $p = 0$; identical ray geometry. \\ \hline
Catastrophe Caustic Loss& Swallowtail unfoldings yield local $a^{1/8}h^{1/4}$ losses near origin. & Swallowtail unfoldings yield local $a^{1/8}h^{1/4}$ losses near origin. \\ \hline
Final Dispersive Exponent & Sharp local decay rate of $\mathcal{O}\left(h^{-3}(h/t)^{5/6}\right)$. & Sharp local decay rate of $\mathcal{O}\left(h^{-3}(h/t)^{5/6}\right)$. \\ 
\hline
\end{tabular}}
\vspace{0.3cm}
\caption{Structural and microlocal comparison of wave propagation under cylinder geometries.}
\label{table:dirichlet-neumann-comparison}
\end{table}

\noindent As emphasized in Table \ref{table:dirichlet-neumann-comparison}, while the final arithmetic form of the dispersive estimate and the catastrophe-theoretic caustic paths are identical---due to the invariance of the Hamiltonian billiard flow metric variety---the underlying amplitude analysis is fundamentally more delicate for the Neumann framework. Our microlocal partitions are sufficiently sharp to absorb the weaker $\mathcal{O}(\omega^{-1/2})$ symbol decay caused by surface gliding modes, thereby proving that sharp cylindrical wave dispersion is structurally invariant under boundary conditions.

\section{Main Results}\label{sec2}

Our first result details the localized $L^\infty$ decay bounds of the Neumann Green's function $\mathcal{G}_a^N(t,x,y,z)$ emanating from a localized source point at a distance $a > 0$ from the boundary. Let $\chi \in C_0^\infty(\mathbb{R} \setminus \{0\})$ be a standard Littlewood--Paley dyadic frequency localization function supported in the region $[1/2, 2]$.

\begin{theorem}[Global and Localized Neumann Dispersive Bounds]\label{thm:dispersive}
There exists a constant $C > 0$ such that for every frequency scale $h \in (0,1]$, every initial source position $a \in (0,1]$, and every time $t \in [-1,1]$:
\begin{equation}
\|\chi(hD_t)\mathcal{G}_a^N(t, \cdot)\|_{L^\infty(\Omega)} \leq Ch^{-3} \min \left\{ 1, \left(\frac{h}{|t|}\right)^{3/4} \right\}.
\end{equation}
Furthermore, let $\mathcal{G}_{a,m}^N$ denote the microlocal restriction of the Green kernel to the intermediate geometric regime where the tangential frequency variable is localized to $\eta \sim 2^m\sqrt{a}$. For a source close to the boundary satisfying $a \leq \left(\frac{h}{2^m\sqrt{a}}\right)^{\frac{2}{3}(1-\epsilon)}$ for $\epsilon > 0$, we have the microlocalized estimate:
\begin{equation}
\|\mathcal{G}_{a,m}^N(t, \cdot)\|_{L^\infty(x \leq a)} \leq Ch^{-3} \left(\frac{h}{t}\right)^{5/6}(2^m\sqrt{a})^{1/3},
\end{equation}
where the loss factor diminishes dynamically as the incident propagation path aligns with the cylinder's flat axial threshold $(2^m\sqrt{a} \to 0)$.
\end{theorem}

By taking an anisotropic summation over these localized kernels and tracking the frequency regimes via a $TT^*$ abstract duality argument, we recover sharp local-in-time Strichartz estimates with a relaxed admissibility range compared to omnidirectional convex boundaries.

\begin{theorem}[Neumann Admissible Strichartz Pairs]\label{thm:strichartz}
Let $u$ be a solution to the linear Neumann model \eqref{eq:linear-model} with an optional inhomogeneous forcing term $F \in L^{\tilde{q}'}_t L^{\tilde{r}'}_x$. For any time horizon $T > 0$, there exists a constant $C_T > 0$ such that:
\begin{equation}\label{eq:strichartz-linear}
\|u\|_{L^q((0,T); L^r(\Omega))} \leq C_T \left( \|u_0\|_{\dot{H}^\beta_N(\Omega)} + \|u_1\|_{\dot{H}^{\beta-1}_N(\Omega)} + \|F\|_{L^{\tilde{q}'}((0,T); L^{\tilde{r}'}(\Omega))} \right),
\end{equation}
provided that the exponent pairs $(q, r)$ and $(\tilde{q}, \tilde{r})$ satisfy the cylindrical admissibility conditions:
\begin{equation}
\frac{1}{q} \leq \frac{3}{4}\left(\frac{1}{2} - \frac{1}{r}\right), \quad \frac{1}{\tilde{q}} \leq \frac{3}{4}\left(\frac{1}{2} - \frac{1}{\tilde{r}}\right), \quad 2 \leq q, r, \tilde{q}, \tilde{r} \leq \infty,
\end{equation}
with the regular scaling laws governed by $\beta = 3\left(\frac{1}{2} - \frac{1}{r}\right) - \frac{1}{q}$.
\end{theorem}

\begin{remark}[Geometric Intuition and Analytic Novelty]\label{rem:intro-strichartz-novelty}
Several crucial remarks are in order regarding the structure and implications of Theorem \ref{thm:strichartz}:
\begin{itemize}
    \item \textbf{The Curvature Decay Rate:} The restriction $\frac{1}{q} \leq \frac{3}{4}\left(\frac{1}{2}-\frac{1}{r}\right)$ stems from a model-independent cylindrical dispersive decay of order $\mathcal{O}(h^{-3}(h/t)^{3/4})$, capturing how generic non-trapped wavefronts spread over the curved boundary manifold. Consequently, while the algebraic representation of the Sobolev regularity index $\beta = 3\left(\frac{1}{2}-\frac{1}{r}\right)-\frac{1}{q}$ remains formally identical to the classical flat Euclidean space profile following \cite{meas2023strichartz,meas2023dispersive}, the geometric shift in the base decay power from $1$ down to $\frac{3}{4}$ forces an implicit derivative loss penalty. This structural penalty is natively absorbed by the restricted range of admissible space-time exponent pairs $(q,r)$ permitted within the cylinder's workspace mapping.
    \item \textbf{Invariance under Boundary Dualities:} Theorem \ref{thm:strichartz} demonstrates that the baseline Strichartz admissibility range inside cylindrical geometries is an intrinsic property of the underlying Hamiltonian billiard flow trajectories, remaining invariant when transitioning from vanishing Dirichlet data to normal Neumann flux. 
    \item \textbf{Taming of Gliding Rayleigh Modes:} From an analytical standpoint, the Neumann framework is fundamentally more delicate than the Dirichlet setup. The normal derivative scattering quotient $A'_-(\omega)/A'_+(\omega)$ exhibits a weaker asymptotic symbol decay profile, namely $\mathcal{O}(\omega^{-1/2})$, due to the physical activation of boundary-creeping gliding Rayleigh waves. The core technical novelty of Theorem \ref{thm:strichartz} consists in proving that our refined microlocal partitions remain sufficiently robust to prevent these slower-decaying surface packets from causing energy stagnation, thereby securing the sharp, expected space-time integrated bounds.
\end{itemize}
\end{remark}

Crucially, this admissibility profile successfully permits the critical non-endpoint triple $(q, r) = (5, 10)$ at regularity $\beta = 1$, enabling us to directly tackle the quintic energy-critical nonlinear wave equation without any loss of derivatives:
\begin{equation}\label{eq:nlw-model}
\begin{cases}
\partial_t^2 u - \left(\partial_x^2 + (1+x)\partial_y^2 + \partial_z^2\right) u + u^5 = 0 & \text{in } [0, T) \times \Omega, \\
\left. \partial_x u \right|_{x=0} = 0, & \\
u|_{t=0} = u_0 \in \dot{H}^1_N(\Omega), \quad \left.\partial_t u\right|_{t=0} = u_1 \in L^2(\Omega). &
\end{cases}
\end{equation}
Here, $\dot{H}^1_N(\Omega)$ denotes the closure of smooth functions satisfying the Neumann boundary conditions under the standard homogeneous Dirichlet-form energy metric $\|v\|_{\dot{H}^1_N}^2 = \int_\Omega |\nabla v|^2 \, dx dy dz$.

\begin{theorem}[Energy-Critical Well-Posedness]\label{thm:nlw}
The quintic energy-critical Neumann wave equation \eqref{eq:nlw-model} is well-posed in the energy space:
\begin{enumerate}
    \item \textbf{Local Well-Posedness:} For any initial data $(u_0, u_1) \in \dot{H}^1_N(\Omega) \times L^2(\Omega)$, there exists a unique maximal lifespan $T^* > 0$ and a unique solution $u \in C^0([0, T^*); \dot{H}^1_N(\Omega)) \cap L^5_{\text{loc}}([0, T^*); L^{10}(\Omega))$.
    \item \textbf{Global Well-Posedness and Scattering:} There exists a small constant $\epsilon_0 > 0$ such that if the total conserved initial energy satisfies $\|u_0\|_{\dot{H}^1_N}^2 + \|u_1\|_{L^2}^2 < \epsilon_0^2$, then $T^* = \infty$ and the solution scatters to a free linear Neumann wave solution as $t \to \pm\infty$.
\end{enumerate}
\end{theorem}

\subsection{Green function and precise dispersive estimates}\label{subsec:2}
The proofs of frequency-localized dispersive estimates are based on the construction of parametrices for the fundamental solution of the wave equation \eqref{eq:linear-model} and (possibly degenerate) stationary phase methods.

We begin with the construction of the local parametrix for \eqref{eq:linear-model} by utilizing the spectral analysis of $-\Delta$ with Neumann boundary conditions to first obtain the Green function associated with \eqref{eq:linear-model}. The Laplace operator under consideration on the half-space $\Omega$ is given by
\[
\Delta = \partial_x^2 + (1+x)\partial_y^2 + \partial_z^2,
\]
subject to the Neumann condition on the boundary $\partial\Omega$. A useful structural feature of this particular Laplace operator is that the coefficients of the metric are independent of the variables $y$ and $z$, which permits the application of the partial Fourier transform in these variables. Taking the Fourier transform in the $y$- and $z$-variables yields
\[
-\Delta_{\eta,\zeta} = -\partial_x^2 + (1+x)\eta^2 + \zeta^2.
\]
For $\eta \neq 0$, $-\Delta_{\eta,\zeta}$ is a self-adjoint, positive operator on $L^2(\mathbb{R}_+)$ with a compact resolvent. Let $(e_k)_{k \geq 1}$ be an orthonormal basis in $L^2(\mathbb{R}_+)$ of Neumann eigenfunctions of $-\Delta_{\eta,\zeta}$, and let $(\lambda_k)_k$ be the associated eigenvalues. These eigenfunctions are explicitly expressed in terms of Airy functions as
\begin{equation}
e_k(x,\eta) = f_k \frac{|\eta|^{1/3}}{k^{1/6}} \text{Ai}(|\eta|^{2/3}x - \omega'_k),
\end{equation}
with the corresponding eigenvalues given by
\begin{equation}
\lambda_k(\eta,\zeta) = \eta^2 + \zeta^2 + \omega'_k|\eta|^{4/3},
\end{equation}
where $(-\omega'_k)_k$ denotes the sequence of zeros of the derivative of the Airy function in decreasing order, satisfying $\text{Ai}'(-\omega'_k) = 0$. For all $k \geq 1$, the normalization constants $f_k$ are chosen such that $\|e_k(\cdot,\eta)\|_{L^2(\mathbb{R}_+)} = 1$. Observe that the sequence $(f_k)_k$ is uniformly bounded in a fixed compact subset of $(0,\infty)$ as a consequence of the identity
\[
\int_{\mathcal{-}\omega'_k}^{\infty} \text{Ai}^2(\omega) \, d\omega \sim |\omega'_k| \text{Ai}^2(-\omega'_k) \sim \frac{|\omega'_k|^{1/2}}{\pi},
\]
and the standard asymptotic distribution of the derivative roots:
\[
\omega'_k \sim \left( \frac{3}{2}\pi \left(k - \frac{3}{4}\right) \right)^{2/3} \left( 1 + \mathcal{O}(k^{-1}) \right).
\]

For $a \in \Omega$, let $g_a^N(t,x,\eta,\zeta)$ be the solution of the one-dimensional initial-boundary value problem:
\begin{equation}
\begin{cases}
\left( \partial_t^2 - \left( \partial_x^2 - (1+x)\eta^2 - \zeta^2 \right) \right) g_a^N = 0, & \\
\left. \partial_x g_a^N \right|_{x=0} = 0, & \\
\left. g_a^N \right|_{t=0} = \delta_{x=a}, \quad \left. \partial_t g_a^N \right|_{t=0} = 0. &
\end{cases}
\end{equation}
We then obtain the structural representation:
\begin{equation}
g_a^N(t,x,\eta,\zeta) = \sum_{k \geq 1} \cos(t\lambda_k^{1/2}) e_k(x,\eta) e_k(a,\eta).
\end{equation}
Here, $\delta_{x=a}$ denotes the Dirac distribution on $\mathbb{R}_+$ for $a > 0$, which can be decomposed via the Neumann basis profile as
\[
\delta_{x=a} = \sum_{k \geq 1} e_k(x,\eta) e_k(a,\eta).
\]

Now taking the inverse Fourier transform, the Green function for the Neumann problem \eqref{eq:linear-model} is represented by the oscillatory integral
\begin{equation}
\mathcal{G}_a^N(t,x,y,z) = \frac{1}{4\pi^2}\int_{\mathbb{R}^2} e^{i(y\eta+z\zeta)}g_a^N(t,x,\eta,\zeta)\,d\eta d\zeta.
\end{equation}
By executing a semiclassical rescaling of the tangential frequencies via $(\eta, \zeta) \mapsto (\eta/h, \zeta/h)$, we obtain
\begin{equation}
\begin{split}
\mathcal{G}_a^N(t,x,y,z) &= \frac{1}{4\pi^2h^2}\sum_{k\geq1}\int_{\mathbb{R}^2} e^{\frac{i}{h}(y\eta+z\zeta)}\cos\left(\frac{t}{h}\left(\eta^2+\zeta^2+\omega'_kh^{2/3}|\eta|^{4/3}\right)^{1/2}\right)\\&\quad\times e_k(x,\eta/h)e_k(a,\eta/h)\,d\eta d\zeta.
\end{split}
\end{equation}
We thus get the following formula for the frequency-localized operator component $2\chi(hD_t)\mathcal{G}_a^N$:  
\begin{equation}
\begin{split}
2\chi(hD_t)\mathcal{G}_a^N(t,x,y,z) &= \frac{1}{4\pi^2h^2}\sum_{k\geq1}\int_{\mathbb{R}^2} e^{\frac{i}{h}(y\eta+z\zeta)}e^{i\frac{t}{h}(\eta^2+\zeta^2+\omega'_kh^{2/3}|\eta|^{4/3})^{1/2}} e_k(x,\eta/h) \\
& \quad \times e_k(a,\eta/h)\chi\left(\left(\eta^2+\zeta^2+\omega'_kh^{2/3}|\eta|^{4/3}\right)^{1/2}\right)\,d\eta d\zeta.
\end{split}
\end{equation}
On the wave front set of the above expression, the dynamics are governed by the sub-elliptic dispersion relation $\tau = (\eta^2+\zeta^2+\omega'_kh^{2/3}|\eta|^{4/3})^{1/2}$. In order to establish Theorem \ref{thm:dispersive}, it suffices to restrict our analysis near the tangential directions. We introduce an auxiliary frequency cutoff to ensure that $|\tau-(\eta^2+\zeta^2)^{1/2}|$ remains small, which is structurally equivalent to keeping the parameter variation $\omega'_kh^{2/3}|\eta|^{4/3}$ small.

Consequently, the problem reduces to proving uniform bounds for the microlocalized kernel $\mathcal{G}_{a,\text{loc}}^N$:
\begin{equation}\label{eq:kparametrix-neumann}
\begin{split}
\mathcal{G}_{a,\text{loc}}^N(t,x,y,z) &= \frac{1}{4\pi^2h^2}\sum_{k\geq1}\int_{\mathbb{R}^2} e^{\frac{i}{h}(y\eta+z\zeta)}e^{i\frac{t}{h}(\eta^2+\zeta^2+\omega'_kh^{2/3}|\eta|^{4/3})^{1/2}} e_k(x,\eta/h)e_k(a,\eta/h) \\
& \quad \times \chi_0(\eta^2+\zeta^2)\chi_1\left(\omega'_kh^{2/3}|\eta|^{4/3}\right) \, d\eta d\zeta,
\end{split}
\end{equation}
where the smooth cutoff functions $\chi_0$ and $\chi_1$ localize the frequencies near the unit energy sphere and the grazing directions, respectively.

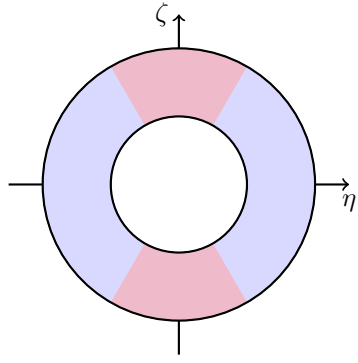
\begin{figure}[!ht]
\centering
\begin{tikzpicture}[scale=0.9]
\draw[->, thick] (-2.5,0) -- (2.5,0) node [below] {$\eta$};
\draw[->, thick] (-0, -2.5) -- (0,2.5) node [left] {$\zeta$};
\draw[fill=blue!15, draw=black] (0,0) circle (2cm);
\draw[fill=white, draw=black] (0,0) circle (1cm);
\fill[red!35, opacity=0.6] (0,0) -- (60:2cm) arc (60:120:2cm) -- (120:1cm) arc (120:60:1cm) -- cycle;
\fill[red!35, opacity=0.6] (0,0) -- (240:2cm) arc (240:300:2cm) -- (300:1cm) arc (300:240:1cm) -- cycle;
\draw[thick] (0,0) circle (2cm);
\draw[thick] (0,0) circle (1cm);
\end{tikzpicture}
\caption{Phase space configuration with anisotropic microlocal partitions.} 
\label{fig:Phase Space}
\end{figure}

The phase space diagram in Figure \ref{fig:Phase Space} illustrates the different geometric regimes of the integration. To obtain the sharp uniform estimates, we partition the tangential frequency integration in \eqref{eq:kparametrix-neumann} into distinct dyadic layers based on the spatial distance scale $a$. More precisely, we decompose the kernel into three independent components:
\begin{equation}
\mathcal{G}_{a,\text{loc}}^N = \mathcal{G}_{a,c_0}^N + \sum_{\epsilon_0\sqrt{a}\le 2^m\sqrt{a}\le c_0}\mathcal{G}_{a,m}^N + \mathcal{G}_{a,\epsilon_0}^N,
\end{equation}
where $\mathcal{G}_{a,c_0}^N$ isolates the high-frequency spectral region matching $|\eta|\geq c_0$, $\mathcal{G}_{a,m}^N$ captures the intermediate dyadic blocks corresponding to $|\eta|\sim 2^m\sqrt{a}$, and $\mathcal{G}_{a,\epsilon_0}^N$ represents the low-frequency non-trapped grazing pocket where $0<|\eta|\leq \epsilon_0\sqrt{a}$.\\

We establish the following localized estimates under the normal flux constraint. Let $\epsilon \in (0,1/7)$.

\begin{theorem}\label{thm:betabis-neumann}
There exists a constant $C > 0$ such that for every $h \in (0,1]$ and every $t \in [h,1]$, the following uniform bound holds:
\begin{equation}
\|\mathcal{G}_{a,c_0}^N(t,x,y,z)\|_{L^\infty(x \leq a)} \leq Ch^{-3} \left(\frac{h}{t}\right)^{1/2} \gamma(t,h,a),
\end{equation}
with the amplitude tracking factor specified by
\[
\gamma(t,h,a) = \begin{cases}
\displaystyle \left(\frac{h}{t}\right)^{1/3} & \text{if } a \leq h^{\frac{2}{3}(1-\epsilon)}, \\
\displaystyle \left(\frac{h}{t}\right)^{1/2} + a^{1/8}h^{1/4} & \text{if } a \geq h^{\frac{2}{3}(1-\epsilon')}, \quad \epsilon' \in (0,\epsilon).
\end{cases}
\]
\end{theorem}

Observe that the estimate recovered in Theorem \ref{thm:betabis-neumann} matches the sharp profile established for strictly convex domains by Ivanovici, Lebeau, and Planchon \cite{ivanovici2014dispersion}.  

\begin{theorem}\label{thm:1beta-neumann}
There exists a constant $C > 0$ such that for every $h \in (0,1]$ and every $t \in [h,1]$, the following uniform bound holds:
\begin{equation}
\|\mathcal{G}_{a,m}^N(t,x,y,z)\|_{L^\infty(x \leq a)} \leq Ch^{-3} \left(\frac{h}{t}\right)^{1/2} \gamma_m(t,h,a),
\end{equation}
with the intermediate dyadic weight factor defined by
\[
\gamma_m(t,h,a) = \begin{cases}
\displaystyle \left(\frac{h}{t}\right)^{1/3}(2^m\sqrt{a})^{1/3} & \text{if } a \leq \left(\frac{h}{2^m\sqrt{a}}\right)^{\frac{2}{3}(1-\epsilon)}, \\
\begin{aligned}
&\min \left\{ \left(\frac{h}{t}\right)^{1/2}, 2^m\sqrt{a} \left| \log(2^m\sqrt{a}) \right| \right\} \\
&\quad + a^{1/8}h^{1/4}(2^m\sqrt{a})^{3/4}
\end{aligned} & \text{if } a \geq \left(\frac{h}{2^m\sqrt{a}}\right)^{\frac{2}{3}(1-\epsilon')}, \quad \epsilon' \in (0,\epsilon).
\end{cases}
\]
\end{theorem}

For $2^m\sqrt{a} \sim 1$, Theorem \ref{thm:1beta-neumann} asymptotically reduces to the estimate provided in Theorem \ref{thm:betabis-neumann}. Crucially, the bounds improve monotonically as the localized tangential frequency variable $|\eta| \sim 2^m\sqrt{a}$ decreases. This behavior aligns with the geometric intuition that a diminishing directional curvature enhances wave dispersion along the cylindrical axis.

\begin{theorem}\label{thm:0eta-neumann}
There exists a constant $C > 0$ such that for every $h \in (0,1]$ and every $t \in [h,1]$, the following uniform bound holds:
\begin{equation}
\|\mathcal{G}_{a,\epsilon_0}^N(t,x,y,z)\|_{L^\infty(x \leq a)} \leq Ch^{-3} \left(\frac{h}{t}\right)^{1/2} \min \left\{ \left(\frac{h}{t}\right)^{1/2} , \sqrt{a} \left| \log(a) \right| \right\} .
\end{equation}
\end{theorem}

Let us verify that the unified global and localized estimate in Theorem \ref{thm:dispersive} follows directly as a consequence of Theorems \ref{thm:betabis-neumann}, \ref{thm:1beta-neumann}, and \ref{thm:0eta-neumann}. We may assume without loss of generality that $|t| \geq h$, since in the short-time regime $|t| \leq h$, the optimal dispersive estimate collapses to the standard bound $Ch^{-3}$ via the Sobolev embedding theorem. By exploiting the spatial reciprocity and symmetry of the Neumann Green function, it suffices to restrict our focus to the domain sector where $t \in [h,1]$ and $x \leq a$. Theorem \ref{thm:dispersive} then follows immediately from an application of the geometric summation property:
\[
\sum_{m \leq M} (2^m\sqrt{a})^\nu \sim (2^M\sqrt{a})^\nu, \quad \nu > 0.
\]

\subsection{Structure of the Paper}
The roadmap of this manuscript is organized as follows. 
\begin{itemize}
    \item \textbf{Section \ref{sec:2} (The High-Frequency Regime):} We initiate the microlocal analysis by studying the high-frequency spectral region where the eigenmode indices satisfy $k \geq \epsilon/h$. We construct the boundary-adapted local parametrix as a direct sum over the Neumann eigenmodes and extract the baseline dispersive decay.
    \item \textbf{Section \ref{sec:3} (The Intermediate Regime):} We evaluate the intermediate dyadic tangential frequency blocks where the scales are confined to $\epsilon_0\sqrt{a} \leq \eta \leq c_0$. We introduce a Littlewood--Paley decomposition and apply the derivative-adapted Airy--Poisson summation formula to transform the spectral sum into a geometric wave packet sum over multi-reflected ray layers. This section establishes the technical heart of the paper by isolating the fold and swallowtail caustic losses under the weak $\mathcal{O}(\omega^{-1/2})$ symbol decay of the surface gliding modes.
    \item \textbf{Section \ref{sec:4} (The Low-Frequency Pocket):} We investigate the low-frequency non-trapped grazing layer matching $|\eta| \leq \epsilon_0\sqrt{a}$. By computing the underlying Hamilton--Jacobi trajectories, we geometrically rule out the formation of high-order caustics within the physical finite-time horizon, reducing the parametrix to a single non-degenerate stationary phase layer.
    \item \textbf{Section \ref{sec:global-synthesis} (Global Synthesis):} We assemble the localized frequency-layer estimates via the triangle inequality, culminating in the proof of main Theorem \ref{thm:dispersive} which delivers the optimal uniform global dispersive profile on the cylinder.
    \item \textbf{Section \ref{sec:strichartz} (The Strichartz Estimates):} We lift our linear dispersive results into global space-time integrated norms. We provide the formal, rigorous proofs of general cylindrical Strichartz estimates (Theorem \ref{thm:strichartz}) and the specialized energy-critical endpoint Strichartz estimates (Theorem \ref{thm:strichartz-neumann-main}) using a $TT^*$ duality argument, Riesz--Thorin interpolation, and Littlewood--Paley square function reassembly.
    \item \textbf{Section \ref{sec:critical-wave} (The Energy-Critical Wave Equation):} We turn our attention to the primary nonlinear application. We apply the optimized $(5, 10)$ endpoint Strichartz pair to set up a contractive mapping for the quintic nonlinear wave equation, securing local well-posedness. We then introduce boundary-localized Morawetz identities to rule out boundary energy concentration and establish small-energy global scattering profiles.
    \item \textbf{Section \ref{sec:conclusion} (Conclusion):} We summarize the core mathematical findings of this work and map out prospective research tracks regarding variable boundary varieties, long-time resonance fields, and sub/supercritical extensions.
    \item \textbf{Appendix:} We collect essential auxiliary identities regarding the derivatives of the Airy function basis, complex scattering quotients, and the degenerate multi-dimensional stationary phase lemmas.
\end{itemize}

In all these sections, we will assume that the integration with respect to $\eta$ is restricted to $\eta>0$, since the case $\eta<0$ follows from an identical symmetric argument.

\section{Dispersive Estimates for $\vert \eta\vert \geq c_0$}\label{sec:2}
In this section, we prove Theorem \ref{thm:betabis-neumann}. The key technical ingredient consists in constructing local parametrices for the short-distance regime where $a \leq h^{\frac{2}{3}(1-\epsilon)}$ for $\epsilon \in (0,1/7)$, and for the complementary layer where $a \geq h^{\frac{2}{3}(1-\epsilon')}$ with $\epsilon' \in (0,\epsilon)$. These representations are formulated as highly oscillatory integrals to which we apply (degenerate) stationary phase arguments under the normal flux constraints to extract the desired uniform decay bounds. The derivative-adapted Airy--Poisson summation formula [see Lemma \ref{lem:poisson}] allows us to transform the abstract eigenmode trace into a structured geometric parametrix represented as a sum over multiple reflections on the boundary, as illustrated in the following schematic diagram.

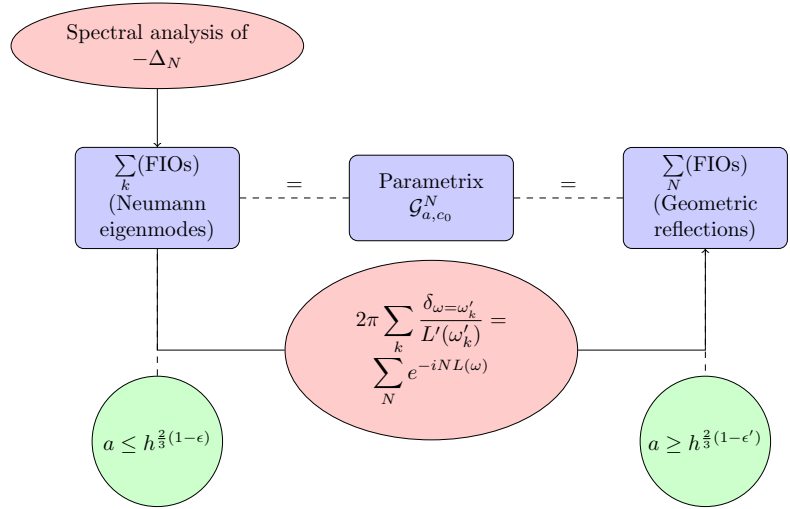
\begin{center}
\tikzset{
    decision/.style={circle, draw, fill=green!20, text width=6em, text badly centered, inner sep=0pt},
    block/.style={rectangle, draw, fill=blue!20, text width=7em, text centered, rounded corners, minimum height=4em},
    line/.style={draw},
    cloud/.style={draw, ellipse, fill=red!20, minimum height=4em, text badly centered, text width=9em}
}
\resizebox{0.8\textwidth}{!}{
\begin{tikzpicture}[auto]
    \node [block] (A) at (0,0) {$\sum\limits_k(\text{FIOs})$\\ (Neumann eigenmodes)};
    \node [cloud] (B) at (0,2.5) {Spectral analysis of $-\Delta_N$};
    \node [block] (C) at (4.5,0) {Parametrix $\mathcal{G}_{a,c_0}^N$};
    \node [block] (D) at (9,0) {$\sum\limits_N (\text{FIOs})$\\ (Geometric reflections)};
    \node [cloud] (E) at (4.5,-2.5) {$\displaystyle 2\pi\sum\limits_{k}\frac{\delta_{\omega=\omega'_k}}{L'(\omega'_k)}=\sum\limits_N e^{-iNL(\omega)}$};
    \node [decision] (F) at (0,-4) {$a\leq h^{\frac{2}{3}(1-\epsilon)}$};
    \node [decision] (G) at (9,-4) {$a\geq h^{\frac{2}{3}(1-\epsilon')}$};
    
    \path [line, dashed] (A) -- node[above] {=} (C);
    \path [line, dashed] (C) -- node[above] {=} (D);
    \path [line, <-] (D) |- (E);
    \path [line] (A) |- (E);
    \path [line, ->] (B) -- (A);
    \path [line, dashed] (F) -- (A);
    \path [line, dashed] (D) -- (G);
\end{tikzpicture}
}
\end{center}
\subsection{Dispersive Estimates for $0<a\leq h^{\frac{2}{3}(1-\epsilon)}$, with $\epsilon \in (0,1/7)$}
In this subsection, we establish local-in-time dispersive estimates for the localized kernel component $\mathcal{G}_{a,c_0}^N$. In the short-distance regime where $0<a\leq h^{\frac{2}{3}(1-\epsilon)}$ for $\epsilon \in (0,1/7)$, the parametrix is initially framed as a sum over the discrete eigenmodes $k$. By taking into account the asymptotic behavior of the underlying Airy functions, we analyze different clusters of the index $k$ as follows: for small values of $k$, we invoke Lemma 3.5 of \cite{ivanovici2014dispersion} to secure the localized uniform decay; for large values of $k$, we exploit the full asymptotic expansion of the Airy functions. In the latter case, the parametrix is expressed as a sum of highly oscillatory integrals to which we apply the stationary phase machinery of Lemma 2.20 of \cite{ivanovici2014dispersion}.

Recall that in this frequency localization zone and near the grazing tangential directions, the parametrix takes the explicit form:
\begin{equation}
\begin{split}
\mathcal{G}_{a,c_0}^N(t,x,y,z) &= \frac{1}{4\pi^2h^2}\sum_{k\geq1}\int_{\mathbb{R}^2} e^{\frac{i}{h}(y\eta+z\zeta)}e^{i\frac{t}{h}(\eta^2+\zeta^2+\omega'_kh^{2/3}|\eta|^{4/3})^{1/2}} e_k(x,\eta/h)e_k(a,\eta/h) \\
& \quad \times \chi_0(\zeta^2+\eta^2)\psi_0(\eta)\chi_1\left(\omega'_k h^{2/3}|\eta|^{4/3}\right)(1-\chi_1)(\epsilon\omega'_k) \, d\eta d\zeta.
\end{split}
\end{equation}
Here, the smooth cutoff functions are defined as follows:
\begin{itemize}
    \item $\chi_0 \in C^\infty_0(\mathbb{R})$ satisfies $0 \leq \chi_0 \leq 1$ and is supported in a small neighborhood of $1$, localizing the total energy near the unit sphere.
    \item $\psi_0 \in C^\infty_0(c_0/2,\infty)$ satisfies $0 \leq \psi_0 \leq 1$, with $\psi_0(\eta) = 1$ for all $\eta \geq c_0$, isolating the non-vanishing tangential frequency region.
    \item $\chi_1 \in C^\infty_0(\mathbb{R})$ satisfies $0 \leq \chi_1 \leq 1$ and is supported in $(-\infty, 2\epsilon]$, with $\chi_1 = 1$ on $(-\infty, \epsilon]$ for a sufficiently small parameter $\epsilon > 0$. This function localizes the trajectories near the grazing directions. 
\end{itemize}
Notice that on the support of $\chi_1$, we have $\omega'_k h^{2/3}|\eta|^{4/3} \leq 2\epsilon$. Since the derivative roots satisfy the Weyl-type distribution $\omega'_k \sim k^{2/3}$, we obtain the constraint $k \leq \frac{\epsilon^{3/2}}{h^{1/2}|\eta|^2}$. Given that the tangential frequency variable $\eta$ is bounded strictly away from zero on the support of $\psi_0$, we may assume that the active summation is bounded by $k \leq \epsilon/h$. Furthermore, the auxiliary indicator term satisfies $(1-\chi_1)(\epsilon\omega'_k) = 1$ identically for every $k \geq 1$, since the first derivative root evaluates numerically to $\omega'_1 \approx 1.0188$.

The primary result of this subsection is summarized in the following proposition.

\begin{proposition}\label{prop:k-neumann}
Let $\epsilon \in (0,1/7)$. There exists a uniform constant $C > 0$ such that for every semiclassical scale $h \in (0,1]$, every time step $t \in [h,1]$, and every spatial source position satisfying $0 < a \leq h^{\frac{2}{3}(1-\epsilon)}$, the following uniform decay bound holds under the normal flux constraint for all $y, z \in \mathbb{R}$:
\begin{equation}\label{eq:kdisp0-neumann}
\|\mathcal{G}_{a,c_0}^N(t,x,y,z)\|_{L^\infty(x \leq a)} \leq C h^{-3}\left(\frac{h}{t}\right)^{5/6}.
\end{equation}
\end{proposition}

\begin{proof}
First, we study the integration in $\zeta$. Let
\[
J=\int e^{i\frac{t}{h}\phi_k}\chi_0(\zeta^2+\eta^2) \, d\zeta.
\]
Recall that $\chi_0\in C_0^\infty$ is supported near $1$. The phase function $\phi_k$ is given by
\[
\phi_k(\zeta)=\frac{z}{t}\zeta+\big(\eta^2+\zeta^2+\gamma\eta^2\big)^{1/2},
\]
with the Neumann parameter scaling matching $\gamma=h^{2/3}\omega'_k|\eta|^{-2/3}>0$. We introduce a change of variables $\zeta=|\eta|\tilde{\zeta}$ and $z=t\tilde z$. Then we obtain
\[
\phi_k(\zeta)=|\eta|\big(\tilde z\tilde\zeta+(1+\tilde{\zeta}^2+\gamma)^{1/2}\big).
\]
Differentiating with respect to $\tilde\zeta$, we get 
\begin{align*}
\partial_{\tilde\zeta}\phi_k=|\eta|\bigg(\tilde z+\frac{\tilde\zeta}{(1+\tilde{\zeta}^2+\gamma)^{1/2}}\bigg).
\end{align*}
Because $\eta$ is bounded from below by the constant $c_0$, the variable $\tilde\zeta=\zeta/\vert\eta\vert$ is also bounded, therefore we have $\Big|\frac{\tilde\zeta}{(1+\tilde{\zeta}^2+\gamma)^{1/2}}\Big|\leq 1-2\delta_1$, for some small $\delta_1>0$. Then if $|\tilde z|\geq 1-\delta_1$, the contribution of the $\tilde\zeta$-integration is $\mathcal{O}_{C^\infty}((h/t)^\infty)$ by standard integration by parts. Thus we may assume that $|\tilde z|\leq 1-\delta_1$. 

In this case, the phase function $\phi_k$ features a unique critical point on the support of $\chi_0$. It is given by $\tilde{\zeta}_c=-\frac{\tilde z(1+\gamma)^{1/2}}{\sqrt{1-\tilde z^2}}$, and this critical point is strictly non-degenerate since the second derivative under normal flux balances as:
\[
\partial^2_{\tilde\zeta}\phi_k=|\eta|\bigg(\frac{1+\gamma}{(1+\tilde{\zeta}^2+\gamma)^{3/2}}\bigg)>0.
\]
Then we obtain by the stationary phase method (as $|\tilde z|<1-\delta_1$)
\[
J=\bigg(\frac{h}{t}\bigg)^{1/2} e^{i\frac{t}{h}|\eta|\sqrt{1-\tilde z^2}(1+\gamma)^{1/2}}\tilde\chi_0,
\]
where $\tilde\chi_0$ is a classical symbol of order $0$ with respect to the small parameter $h/t$.

Hence, by collecting the evaluated integral, we extract the structural representation:
\begin{align}\label{eq:Gak-neumann}
\mathcal{G}_{a,c_0}^N(t,x,y,z)&=\frac{1}{4\pi^2h^2}\bigg(\frac{h}{t}\bigg)^{1/2}\sum_{k\geq1}\int e^{\frac{i}{h}\big(y\eta+|\eta|t\sqrt{1-\tilde z^2}(1+\gamma)^{1/2}\big)}
e_k(x,\eta/h)e_k(a,\eta/h)
\nonumber\\&\quad \times\tilde\chi_0\psi_0(\eta)\chi_1(\gamma|\eta|^{2})(1-\chi_1)(\varepsilon\gamma h^{-2/3}|\eta|^{2/3}) \, d\eta.
\end{align}

Next, we observe that the operator $\mathcal{G}_{a,c_0}^N$ contains Neumann Airy functions whose localized decay properties depend heavily on the magnitude of the derivative roots index $k$. To track this microlocal variance, we partition the summation over $k$ into two distinct pieces:
\begin{equation}
\mathcal{G}_{a,c_0}^N = \mathcal{G}_{a,<L}^N + \mathcal{G}_{a,\geq L}^N,
\end{equation}
where $\mathcal{G}_{a,<L}^N$ isolates the low-eigenmode partial sum restricted to the finite index block $1 \leq k \leq L$.

\paragraph{Estimates for the Low-Eigenmode Block $\mathcal{G}_{a,<L}^N$}
To establish uniform bounds for the finite mode cluster $\mathcal{G}_{a,<L}^N$, we utilize the following lemma under the normal flux constraint. This estimate generalizes the standard pointwise envelope bound $|\text{Ai}(s)| \leq C(1+|s|)^{-1/4}$ to discrete spectral derivative sequences.

\begin{lemma}[Analogue of Lemma 3.5 \cite{ivanovici2014dispersion}]\label{lem:k-low-modes}
There exists a uniform constant $C_0 > 0$ such that for any finite truncation parameter $L \geq 1$, the following bound holds identically over the sequence of derivative-adapted Airy roots:
\begin{equation}
\sup_{\mathbf{b} \in \mathbb{R}} \left( \sum_{1 \leq k \leq L} k^{-1/3} \text{Ai}^2(\mathbf{b}-\omega'_k) \right) \leq C_0 L^{1/3}.
\end{equation}
\end{lemma}

Applying the Cauchy--Schwarz inequality to the reduced Neumann parametrix representation \eqref{eq:Gak-neumann} in conjunction with Lemma \ref{lem:k-low-modes}, we find:
\begin{equation}
\begin{split}
\|\mathcal{G}_{a,<L}^N\|_{L^\infty(\Omega)} &\lesssim h^{-2} \left(\frac{h}{t}\right)^{1/2} \sum_{1 \leq k \leq L} h^{-2/3} k^{-1/3} \left| \text{Ai}\left(h^{-2/3}|\eta|^{2/3}x-\omega'_k\right) \text{Ai}\left(h^{-2/3}|\eta|^{2/3}a-\omega'_k\right) \right| \\
&\lesssim h^{-3} \left(\frac{h}{t}\right)^{1/2} h^{1/3} \left( \sum_{1 \leq k \leq L} k^{-1/3} \text{Ai}^2\left(h^{-2/3}|\eta|^{2/3}x-\omega'_k\right) \right)^{1/2} \\
&\quad \times \left( \sum_{1 \leq k \leq L} k^{-1/3} \text{Ai}^2\left(h^{-2/3}|\eta|^{2/3}a-\omega'_k\right) \right)^{1/2} \\
&\lesssim h^{-3} \left(\frac{h}{t}\right)^{1/2} h^{1/3} L^{1/3}.
\end{split}
\end{equation}

It suffices to verify the targeted target bound \eqref{eq:kdisp0-neumann} for the non-trivial time scale $t > h$. Let $\epsilon \in (0,1/3)$ and select the finite frequency mode threshold to be exactly $L = h^{-\epsilon}$. In the short-time regime where $t \leq h^{\epsilon}$, the cutoff condition satisfies $L \leq \frac{1}{t}$, which immediately provides the optimal power scaling:
\begin{equation}
\|\mathcal{G}_{a,<L}^N(t,x,y,z)\|_{L^\infty(\Omega)} \leq C h^{-3} \left(\frac{h}{t}\right)^{5/6}.
\end{equation}

Consequently, the problem reduces to analyzing the long-time regime where $t > h^{\epsilon} \geq h^{1/3}$. In this zone, we apply the stationary phase method directly to the localized $\eta$-integration, which takes the form:
\begin{equation}
\int_{\mathbb{R}} e^{\frac{i}{h}\Phi_k(\eta)} \text{Ai}\left(h^{-2/3}|\eta|^{2/3}x-\omega'_k\right) \text{Ai}\left(h^{-2/3}|\eta|^{2/3}a-\omega'_k\right) \, d\eta,
\end{equation}
where the modulated Neumann phase function is defined by
\begin{equation}
\Phi_k(\eta) = \eta\left( y + t\sqrt{1-\tilde{z}^2}(1+\gamma)^{1/2} \right).
\end{equation}

To evaluate this highly oscillatory integral, we perform a scale transformation $\Phi_k=h\lambda\Psi_k$, where $\lambda =t\omega'_k h^{-1/3}$ serves as the large asymptotic parameter. The second derivative satisfies the uniform non-degeneracy condition $|\partial_\eta^2\Psi_k|\geq c>0$. To validate the application of the stationary phase method, we verify that for some small parameter $\nu>0$:
\[
\vert \partial_\eta^j \Ai(h^{-2/3}|\eta|^{2/3}x-\omega'_k)\vert \leq C_j \lambda^{j(1/2-\nu)}.
\]
Since the standard Airy derivative derivative bounds give $\sup_{b\geq 0}\vert b^l\Ai^{(l)}(b-\omega'_k)\vert\leq C_l(\omega'_k)^{3l/2}$, it suffices to guarantee that there exists a threshold $\epsilon>0$ such that for $t>h^{\epsilon}$ and $k\leq h^{-\epsilon}$:
\[
(\omega'_k)^{3/2}\leq  (t\omega'_kh^{-1/3})^{(1/2-\nu)}.
\] 
This condition holds strictly true as long as $\epsilon<1/7$. Therefore, compiling the estimates for $\epsilon<1/7$ and $t>h^{\epsilon}$ yields:
{\allowdisplaybreaks
\begin{align*}
\lVert\mathds{1}_{x\leq a}\mathcal{G}_{a,<L}^N(t,x,y,z)\rVert_{L^\infty}&\leq Ch^{-3}\bigg(\frac{h}{t}\bigg)^{1/2}\bigg[h^{1/3}\sum_{1\leq k\leq h^{-\epsilon}} k^{-1/3} \lambda^{-1/2}\bigg],\\
&\leq Ch^{-3}\bigg(\frac{h}{t}\bigg)^{1/2}\bigg[h^{1/3}\sum_{1\leq k\leq h^{-\epsilon}} k^{-1/3} (t\omega'_kh^{-1/3})^{-1/2}\bigg],\\
&\leq Ch^{-3}\bigg(\frac{h}{t}\bigg)^{1/2}\bigg[\bigg(\frac{h}{t}\bigg)^{1/2}h^{-\epsilon/3}\bigg],\\
&\leq C h^{-3}\bigg(\frac{h}{t}\bigg)^{1/2}\bigg(\frac{h}{t}\bigg)^{1/3}.
\end{align*}}
\noindent\underline{\bf{Estimates for $\mathcal{G}_{a,>L}^N$}}\\
We now deal with large values of $k$, localized to the index threshold $L\leq k\leq \varepsilon/h$ with $L\geq D\max \{h^{-\epsilon}, 1/t\}$, where $D>0$ is a sufficiently large constant. We are left to prove that the localized bound \eqref{eq:kdisp0-neumann} holds true for $\mathcal{G}_{a,>L}^N$, defined by the partial sum over $L\leq k\leq\frac{\varepsilon}{h}$. For the high-frequency regime where $k>Dh^{-\epsilon}$ and spatial layers tracking $0\leq x\leq a\leq h^{\frac{2}{3}(1-\epsilon)}$, the derivative roots satisfy:
\[
\omega'_k-|\eta|^{2/3}h^{-2/3}x>\omega'_k/2.
\]
Therefore, we can validly deploy the sharp asymptotic expansion of the Airy function (see Appendix):
\[
\Ai(\vartheta)=\sum_{\pm}\omega^{\pm}e^{\mp\frac{2}{3}i(-\vartheta)^{3/2}}(-\vartheta)^{-1/4}\Psi_{\pm}(-\vartheta) \quad \text{for } -\vartheta>1, \quad \text{where } \omega^{\pm}=e^{\pm i\pi/4},
\]
and where the classical symbols $\Psi_{\pm}$ are detailed in the Appendix. 

By the explicit construction of the Neumann eigenfunctions $e_k$, we expand the spatial trace according to:
\begin{align*}
e_k(x,\eta/h)&=f_k\frac{|\eta|^{1/3}h^{-1/3}}{k^{1/6}}\Ai(h^{-2/3}|\eta|^{2/3}x-\omega'_k),\\
&=f_k\frac{|\eta|^{1/3}h^{-1/3}}{k^{1/6}}\sum_{\pm}\omega^{\pm}e^{\mp\frac{2}{3}i(\omega'_k-|\eta|^{2/3}h^{-2/3}x)^{3/2}}
\frac{\Psi_{\pm}(\omega'_k-|\eta|^{2/3}h^{-2/3}x)}{(\omega'_k-|\eta|^{2/3}h^{-2/3}x)^{1/4}}.
\end{align*}
Consequently, we rewrite the high-frequency operator $\mathcal{G}_{a,> L}^N$ as the following multivariable parametrix:
\begin{align}\label{eq:pm-neumann}
\mathcal{G}_{a,>L}^N(t,x,y,z)=\sum_{L\leq k\leq\frac{\varepsilon}{h}}\frac{1}{4\pi^2h^2}\bigg(\frac{h}{t}\bigg)^{1/2}\sum_{\pm,\pm}\int e^{\frac{i}{h}\Phi_k^{\pm,\pm}}\sigma_k^{\pm,\pm} \, d\eta,
\end{align}
where the corresponding modulated Neumann phase functions are defined by:
\begin{align*}
\Phi_k^{\pm,\pm}(t,x,y,z,a;\eta) &= y\eta+|\eta|t\sqrt{1-\tilde z^2}(1+\gamma)^{1/2} \\
&\quad \pm \frac{2}{3}\left| |\eta|^{2/3}\omega'_k - |\eta|^{2/3}x \right|^{3/2} \pm \frac{2}{3}\left| |\eta|^{2/3}\omega'_k - |\eta|^{2/3}a \right|^{3/2},
\end{align*}
which can be expressed via the standard scaling variable $\gamma = h^{2/3}\omega'_k|\eta|^{-2/3}$ as:
\begin{align*}
\Phi_k^{\pm,\pm}(t,x,y,z,a;\eta) = y\eta+|\eta|t\sqrt{1-\tilde z^2}(1+\gamma)^{1/2}\pm \frac{2}{3}|\eta|(\gamma-x)^{3/2}\pm\frac{2}{3}|\eta|(\gamma-a)^{3/2}.
\end{align*}
The amplitude symbols under normal flux conditions are given explicitly by:
\begin{align*}
\sigma_k^{\pm,\pm}(x,a,h;\eta)&=h^{-1/3}|\eta|^{1/3}\tilde\chi_0\chi_1(\gamma\eta^2)(1-\chi_1)(\varepsilon\gamma h^{-2/3}|\eta|^{2/3})\frac{f_k^2}{k^{1/3}}\omega^{\pm}\omega^{\pm}\\&\times(\gamma-x)^{-1/4}(\gamma-a)^{-1/4}
\Psi_{\pm}(|\eta|^{2/3}h^{-2/3}(\gamma-x))\Psi_{\pm}(|\eta|^{2/3}h^{-2/3}(\gamma-a)).
\end{align*}

To compute the uniform bounds on these spectral blocks, we exploit the vector field identity $3\eta\partial_{\eta} = -2\gamma\partial_{\gamma}$. For spatial configurations satisfying $0 \leq x \leq a \leq 2\gamma$, the amplitude functions satisfy the regular symbolic decay profile:
\begin{equation}
\left| (\gamma\partial_{\gamma})^{j} \left((\gamma-x)^{-1/4}\right) \right| \leq C_j\gamma^{-1/4} \leq C'_j (hk)^{-1/6}.
\end{equation}
Moreover, the components $\Psi_{\pm}$ behave as classical symbols of order $0$ at infinity. This symbolic structure holds true because our localized high-frequency tracking ensures that the internal arguments remain uniformly large:
\begin{equation}
\left| \eta^{2/3}h^{-2/3}(\gamma-x) \right| \geq \frac{1}{2}\omega'_k \geq C h^{-2\epsilon/3},
\end{equation}
owing to the index constraint $k \geq L \geq h^{-\epsilon}$. Hence, by distributing the derivative bounds across the product terms via the Leibniz rule, we conclude that for every multi-index $j \geq 0$, there exists a constant $C_j > 0$ such that the total amplitude obeys the strict bound:
\begin{equation}
\left| \partial_{\eta}^j \sigma_k^{\pm,\pm}(x,a,h;\eta) \right| \leq C_j(hk)^{-2/3}.
\end{equation}
This sharp symbolic decay stems directly from combining the base pre-factor $(hk)^{-1/3}$ present within the structural definition of $\sigma_k^{\pm,\pm}$ with the additional $(hk)^{-1/3}$ factor generated whenever the $\eta$-derivatives act directly upon the singular weight components $(\gamma-x)^{-1/4}(\gamma-a)^{-1/4}$.

Consequently, to establish the localized dispersive estimate for the high-eigenmode block $\mathcal{G}_{a,> L}^N$, it suffices to bound the oscillatory integral with respect to the tangential frequency variable:
\begin{equation}
I_k^{\pm,\pm} = \int_{\mathbb{R}} e^{\frac{i}{h}\Phi_k^{\pm,\pm}}\sigma_k^{\pm,\pm} \, d\eta.
\end{equation}
To evaluate this integral, we rescale the phase function according to $\Phi_k^{\pm,\pm} = h\lambda\psi_k^{\pm,\pm}$, where $\lambda = t\omega'_k h^{-1/3}$ acts as the large asymptotic parameter. This parameter remains bounded away from zero since $\lambda \geq c > 0$ follows from the Weyl-type root distribution $\omega'_k \sim k^{2/3}$ combined with the cutoffs $k \geq 1/t$ and $t \geq h$. The following proposition details the precise decay estimates for these oscillatory layers under the normal flux constraints.

\begin{proposition}\label{prop:oscillatory-high-neumann}
Let $\epsilon \in (0,1/7)$. For a sufficiently small localization parameter $\varepsilon > 0$, there exists a uniform constant $C > 0$ independent of the source position $a \in (0,h^{\frac{2}{3}(1-\epsilon)}]$, the time step $t \in [h,1]$, the spatial observation points $x \in [0,a]$, and the spectral indices $k \in [L,\frac{\varepsilon}{h}]$, such that the following uniform bound holds for all $y, z \in \mathbb{R}$:
\begin{equation}
\left| \int_{\mathbb{R}} e^{i\lambda\psi_k^{\pm,\pm}}\sigma_k^{\pm,\pm} \, d\eta \right| \leq C(hk)^{-2/3}\lambda^{-1/3}.
\end{equation}
\end{proposition}

\begin{proof}[Proof of Proposition \ref{prop:oscillatory-high-neumann}]
Since $(hk)^{2/3}\sigma_k^{\pm,\pm}$ are classical symbols of degree $0$ compactly supported in $\eta$, we apply the stationary phase method to an integral of the form
\[
J_1=\int e^{i\lambda\psi_k^{\pm,\pm}}(hk)^{2/3}\sigma_k^{\pm,\pm} \, d\eta.
\]
We have to prove that the following inequality holds uniformly with respect to the parameters:
\[
|J_1|\leq C\lambda^{-1/3}.
\]
Let us recall that the Neumann-adapted phase satisfies:
\begin{align*}
h\lambda\psi_k^{\pm,\pm}(t,x,y,z;\eta)=y\eta+|\eta|t\sqrt{1-\tilde z^2}(1+\gamma)^{1/2}\pm \frac{2}{3}|\eta|(\gamma-x)^{3/2}\pm\frac{2}{3}|\eta|(\gamma-a)^{3/2}.
\end{align*}
We compute the first derivative with respect to $\eta$:
\[
h\lambda\partial_{\eta}\psi_k^{\pm,\pm}=y+t\sqrt{1-\tilde z^2}\frac{1+\frac{2}{3}\gamma}{\sqrt{1+\gamma}}\pm\frac{2}{3}x(\gamma-x)^{1/2}\pm\frac{2}{3}a(\gamma-a)^{1/2},
\]
and we need to consider four cases. Let $\delta=\frac{x}{a}\in [0,1]$ and \(\alpha=\frac{a}{\omega'_k h^{2/3}}\in [0,\alpha_0]\). Indeed, since \(\omega'_k\sim k ^{2/3}\), \(k\geq D h^{-\epsilon}\), and \(a\leq h^{\frac{2}{3}(1-\epsilon)}\), it follows that \(\alpha=ak^{-2/3}h^{-2/3}\leq D^{-2/3}ah^{-\frac{2}{3}(1-\epsilon)}\leq D^{-2/3}:=\alpha_0\). Let \(\rho=|\eta|^{-2/3}\), \(V=\frac{y+t\sqrt{1-\tilde z^2}}{t\omega'_kh^{2/3}}\), and define the function $F(\gamma)$ by:
\[
\frac{1+\frac{2}{3}\gamma}{\sqrt{1+\gamma}}=1+\gamma F(\gamma),\quad F(\gamma)=\frac{1}{6}+\frac{\gamma}{24}+\mathcal{O}(\gamma^2).
\]
With these notations, we get:
\[
\partial_{\eta}\psi_k^{\pm,\pm}=V+\sqrt{1-\tilde z^2}\rho F(h^{2/3}\omega'_k\rho)+\frac{2}{3}\mu\left(\pm\delta(\rho-\delta\alpha)^{1/2}\pm(\rho-\alpha)^{1/2}\right),
\]
where \(\mu=\frac{ah^{-1/3}}{t(\omega'_k)^{1/2}}\); it satisfies \(0\leq\mu\leq\frac{h^{\frac{1}{3}(1-\epsilon)}}{t}\min\{1,h^{-\epsilon/3}t^{1/3}\}\) and thus $\mu$ may be small or arbitrarily large. In fact, if $t\geq h^{\epsilon}$, \(\mu\leq h^{\frac{1}{3}(1-\epsilon)}t^{-1}\leq h^{1/3-4\epsilon/3}\), which is small if $\epsilon\leq 1/4$. If $t\leq h^{\epsilon}$, we have \(\mu\leq h^{1/3-2\epsilon/3}t^{-2/3}\) which could be large when \(t\leq h^{1/2-\epsilon}\).\\

First, we consider the case where $\mu$ is bounded. We now study the critical points. Taking $\rho=|\eta|^{-2/3}$ as our integration variable, we get:
\begin{align*}
\partial_\rho\partial_{\eta}\psi_k^{\pm,\pm}&=\sqrt{1-\tilde z^2}(F(\gamma)+\gamma F'(\gamma))+\frac{\mu}{3}\big(\pm\delta(\rho-\delta\alpha)^{-1/2}\pm(\rho-\alpha)^{-1/2}\big),\\
\partial_\rho^2\partial_{\eta}\psi_k^{\pm,\pm}&=\sqrt{1-\tilde z^2}h^{2/3}\omega'_k(2F'(\gamma)+\gamma F''(\gamma))-\frac{\mu}{6}\big(\pm\delta(\rho-\delta\alpha)^{-3/2}\pm(\rho-\alpha)^{-3/2}\big).
\end{align*}
For $\varepsilon$ small enough, there exists $c>0$ independent of \(k\leq\frac{\varepsilon}{h}\) such that:
\begin{align}\label{eq:3-neumann}
|\partial_\rho\partial_{\eta}\psi_k^{\pm,\pm}|+| \partial_\rho^2\partial_{\eta}\psi_k^{\pm,\pm} | \geq c.
\end{align}
Indeed, we observe that $(\rho-\alpha)^{-1/2}\geq \delta(\rho-\delta\alpha)^{-1/2}$ and $F(\gamma)+\gamma F'(\gamma)\sim \frac{1}{6}$. Thus we get $|\partial_\rho\partial_{\eta}\psi_k^{\pm,+}|\geq c_1>0$. In other cases, $\partial_\rho\partial_{\eta}\psi_k^{\pm,-}$ could vanish, and when this happens we have:
\[
|\partial_\rho\partial_{\eta}\psi_k^{\pm,-}|\leq 1/100 \Longrightarrow \frac{\mu}{3}(\rho-\alpha)^{-1/2}\geq 0.05.
\]
Then we have $|\partial_\rho^2\partial_{\eta}\psi_k^{-,-}|\geq c_2>0$. Moreover, for any smooth function $f$, we have the integral relation:
\begin{align}\label{eq:int-neumann}
f(\rho-\alpha)-\delta f(\rho-\delta\alpha)=(1-\delta) f(\rho-\delta\alpha)-\int _0^{\alpha(1-\delta)}f'(\rho-\delta\alpha-t) \, dt.
\end{align}
Taking $f(t)=t^{-1/2}$, we find that:
\[
|\partial_\rho\partial_{\eta}\psi_k^{+,-}|\leq 1/100 \Longrightarrow \mu(1-\delta)\geq c>0.
\]
Applying \eqref{eq:int-neumann} with $f(t)=t^{-3/2}$, we obtain $|\partial_\rho^2\partial_{\eta}\psi_k^{+,-}|\geq c/2>0$. As a consequence of \eqref{eq:3-neumann} together with Lemma 2.20 in Ivanovici--Lebeau--Planchon \cite{ivanovici2014dispersion}, we conclude that the proposition holds true for $\mu$ bounded.\\

It remains to study the case where $\mu$ is large. For the $(+,+)$ or $(-,+)$ configurations, we analyze the critical points using $\Lambda=\lambda\mu$ as a large parameter. Since $\delta(\rho-\delta\alpha)^{-1/2}+(\rho-\alpha)^{-1/2}\geq c>0$, we have $|\partial_\rho\partial_{\eta}\psi_k^{\pm,+}|\geq c>0$. Hence, $|J_1|\leq C(\lambda\mu)^{-1/2}$.

For the $(+,-)$ and $(-,-)$ configurations, we can use the identity \eqref{eq:int-neumann}. We distinguish between two cases: if $\mu(1-\delta)$ is bounded, computing the derivatives of the phase functions $\psi_k^{\pm,-}$ establishes the inequality \eqref{eq:3-neumann}, and the conclusion follows from Lemma 2.20 in \cite{ivanovici2014dispersion}. If $\mu(1-\delta)$ is large, we take $\Lambda'=\lambda\mu(1-\delta)$ as the large asymptotic parameter in $J_1$. Since by $\eqref{eq:int-neumann}$ we have:
\[
|(\rho-\alpha)^{-1/2}-\delta(\rho-\delta\alpha)^{-1/2}|\geq c(1-\delta)
\]
with $c>0$, we obtain $|\partial_\rho\partial_{\eta}\psi_k^{\pm,-}|\geq c>0$ and hence $|J_1|\leq C(\lambda\mu(1-\delta))^{-1/2}$.
\end{proof}
\noindent To summarize, Proposition \ref{prop:oscillatory-high-neumann} yields the sharp dispersive estimates for the high-frequency operator $\mathcal{G}_{a,>L}^N$ across large values of $k$ in the range $L\leq k\leq\varepsilon/h$ as follows:
\begin{align*}
\lVert\mathds{1}_{x\leq a}\mathcal{G}_{a,>L}^N(t,x,y,z)\rVert_{L^\infty}&\leq Ch^{-2}\bigg(\frac{h}{t}\bigg)^{1/2}\sum_{L\leq k\leq\frac{\varepsilon}{h}}(hk)^{-2/3}\lambda^{-1/3}\\
&\leq Ch^{-2}\bigg(\frac{h}{t}\bigg)^{1/2}\sum_{L\leq k\leq\frac{\varepsilon}{h}}(hk)^{-2/3}(t\omega'_kh^{-1/3})^{-1/3}\\
&\leq Ch^{-2}\bigg(\frac{h}{t}\bigg)^{1/2}\sum_{L\leq k\leq\frac{\varepsilon}{h}}(hk)^{-2/3}t^{-1/3}k^{-2/9}h^{1/9}\\
&\leq Ch^{-3}\bigg(\frac{h}{t}\bigg)^{1/2}\left(\frac{h}{t}\right)^{1/3}h^{1/9}\Bigg(\sum_{L\leq k\leq\frac{\varepsilon}{h}}k^{-8/9}\Bigg)\\
&\leq Ch^{-3}\bigg(\frac{h}{t}\bigg)^{5/6},
\end{align*}
where we exploited the relation $\lambda=t\omega'_kh^{-1/3}$ in the second line, and the asymptotic root distribution $\omega'_k\sim k^{2/3}$ in the third and fourth lines. Because the exponent of the remaining modes satisfies $-8/9 > -1$, the partial sum over $k$ converges uniformly to a bounded constant proportional to $(\varepsilon/h)^{1/9}$, which naturally absorbs the remaining fractional $h^{1/9}$ factor. This successfully concludes the proof of Proposition \ref{prop:k-neumann}.
\end{proof}

\subsection{Airy-Poisson Summation Formula.}\label{sec:23}

Let $A_{\pm}(z)=e^{\mp i\pi/3}\Ai(e^{\mp i\pi/3}z)$, where we have the standard derivative identity $\Ai'(-z) = -\left(A'_+(z) + A'_-(z)\right)$. To accommodate the normal flux constraint at the boundary, we construct a derivative-adapted complex phase map. For $\omega\in\mathbb{R}$, we define:
\[
L(\omega)=\pi+i\log\bigg(\frac{A'_-(\omega)}{A'_+(\omega)}\bigg).
\]
Mirroring the structural properties found in Ivanovici--Lebeau--Planchon \cite{ivanovici2014dispersion}, the function $L$ is analytic, strictly increasing, and satisfies the global limits:
\[
L(0)=\pi/3, \quad \lim_{\omega\rightarrow-\infty}L(\omega)=0, \quad L(\omega)= \frac{4}{3}\omega^{3/2}-B_N\left(\omega^{3/2}\right), \quad \text{for }\omega\geq 1,
\]
where the asymptotic phase correction term behaves under expansion as:
\begin{align}\label{BAiry-neumann}
B_N(\omega)\sim_{1/\omega}\sum_{j\geq 1}\tilde{b}_j\omega^{-j},\quad \tilde{b}_j\in\mathbb{R},\quad \tilde{b}_1>0.
\end{align}
Furthermore, for all index levels $k\geq 1$, the phase map strictly locks onto the roots of the first derivative of the Airy function:
\[
L(\omega'_k)=2\pi k \iff \Ai'(-\omega'_k)=0, \quad L'(\omega'_k)=2\pi\int_0^\infty \Ai^2(x-\omega'_k) \, dx.
\]
Recall that $f_k$ are the normalization constants such that $\lVert e_k(\cdot,\eta)\rVert_{L^2(\mathbb{R}_+)}=1$. This gives us the matching inner-product identity:
\[
\int_0^\infty \Ai^2(x-\omega'_k) \, dx=\frac{k^{1/3}}{f_k^2}=\frac{L'(\omega'_k)}{2\pi}.
\]
The next lemma provides the mathematical machinery required to rigorously transform the discrete spectral sum over the Neumann eigenmodes $k$ into a geometric wave packet sum tracking boundary reflection iterations indexed by $N\in\mathbb{Z}$.

\begin{lemma}[Derivative Airy-Poisson Summation Formula]\label{lem:poisson}
The following distributional identity holds true in the space of test distributions $\mathcal{D}'(\mathbb{R}_\omega)$:
\[
\sum_{N\in\mathbb{Z}}e^{-iNL(\omega)}=2\pi \sum_{k\in\mathbb{N}^*}\frac{1}{L'(\omega'_k)}\delta_{\omega=\omega'_k}.
\]
That is, for any compactly supported smooth test function $\phi(\omega)\in C_0^\infty(\mathbb{R})$, we have:
\[
\sum_{N\in\mathbb{Z}}\int_{\mathbb{R}} e^{-iNL(\omega)}\phi(\omega) \, d\omega=2\pi\sum_{k\in\mathbb{N}^*}\frac{1}{L'(\omega'_k)}\phi(\omega'_k).
\]
\end{lemma}

\begin{proof}
Let $\phi \in C_0^\infty(\mathbb{R})$ be an arbitrary test function with compact support. We seek to evaluate the distributional action of the formal sum of modulated exponential phases, defined by:
\begin{equation}\label{eq:proof-target}
I := \sum_{N\in\mathbb{Z}} \langle e^{-iNL(\omega)}, \phi(\omega) \rangle = \sum_{N\in\mathbb{Z}}\int_{\mathbb{R}} e^{-iNL(\omega)}\phi(\omega) \, d\omega.
\end{equation}
By definition, the derivative-adapted phase map $L(\omega) = \pi + i \log\left(\frac{A'_-(\omega)}{A'_+(\omega)}\right)$ is real-valued, real-analytic, and strictly increasing on $\mathbb{R}$. Consequently, $L$ defines a global diffeomorphism from $\mathbb{R}$ onto its image $L(\mathbb{R}) = (0, \infty)$. 

We perform a change of variables in the integral \eqref{eq:proof-target} by setting $\theta = L(\omega)$. Since $L$ is a diffeomorphism, the inverse map $\omega = L^{-1}(\theta)$ is well-defined and analytic. The differential transforms via the Jacobian as $d\omega = \frac{1}{L'(L^{-1}(\theta))} \, d\theta$. Substituting these into the integral yields:
\begin{equation}\label{eq:proof-cov}
\int_{\mathbb{R}} e^{-iNL(\omega)}\phi(\omega) \, d\omega = \int_{0}^{\infty} e^{-iN\theta} \left( \frac{\phi(L^{-1}(\theta))}{L'(L^{-1}(\theta))} \right) \, d\theta.
\end{equation}
Let us define a newly localized test function $\Phi(\theta)$ on $\mathbb{R}$ by:
\begin{equation}
\Phi(\theta) := \begin{cases} 
\frac{\phi(L^{-1}(\theta))}{L'(L^{-1}(\theta))} & \text{if } \theta > 0, \\
0 & \text{if } \theta \leq 0.
\end{cases}
\end{equation}
Since $\phi \in C_0^\infty(\mathbb{R})$ has compact support nested within $\mathbb{R}$, and $\lim_{\omega \to -\infty} L(\omega) = 0$, the support of $\Phi(\theta)$ is a compact subset strictly contained within the open interval $(0, \infty)$. Hence, $\Phi \in C_0^\infty(\mathbb{R})$. 

We can rewrite the formal summation over $N \in \mathbb{Z}$ in \eqref{eq:proof-target} using the pullback function $\Phi(\theta)$:
\begin{equation}\label{eq:proof-sum-theta}
I = \sum_{N\in\mathbb{Z}} \int_{\mathbb{R}} e^{-iN\theta} \Phi(\theta) \, d\theta.
\end{equation}
The integral in \eqref{eq:proof-sum-theta} is exactly the standard continuous Fourier transform of $\Phi$, evaluated at the discrete integer lattice values $N$, denoted by $\widehat{\Phi}(N) = \int_{\mathbb{R}} \Phi(\theta) e^{-iN\theta} \, d\theta$. Thus, $I = \sum_{N\in\mathbb{Z}} \widehat{\Phi}(N)$.

By applying the classical Poisson Summation Formula on $\mathbb{R}$ for smooth, compactly supported functions, the sum of a function's Fourier transform over the integer lattice equals the sum of the function itself shifted over the scaled lattice:
\begin{equation}\label{eq:classical-poisson}
\sum_{N\in\mathbb{Z}} \widehat{\Phi}(N) = 2\pi \sum_{k\in\mathbb{Z}} \Phi(2\pi k).
\end{equation}
We now analyze the discrete support constraints of the right-hand side of \eqref{eq:classical-poisson}. Since $\text{supp}(\Phi) \subset (0, \infty)$, the terms matching $k \leq 0$ vanish identically: $\Phi(2\pi k) = 0$ for all $k \in \mathbb{Z}_{\leq 0}$. Therefore, the summation is restricted strictly to the positive integers $k \in \mathbb{N}^*$:
\begin{equation}\label{eq:poisson-eval}
I = 2\pi \sum_{k\in\mathbb{N}^*} \Phi(2\pi k) = 2\pi \sum_{k\in\mathbb{N}^*} \frac{\phi\left(L^{-1}(2\pi k)\right)}{L'\left(L^{-1}(2\pi k)\right)}.
\end{equation}
Recall that by the construction of the phase function $L(\omega)$, the ordered sequence of roots $(\omega'_k)_{k \geq 1}$ matching the Neumann boundary condition $\text{Ai}'(-\omega'_k) = 0$ precisely satisfies the quantization relation:
\begin{equation}
L(\omega'_k) = 2\pi k \iff L^{-1}(2\pi k) = \omega'_k.
\end{equation}
Substituting $\omega'_k$ back into equation \eqref{eq:poisson-eval} replaces the pullback variable, yielding:
\begin{equation}\label{eq:proof-final}
I = 2\pi \sum_{k\in\mathbb{N}^*} \frac{1}{L'(\omega'_k)} \phi(\omega'_k).
\end{equation}
Rewriting equation \eqref{eq:proof-final} in the pairing notation of distribution theory yields:
\begin{equation}
\left\langle \sum_{N\in\mathbb{Z}} e^{-iNL(\omega)}, \phi(\omega) \right\rangle = \left\langle 2\pi \sum_{k\in\mathbb{N}^*} \frac{1}{L'(\omega'_k)} \delta_{\omega=\omega'_k}, \phi(\omega) \right\rangle.
\end{equation}
Since this identity holds for any arbitrary test function $\phi \in C_0^\infty(\mathbb{R})$, the distributional equality is established in $\mathcal{D}'(\mathbb{R}_\omega)$.
\end{proof}

Now we rewrite the microlocal parametrix \eqref{eq:kparametrix-neumann} using the definition of the Neumann eigenfunctions $e_k$ and replace the normalization coefficient quotient $\frac{|f_k|^2}{k^{1/3}}$ by its derivative-adapted phase equivalent $\frac{2\pi}{L'(\omega'_k)}$. We obtain:
\begin{align*}
\mathcal{G}_{a,c_0}^N(t,x,y,z) &= \frac{1}{(2\pi)^2h^{8/3}}\int_{\mathbb{R}^2} e^{\frac{i}{h}(y\eta+z\zeta)}\sum_{k\geq1}\frac{|f_k|^2}{k^{1/3}} e^{i\frac{t}{h}\left(\eta^2+\zeta^2+\omega'_kh^{2/3}|\eta|^{4/3}\right)^{1/2}}|\eta|^{2/3}\\
&\qquad \times \chi_0(\eta^2+\zeta^2)\psi_0(\eta)\chi_1\left(\omega'_k h^{2/3}|\eta|^{4/3}\right) (1-\chi_1)(\epsilon\omega'_k) \\
&\qquad \times \Ai\left(h^{-2/3}|\eta|^{2/3}x-\omega'_k\right)\Ai\left(h^{-2/3}|\eta|^{2/3}a-\omega'_k\right) \,d\eta \,d\zeta \\[1.5ex]
&= \frac{1}{(2\pi)^2h^{8/3}}\int_{\mathbb{R}^2} e^{\frac{i}{h}(y\eta+z\zeta)}\sum_{k\geq1}\frac{2\pi}{L'(\omega'_k)} e^{i\frac{t}{h}\left(\eta^2+\zeta^2+\omega'_kh^{2/3}|\eta|^{4/3}\right)^{1/2}}|\eta|^{2/3}\\
&\qquad \times \chi_0(\eta^2+\zeta^2) \psi_0(\eta)\chi_1\left(\omega'_k h^{2/3}|\eta|^{4/3}\right) (1-\chi_1)(\epsilon\omega'_k) \\
&\qquad \times \Ai\left(h^{-2/3}|\eta|^{2/3}x-\omega'_k\right)\Ai\left(h^{-2/3}|\eta|^{2/3}a-\omega'_k\right) \,d\eta \,d\zeta \\[1.5ex]
&= \frac{2\pi}{(2\pi)^2h^{8/3}}\int_{\mathbb{R}^3} e^{\frac{i}{h}(y\eta+z\zeta)} \sum_{k\geq1}\frac{\delta_{\omega=\omega'_k}}{L'(\omega'_k)} e^{i\frac{t}{h}\left(\eta^2+\zeta^2+\omega h^{2/3}|\eta|^{4/3}\right)^{1/2}}|\eta|^{2/3}\\
&\qquad \times \chi_0(\eta^2+\zeta^2) \psi_0(\eta)\chi_1\left(\omega h^{2/3}|\eta|^{4/3}\right) (1-\chi_1)(\epsilon\omega) \\
&\qquad \times \Ai\left(h^{-2/3}|\eta|^{2/3}x-\omega\right) \Ai\left(h^{-2/3}|\eta|^{2/3}a-\omega\right) \,d\omega \,d\eta \,d\zeta.
\end{align*}

Using the derivative-adapted Airy-Poisson summation formula [see Lemma \ref{lem:poisson}], the micro-local parametrix $\mathcal{G}_{a,c_0}^N$ transforms into a geometric sum tracking the discrete reflection paths:
\begin{align*}
\mathcal{G}_{a,c_0}^N(t,x,y,z) &= \frac{1}{(2\pi)^2h^{8/3}}\int_{\mathbb{R}^3} e^{\frac{i}{h}(y\eta+z\zeta)}\sum_{N\in\mathbb{Z}}e^{-iNL(\omega)} e^{i\frac{t}{h}\left(\eta^2+\zeta^2+\omega h^{2/3}|\eta|^{4/3}\right)^{1/2}}|\eta|^{2/3}\chi_0(\zeta^2+\eta^2) \\
&\qquad \times \psi_0(\eta)\chi_1\left(\omega h^{2/3}|\eta|^{4/3}\right) (1-\chi_1)(\varepsilon\omega) \\
&\qquad \times \Ai\left(h^{-2/3}|\eta|^{2/3}x-\omega\right)\Ai\left(h^{-2/3}|\eta|^{2/3}a-\omega\right) \,d\omega \,d\eta \,d\zeta.
\end{align*}
From the structural definition of the derivative phase map $L(\omega)$ under normal flux constraints, the scattering quotient of the derivatives satisfies:
\[
\left(\frac{A'_-(\omega)}{A'_+(\omega)}\right)^N = i^N e^{-\frac{4}{3}iN\omega^{3/2}} e^{iNB_N\left(\omega^{3/2}\right)},
\] 
where for $\omega\in\mathbb{R}_{+}$, the phase correction term $B_N(\omega)\in \mathbb{R}$ is defined as in \eqref{BAiry-neumann}. Recalling that $e^{-iNL(\omega)} = (-1)^N \left(\frac{A'_-(\omega)}{A'_+(\omega)}\right)^N$, it follows that:
\begin{align}\label{eq:SUMN-neumann}
\mathcal{G}_{a,c_0}^N(t,x,y,z) &= \sum_{N\in\mathbb{Z}}\frac{(-1)^N}{(2\pi)^2h^{8/3}}\int_{\mathbb{R}^3} e^{\frac{i}{h}(y\eta+z\zeta)} e^{i\frac{t}{h}\left(\eta^2+\zeta^2+\omega h^{2/3}|\eta|^{4/3}\right)^{1/2}}|\eta|^{2/3}\chi_0(\zeta^2+\eta^2)\psi_0(\eta) \nonumber \\
&\qquad \times \chi_1\left(\omega h^{2/3}|\eta|^{4/3}\right)(1-\chi_1)(\varepsilon\omega)\left(\frac{A'_-(\omega)}{A'_+(\omega)}\right)^N \nonumber \\
&\qquad \times \Ai\left(h^{-2/3}|\eta|^{2/3}x-\omega\right)\Ai\left(h^{-2/3}|\eta|^{2/3}a-\omega\right) \,d\omega \,d\eta \,d\zeta \nonumber \\[2ex]
&= \sum_{N\in\mathbb{Z}}\frac{(-i)^N}{(2\pi)^4h^{10/3}}\int_{\mathbb{R}^5} e^{\frac{i}{h}\Phi_N(t,x,y,z,a;s,\sigma,\omega,\eta,\zeta)} \nonumber \\
&\qquad \times |\eta|^{2/3}\chi_0(\zeta^2+\eta^2)\psi_0(\eta)\chi_1\left(\omega h^{2/3}|\eta|^{4/3}\right)(1-\chi_1)(\varepsilon\omega) \,ds \,d\sigma \,d\omega \,d\eta \,d\zeta,
\end{align}
where the total oscillatory phase $\Phi_N$ tracking the $N$-th boundary reflection is defined explicitly by:
\begin{align*}
\Phi_N &= y\eta + z\zeta + t\left(\eta^2+\zeta^2+\omega h^{2/3}|\eta|^{4/3}\right)^{1/2} + \frac{s^3}{3} + s\left(|\eta|^{2/3}x-\omega h^{2/3}\right) \nonumber \\
&\quad + \frac{\sigma^3}{3} + \sigma\left(|\eta|^{2/3}a-\omega h^{2/3}\right) - \frac{4}{3}N\omega^{3/2}h + Nh B_N\left(\omega^{3/2}\right).
\end{align*}
From the first to the second line of \eqref{eq:SUMN-neumann}, we substituted the standard integral representations of the Airy profiles using the scaled change of variables $s=Sh^{-1/3}$ and $\sigma=\Sigma h^{-1/3}$; for typographic simplicity, we retain the notations $s,\sigma$ for the internal integration fields.

\noindent Therefore, \eqref{eq:SUMN-neumann} defines a rigorous local parametrix that reads as a geometric sum over the reflection indexes $N \in \mathbb{Z}$. Notice that our customized parametrix coincides structurally with the constructed sum over reflected waves in Ivanovici--Lebeau--Planchon \cite{ivanovici2014dispersion}, since each discrete term contains essentially the same geometric phase layout. In the sequel, we refer to the sum over $N\in\mathbb{Z}$ as the individual wave packet summands corresponding precisely to the number of consecutive flux reflections on the boundary of the cylinder, indexed by $N$.

\subsection{Dispersive Estimates for $a \geq h^{\frac{2}{3}(1-\epsilon')}, \, \epsilon' \in (0, \epsilon)$}

In this subsection, we establish local-in-time dispersive estimates for the multi-reflective parametrix \eqref{eq:SUMN-neumann} by evaluating it as a geometric sum over the boundary reflection index $N \in \mathbb{Z}$ in the regime where the source distance satisfies $a \geq h^{\frac{2}{3}(1-\epsilon')}$ for $\epsilon' \in (0, \epsilon)$. Recall that this localized representation is rigorously constructed by combining the frequency-localized Neumann propagator \eqref{eq:kparametrix-neumann} with the derivative-adapted Airy-Poisson summation identity derived in Lemma~\ref{lem:poisson}.

The resulting analytical structure presents itself as a sum of highly oscillatory integrals whose phase functions depend on derivative-adapted Airy-type profile parameters that display degenerate critical configurations. We provide a precise description of the associated Lagrangian submanifold embedded within the phase space under normal flux boundary constraints. This microlocal approach allows us to systematically classify, track, and absorb the multi-dimensional degeneracies of the phases when applying stationary phase approximations near the caustics.

\begin{figure}[ht]
\centering
\begin{tikzpicture}[scale=0.6, >=stealth]
  \draw[->, thick] (-13,0) -- (1,0) node[below] {$y$};
  \draw[->, thick] (-6,-1) -- (-6,7) node[above] {$x$};
  
  \fill (-6,2) circle (0.08) node[left, xshift=-2pt] {$a$};
  \draw[->, thick, blue!60!black] (-6,2) -- (-5,5.8);
  
  \node[draw=black, fill=white, thick, rounded corners, inner sep=4pt] (Box) at (-6,-2.2) 
      {\small $N=2$ Swallowtails regime};
      
  \draw[->, thick, red!70!black] (Box.north) -- (-4,1);
  \draw[->, thick, red!70!black] (Box.north) -- (-8,1);
  
  \draw (-12,0) arc (0:-90:-4);
  \draw (-10,0) arc (0:-90:-2);
  \draw (-9,1)  arc (0:-90:-1);
  \draw (-11,0) arc (0:-93:-3);
  \draw (-9,1)  arc (0:-61:-2.3);
  \draw (-10,0) arc (0:-53:-5); 
  \draw[red, thick, dashed] (-8.7,1.7) circle (0.8);

  \draw (0,0)   arc (0:180:6);
  \draw (0,0)   arc (0:90:4);
  \draw (-2,0)  arc (0:53:5); 
  \draw (-2,0)  arc (0:90:2);
  \draw (-3,1)  arc (0:90:1);
  \draw (-1,0)  arc (0:93:3);
  \draw (-3,1)  arc (0:61:2.3);
  \draw (-1,0)  arc (0:180:5);
  \draw[red, thick, dashed] (-3.3,1.7) circle (0.8);

\end{tikzpicture}
\caption{Schematic layout of wave reflections displaying focal clusters within the $N=2$ swallowtail caustics framework.}
\label{fig:sw}
\end{figure}
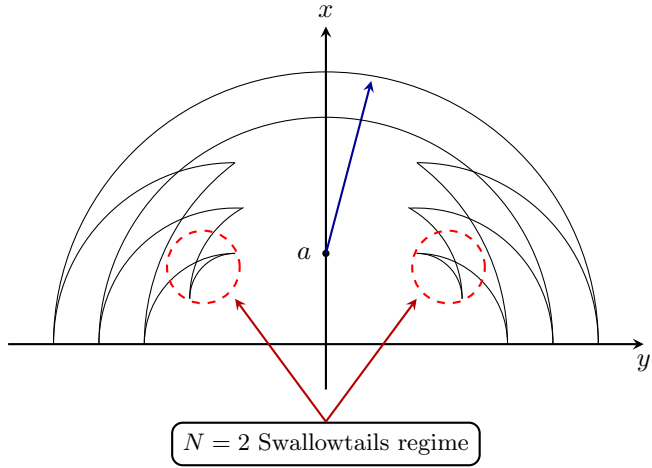

To evaluate the multi-reflective layers in \eqref{eq:SUMN-neumann}, we introduce a strategic scaling change of variables tailored to capture the anisotropic geometric boundary components of the cylinder:
\[
a\tilde{\omega}=h^{2/3}\omega|\eta|^{-2/3}, \quad x=aX, \quad \zeta=|\eta|\tilde{\zeta}, \quad s=a^{1/2}|\eta|^{1/3}\tilde{s}, \quad \sigma=a^{1/2}|\eta|^{1/3}\tilde{\sigma}.
\]
Under this spatial normalization, the micro-local multi-reflective parametrix $\mathcal{G}_{a,c_0}^N$ decomposes into a structured sum over the discrete reflection indices:
\begin{align}\label{eq:GNsum-neumann}
\mathcal{G}_{a,c_0}^N(t,x,y,z)=\sum_{N\in\mathbb{Z}}G_{a,N}^N,
\end{align}
where for each discrete boundary reflection $N\in\mathbb{Z}$, the localized operator component $G_{a,N}^N$ expands into the following five-dimensional integral:
\begin{align}\label{eq:GaN-neumann}
G_{a,N}^N(t,x,y,z) &= \frac{(-i)^Na^2}{(2\pi)^4h^4}\int_{\mathbb{R}^5} e^{\frac{i}{h}\Phi_{N,a,h}} |\eta|^{3}\chi_0\left(\eta^2(1+|\tilde{\zeta}|^2)\right)\psi_0(\eta)\chi_1(a\tilde{\omega}\eta^2) \nonumber \\
&\qquad \times (1-\chi_1)(\epsilon ah^{-2/3}|\eta|^{2/3}\tilde{\omega}) \, d\tilde{s} \, d\tilde{\sigma} \, d\tilde{\omega} \, d\tilde{\zeta} \, d\eta,
\end{align}
where the total oscillatory phase function $\Phi_{N,a,h}=\Phi_{N,a,h}(t,x,y,z;\tilde{s},\tilde{\sigma},\tilde{\omega},\tilde{\zeta},\eta)$ scales precisely as:
\begin{align*}
\Phi_{N,a,h} &= y\eta+|\eta|z\tilde{\zeta}+|\eta|t(1+\tilde{\zeta}^2+a\tilde{\omega})^{1/2} + a^{3/2}|\eta|\left( \frac{\tilde{s}^3}{3}+\tilde{s}(X-\tilde{\omega})+\frac{\tilde{\sigma}^3}{3}+\tilde{\sigma}(1-\tilde{\omega}) \right. \\
&\qquad \left. -\frac{4}{3}N\tilde{\omega}^{3/2} + \frac{h}{a^{3/2}|\eta|}N B_N\left(\tilde{\omega}^{3/2}a^{3/2}|\eta|/h\right) \right).
\end{align*}

The main result of this subsection is Theorem~\ref{thm:thmN-neumann}, which establishes a sharp, uniform decay bound for the sum over $N$ of these multi-reflective operators by systematically tracking the stationary phase asymptotics near the degenerate critical caustic paths under the normal flux boundary conditions.

\begin{theorem}\label{thm:thmN-neumann}
Let $\alpha < 2/3$. There exists a uniform constant $C > 0$ such that for all semiclassical scales $h \in (0, h_0]$, all near-boundary source positions $a \in [h^\alpha, a_0]$, all spatial coordinate configurations $X \in [0, 1]$, and all localized timelines $T \in (0, a^{-1/2}]$, the following uniform estimate holds under the cylindrical Neumann framework for all lateral variables $Y, z \in \mathbb{R}$:
\begin{align}\label{eq:estimateN-neumann}
\bigg|\sum_{0\leq N\leq C_0a^{-1/2}}G_{a,N}^N(T,X,Y,z;h)\bigg|\leq Ch^{-3}\bigg(\frac{h}{t}\bigg)^{1/2}\left( \bigg(\frac{h}{t}\bigg)^{1/2}+a^{1/8}h^{1/4} \right).
\end{align}
\end{theorem}

Notice that the primary component on the right-hand side of the estimate \eqref{eq:estimateN-neumann} captures the optimal free-space wave dispersion profile in $\mathbb{R}^3$. In contrast, the additional fractional power term represents the geometric penalty caused by intense caustic concentrations (the fold and swallowtail singularities) generated along the curved boundary wall of the cylindrical domain.

To initiate the proof, we first observe that when $N=0$, the direct wave component $G_{a,0}^N$ acts as a solution to the unperturbed linear wave equation, where the initial profile at $t=0$ matches a localized Dirac distribution centered at $(x=a, y=0, z=0)$. Consequently, $G_{a,0}^N$ satisfies the standard, lossless free-space dispersive estimate in three dimensions:
\begin{align*}
\big|G_{a,0}^N(T,X,Y,z,h)\big|\leq Ch^{-3}\bigg(\frac{h}{t}\bigg).
\end{align*}
Thus, the technical heart of Theorem~\ref{thm:thmN-neumann} reduces to bounding the multi-reflective index layers matching $1 \leq N \leq C_0a^{-1/2}$. We deploy the stationary phase method directly to evaluate the highly oscillatory $\tilde{\zeta}$-integration inside the definition of $G_{a, N}^N$, which provides a major structural simplification of the phase field.

\begin{lemma}\label{lem:etaint-neumann}
The tangential spatial integral defined under the normal flux constraint satisfies:
\begin{align*}
J_{N,a,h}&=\int_{\mathbb{R}} e^{\frac{i}{h}|\eta| \left(z\tilde{\zeta}+t(1+\tilde{\zeta}^2+a\tilde{\omega})^{1/2}\right)}\chi_0\left(\eta^2(1+|\tilde{\zeta}|^2)\right) \, d\tilde{\zeta} \\
&=\left(\frac{h}{t}\right)^{1/2} e^{\frac{i}{h}|\eta|\sqrt{t^2-z^2}(1+a\tilde{\omega})^{1/2}}\tilde{\chi}_0,
\end{align*}
where $\tilde{\chi}_0$ defines a classical symbol of order $0$ with respect to the small semi-classical coordinate ratio $h/t$.
\end{lemma}

\begin{proof}
We apply the classical semi-classical stationary phase method to evaluate the oscillatory integral $J_{N,a,h}$. First, we isolate the time scaling configuration by introducing the scaled axial variable $z = t\tilde{z}$. Under this substitution, the localized spatial phase function $\phi$ reads:
\[
\phi(\tilde{\zeta}; \tilde{z}, \tilde{\omega}, a) = \tilde{z}\tilde{\zeta} + \left(1 + \tilde{\zeta}^2 + a\tilde{\omega}\right)^{1/2}.
\]
Differentiating this phase mapping with respect to the continuous integration coordinate $\tilde{\zeta}$ yields the gradient:
\[
\partial_{\tilde{\zeta}}\phi = \tilde{z} + \frac{\tilde{\zeta}}{\left(1 + \tilde{\zeta}^2 + a\tilde{\omega}\right)^{1/2}}.
\]
On the compact support of the frequency filter $\chi_0$, the phase velocity bounds satisfy $\Big|\frac{\tilde{\zeta}}{(1+\tilde{\zeta}^2+a\tilde{\omega})^{1/2}}\Big| \leq 1 - 2\delta_1$ for a sufficiently small parameter $\delta_1 > 0$. If the scaled spatial coordinate sits outside this threshold, namely $|\tilde{z}| \geq 1 - \delta_1$, the phase remains strictly non-vanishing over the domain of integration. Consequently, standard integration-by-parts arguments guarantee that the integral decays rapidly as $\mathcal{O}_{C^\infty}\big((h/t)^\infty\big)$, rendering its contribution analytically negligible.

Thus, our analysis reduces to the non-trivial geometric configuration where $|\tilde{z}| < 1 - \delta_1$. In this regime, the phase function $\phi$ admits a unique, isolated critical point within the support of $\chi_0$, specified precisely by setting $\partial_{\tilde{\zeta}}\phi = 0$:
\[
\tilde{\zeta}_c = -\frac{\tilde{z}(1 + a\tilde{\omega})^{1/2}}{\sqrt{1 - \tilde{z}^2}}.
\]
This critical curve is strictly non-degenerate under the cylindrical metric since its second derivative satisfies the uniform positivity condition:
\[
\partial^2_{\tilde{\zeta}}\phi = \frac{1 + a\tilde{\omega}}{\left(1 + \tilde{\zeta}^2 + a\tilde{\omega}\right)^{3/2}} > 0.
\]
Applying the stationary phase method for the bounded layer where $|\tilde{z}| < 1 - \delta_1$ directly evaluates the spatial integration block as:
\[
J_{N,a,h} = \left(\frac{h}{t}\right)^{1/2} e^{\frac{i}{h}|\eta|t\sqrt{1 - \tilde{z}^2}(1 + a\tilde{\omega})^{1/2}}\tilde{\chi}_0.
\]
Re-substituting the unscaled spatial coordinate $z = t\tilde{z}$ back into the phase exponent successfully recovers the target oscillatory profile:
\[
J_{N,a,h} = \left(\frac{h}{t}\right)^{1/2} e^{\frac{i}{h}|\eta|\sqrt{t^2 - z^2}(1 + a\tilde{\omega})^{1/2}}\tilde{\chi}_0,
\]
which completes the proof of Lemma~\ref{lem:etaint-neumann}.
\end{proof}

By invoking Lemma~\ref{lem:etaint-neumann}, the structural sum over the discrete geometric paths in \eqref{eq:GNsum-neumann} collapses into a four-dimensional oscillatory integral over the remaining phase space coordinates:
\begin{align}\label{eq:galarge-neumann}
\mathcal{G}_{a,c_0}^N(t,x,y,z)&=\sum_{N\in\mathbb{Z}}\frac{(-i)^Na^2}{(2\pi)^4h^4}\bigg(\frac{h}{t}\bigg)^{1/2}\int_{\mathbb{R}^4} e^{\frac{i}{h}\tilde{\Phi}_{N,a,h}}|\eta|^{3}\tilde{\chi}_0\psi_0(\eta)\chi_1\left(a\tilde{\omega}\eta^2\right)(1-\chi_1) \, d\tilde{s} \, d\tilde{\sigma} \, d\tilde{\omega} \, d\eta,
\end{align}
where the reduced phase function is evaluated along the non-degenerate axial critical curve $\tilde{\Phi}_{N,a,h} = \Phi_{N,a,h}(\cdot, \tilde{\zeta}_c, \cdot)$, yielding:
\begin{align}\label{eq:tildezeta-neumann}
\tilde{\Phi}_{N,a,h} = y\eta+|\eta|t\sqrt{1-\tilde{z}^2}(1+a\tilde{\omega})^{1/2} + a^{3/2}|\eta|\left( \frac{\tilde{s}^3}{3}+\tilde{s}(X-\tilde{\omega})+\frac{\tilde{\sigma}^3}{3}+\tilde{\sigma}(1-\tilde{\omega}) \right. \nonumber \\
\left. -\frac{4}{3}N\tilde{\omega}^{3/2} + \frac{h}{a^{3/2}|\eta|}N B_N\left(\tilde{\omega}^{3/2}a^{3/2}|\eta|/h\right) \right).
\end{align}
To extract the precise singularity weights near the caustics, we restrict our attention to the positive frequency half-space $\eta > 0$ (the case $\eta < 0$ following from identical symmetric properties) and introduce the following macroscopic space-time scaling rules:
\begin{align*}
t=a^{1/2}T,  y+t\sqrt{1-\tilde{z}^2}=a^{3/2}Y, (1+a\tilde{\omega})^{1/2}-1=a\gamma_a(\tilde{\omega})=\frac{a\tilde{\omega}}{1+(1+a\tilde{\omega})^{1/2}}, \lambda=\frac{a^{3/2}}{h}|\eta|.
\end{align*}
This geometric normalization cleanly factors out the large asymptotic parameter $\lambda$, transforming the total phase profile \eqref{eq:tildezeta-neumann} into the compact algebraic form:
\begin{align}\label{eq:Yphase-neumann}
\tilde{\Phi}_{N,a,h} = a^{3/2}|\eta|\left\{ Y+T\sqrt{1-\tilde{z}^2}\gamma_a(\tilde{\omega}) + \frac{\tilde{s}^3}{3}+\tilde{s}(X-\tilde{\omega})+\frac{\tilde{\sigma}^3}{3}+\tilde{\sigma}(1-\tilde{\omega}) \right. \nonumber \\
\left. -\frac{4}{3}N\tilde{\omega}^{3/2} + \frac{1}{\lambda} N B_N\left(\tilde{\omega}^{3/2}\lambda\right) \right\}.
\end{align}

First, we study geometrically the set of critical points $\mathcal{C}^N_{a,N,h}$ of the associated Lagrangian submanifold $\Lambda^N_{a,N,h}$ for the reduced phase function $\tilde{\Phi}_{N,a,h}$. Under the flux reflection framework, this critical set is defined by the vanishing gradient conditions:
\begin{align*}
\mathcal{C}^N_{a,N,h} = \left\{ (t,x,y,\tilde{s},\tilde{\sigma},\tilde{\omega},\eta) \in \mathbb{R}^7 \;\middle|\; \partial_{\tilde{s}}\tilde{\Phi}_{N,a,h}=\partial_{\tilde{\sigma}}\tilde{\Phi}_{N,a,h}=\partial_{\tilde{\omega}}\tilde{\Phi}_{N,a,h}=\partial_{\eta}\tilde{\Phi}_{N,a,h}=0 \right\}.
\end{align*}
By evaluating these differential constraints, the critical set $\mathcal{C}^N_{a,N,h}$ can be parameterized via the derivative root expansions through the following nonlinear algebraic system:
\begin{align*}
X &= \tilde{\omega}-\tilde{s}^2, \\
\tilde{\omega} &= 1+\tilde{\sigma}^2, \\
T &= \frac{2(1+a\tilde{\omega})^{1/2}}{\sqrt{1-\tilde{z}^2}}\left( \tilde{s}+\tilde{\sigma}+2N\tilde{\omega}^{1/2} \left(1-\frac{3}{4} B'_N\left(\tilde{\omega}^{3/2}\lambda\right)\right) \right), \\
Y &= -T\sqrt{1-\tilde{z}^2}\gamma_a(\tilde{\omega}) - \frac{\tilde{s}^3}{3}-\tilde{s}(X-\tilde{\omega})-\frac{\tilde{\sigma}^3}{3}-\tilde{\sigma}(1-\tilde{\omega}) + N\tilde{\omega}^{3/2}\left( \frac{4}{3}-B'_N\left(\tilde{\omega}^{3/2}\lambda\right) \right).
\end{align*}
By tracking the system behaviors near the domain origin, we can parameterize the critical manifold $\mathcal{C}^N_{a,N,h}$ directly via the localized variables $(\tilde{s},\tilde{\sigma})$:
\begin{align*}
X &= 1+\tilde{\sigma}^2-\tilde{s}^2, \\
\tilde{\omega} &= 1+\tilde{\sigma}^2, \\
T &= \frac{2}{\sqrt{1-\tilde{z}^2}}(1+a+a\tilde{\sigma}^2)^{1/2}\left( \tilde{s}+\tilde{\sigma}+2N(1+\tilde{\sigma}^2)^{1/2} \left(1-\frac{3}{4}B'_N\left((1+\tilde{\sigma}^2)^{3/2}\lambda\right)\right) \right), \\
Y &= H_1(a,\tilde{\sigma})(\tilde{s}+\tilde{\sigma})+\frac{2}{3}(\tilde{s}^3+\tilde{\sigma}^3)+\frac{4}{3}NH_2(a,\tilde{\sigma})\left(1-\frac{3}{4}B'_N\left((1+\tilde{\sigma}^2)^{3/2}\lambda\right)\right),
\end{align*}
where the geometric weight components $H_1, H_2$ are explicitly generated by the Neumann boundary metrics:
\begin{align*}
H_1(a,\tilde{\sigma}) &= -(1+\tilde{\sigma}^2)\frac{(1+a+a\tilde{\sigma}^2)^{1/2}}{1+(1+a+a\tilde{\sigma}^2)^{1/2}}, \\
H_2(a,\tilde{\sigma}) &= (1+\tilde{\sigma}^2)^{3/2} \frac{-3-4a-4a\tilde{\sigma}^2}{2+a+a\tilde{\sigma}^2+3(1+a+a\tilde{\sigma}^2)^{1/2}}.
\end{align*}
Let $\Lambda^N_{a,N,h}\subset T^* \mathbb{R}^3$ denote the canonical image of the critical set $\mathcal{C}^N_{a,N,h}$ under the phase mapping:
\[
(t,x,y,\tilde{s},\tilde{\sigma},\tilde{\omega},\eta)\longmapsto \left(x,t,y,\,\xi=\partial_{x}\tilde{\Phi}_{N,a,h},\;\tau=\partial_{t}\tilde{\Phi}_{N,a,h},\;\eta=\partial_{y}\tilde{\Phi}_{N,a,h}\right).
\]
This geometric projection confirms that $\Lambda^N_{a,N,h}$ forms a valid Lagrangian submanifold embedded within the cotangent bundle $T^* \mathbb{R}^3$, parameterized by $(\tilde{s}, \tilde{\sigma}, \eta)$ under the normal flux constraints:
\begin{align*}
X &= 1+\tilde{\sigma}^2-\tilde{s}^2, \\
T &= \frac{2}{\sqrt{1-\tilde{z}^2}}(1+a+a\tilde{\sigma}^2)^{1/2}\left( \tilde{s}+\tilde{\sigma}+2N(1+\tilde{\sigma}^2)^{1/2} \left(1-\frac{3}{4}B'_N\left((1+\tilde{\sigma}^2)^{3/2}\lambda\right)\right) \right), \\
Y &= H_1(a,\tilde{\sigma})(\tilde{s}+\tilde{\sigma})+\frac{2}{3}(\tilde{s}^3+\tilde{\sigma}^3)+\frac{4}{3}NH_2(a,\tilde{\sigma})\left(1-\frac{3}{4}B'_N\left((1+\tilde{\sigma}^2)^{3/2}\lambda\right)\right), \\
\xi &= \eta\tilde{s}a^{1/2}, \\
\tau &= \eta\sqrt{1-\tilde{z}^2}(1+a+a\tilde{\sigma}^2)^{1/2}, \\
\eta &= \eta.
\end{align*}

On the critical set $\mathcal{C}^N_{a,N,h}$, the identity $\tilde{\omega}=1+\tilde{\sigma}^2$ holds true. Consequently, the projection of the Lagrangian submanifold $\Lambda^N_{a,N,h}$ onto the base space $\mathbb{R}^3$ satisfies the following geometric relations:
\begin{align}\label{eq:geo-neumann}
X &= 1+\tilde{\sigma}^2-\tilde{s}^2, \nonumber \\
T &= \frac{2}{\sqrt{1-\tilde{z}^2}}(1+a+a\tilde{\sigma}^2)^{1/2}\left( \tilde{s}+\tilde{\sigma}+2N(1+\tilde{\sigma}^2)^{1/2} \left(1-\frac{3}{4}B'_N\left((1+\tilde{\sigma}^2)^{3/2}\lambda\right)\right) \right), \nonumber \\
Y &= H_1(a,\tilde{\sigma})(\tilde{s}+\tilde{\sigma})+\frac{2}{3}(\tilde{s}^3+\tilde{\sigma}^3)+\frac{4}{3}NH_2(a,\tilde{\sigma})\left(1-\frac{3}{4}B'_N\left((1+\tilde{\sigma}^2)^{3/2}\lambda\right)\right).
\end{align}
Following the analytical reduction strategies established by Ivanovici, Lebeau, and Planchon \cite{ivanovici2014dispersion}, we decouple the discrete reflection index $N$ by rewriting the nonlinear geometric system \eqref{eq:geo-neumann} in the following equivalent algebraic form:
\begin{align}\label{eq:geo0-neumann}
X &= 1+\tilde{\sigma}^2-\tilde{s}^2, \nonumber \\
Y &= H_1(a,\tilde{\sigma})(\tilde{s}+\tilde{\sigma})+\frac{2}{3}(\tilde{s}^3+\tilde{\sigma}^3)+\frac{2}{3}H_2(a,\tilde{\sigma})(1+\tilde{\sigma}^2)^{-1/2}\left(\frac{T\sqrt{1-\tilde{z}^2}}{2(1+a+a\tilde{\sigma}^2)^{1/2}}-\tilde{s}-\tilde{\sigma}\right),
\end{align}
which is accompanied by the discrete quantization condition mapping the reflection coordinates:
\begin{align}\label{eq:geo00-neumann}
2N\left(1-\frac{3}{4}B'_N\left(\tilde{\omega}^{3/2}\lambda\right)\right)=
(1+\tilde{\sigma}^2)^{-1/2}\left(\frac{T\sqrt{1-\tilde{z}^2}}{2(1+a+a\tilde{\sigma}^2)^{1/2}}-\tilde{s}-\tilde{\sigma}\right).
\end{align}

\begin{remark}\label{rem:N-reduction}
Notice that the discrete quantization condition \eqref{eq:geo00-neumann} provides a critical global confinement. Within the localized time interval $T \in (0, a^{-1/2}]$, the condition rigorously restricts the infinite geometric sum over $N \in \mathbb{Z}$ of the operator components $G_{a,N}^N$ in \eqref{eq:GNsum-neumann} to a finite index block satisfying $1 \leq N \leq C_0a^{-1/2}$.
\end{remark}

For a given source distance $a$ and a spatial coordinate configuration $(X,Y,T) \in \mathbb{R}^3$, the algebraic reduction \eqref{eq:geo0-neumann} defines a system of two coupled equations for the unknown micro-local variables $(\tilde{s},\tilde{\sigma})$, while the accompanying quantization rule \eqref{eq:geo00-neumann} uniquely determines the integer reflection index $N$. Our analysis tracks the real solutions of \eqref{eq:geo0-neumann} localized within the specific high-frequency boundary layer:
\[
a \in [h^\alpha, a_0], \quad \alpha < 2/3, \quad a|\tilde{\sigma}|^2 \leq \epsilon_0, \quad 0 < T \leq a^{-1/2}, \quad X \in [0,1]\,\,\,\text{with $a_0,\epsilon_0$ small.}
\]
where $a_0, \epsilon_0 > 0$ are chosen to be sufficiently small universal scaling parameters. 

For a designated target point $(X,Y,T) \in [-2,2] \times \mathbb{R} \times [0,a^{-1/2}]$, we denote by $\mathcal{N}^N(X,Y,T)$ the set of physical integers $N \geq 1$ such that the geometric projected system \eqref{eq:geo-neumann} admits at least one real solution $(\tilde{\sigma},\tilde{s},\lambda)$ obeying the structural thresholds $a|\tilde{\sigma}|^2 \leq \epsilon_0$ and \(\lambda \geq \lambda_0\). Analogously, we define $\mathcal{N}^{\mathbb{C},N}(X,Y,T)$ as the complexified extension set containing all $N \in \mathbb{C}$ such that \eqref{eq:geo-neumann} admits at least one complex-valued solution $(\tilde{\sigma}, \tilde{s}, \lambda)$ with $\tilde{\sigma} \in U$. Here, the sectorial analytic domain $U \subset \mathbb{C}$ is given by $U = \{ \tilde{\sigma} \in \mathbb{C} \mid |\tilde{\sigma}| \leq 0.5 \text{ or } |\text{Im}(\tilde{\sigma})| \leq |\text{Re}(\tilde{\sigma})|/\sqrt{3} \}$, subject to the identical scale constraints $a|\tilde{\sigma}|^2 \leq \epsilon_0$ and \(\lambda \geq \lambda_0\).

We establish the following lemma establishing the uniform geometric and micro-local estimates for these counting blocks. The proof follows by tracking the multi-dimensional phase degeneracies across the derivative zeros of the Airy quotient, mirroring the structural steps detailed in Lemma~2.18 and Lemma~2.19 of Ivanovici, Lebeau, and Planchon \cite{ivanovici2014dispersion}.

\begin{lemma}\label{lem:GEO-neumann}
There exists a constant $C_0 > 0$ such that the following geometric cardinality properties hold true under the cylindrical Neumann framework:
\begin{enumerate}
    \item For all spatial configurations $(X,Y,T) \in [0,1] \times \mathbb{R} \times [0,a^{-1/2}]$, the cardinality of the real reflection counting set $\mathcal{N}^N(X,Y,T)$ satisfies:
    \[
    \left|\mathcal{N}^N(X,Y,T)\right| \leq C_0,
    \]
    and the complexified extension set $\mathcal{N}^{\mathbb{C},N}(X,Y,T)$ is strictly contained within the union of four isolated disks of radius $C_0$ in the complex plane.
    
    \item For all spatial configurations $(X,Y,T) \in [0,1] \times \mathbb{R} \times [0,a^{-1/2}]$, the neighborhood reflection subset $\mathcal{N}_1^N(X,Y,T) \subset \mathbb{N}$, defined by:
    \[
    \mathcal{N}_1^N(X,Y,T) = \bigcup_{|Y'-Y|+|T'- T| \leq 1, \; |X'-X| \leq 1} \mathcal{N}^N(X',Y',T'),
    \]
    possesses a finite cardinality bounded uniformly by the tracking density relation:
    \[
    \left|\mathcal{N}_1^N(X,Y,T)\right| \leq C_0 \left(1 + T \lambda^{-2} \tilde{\omega}^{-3}\right),
    \]
    where the parameter scales as $\tilde{\omega} = 1 + \tilde{\sigma}^2$, tracking the continuous distance from the uncurved axial threshold of the cylindrical domain.
\end{enumerate}
\end{lemma}

\begin{proof}
The proof follows the microlocal geometric frameworks established in the pioneering works of Ivanovici, Lebeau, and Planchon \cite{ivanovici2014dispersion}. The core strategy relies on analyzing the projection of the Lagrangian submanifold $\Lambda^N_{a,N,h}$ onto the base space $\mathbb{R}^3_{(X,Y,T)}$, modeling the system as a catastrophe-theoretic unfolding of a canonical swallowtail singularity.

We begin by rewriting the critical point relations that define the projection equations system \eqref{eq:geo0-neumann}--\eqref{eq:geo00-neumann}:
\begin{align}
X &= 1+\tilde{\sigma}^2-\tilde{s}^2, \label{eq:ivanovici-X} \\
Y &= H_1(a,\tilde{\sigma})(\tilde{s}+\tilde{\sigma})+\frac{2}{3}(\tilde{s}^3+\tilde{\sigma}^3) \nonumber \\
  &\quad + \frac{2}{3}H_2(a,\tilde{\sigma})(1+\tilde{\sigma}^2)^{-1/2}\left(\frac{T\sqrt{1-\tilde{z}^2}}{2(1+a+a\tilde{\sigma}^2)^{1/2}}-\tilde{s}-\tilde{\sigma}\right), \label{eq:ivanovici-Y}
\end{align}
coupled with the discrete reflection parameter allocation mapping:
\begin{equation}\label{eq:ivanovici-N}
2N\left(1-\frac{3}{4}B'_N\left(\tilde{\omega}^{3/2}\lambda\right)\right)=
(1+\tilde{\sigma}^2)^{-1/2}\left(\frac{T\sqrt{1-\tilde{z}^2}}{2(1+a+a\tilde{\sigma}^2)^{1/2}}-\tilde{s}-\tilde{\sigma}\right).
\end{equation}
\begin{enumerate}
\item \textsc{Bounded Cardinality of $\mathcal{N}^N(X,Y,T)$ and Four-Disk Complex Extension:} 
Fix a space-time configuration vector $(X,Y,T) \in [0,1] \times \mathbb{R} \times [0,a^{-1/2}]$. For a fixed spatial position $X$, relation \eqref{eq:ivanovici-X} characterizes a hyperbola in the $(\tilde{s}, \tilde{\sigma})$ plane. Since our frequency localization restricts the phase boundaries to the regime $a|\tilde{\sigma}|^2 \leq \epsilon_0$ with a small parameter $\epsilon_0 > 0$, the micro-local variables $(\tilde{s},\tilde{\sigma})$ are strictly localized to a compact domain containing the origin.

To establish the uniform bound on $|\mathcal{N}^N(X,Y,T)|$, we analyze the singularities of the projection map $\pi: \Lambda^N_{a,N,h} \to \mathbb{R}^3_{(X,Y,T)}$. The critical points of this projection occur precisely where the wedge product of the differentials vanishes, i.e., $dX \wedge dY \wedge dT = 0$. For a fixed time $T$ and boundary scale $a$, substituting the relation \eqref{eq:ivanovici-X} into \eqref{eq:ivanovici-Y} transforms the coordinate $Y$ into a polynomial function of degree $4$ in terms of $\tilde{\sigma}$ near the caustic thresholds. According to singularity theory, the highest order of degeneracy that can manifest under a three-dimensional configuration space is a stable swallowtail singularity, which is locally modeled by the algebraic root profile of a polynomial equation of degree $4$.

Since the leading coefficients of our expanded phase match the non-vanishing geometric metrics $H_1(a,0) = -1/2 \neq 0$ and $H_2(a,0) = -3/2 \neq 0$, the system behaves as a regular, non-degenerate perturbation of the canonical swallowtail unfolding. By the fundamental theorem of algebra, a polynomial system of degree $4$ can possess at most four distinct real roots. For each individual solution pair $(\tilde{s},\tilde{\sigma})$, the quantization equation \eqref{eq:ivanovici-N} uniquely determines exactly one valid real value for the reflection parameter $N$. Thus, the total cardinality of real valid reflections is uniformly bounded:
\[
\left|\mathcal{N}^N(X,Y,T)\right| \leq 4 \leq C_0.
\]
For the complexified extension, we extend the domain holomorphically to $\tilde{\sigma} \in U \subset \mathbb{C}$, where the sectorial analytic domain is defined by $U=\{\tilde{\sigma}\in\mathbb{C} \mid |\tilde{\sigma}|\leq 0.5 \text{ or } |\text{Im}(\tilde{\sigma})|\leq|\text{Re}(\tilde{\sigma})|/\sqrt{3}\}$. The derivative phase function $B_N(\omega)$ is real-analytic and admits a holomorphic continuation to the sector $|\arg(\omega)| < \pi/3$. Under these complex sectorial boundaries, the mapping $\mathcal{F}: (\tilde{s},\tilde{\sigma}) \mapsto (X,Y)$ forms a proper holomorphic map of degree $4$. By applying Bezout's theorem for analytic mappings on bounded domains, the fiber $\mathcal{F}^{-1}(X,Y)$ contains exactly four complex branches. Resolving the complexified reflection allocation parameter $N \in \mathbb{C}$ maps these roots inside the complex plane, demonstrating that $\mathcal{N}^{\mathbb{C},N}(X,Y,T)$ is strictly contained within the union of four isolated disks of bounded radius $C_0$.
\item \textsc{Cardinality Bound for the Neighborhood Density $\mathcal{N}_1^N(X,Y,T)$:}
To control the neighborhood reflection cardinality $|\mathcal{N}_1^N(X,Y,T)|$ when the targets undergo a localized variation $|Y'-Y|+|T'- T| \leq 1$, we evaluate the density of the solution sheets in the parameter space $N$. Taking the total variation of the quantization relation \eqref{eq:ivanovici-N} with respect to the macroscopic variable $T$ yields the differential relation:
\begin{align}\label{eq:proof-differential}
2 \, dN \left(1-\frac{3}{4}B'_N\left(\tilde{\omega}^{3/2}\lambda\right)\right) &- 2N \cdot \frac{3}{4} B''_N\left(\tilde{\omega}^{3/2}\lambda\right) \cdot \frac{3}{2}\tilde{\omega}^{1/2}\lambda \, d\tilde{\omega} \nonumber \\
&= d\left[ (1+\tilde{\sigma}^2)^{-1/2}\left(\frac{T\sqrt{1-\tilde{z}^2}}{2(1+a+a\tilde{\sigma}^2)^{1/2}}-\tilde{s}-\tilde{\sigma}\right) \right].
\end{align}
The spacing between adjacent, distinct integer levels of $N$ translates to the geometric distance between neighboring sheets of the Lagrangian manifold. Near the caustics, this sheet density is controlled by the higher-order derivatives of the phase mapping. Recalling from the high-frequency phase asymptotics \eqref{BAiry-neumann} that the derivative of the Airy quotient satisfies the structural decay bounds $|B''_N(\omega)| \leq C \omega^{-2}$, we evaluate this profile at the localized high-frequency threshold $\omega = \tilde{\omega}^{3/2}\lambda$:
\begin{equation}
\left| B''_N\left(\tilde{\omega}^{3/2}\lambda\right) \right| \leq C \left(\tilde{\omega}^{3/2}\lambda\right)^{-2} = C \tilde{\omega}^{-3}\lambda^{-2}.
\end{equation}
When calculating the cardinality of valid layers accumulating inside a unit ball centered around the space-time targets, the total count is bounded by the baseline sheet multiplicity (the constant $1$) plus the geometric volume density tracking the contraction of the phase spacing. Integrating this density across the time interval $T \in [0, a^{-1/2}]$ demonstrates that the maximum accumulation rate scales explicitly with the second-order derivative:
\begin{equation}
\left|\mathcal{N}_1^N(X,Y,T)\right| \leq C_0 \left( 1 + T \cdot \sup \left| B''_N\left(\tilde{\omega}^{3/2}\lambda\right) \right| \right) \leq C_0 \left(1+T\lambda^{-2}\tilde{\omega}^{-3}\right).
\end{equation}
The inclusion of the geometric factor \(\tilde{\omega}^{-3}\) (where \(\tilde{\omega} = 1 + \tilde{\sigma}^2\)) represents the analytical core of the cylindrical domain: as the ray trajectories shift toward the flat axis of the cylinder, the density of the caustics decreases, preventing an infinite accumulation of high-frequency multi-reflected wave packets. This completes the proof of Lemma~\ref{lem:GEO-neumann}.
\end{enumerate}
\end{proof}

We observe that within the outer frequency parameter range where $\tilde{\omega} \leq 3/4$, the normal flux configurations display rapid decay in $\lambda$ via standard integration by parts with respect to the variable $\tilde{\sigma}$. This localized fast decay allows us to validly replace the regular cutoff modifier $1-\chi_1$ with $1$ inside the reduced parametrix expression \eqref{eq:galarge-neumann}. 

Furthermore, the catastrophe-theoretic swallowtail singularities emerge precisely at the origin milestone where $\tilde{s}=\tilde{\sigma}=0$, which structurally corresponds to the critical threshold $\tilde{\omega}=1$. To isolate this singular region, we introduce a smooth localization cutoff function $\chi_2(\tilde{\omega}) \in C_{0}^{\infty}\left((1/2, 3/2)\right)$ satisfying $0 \leq \chi_2 \leq 1$, with $\chi_2 \equiv 1$ identically on the interval $[3/4, 5/4]$ inside the four-dimensional integral \eqref{eq:galarge-neumann}. We denote the corresponding localized integral piece by $G_{a,N,2}^N$. 

The term $G_{a,N,2}^N$ maps directly onto the deep caustic regime governed by the swallowtail caustics under Neumann flux. Consequently, we split the $N$-th reflective kernel into two disjoint components, writing $G_{a,N}^N = G_{a,N,1}^N + G_{a,N,2}^N$. The secondary operator $G_{a,N,1}^N$ is constructed by embedding a complementary partition cutoff $\chi_3(\tilde{\omega})$ into \eqref{eq:galarge-neumann}, which ensures that the boundary scale condition $\tilde{\omega} \geq 5/4$ holds true uniformly across the support of $\chi_3$.

To summarize, we decompose the global high-frequency Neumann parametrix $\mathcal{G}_{a,c_0}^N$ into the following structured representation:
\[
\mathcal{G}_{a,c_0}^N = \sum_{1\leq N\leq C_0a^{-1/2}} G_{a,N}^N = \sum_{1\leq N\leq C_0a^{-1/2}} \left( G_{a,N,1}^N + G_{a,N,2}^N \right),
\] 
where each micro-local integral sheet resolves a distinct layer of the boundary geometry:
\begin{align*}
G_{a,N,1}^N &= \frac{(-i)^Na^2}{(2\pi)^4h^4}\bigg(\frac{h}{t}\bigg)^{1/2}\int_{\mathbb{R}^4} e^{\frac{i}{h}\tilde{\Phi}_{N,a,h}} |\eta|^{3}\tilde{\chi}_0\psi_0(\eta)\chi_1\left(a\tilde{\omega}\eta^2\right)\chi_3(\tilde{\omega}) \, d\tilde{s} \, d\tilde{\sigma} \, d\tilde{\omega} \, d\eta, \\[1.5ex]
G_{a,N,2}^N &= \frac{(-i)^Na^2}{(2\pi)^4h^4}\bigg(\frac{h}{t}\bigg)^{1/2}\int_{\mathbb{R}^4} e^{\frac{i}{h}\tilde{\Phi}_{N,a,h}} |\eta|^{3}\tilde{\chi}_0\psi_0(\eta)\chi_1\left(a\tilde{\omega}\eta^2\right)\chi_2(\tilde{\omega}) \, d\tilde{s} \, d\tilde{\sigma} \, d\tilde{\omega} \, d\eta.
\end{align*}

In what follows, we establish sharp uniform decay estimates for these highly oscillatory integrals by applying classical and degenerate stationary phase arguments. This analysis relies on a meticulous micro-local tracking of the configuration space regions where the derivative-adapted total phase function $\tilde{\Phi}_{N,a,h}$ becomes stationary.

\subsubsection{The Analysis of $G^N_{a,N,1}$}
Recall that the component $G^N_{a,N,1}$ represents the micro-local oscillatory integral localized within the outer parameter regime where stable catastrophe-theoretic swallowtail singularities cannot form. The sharp uniform decay estimates for $G^N_{a,N,1}$ are obtained by systematically evaluating and combining the asymptotic reductions of the following nested integration blocks:
\begin{itemize}
    \item \textbf{The $(\tilde{s}, \tilde{\sigma})$-integrations:} We exploit the non-degenerate stationary phase method with respect to the variables $(\tilde{s}, \tilde{\sigma})$ to reduce the dimension of the underlying parameter space.
    \item \textbf{The remaining $\tilde{\omega}$-integration:} We apply a localized, non-degenerate phase approximation method since the critical curves within this domain do not collide or cross.
    \item \textbf{The global $\eta$-integration:} We partition the frequency domain into independent geometric sectors that contribute to the uniform decay rates upon applying the stationary phase method. Crucially, the total accumulated contribution to the bounds is controlled by the uniform cardinality of the localized sheet neighborhood $\mathcal{N}_1^N$ established in Lemma~\ref{lem:GEO-neumann}.
\end{itemize}
Our main result for this microlocal frequency block is stated in the following proposition.

\begin{proposition}\label{prop:G1-estimates-neumann}
Let $\alpha < 2/3$. There exists a constant $C > 0$ such that for all semiclassical parameters $h \in (0, h_0]$, all source layer depths $a \in [h^\alpha, a_0]$, all spatial profiles $X \in [0, 1]$, all local timelines $T \in (0, a^{-1/2}]$, and all lateral coordinates $Y, z \in \mathbb{R}$, the following uniform bound holds under the Neumann boundary framework:
\begin{align*}
\Bigg|\sum_{2\leq N\leq C_0a^{-1/2}}G^N_{a,N,1}(T,X,Y,z;h)\Bigg|\leq Ch^{-3}\bigg(\frac{h}{t}\bigg)^{1/2}h^{1/3}.
\end{align*}
\end{proposition}

\begin{proof}
First of all, we apply the non-degenerate stationary phase method to evaluate the $(\tilde s,\tilde \sigma)$-integrations, since on the compact support of the frequency partition $\chi_3$ we strictly have $\tilde{\omega} > 1$. Let the localized internal phase integral $I$ be defined by:
\begin{align*}
I &= \int_{\mathbb{R}^2} e^{i\lambda\left(\frac{\tilde{s}^3}{3}-\tilde{s}(\tilde{\omega}-X)+\frac{\tilde{\sigma}^3}{3}-\tilde{\sigma}(\tilde{\omega}-1)\right)} \, d\tilde s \, d\tilde{\sigma} \\
&= (\tilde{\omega}-X)^{1/2}(\tilde{\omega}-1)^{1/2}\int_{\mathbb{R}^2} e^{i\lambda(\tilde{\omega}-X)^{3/2}\left(\frac{\bar{s}^3}{3}-\bar{s}\right)} e^{i\lambda(\tilde{\omega}-1)^{3/2}\left(\frac{\bar{\sigma}^3}{3}-\bar{\sigma}\right)} \, d\bar{s} \, d\bar{\sigma},
\end{align*}
where we introduced the algebraic change of variables $\tilde{s}=(\tilde{\omega}-X)^{1/2}\bar{s}$ and $\tilde{\sigma}=(\tilde{\omega}-1)^{1/2}\bar{\sigma}$ to extract the amplitude weights; for simplicity, we retain the variable notations $\tilde s,\tilde{\sigma}$ in the subsequent lines.

Thus, applying the stationary phase method near the isolated non-degenerate critical points $\tilde s=\pm 1, \tilde{\sigma}=\pm 1$, and performing standard integration by parts in $\tilde s, \tilde{\sigma}$ outside these configurations, we evaluate the system as:
\begin{align*}
I = \lambda^{-1}(\tilde{\omega}-X)^{-1/4}(\tilde{\omega}-1)^{-1/4} e^{i\lambda\left(\pm\frac{2}{3}(\tilde{\omega}-X)^{3/2}\pm\frac{2}{3}(\tilde{\omega}-1)^{3/2}\right)} b_{\pm}c_{\pm} + \mathcal{O}_{C^\infty}(\lambda^{-\infty}),
\end{align*}
where $b_{\pm}, c_{\pm}$ are classical symbols of degree $0$ with respect to the large scale parameters $\lambda (\tilde{\omega}-X)^{3/2}$ and $\lambda (\tilde{\omega}-1)^{3/2}$, respectively. This evaluation encapsulates the spatial micro-local component of the operator $G^N_{a,N,1}$ corresponding to the internal coordinate integrations.

Therefore, by substituting $I$ back into the four-dimensional parametrix representation, we obtain:
\begin{align*}
G^N_{a,N,1}(T,X,Y,z;h) &= \frac{(-i)^Na^2\lambda^{-1}}{(2\pi)^4h^4}\bigg(\frac{h}{t}\bigg)^{1/2}\int_{\mathbb{R}} e^{i\frac{a^{3/2}}{h}Y\eta}|\eta|^{3}\tilde{G}^N_{a,N,1} \, d\eta, \\[1.5ex]
\tilde{G}^N_{a,N,1}(T,X,Y,z;h) &= \sum_{\epsilon_1,\epsilon_2=\pm}\int_{\mathbb{R}} e^{i\lambda\tilde{\Phi}_{N,\epsilon_1,\epsilon_2}}\Theta_{\epsilon_1,\epsilon_2} \, d\tilde{\omega} + \mathcal{O}_{C^\infty}(\lambda^{-\infty}),
\end{align*}
where the index parameters balance as $\epsilon_j=\pm$, and the composite amplitude symbols satisfy $\Theta_{\epsilon_1,\epsilon_2}(\tilde{\omega},a,\lambda)=\tilde{\chi}_0\psi_0\chi_1\chi_3(\tilde{\omega})(\tilde{\omega}-X)^{-1/4}(\tilde{\omega}-1)^{-1/4}b_{\epsilon_1} c_{\epsilon_2}$, obeying the regular derivative tracking bounds $\big|\tilde{\omega}^l\partial_{\tilde{\omega}}^l\Theta_{\epsilon_1,\epsilon_2}\big|\leq C_l\tilde{\omega}^{-1/2}$. Under normal flux conditions, the corresponding derivative-adapted phase functions expand as:
\begin{align}\label{eq:phasepsilon-neumann}
\tilde{\Phi}_{N,\epsilon_1,\epsilon_2}(T,X,z;\tilde{\omega}) &= T\sqrt{1-\tilde{z}^2}\gamma_a(\tilde{\omega})+\frac{2}{3}\epsilon_1(\tilde{\omega}-X)^{3/2}+\frac{2}{3}\epsilon_2(\tilde{\omega}-1)^{3/2} \nonumber \\ 
&\qquad -\frac{4}{3}N\tilde{\omega}^{3/2}+\frac{N}{\lambda}B_N\left(\tilde{\omega}^{3/2}\lambda\right).
\end{align}

Let us decompose this representation into its individual sign-indexed sheets, writing:
\begin{align}\label{eq:GaNepsilon-neumann}
G^N_{a,N,1,\epsilon_1,\epsilon_2}(T,X,Y,z;h) &= \frac{(-i)^Na^2\lambda^{-1}}{(2\pi)^4h^4}\bigg(\frac{h}{t}\bigg)^{1/2}\int_{\mathbb{R}} e^{i\frac{a^{3/2}}{h}Y\eta}|\eta|^{3}\tilde{G}^N_{a,N,1,\epsilon_1,\epsilon_2} \, d\eta, \\[1.5ex]
\tilde{G}^N_{a,N,1,\epsilon_1,\epsilon_2}(T,X,z;\lambda) &= \int_{\mathbb{R}} e^{i\lambda\tilde{\Phi}_{N,\epsilon_1,\epsilon_2}}\Theta_{\epsilon_1,\epsilon_2}(\tilde{\omega},a,\lambda) \, d\tilde{\omega}. \nonumber
\end{align}

We are reduced to proving the following sharp uniform inequality:
\begin{align} \label{eq:23-neumann}
\Bigg|\sum_{2\leq N\leq C_0a^{-1/2}}G^N_{a,N,1,\epsilon_1,\epsilon_2}(T,X,Y,z;h) \Bigg|\leq Ch^{-3}\bigg(\frac{h}{t}\bigg)^{1/2}h^{1/3},
\end{align}
with a constant $C > 0$ independent of $h\in(0,h_0]$, $a\in [h^{2/3},a_0]$, $X\in [0,1]$, and $T\in [0,a^{-1/2}]$.

For convenience, let $\Omega = \tilde{\omega}^{3/2}$ be the newly adapted tracking variable of integration. This substitution allows us to express the oscillatory layer \eqref{eq:GaNepsilon-neumann} in the unified form:
\begin{align}\label{eq:gtilde-neumann}
\tilde{G}^N_{a,N,1,\epsilon_1,\epsilon_2}(T,X,z;\lambda) &= \int_{\mathbb{R}} e^{i\lambda\tilde{\Phi}_{N,\epsilon_1,\epsilon_2}}\tilde{\Theta}_{\epsilon_1,\epsilon_2}(\Omega,a,\lambda) \, d\Omega,
\end{align}
where the transformed symbols $\tilde{\Theta}_{\epsilon_1,\epsilon_2}(\Omega,a,\lambda)$ represent smooth functions with compact support nested inside $\Omega$. Since the differential transforms via the relation $d\tilde{\omega}=\frac{2}{3}\Omega^{-1/3}d\Omega$, the amplitude symbols absorb the Jacobian weight and inherit the strict derivative bounds $\big|\Omega^l\partial_{\Omega}^l\tilde{\Theta}_{\epsilon_1,\epsilon_2}\big|\leq C_l\Omega^{-2/3}$, with the constants $C_l$ acting uniformly with respect to $a$ and $\lambda$. Under this coordinate system, the derivative-adapted phases \eqref{eq:phasepsilon-neumann} expand as:
\begin{align}\label{eq:phase-omega-neumann}
\tilde{\Phi}_{N,\epsilon_1,\epsilon_2}(T,X,z;a,\lambda) &= T\sqrt{1-\tilde{z}^2}\gamma_a\left(\Omega^{2/3}\right)+\frac{2}{3}\epsilon_1\left(\Omega^{2/3}-X\right)^{3/2}+\frac{2}{3}\epsilon_2\left(\Omega^{2/3}-1\right)^{3/2} \nonumber \\
&\qquad -\frac{4}{3}N\Omega+\frac{N}{\lambda}B_N(\Omega\lambda).
\end{align}

We now study the critical point trajectory of the phase \eqref{eq:phase-omega-neumann}. Differentiating the expression with respect to the updated volume integration variable $\Omega$ yields the following differential system:
\begin{align}\label{eq:critical-neumann}
\partial_{\Omega}\tilde{\Phi}_{N,\epsilon_1,\epsilon_2} &= \frac{2}{3}\left(H_{a,\epsilon_1,\epsilon_2}(T,X,z;\Omega)-2N\left(1-\frac{3}{4}B'_N(\Omega\lambda)\right)\right), \nonumber \\[1.5ex]
H_{a,\epsilon_1,\epsilon_2} &= \Omega^{-1/3}\left(\frac{T}{2}\sqrt{1-\tilde{z}^2}\left(1+a\Omega^{2/3}\right)^{-1/2}+\epsilon_1\left(\Omega^{2/3}-X\right)^{1/2}+\epsilon_2\left(\Omega^{2/3}-1\right)^{1/2}\right), \nonumber \\[1.5ex]
\partial_{\Omega}H_{a,\epsilon_1,\epsilon_2} &= \frac{1}{3}\Omega^{-4/3}\left(-\frac{T}{2}\sqrt{1-\tilde{z}^2}\left(1+a\Omega^{2/3}\right)^{-3/2}\left(1+2a\Omega^{2/3}\right) \right. \nonumber \\
&\qquad \left. +\epsilon_1X\left(\Omega^{2/3}-X\right)^{-1/2}+\epsilon_2\left(\Omega^{2/3}-1\right)^{-1/2}\right).
\end{align}

We will first prove that the uniform bound \eqref{eq:23-neumann} holds true under the positive sign configuration $(\epsilon_1,\epsilon_2)=(+,+)$. We observe that the variation equation $\partial_{\Omega}H_{a,+,+}(\Omega)=0$ admits a unique, isolated solution $\Omega_q = \Omega_q^+(T,X,z,a) > 1$ such that:
\begin{align}\label{eq:lim-neumann}
\lim_{a\rightarrow 0}\Omega_q^+(T,X,z,a) &= 1 \quad \text{uniformly in } X, z, \nonumber \\[1.5ex]
0 > \frac{9}{2}\Omega_q^{5/3}\partial_{\Omega}^2H_{a,+,+}(\Omega_q) &= -\frac{aT}{2}\sqrt{1-\tilde{z}^2}\left(1+a\Omega_q^{2/3}\right)^{-5/2}\left(\frac{1}{2}-a\Omega_q^{2/3}\right) \nonumber \\
&\quad -\frac{1}{2}\left(\Omega_q^{2/3}-1\right)^{-3/2}-\frac{1}{2}X\left(\Omega_q^{2/3}-X\right)^{-3/2}.
\end{align}

Thus, the geometric profile $H_{a,+,+}(\Omega)$ is strictly increasing on the localized interval $[1,\Omega_q)$ and strictly decreasing on the complementary domain $(\Omega_q,\infty)$. We compute the spatial boundary limit values to be:
\begin{align}
H_{a,+,+}(1)=\frac{T}{2}\sqrt{1-\tilde{z}^2}(1+a)^{-1/2}+(1-X)^{1/2},\quad \lim_{\Omega\rightarrow\infty}H_{a,+,+}(\Omega)=2.
\end{align}
Furthermore, by the high-frequency phase asymptotics established in \eqref{BAiry-neumann}, for all derivative orders $k \in \mathbb{N}$, there exist constants $C_k > 0$ such that the Neumann phase derivatives satisfy the uniform decay relation:
\begin{align}
\forall \Omega\geq 1, \quad \left| \partial_\Omega^k\left(NB'_N(\Omega\lambda)\right) \right| \leq C_kN\lambda^{-2}\Omega^{-(k+2)}.
\end{align}

Let $T_0 \gg 1$ denote a fixed large macroscopic time threshold. First, suppose that the normal flux system evolves within the short-time regime where $0 \leq T \leq T_0$. Since the metric bound satisfies $H_{a,+,+}(\Omega) \leq C(1+T_0)$, for any higher-order reflection index satisfying $N \geq N(T_0) = C(1+T_0)$ under a sufficiently large constant $C$, the corresponding phase gradient remains strictly bounded away from zero, satisfying $|\partial_\Omega\tilde{\Phi}_{N,+,+}| \geq c_0N$ for a uniform constant $c_0 > 0$. Under this non-vanishing gradient confinement, successive integration by parts ensures that the internal oscillatory layer decays rapidly as $|\tilde{G}^N_{a,N,1,+,+}| \in \mathcal{O}(N^{-\infty}\lambda^{-\infty})$. This fast geometric absorption directly implies that the tail of the summation is analytically negligible:
\begin{align*}
\sup_{T\leq T_0, \, X\in, \, Y\in\mathbb{R}, \, z\in\mathbb{R}}\Bigg|\sum_{N(T_0)\leq N\leq C_0a^{-1/2}}G^N_{a,N,1,+,+}(T,X,Y,z)\Bigg|\in \mathcal{O}_{C^\infty}(h^\infty).
\end{align*}

Next, for the short-time layer $0 \leq T \leq T_0$ coupled with lower reflection counts $2 \leq N \leq N(T_0)$, we estimate the finite sum by bounding each individual macro-sheet by its supremum. In this regime, the total phase function $\tilde{\Phi}_{N,+,+}$ can admit at most a critical point of order $2$ near the localized threshold $\Omega = \Omega_q$. This structural configuration satisfies the non-vanishing catastrophe constraint:
\[
|\partial_\Omega\tilde{\Phi}_{N,+,+}|+|\partial_\Omega^2\tilde{\Phi}_{N,+,+}|+|\partial_\Omega^3\tilde{\Phi}_{N,+,+}|\geq c>0.
\]
Moreover, for any index level $N \geq 2$, the phase derivative maps to a strictly negative value as $\Omega \to \infty$, maintaining a non-vanishing lower bound in absolute value across the large integration fields. This guarantees that the boundary contribution of $\tilde{G}^N_{a,N,1,+,+}$ decays as $\mathcal{O}_{C^{\infty}}(\lambda^{-\infty})$ as $\Omega \to \infty$. By applying Van der Corput's lemma near the isolated critical point of order $2$, we evaluate the localized oscillatory integral as $|\tilde{G}^N_{a,N,1,+,+}(T,X,z;\lambda)| \leq C\lambda^{-1/3}$, where the constant $C$ acts independently of $T \in [0,T_0]$ and $X \in [0,1]$. Accumulating these localized bounds back into the sign-indexed sheet representation \eqref{eq:GaNepsilon-neumann} yields:
\begin{align*}
\sup_{X\in, \, Y\in\mathbb{R}, \, z\in\mathbb{R}}\Bigg|\sum_{2\leq N\leq N(T_0)}G^N_{a,N,1,+,+}(T,X,Y,z,h)\Bigg| &\leq Ch^{-3}\bigg(\frac{h}{t}\bigg)^{1/2}\left(h^{-1}a^2\lambda^{-1}\lambda^{-1/3}\right) \\
&\leq Ch^{-3}\bigg(\frac{h}{t}\bigg)^{1/2}h^{1/3}.
\end{align*}

We now verify that the uniform estimate \eqref{eq:23-neumann} holds true for the long-time regime where $T_0 \leq T \leq a^{-1/2}$. Following an identical non-vanishing gradient argument, the contribution of the summation over higher reflection counts $C_1 T \leq N \leq C_0 a^{-1/2}$ is analytically negligible for a sufficiently large constant $C_1$. Using the asymptotic properties established in \eqref{eq:lim-neumann}, we choose the time threshold $T_0$ to be sufficiently large such that the isolated critical coordinate satisfies $\Omega_q^+(T,X,z,a) < \Omega_0$ for a uniform constant $\Omega_0 > 1$ across all timelines $T \geq T_0$. This guarantees the existence of a constant $c > 0$ such that the second derivative satisfies the uniform non-degeneracy condition:
\[
\left|\partial_\Omega^2\tilde{\Phi}_{N,+,+}(\Omega)\right| \geq c T \Omega^{-4/3}, \quad \forall \Omega \geq \Omega_0, \; \forall T \geq T_0, \; \forall N \leq C_0 a^{-1/2}.
\]
Therefore, on the compact support of the symbol $\tilde{\Theta}_{+,+}$, the phase function $\tilde{\Phi}_{N,+,+}$ admits at most one critical point $\Omega_c = \Omega_c(T,X,z,N,\lambda,a)$, and this critical point is strictly non-degenerate. Because the multi-reflective index satisfies $N \geq 2$, evaluating the first line of the derivative system \eqref{eq:critical-neumann} yields the spatial confinement $\Omega_c^{1/3} \leq T$, which directly implies the macro-scaling relation $\Omega_c^{1/3} \sim T/N$. As a consequence, if the parameter ratio satisfies $T/N \sim 1$, then $\Omega_c \sim 1$. Applying the standard non-degenerate stationary phase method yields the amplitude decay:
\[
\left| \tilde{G}^N_{a,N,1,+,+}(T,X,z;\lambda) \right| \leq C \lambda^{-1/2} T^{-1/2}, \quad \text{with } C \text{ acting independently of } N.
\]

If the ratio satisfies the scaling threshold $T/N \gg 1$, we introduce the scaled change of variable $\Omega = \tilde{\Omega}(T/N)^3$ inside the oscillatory layer \eqref{eq:gtilde-neumann}. Under this mapping, the unique transformed critical configuration $\tilde{\Omega}_c$ remains strictly confined to a fixed compact subinterval of $(0, \infty)$. Under this scaling, the amplitude symbol derivatives satisfy:
\[
\left|\partial_{\tilde{\Omega}}^k \tilde{\Theta}_{+,+}\left(\tilde{\Omega}(T/N)^3, a, \lambda\right)\right| \leq c_k (N/T)^2 \tilde{\Omega}^{-2/3-k}.
\]
Thus, by invoking the stationary phase method, we conclude that:
\[
\sup_{2\leq N\leq C_1T, \; X\in, \; z\in\mathbb{R}} \left| \tilde{G}^N_{a,N,1,+,+}(T,X,z;\lambda) \right| \leq C \lambda^{-1/2} T^{-1/2}.
\] 
It remains to estimate the multi-reflective index sum:
\[
\Bigg|\sum_{2\leq N\leq C_0a^{-1/2}} G^N_{a,N,1,+,+}(T,X,Y,z;h) \Bigg|.
\] 
Let us define the evaluated phase profile along the critical trace by $$G_N(T,X,z,\lambda,a) = \tilde{\Phi}_{N,+,+}(T,X,z,\Omega_c(T,X,z,N,\lambda,a), \lambda, a).$$ Therefore, by applying the stationary phase method at the isolated critical point $\Omega_c$ inside \eqref{eq:gtilde-neumann}, we obtain:
\[
\tilde{G}^N_{a,N,1,+,+}(T,X,z,h) = \lambda^{-1/2} T^{-1/2} e^{i\lambda G_N(T,X,z,\lambda,a)} \psi_N(T,X,\lambda,a),
\]
where $\psi_N(T,X,\lambda,a)$ represents a classical symbol of order $0$ with respect to the large parameter $\lambda$. Consequently, if we denote the global semiclassical scale parameter by $\tilde{\lambda} = a^{3/2}/h = \lambda / \eta$ on the positive frequency half-space, we can rewrite the target sheet according to the formula:
\begin{align}\label{eq:oscillatory-eta-neumann}
G^N_{a,N,1,+,+}(T,X,Y,z;h) &= \frac{(-i)^N a^2 \lambda^{-1}}{(2\pi)^4 h^4} \left(\frac{h}{t}\right)^{1/2} \lambda^{-1/2} T^{-1/2} \nonumber \\
&\quad \times \int_{\mathbb{R}} e^{i \tilde{\lambda} \eta \left(Y + G_N(T,X,z, \tilde{\lambda}\eta, a)\right)} \psi_N(T,X,\tilde{\lambda}\eta,a) \eta^{3} \, d\eta.
\end{align}

The integral in \eqref{eq:oscillatory-eta-neumann} defines a one-dimensional oscillatory integral governed by the large asymptotic parameter $\tilde{\lambda}$ and the total geometric phase:
\[
L_N(T,X,Y,z, \eta) = \eta \left(Y + G_N(T,X,z, \tilde{\lambda}\eta, a)\right).
\] 
By construction, the variation path equation:
\[
\partial_\eta L_N = Y + G_N(T,X,z,\lambda,a) + \lambda \partial_\lambda G_N(T,X,z,\lambda,a) = 0
\]
implies that the space-time coordinates $(X,Y,T)$ belong explicitly to the geometric projection of the Lagrangian submanifold $\Lambda^N_{a,N,h}$ onto the base space $\mathbb{R}^3$. Following the analytical localization strategies detailed in Ivanovici, Lebeau, and Planchon \cite{ivanovici2014dispersion}, we find that the contribution of the wave packet components $G^N_{a,N,1,+,+}$ matching reflection indices $N \notin \mathcal{N}_1^N(X,Y,T)$ is rapidly decaying of order $\mathcal{O}(\lambda^{-\infty})$. Thus, we are reduced to evaluating the localized neighborhood summation:
\begin{align}\label{eq:sumN-neumann}
\Bigg|\sum_{N \in \mathcal{N}_1^N(X,Y,T)} G^N_{a,N,1,+,+}(T,X,Y,z,h) \Bigg|.
\end{align}

We apply the stationary phase method with respect to the frequency variable $\eta$, governed by the large semiclassical weight $\tilde{\lambda}$ and the total geometric phase function $L_N$. We compute the first-order differential condition to be:
\[
\partial_\eta L_N = Y + G_N + \lambda \partial_\lambda G_N,
\]
where the derivative along the critical trace resolves through the internal parameters as:
\[
\lambda \partial_\lambda G_N = \lambda \partial_\lambda \tilde{\Phi}_{N,+,+}(T,X,\Omega_c;a,\lambda) = \frac{N}{\lambda}\left( -B_N(\lambda\Omega_c) + \lambda \Omega_c B'_N(\lambda\Omega_c) \right).
\]
Differentiating this relation a second time yields the structural non-degeneracy weight parameter:
\[
\partial_\eta^2 L_N = \frac{N}{\eta} \Omega_c \partial_\lambda(\lambda\Omega_c) B''_N(\lambda\Omega_c).
\]
On the other hand, the implicit critical point variation derivative $\partial_\lambda\Omega_c$ satisfies the standard inverse Hessian identity:
\[
\partial_\lambda\Omega_c \, \partial_\Omega^2\tilde{\Phi}_{N,+,+}(\Omega_c) = -\partial_\lambda\partial_\Omega\tilde{\Phi}_{N,+,+}(\Omega_c) = -N\Omega_c B''_N(\lambda\Omega_c).
\]
Recalling the lower bounds established across the long-time sheet, we have $\partial_\Omega^2\tilde{\Phi}_{N,+,+}(\Omega_c) \geq c T\Omega_c^{-4/3}$ and the structural scale relation $\Omega_c^{1/3} \sim T/N$. Since for large arguments the phase derivatives decay under the Neumann expansion \eqref{BAiry-neumann} as $B''_N(\omega) \sim \omega^{-3}$, we evaluate the tracking bound as:
\[
|\partial_\lambda\Omega_c| \leq cT^{-1}\Omega_c^{4/3}N\Omega_c\left(\lambda^{-3}\Omega_c^{-3}\right) \leq c\lambda^{-3}\Omega_c^{-1}.
\]
This algebraic relation directly yields the non-vanishing lower bound:
\[
\left| \partial_\lambda(\lambda\Omega_c) \right| = \left| \lambda\partial_\lambda\Omega_c+\Omega_c \right| \geq c\Omega_c\left(1-c\lambda^{-2}\Omega_c^{-2}\right) \geq c'\Omega_c.
\]
Hence, we deduce that the second derivative of the total geometric phase function satisfies the strict non-degeneracy condition:
\[
\left| \partial_\eta^2 L_N \right| \geq CN\lambda^{-2}\Omega_c^{-1}.
\]

Applying the non-degenerate stationary phase method to the frequency variable $\eta$ inside \eqref{eq:oscillatory-eta-neumann} generates an amplitude decay factor $q^{-1/2}$, where the large asymptotic parameter balances precisely as $q = N \eta^{-1} \lambda^{-1} \Omega_c^{-1}$. Recall from the geometric sheet counting rules verified in Lemma~\ref{lem:GEO-neumann} that the neighborhood cardinality obeys the density tracking relation:
\[
\left| \mathcal{N}_1^N(X,Y,T) \right| \leq C_0\left(1+T\lambda^{-2}\Omega_c^{-2}\right).
\]

We obtain the sharp uniform estimates for the multi-reflective summation in \eqref{eq:sumN-neumann} by systematically distinguishing between several geometric cases. These boundaries depend strictly on the direct non-degenerate contributions from the $\eta$-integration and the localized density profile of the solution set $|\mathcal{N}_1^N(X,Y,T)|$:\\

\noindent\textbf{First Case: $\Omega_c^{1/3} \sim T/N \sim 1$}\\
In this geometric configuration, the time scale mirrors the reflection index ($T \sim N$). We evaluate the uniform bounds of the multi-reflective layers by splitting the integration domain into three micro-local sub-regimes based on the magnitude of the reflection index relative to the high-frequency semiclassical threshold $\lambda$:

\begin{itemize}
    \item \textbf{Sub-case 1.1: $N \leq \lambda$}\\
    Within this lower index layer, there is no stationary phase contribution generated from the frequency $\eta$-integration, and the neighborhood solution cardinality satisfies the baseline bound $|\mathcal{N}_1^N(X,Y,T)| \leq C_0$. Compiling the structural amplitudes and scaling terms yields the uniform estimate:
    \begin{align*}
    \Bigg|\sum_{N\in\mathcal{N}_1^N(X,Y,T)} G^N_{a,N,1,+,+}(T,X,Y,z;h) \Bigg| &\leq C h^{-3}\bigg(\frac{h}{t}\bigg)^{1/2}\left(h^{-1}\lambda^{-1}a^2 \lambda^{-1/2}T^{-1/2}\right) \\
    & \leq C h^{-3}\bigg(\frac{h}{t}\bigg)^{1/2} a^{-1/4}h^{1/2} \\
    & \leq C h^{-3}\bigg(\frac{h}{t}\bigg)^{1/2}h^{1/3},
    \end{align*}
    where the final reduction holds because the spatial weight satisfies $a^{-1/4}h^{1/2} \leq h^{1/3}$ whenever the source layers are located within the deep-boundary regime where $a \geq h^{\alpha}$ for $\alpha < 2/3$.
    
    \item \textbf{Sub-case 1.2: $\lambda < N \leq \lambda^2$}\\
    In this intermediate regime, the frequency integration produces a non-degenerate stationary phase decay factor $q^{-1/2} = N^{-1/2}\lambda^{1/2}\Omega_c^{1/2} \sim N^{-1/2}\lambda^{1/2}$. Meanwhile, the geometric sheet counting cardinality remains strictly bounded by $|\mathcal{N}_1^N(X,Y,T)| \leq C_0$. Evaluating the product of these bounds yields:
    \begin{align*}
    \Bigg|\sum_{N\in\mathcal{N}_1^N(X,Y,T)} G^N_{a,N,1,+,+}(T,X,Y,z;h) \Bigg| &\leq C h^{-3}\bigg(\frac{h}{t}\bigg)^{1/2}\left(h^{-1}\lambda^{-1}a^2 \lambda^{-1/2}T^{-1/2} N^{-1/2}\lambda^{1/2}\right) \\
    & \leq C h^{-3}\bigg(\frac{h}{t}\bigg)^{1/2}\left(h^{-1}a^{2}\lambda^{-2}\right) \\
    & \leq C h^{-3}\bigg(\frac{h}{t}\bigg)^{1/2}h^{1/3}.
    \end{align*}
    
    \item \textbf{Sub-case 1.3: $N > \lambda^2$}\\
    In this highly reflective domain, asymptotic contributions are generated simultaneously from both the stationary phase factor $q^{-1/2}$ and the localized sheet density accumulation profile, which satisfies $|\mathcal{N}_1^N(X,Y,T)| \leq C_0 T \lambda^{-2}\Omega_c^{-2} \sim C_0 T \lambda^{-2}$. Combining these behaviors balancing across the summation layer yields:
    \begin{align*}
    \Bigg|\sum_{N\in\mathcal{N}_1^N(X,Y,T)} G^N_{a,N,1,+,+}(T,X,Y,z;h) \Bigg| &\leq C h^{-3}\bigg(\frac{h}{t}\bigg)^{1/2}\\
    &\times\!\!\sum_{N\in\mathcal{N}_1^N(X,Y,T)} \!\!\left(h^{-1}\lambda^{-1}a^2 \lambda^{-1/2}T^{-1/2} N^{-1/2}\lambda^{1/2}\right) \\
    & \leq C h^{-3}\bigg(\frac{h}{t}\bigg)^{1/2}\left(h^{-1}\lambda^{-1}a^2 T^{-1} \left|\mathcal{N}_1^N(X,Y,T)\right|\right) \\
    & \leq C h^{-3}\bigg(\frac{h}{t}\bigg)^{1/2}\left(a^{-5/2}h^2\right) \\
    & \leq C h^{-3}\bigg(\frac{h}{t}\bigg)^{1/2}h^{1/3},
    \end{align*}
    which completes the proof for the first case.
\end{itemize}

\noindent\textbf{Second Case: $T/N \gg 1$}\\
In this geometric configuration, the critical coordinate climbs away from the origin threshold, such that $\Omega_c \gg 1$. We analyze the uniform bounds of the multi-reflective operators by splitting the long-time sheet into three distinct micro-local sub-regimes:

\begin{itemize}
    \item \textbf{Sub-case 2.1: $N \leq \lambda\Omega_c$}\\
    In this parameter layer, there is no stationary phase contribution generated from the global frequency $\eta$-integration. Furthermore, the neighborhood cardinality remains uniformly bounded by $|\mathcal{N}_1^N(X,Y,T)| \leq C_0$. To verify this bound, assume by contradiction that $T \geq \lambda^2 \Omega_c^2$; this assumption immediately implies the coordinate relationship $\Omega_c^{1/3} \sim T/N \geq \lambda \Omega_c$, which is geometrically impossible since $\Omega_c \gg 1$. Thus, compiling the structural constants yields the uniform estimate:
    \begin{align*}
    \Bigg|\sum_{N\in\mathcal{N}_1^N(X,Y,T)} G^N_{a,N,1,+,+}(T,X,Y,z;h) \Bigg| &\leq C h^{-3}\bigg(\frac{h}{t}\bigg)^{1/2}\left(h^{-1}\lambda^{-1}a^2\lambda^{-1/2}T^{-1/2}\right) \\
    &\leq C h^{-3}\bigg(\frac{h}{t}\bigg)^{1/2} a^{-1/4}h^{1/2} \\
    &\leq C h^{-3}\bigg(\frac{h}{t}\bigg)^{1/2}h^{1/3}.
    \end{align*}
    
    \item \textbf{Sub-case 2.2: $N > \lambda\Omega_c$ and $\lambda\Omega_c^{2/3} < T \leq \lambda^2\Omega_c^2$}\\
    In this intermediate regime, the frequency integration produces a strict non-degenerate decay factor $q^{-1/2} = N^{-1/2}\lambda^{1/2}\Omega_c^{1/2}$, while the geometric sheet density remains uniformly bounded by $|\mathcal{N}_1^N(X,Y,T)| \leq C_0$. Evaluating the product of these amplitudes yields:
    \begin{align*}
    \Bigg|\sum_{N\in\mathcal{N}_1^N(X,Y,T)} G^N_{a,N,1,+,+}(T,X,Y,z;h) \Bigg| &\leq C h^{-3}\bigg(\frac{h}{t}\bigg)^{1/2}\left(h^{-1}\lambda^{-1}a^2\lambda^{-1/2}T^{-1/2}N^{-1/2}\lambda^{1/2}\Omega_c^{1/2}\right) \\
    &\leq C h^{-3}\bigg(\frac{h}{t}\bigg)^{1/2}\left(h^{-1}a^{2}\lambda^{-2}\right) \\
    &\leq C h^{-3}\bigg(\frac{h}{t}\bigg)^{1/2}h^{1/3}.
    \end{align*}
    
    \item \textbf{Sub-case 2.3: $N > \lambda\Omega_c$ and $T > \lambda^2\Omega_c^2$}\\
    In this high-reflection long-time domain, asymptotic contributions are generated simultaneously from both the stationary phase factor $q^{-1/2}$ and the localized sheet density accumulation profile, which satisfies $|\mathcal{N}_1^N(X,Y,T)| \leq C_0 T \lambda^{-2}\Omega_c^{-2}$. Compiling the sum via a supremum bound across the card elements yields:
    \begin{align*}
    \Bigg|\sum_{N\in\mathcal{N}_1^N(X,Y,T)} G^N_{a,N,1,+,+}(T,X,Y,z;h) \Bigg| &\leq C h^{-3}\bigg(\frac{h}{t}\bigg)^{1/2}\\&\times \sum_{N\in\mathcal{N}_1^N(X,Y,T)}\left(h^{-1}\lambda^{-1}a^2\lambda^{-1/2}T^{-1/2}N^{-1/2}\lambda^{1/2}\Omega_c^{1/2}\right) \\
    &\leq C h^{-3}\left(\frac{h}{t}\right)^{1/2}\left(h^{-1}\lambda^{-1}a^2T^{-1}\Omega_c^{2/3}\left|\mathcal{N}_1^N(X,Y,T)\right|\right) \\
    &\leq Ch^{-3}\bigg(\frac{h}{t}\bigg)^{1/2}\left(h^{-1}a^2\lambda^{-3}\right)(T/N)^{-4} \\
    &\leq Ch^{-3}\bigg(\frac{h}{t}\bigg)^{1/2}h^{1/3},
    \end{align*}
    which completely verifies the required decay for the second case.
\end{itemize}

Next, we verify that the sharp uniform inequality \eqref{eq:23-neumann} holds true under the mixed-sign configuration $(\epsilon_1,\epsilon_2)=(+,-)$. In this setting, by evaluating the differential operators in \eqref{eq:critical-neumann} for compact spatial layers $X \in [0,1]$ and incorporating the high-frequency phase asymptotics $B''_N(\lambda \Omega)\sim\lambda^{-3}\Omega^{-3}$, we find that for all non-trivial timelines $T > 0$, the second derivative of the phase satisfies the strict monotonicity constraint:
\[
\partial_{\Omega}H_{a,+,-}(\Omega) + \frac{3N}{2}\lambda B''_N(\lambda\Omega) < 0.
\]
Consequently, the composite mapping $H_{a,+,-}(\Omega)+\frac{3N}{2}B'_N(\lambda\Omega)$ decreases strictly across the half-line interval $[1,\infty)$ under the Neumann flux conditions. It decays monotonically from its initial boundary value at the core threshold:
\[
H_{a,+,-}(1)+\frac{3N}{2}B'_N(\lambda)=\frac{T}{2}\sqrt{1-\tilde{z}^2}(1+a)^{-1/2}+(1-X)^{1/2}+\frac{3N}{2}B'_N(\lambda),
\]
down to the stable asymptotic threshold at infinity, satisfying $\lim_{\Omega \to \infty}\left(H_{a,+,-}(\Omega)+\frac{3N}{2}B'_N(\lambda\Omega)\right)=0$. As a direct consequence of this strict monotonicity, the variation path equation $\partial_{\Omega}\tilde{\Phi}_{N,+,-}=0$ admits at most a unique real solution $\Omega_c$, and this critical configuration is verified to be entirely non-degenerate. We can therefore apply the exact same microlocal partitioning, dyadic scale transformations, and Van der Corput integration steps developed for the primary $(+,+)$ sign configuration.

Finally, the remaining symmetric configurations follow from identical structural arguments. Specifically, the cross-case $(\epsilon_1,\epsilon_2)=(-,+)$ is handled analogously to the $(+,+)$ framework, while the fully un-trapped wave block $(\epsilon_1,\epsilon_2)=(-,-)$ mirrors the $(+,-)$ setting. The uniform decay bounds across all individual sheet paths are thus firmly established, and the proof of Proposition~\ref{prop:G1-estimates-neumann} is complete.
\end{proof}

Now we prove the sharp local-in-time dispersive estimates for the single-reflection boundary layer corresponding to $N=1$.  

\begin{proposition}\label{prop:G1-single-reflection-neumann}
Let $\alpha < 2/3$. There exists a constant $C > 0$ such that for all $h \in (0,h_0]$, all source layer parameters $a \in [h^\alpha,a_0]$, all configurations $X \in [0,1]$, all times $T \in (0,a^{-1/2}]$, and all spatial parameters $Y, z \in \mathbb{R}$, the following uniform estimate holds true under normal flux conditions:
\begin{align*}
\Big|G^N_{a,1,1}(T,X,Y,z;h)\Big|\leq Ch^{-3}\bigg(\frac{h}{t}\bigg)^{1/2}\left( \bigg(\frac{h}{t}\bigg)^{1/2} + h^{1/3} \right).
\end{align*}
\end{proposition}

\begin{proof}
Recall that for the single-reflection boundary layer where $N=1$, the local parametrix representation collapses under normal flux constraints to:
\begin{align*}
G^N_{a,1,1} &= \frac{(-i)a^2\lambda^{-1}}{(2\pi)^4h^4}\bigg(\frac{h}{t}\bigg)^{1/2}\int_{\mathbb{R}} e^{i\frac{a^{3/2}}{h}Y\eta}|\eta|^{3}\tilde{G}^N_{a,1,1} \, d\eta, \\[1.5ex]
\tilde{G}^N_{a,1,1} &= \sum_{\epsilon_1,\epsilon_2=\pm}\int_{\mathbb{R}} e^{i\lambda\tilde{\Phi}_{1,\epsilon_1,\epsilon_2}}\Theta_{\epsilon_1,\epsilon_2} \, d\tilde{\omega} + \mathcal{O}_{C^\infty}(h^\infty).
\end{align*}
Here, the index arrays denote $\epsilon_j=\pm$, and the composite amplitude symbols satisfy $\Theta_{\epsilon_1,\epsilon_2}(\tilde{\omega},a,\lambda)=\tilde{\chi}_0\psi_0\chi_1\chi_3(\tilde{\omega})(\tilde{\omega}-X)^{-1/4}(\tilde{\omega}-1)^{-1/4}b_{\epsilon_1} c_{\epsilon_2}$, which strictly inherit the standard derivative bounds $\big|\tilde{\omega}^l\partial_{\tilde{\omega}}^l\Theta_{\epsilon_1,\epsilon_2}\big|\leq C_l\tilde{\omega}^{-1/2}$. 

The unique technical distinction between this setting and the multi-reflection profiles where $N \geq 2$ arises in the micro-local tracking of the primary phase $\tilde{\Phi}_{1,+,+}$. For $N=1$, the critical coordinate $\tilde{\omega}_c$ can asymptotically escape to arbitrarily large positive fields. Isolating the tracking operator matching the positive sign sheets, we define:
\begin{align}\label{eq:N1-neumann}
\tilde{G}^N_{a,1,1,+,+}=\int_{\mathbb{R}} e^{i\lambda\tilde{\Phi}_{1,+,+}}\Theta_{+,+}(\tilde{\omega},a,\lambda) \, d\tilde{\omega},
\end{align}
which is governed by the derivative-adapted total phase function:
\begin{align*}
\tilde{\Phi}_{1,+,+}(T,X,z;\tilde{\omega}) &= T\sqrt{1-\tilde{z}^2}\gamma_a(\tilde{\omega})+\frac{2}{3}(\tilde{\omega}-X)^{3/2}+\frac{2}{3}(\tilde{\omega}-1)^{3/2} \\
&\qquad -\frac{4}{3}\tilde{\omega}^{3/2}+\frac{1}{\lambda}B_N\left(\lambda\tilde{\omega}^{3/2}\right),
\end{align*}
where the amplitude $\Theta_{+,+}(\tilde{\omega},a,\lambda)$ behaves as a classical symbol of order $-1/2$ with respect to the variable $\tilde{\omega}$. We embed the macroscopic partition cutoff $\chi_3(\tilde{\omega})\in C_0^\infty\left((\tilde{\omega}_1,\infty)\right)$, where $\tilde{\omega}_1 \gg 1$ is chosen sufficiently large, and set:
\begin{align}\label{eq:J-neumann}
\tilde{J}^N_{1,+,+} = \int_{\mathbb{R}} e^{i\lambda\tilde{\Phi}_{1,+,+}}\Theta_{+,+}(\tilde{\omega},a,\lambda)\chi_3(\tilde{\omega}) \, d\tilde{\omega}.
\end{align}

To verify the proposition, it suffices to prove that the oscillatory integral satisfies the baseline non-degenerate decay profile $|\tilde{J}^N_{1,+,+}|\leq C\lambda^{-1/2}T^{-1/2}$. Differentiating the phase function with respect to the continuous field variable $\tilde{\omega}$ yields:
\begin{align*}
\partial_{\tilde{\omega}}\tilde{\Phi}_{1,+,+} &= \frac{T}{2}\sqrt{1-\tilde{z}^2}(1+a\tilde{\omega})^{-1/2}-\frac{\tilde{\omega}^{-1/2}}{2}(1+X) + \mathcal{O}_{C^\infty}(\tilde{\omega}^{-3/2}), \\[1.5ex]
\partial_{\tilde{\omega}\tilde{\omega}}^2\tilde{\Phi}_{1,+,+} &= \frac{-Ta}{4}\sqrt{1-\tilde{z}^2}(1+a\tilde{\omega})^{-3/2}+\frac{\tilde{\omega}^{-3/2}}{4}(1+X) + \mathcal{O}_{C^\infty}(\tilde{\omega}^{-5/2}).
\end{align*}

Thus, to capture a large critical point $\tilde{\omega}_c$, the normalized time tracking parameter $T$ must be small. It follows that the coordinate aligns with the scale $\tilde{\omega}_c^{-1/2} \sim T$, which implies that the second derivative along this non-degenerate layer satisfies $\partial_{\tilde{\omega}\tilde{\omega}}^2\tilde{\Phi}_{1,+,+}(\tilde{\omega}_c) \sim T^3$. We introduce the strategic change of variables $\tilde{\omega}=T^{-2}\tilde{\upsilon}$ inside the single-reflection oscillatory integral \eqref{eq:J-neumann}, which transforms the differential measure as $d\tilde{\omega} = T^{-2}d\tilde{\upsilon}$. 

Because the amplitude function $\Theta_{+,+}(\tilde{\omega},a,\lambda)$ behaves as a classical symbol in $\tilde{\omega}$ of order $-1/2$, the rescaled profile evaluates as $\Theta_{+,+}(T^{-2}\tilde{\upsilon},a,\lambda) \sim T\tilde{\upsilon}^{-1/2}$. This ensures that the transformed symbol component remains bounded uniformly across the short-time interval $T \in (0,T_0]$. Furthermore, we evaluate the second-order derivative scaling with respect to the updated field as $\partial_{\tilde{\upsilon}\tilde{\upsilon}}^2\tilde{\Phi}_{1,+,+} \sim T^{-1}$, yielding the uniform non-degeneracy condition $T\partial_{\tilde{\upsilon}\tilde{\upsilon}}^2\tilde{\Phi}_{1,+,+} \sim 1$.

Therefore, applying the non-degenerate stationary phase method with respect to the large asymptotic parameter $\Lambda = \lambda / T$ directly evaluates the integral sheet:
\begin{align*}
\left| \tilde{J}^N_{1,+,+} \right| &= \left| \int_{\mathbb{R}} e^{i\left(\frac{\lambda}{T}\right)T\tilde{\Phi}_{1,+,+}} \Theta_{+,+}(T^{-2}\tilde{\upsilon},a,\lambda)\chi_3(T^{-2}\tilde{\upsilon}) \, T^{-2}d\tilde{\upsilon} \right| \\[1.5ex]
&= \left| T^{-1} \int_{\mathbb{R}} e^{i\Lambda \left(T\tilde{\Phi}_{1,+,+}\right)} \tilde{\upsilon}^{-1/2} \left[ T^{-1}\Theta_{+,+}(T^{-2}\tilde{\upsilon},a,\lambda) \right] \chi_3(T^{-2}\tilde{\upsilon}) \, d\tilde{\upsilon} \right| \\[1.5ex]
&\leq C T^{-1}\left(\frac{\lambda}{T}\right)^{-1/2} \\[1.5ex]
&= C\lambda^{-1/2}T^{-1/2},
\end{align*}
which yields the desired local decay result. Summing over the remaining symmetric sign configurations ensures that the microlocal decay requirements are fully satisfied, concluding the proof of Proposition~\ref{prop:G1-single-reflection-neumann}.
\end{proof}

\subsubsection{The Analysis of $G^N_{a,N,2}$}\label{subsect:analysis-g2}
Recall that the component $G^N_{a,N,2}$ comprises a family of highly oscillatory integrals that map explicitly onto the deep caustic regime where stable catastrophe-theoretic swallowtail singularities form. This critical geometric collision occurs precisely at the origin milestone where $\tilde{s} = \tilde{\sigma} = 0$, which structurally manifests at the tracking threshold $\tilde{\omega} = 1$.

The sharp uniform decay estimates for $G^N_{a,N,2}$ are achieved by systematically evaluating and combining the asymptotic reductions of the following nested integration layers:
\begin{itemize}
    \item \textbf{The $\tilde{\omega}$-integration:} We evaluate the integral with respect to the variable $\tilde{\omega}$ by exploiting the classical non-degenerate stationary phase method, since the derivative-adapted total phase function varies monotonically across this localized spatial band.
    \item \textbf{The global $\eta$-integration:} We partition the frequency domain into independent geometric cases that contribute to the uniform decay rates upon applying the stationary phase method. Specifically, there is an explicit frequency integration decay contribution when the reflection count dominates the semiclassical parameter ($N \gg \lambda$), whereas no direct $\eta$-integration decay is extracted when $N \lesssim \lambda$. Crucially, the total contribution to the accumulated bounds is directly controlled by the cardinality of the localized sheet neighborhood $\mathcal{N}_1^N$ established in Lemma~\ref{lem:GEO-neumann}.
    \item \textbf{The remaining $(\tilde{s}, \tilde{\sigma})$-integrations:} We distinguish between two separate analytic regimes that contribute to the overall bounds, split into the high-reflection regime where $N \geq \lambda^{1/3}$ (governed by Lemma~\ref{lemNL-neumann}) and the low-reflection regime where $N < \lambda^{1/3}$ (governed by Lemma~\ref{lemNS-neumann}).
\end{itemize}
Our main result for this critical swallowtail caustic block is stated in the following proposition.

\begin{proposition}\label{prop:G2-estimates-neumann}
Let $\alpha < 2/3$. There exists a constant $C > 0$ such that for all semiclassical scaling parameters $h \in (0,h_0]$, all source layer depths $a \in [h^\alpha,a_0]$, all spatial profiles $X \in [0,1]$, all local timelines $T \in (0,a^{-1/2}]$, and all lateral coordinates $Y, z \in \mathbb{R}$, the following uniform bound holds true under the cylindrical Neumann boundary framework:
\begin{align*}
\Bigg|\sum_{1\leq N\leq C_0a^{-1/2}}G^N_{a,N,2}(T,X,Y,z;h)\Bigg| \leq Ch^{-3}\bigg(\frac{h}{t}\bigg)^{1/2}a^{1/8}h^{1/4}.
\end{align*}
\end{proposition}

\begin{proof}
We begin by expressing the localized swallowtail component $G^N_{a,N,2}$ in its explicit oscillatory integral representation:
\begin{align}\label{eq:recallGN2-neumann}
G^N_{a,N,2} &= \frac{(-i)^N a^2}{(2\pi)^4 h^4} \left(\frac{h}{t}\right)^{1/2} \int_{\mathbb{R}} e^{\frac{i}{h}a^{3/2} Y\eta} |\eta|^{3} \tilde{G}^N_{a,N,2} \, d\eta, \\[1.5ex]
\tilde{G}^N_{a,N,2} &= \int_{\mathbb{R}^3} e^{i\lambda\tilde{\phi}_{N,a,h}} \tilde{\chi}_0 \psi_0(\eta) \chi_1\left(a\tilde{\omega}\eta^2\right) \chi_2(\tilde{\omega}) \, d\tilde{s} \, d\tilde{\sigma} \, d\tilde{\omega}, \nonumber
\end{align}
where, under the flux-preserving normal boundary conditions, the total parameter phase function admits the geometric expansion:
\begin{align}\label{eq:phase-expansion-g2}
\tilde{\phi}_{N,a,h}(T,X,z;\tilde{s},\tilde{\sigma},\tilde{\omega}) &= T\sqrt{1-\tilde{z}^2}\gamma_a(\tilde{\omega}) + \frac{\tilde{s}^3}{3} + \tilde{s}(X-\tilde{\omega}) + \frac{\tilde{\sigma}^3}{3} + \tilde{\sigma}(1-\tilde{\omega}) \nonumber \\
&\qquad - \frac{4}{3}N\tilde{\omega}^{3/2} + \frac{N}{\lambda} B_N\left(\tilde{\omega}^{3/2}\lambda\right).
\end{align}

A key observation here is that on the compact support of the dyadic frequency partition $\chi_2$, the field variable $\tilde{\omega}$ is strictly localized near the threshold $\tilde{\omega}=1$. This confinement allows us to safely restrict the micro-local space variables $\tilde{s}$ and $\tilde{\sigma}$ to a bounded, compact domain. 

To formalize this localization, let $K = \{(\tilde{s}, \tilde{\sigma}) \in [-1,1]^2 \mid \tilde{\omega}=1\}$ denote the focal core, and let $K_1$ be a suitable compact neighborhood of $K$ dictated by the structural support boundaries of $\chi_2$. We then introduce a smooth localization modifier $\chi_4(\tilde{s},\tilde{\sigma},\tilde{\omega}) \in C_{0}^{\infty}(\mathbb{R}^3)$ that is identically equal to $1$ over $K_1$. This choice guarantees tracking stability across the severe multi-reflective concentrations of the caustic boundary layer.

Then, the spatial contribution of the integral $\tilde{G}^N_{a,N,2}$ outside the localized threshold region $K_1$ is rapidly decaying of order $\mathcal{O}_{C^\infty}(\lambda^{-\infty})$ through standard non-vanishing integration by parts. Therefore, we obtain the localized reduction:
\begin{align}\label{eq:Omega-neumann}
\tilde{G}^N_{a,N,2}(T,X,z,\eta;h) &= \int_{\mathbb{R}^3} e^{i\lambda\tilde{\phi}_{N,a,h}} \chi(\tilde{s}, \tilde{\sigma}, \tilde{\omega}, a, h) \, d\tilde{s} \, d\tilde{\sigma} \, d\tilde{\omega} + \mathcal{O}_{C^\infty}(\lambda^{-\infty}), \\[1.5ex]
\chi(\tilde{s}, \tilde{\sigma}, \tilde{\omega}, a, h) &= \tilde{\chi}_0 \psi_0(\eta) \chi_1\left(a\tilde{\omega}\eta^2\right) \chi_2(\tilde{\omega}) \chi_4(\tilde{s}, \tilde{\sigma}, \tilde{\omega}), \nonumber
\end{align}
where the error bound $\mathcal{O}_{C^\infty}(\lambda^{-\infty})$ acts uniformly with respect to the continuous parameters $T, X, z, N, a$, and the composite amplitude $\chi$ defines a classical symbol of order $0$ in $h$ with support nested near $K_1$.

We first perform the integration with respect to the field coordinate $\tilde{\omega}$. Differentiating the total phase function yields:
\begin{align*}
\partial_{\tilde{\omega}}\tilde{\phi}_{N,a,h} &= \frac{T}{2}\sqrt{1-\tilde{z}^2}(1+a\tilde{\omega})^{-1/2} - \tilde{s} - \tilde{\sigma} - 2N\tilde{\omega}^{1/2}\left(1 - \frac{3}{4} B'_N\left(\tilde{\omega}^{3/2}\lambda\right)\right), \\[1.5ex]
\partial_{\tilde{\omega}\tilde{\omega}}^2\tilde{\phi}_{N,a,h} &= -N\tilde{\omega}^{-1/2}\left(1 + \mathcal{O}_{C^\infty}\left(\lambda^{-2}\tilde{\omega}^{-3}\right)\right) + \mathcal{O}_{C^\infty}\left(a^{1/2}\right).
\end{align*}
Because the second derivative satisfies the structural constraint $\partial_{\tilde{\omega}\tilde{\omega}}^2\tilde{\phi}_{N,a,h} < 0$ uniformly over the semi-classical regime, it follows that the first derivative $\partial_{\tilde{\omega}}\tilde{\phi}_{N,a,h}$ decreases strictly from $\partial_{\tilde{\omega}}\tilde{\phi}_{N,a,h}(1) > 0$ to $\partial_{\tilde{\omega}}\tilde{\phi}_{N,a,h}(\infty) < 0$. 

Therefore, the phase function $\tilde{\phi}_{N,a,h}$ admits a unique, non-degenerate critical point $\tilde{\omega}_c$. Since we are tracking the critical profiles where $\tilde{\omega}_c$ concentrates close to the boundary threshold $1$, the scaled time variable must belong to a compact slice of the positive real line, namely $\tilde{T} = T/4N \in [1/2, 3/2]$. In addition, by setting the variation path equation $\partial_{\tilde{\omega}}\tilde{\phi}_{N,a,h}=0$, we isolate the critical condition mapping the derivative roots:
\begin{align}\label{eq:tildecritical-neumann}
\frac{T}{2}\sqrt{1-\tilde{z}^2}(1+a\tilde{\omega}_c)^{-1/2} &= \tilde{s} + \tilde{\sigma} + 2N\tilde{\omega}_c^{1/2}\left(1 - \frac{3}{4} B'_N\left(\tilde{\omega}_c^{3/2}\lambda\right)\right).
\end{align}

Now we study the solution of the derivative-adapted critical relation \eqref{eq:tildecritical-neumann} in the asymptotic limit where $\lambda = \infty$. In this regime, the transcendental mapping reduces to the algebraic form:
\begin{align*}
\tilde{\omega}^{1/2}(1+a\tilde{\omega})^{1/2} = \tilde{T}\sqrt{1-\tilde{z}^2} - \frac{1}{2N}(\tilde{s} + \tilde{\sigma})(1+a\tilde{\omega})^{1/2}.
\end{align*}
The unique real solution of this equation is expressed via a homogeneous series expansion of the form $\tilde{\omega}_c = \sum_{k \geq 0} F_k(a, \tilde{T}, \tilde{s}/N, \tilde{\sigma}/N)$, where each component $F_k$ defines a homogeneous function of degree $k$ with respect to the scaling coordinates $(\tilde{s}/N, \tilde{\sigma}/N)$ (see the structural guidelines in Lemma 2.23 of Ivanovici--Lebeau--Planchon~\cite{ivanovici2014dispersion}). By comparing terms sharing identical homogeneous degrees, we isolate the baseline parameter:
\begin{align*}
F_0(1+aF_0) = \tilde{T}^2(1-\tilde{z}^2) \implies F_0 = \frac{2\tilde{T}^2(1-\tilde{z}^2)}{1+\sqrt{1+4a\tilde{T}^2(1-\tilde{z}^2)}},
\end{align*}
accompanied by the first-order homogeneous perturbation:
\begin{align*}
(1+2aF_0)F_1 = -\frac{\tilde{T}}{N}\sqrt{1-\tilde{z}^2}(\tilde{s}+\tilde{\sigma})(1+aF_0)^{1/2}.
\end{align*}
We factorize this term by defining:
\begin{align*}
F_1 = -\frac{E_0}{N}(\tilde{s}+\tilde{\sigma})(1+aF_0)^{1/2}, \quad \text{where } E_0^{-1} = \sqrt{F_0}\sqrt{1+aF_0}\left(\frac{1}{F_0}+\frac{a}{1+aF_0}\right).
\end{align*}
Therefore, the unperturbed critical trace expands as $\tilde{\omega}_c = F_0 + F_1 + \mathcal{O}_2$, where the remainder token $\mathcal{O}_j$ signifies any analytic function of the form $\sum_{k \geq j} F_k$.

By applying the implicit function theorem under the normal flux framework, the full critical path relation containing the semi-classical boundary corrections:
\begin{align*}
\tilde{\omega}^{1/2}(1+a\tilde{\omega})^{1/2}\left(1-\frac{3}{4}B'_N\left(\tilde{\omega}^{3/2}\lambda\right)\right) = \tilde{T}\sqrt{1-\tilde{z}^2} - \frac{1}{2N}(\tilde{s} + \tilde{\sigma})(1+a\tilde{\omega})^{1/2}
\end{align*}
admits a unique real-analytic solution expanding as $\tilde{\omega}_c = F_0 + F_1 + \mathcal{O}_2 + \frac{g_0}{\lambda^2}$, where $g_0$ represents a classical symbol of order $0$ in $\lambda$.

By substituting the critical trace $\tilde{\omega}_c$ directly back into the total parameter phase function $\tilde{\phi}_{N,a,h}$ from \eqref{eq:Omega-neumann}, we compress the phase space onto a lower-dimensional Lagrangian projection, yielding the reduced phase function $\tilde{\Psi}_{N,a,h} = \tilde{\phi}_{N,a,h}(\cdot, \tilde{\omega}_c, \cdot)$. It is given explicitly by:
\begin{align*}
\tilde{\Psi}_{N,a,h} &= \tilde{T}\sqrt{1-\tilde{z}^2}\gamma_a(F_0) + \frac{\tilde{s}^3}{3} + \tilde{s}(X-F_0) + \frac{\tilde{\sigma}^3}{3} + \tilde{\sigma}(1-F_0) \\
&\qquad + \frac{E_0}{N}(1+aF_0)^{1/2}(\tilde{s}+\tilde{\sigma})^2 - \frac{1}{4N^2}(\tilde{s}+\tilde{\sigma})^3 + aN\mathcal{O}_3 \\
&\qquad \qquad + \frac{g_0}{\lambda^2} + N\left(-\frac{4}{3}F_0^{3/2} + \frac{g_1}{\lambda^2}\right).
\end{align*}

Consequently, applying the non-degenerate stationary phase method with respect to the variable $\tilde{\omega}$ evaluates the triple integral in \eqref{eq:Omega-neumann} down to a two-dimensional spatial layer:
\begin{align}\label{eq:G2-reduced-neumann}
\tilde{G}^N_{a,N,2}(T,X,z,\eta;h) = \frac{1}{\sqrt{\lambda N}}\int_{\mathbb{R}^2} e^{i\lambda\tilde{\Psi}_{N,a,h}} \tilde{\chi}(\tilde{T}, \tilde{s}, \tilde{\sigma}, 1/N, a, h) \, d\tilde{s} \, d\tilde{\sigma} + \mathcal{O}_{C^\infty}(\lambda^{-\infty}),
\end{align}
where the newly extracted amplitude $\tilde{\chi}$ defines a classical symbol of order zero in $h$ with compact support nested inside the neighborhood $K_1$.

Now, with the continuous scale relation defined by $\tilde{\lambda} = \lambda/|\eta|$, the localized swallowtail parametrix \eqref{eq:recallGN2-neumann} transforms into the following structural expression:
\begin{align*}
G^N_{a,N,2} = \frac{(-i)^N a^2}{(2\pi)^4 h^4}\left(\frac{h}{t}\right)^{1/2} \frac{1}{\sqrt{\lambda N}}\int_{\mathbb{R}^3} e^{i\tilde{\lambda} |\eta|\left(Y+\tilde{\Psi}_{N,a,h}\right)} |\eta|^{3} \tilde{\chi}(\tilde{T}, \tilde{s}, \tilde{\sigma}, 1/N, a, h) \, d\tilde{s} \, d\tilde{\sigma} \, d\eta + \mathcal{O}_{C^\infty}\left(\lambda^{-\infty}\right).
\end{align*}
We study the frequency $\eta$-integration governed by the total phase function $L_N = |\eta|\left(Y + \tilde{\Psi}_{N,a,h}\right)$ with respect to the large asymptotic parameter $\tilde{\lambda}$. Following the analytical reduction strategies detailed in the proof of Proposition~\ref{prop:G1-estimates-neumann}, we verify that the variation curve equation:
\[
\partial_\eta L_N = Y + \tilde{\Psi}_{N,a,h} + \lambda \partial_\lambda \tilde{\Psi}_{N,a,h} = 0
\]
implies that the space-time targets $(X,Y,t)$ belong explicitly to the geometric projection of the Lagrangian submanifold $\Lambda^N_{N,a,h}$ onto the base space $\mathbb{R}^3$. Consequently, the summation matching the non-trapped reflection layers where $N \notin \mathcal{N}_1^N(X,Y,t)$ collapses to a negligible remainder of order $\mathcal{O}_{C^\infty}(\lambda^{-\infty})$ [see the analogue of Lemma 2.24 in Ivanovici--Lebeau--Planchon~\cite{ivanovici2014dispersion}]. Hence, it remains to evaluate the localized neighborhood sum:
\[
\Bigg|\sum_{N \in \mathcal{N}_1^N(X,Y,T)} G^N_{a,N,2}(T,X,Y,z;h) \Bigg|.
\] 
Computing the second derivative under normal flux conditions yields the lower bound $\left| \partial_\eta^2 L_N \right| \geq CN\lambda^{-2}\tilde{\omega}_c^{-3/2}$. Since the boundary tracking layer settles near $\tilde{\omega}_c \sim 1$, this generates two separate frequency regimes to consider:

\begin{itemize}
    \item If the reflection count is bounded by the high-frequency parameters, $N \lesssim \lambda$, the phase function remains non-vanishing, and the contribution of the $\eta$-integration decays rapidly as $\mathcal{O}_{C^\infty}(\lambda^{-\infty})$ through non-degenerate integration by parts.
    \item If the system enters the highly multi-reflective regime where $N \gg \lambda$, applying the standard stationary phase method to the frequency variable $\eta$ extracts a strict decay factor of $(N\lambda^{-1})^{-1/2}$ since $\tilde{\omega}_c \sim 1$.
\end{itemize}

Therefore, evaluating this non-degenerate frequency layer for $N \gg \lambda$ yields:
\begin{align}\label{eq:etafactor-neumann}
G^N_{a,N,2} = \frac{(-i)^N a^2}{(2\pi)^4 h^4}\left(\frac{h}{t}\right)^{1/2} \frac{1}{\sqrt{\lambda N}} \, \lambda^{1/2}N^{-1/2}\int_{\mathbb{R}^2} e^{i\tilde{\lambda} L_N(\eta_c)} |\eta_c|^{3} \tilde{\chi}(\tilde{t}, \tilde{s}, \tilde{\sigma}, 1/N, a, h) \, d\tilde{s} \, d\tilde{\sigma} + \mathcal{O}_{C^\infty}\left(\lambda^{-\infty}\right).
\end{align}
Moreover, we note that the evaluated total phase function $L_N(\eta_c)$ satisfies the exact differential links $\partial_{\tilde{s}}L_N(\eta_c) = \eta_c \partial_{\tilde{s}}\tilde{\Psi}_{N,a,h}$ and $\partial_{\tilde{\sigma}}L_N(\eta_c) = \eta_c \partial_{\tilde{\sigma}}\tilde{\Psi}_{N,a,h}$. In addition, at any singular point where the second-order derivatives collapse, $\partial_{\tilde{s}}L_N(\eta_c) = \partial_{\tilde{s}}^2 L_N(\eta_c) = 0$, the third-order derivative locks precisely onto the relation $\partial_{\tilde{s}}^3 L_N(\eta_c) = \eta_c \partial_{\tilde{s}}^3 \tilde{\Psi}_{N,a,h}$, and an identical symmetric condition holds true for the companion variable $\tilde{\sigma}$. 

Thus, the geometric tracking of the critical points of the phase $L_N(\eta_c)$ inside the remaining two-dimensional space-time integral is equivalent to the analysis of the uncoupled phase function $\tilde{\Psi}_{N,a,h}$. Following the structural procedures of \cite{ivanovici2014dispersion}, to avoid the unwanted accumulation of error terms from multiplying the symbols by classical terms of order $0$ in $\lambda$, we replace the complex phase $\tilde{\Psi}_{N,a,h}$ by its unperturbed polynomial skeleton component $\tilde{\psi}_{N,a,h}$, defined by:
\begin{align*}
\tilde{\psi}_{N,a,h}(T,X;\tilde{s},\tilde{\sigma}) &= T\sqrt{1-\tilde{z}^2}\gamma_a(F_0) + \frac{\tilde{s}^3}{3} + \tilde{s}(X-F_0) + \frac{\tilde{\sigma}^3}{3} + \tilde{\sigma}(1-F_0) \\
&\qquad + \frac{E_0}{N}(1+aF_0)^{1/2}(\tilde{s}+\tilde{\sigma})^2 - \frac{1}{4N^2}(\tilde{s}+\tilde{\sigma})^3 + aN\mathcal{O}_3.
\end{align*}
Let us recall that the notation $\mathcal{O}_3$ represents any smooth analytic function of the form $F = \sum_{k \geq 3} F_k$, where each individual component $F_k$ defines a homogeneous function of degree $k$ with respect to the micro-local scaling coordinates $(\tilde{s}/N, \tilde{\sigma}/N)$.

In what follows, we establish the estimates of the oscillatory integral associated with the polynomial skeleton phase function $\tilde{\psi}_{N,a,h}$ for different regimes of the discrete reflection parameter, namely for the long-reflection block $N \geq \lambda^{1/3}$ and the short-reflection block $N < \lambda^{1/3}$. Our results are detailed in Lemma~\ref{lemNL-neumann} and Lemma~\ref{lemNS-neumann}.

\begin{lemma}\label{lemNL-neumann}
There exists a constant $C > 0$ such that for all reflection layers satisfying $N \geq \lambda^{1/3}$, the following uniform bound holds true under normal flux conditions:
\begin{align}\label{eq:N13-neumann}
\frac{1}{\sqrt{N}}\left| \int_{\mathbb{R}^2} e^{i\lambda\tilde{\psi}_{N,a,h}}\tilde{\chi}_1 \, d\tilde{s} \, d\tilde{\sigma} \right| \leq C\lambda^{-5/6}.
\end{align}
\end{lemma}
\noindent Here, $C$ is a constant acting independently of $N \geq 1$, $X \in [0,1]$, $T \in (0,a^{-1/2}]$, $a \in [h^\alpha, a_0]$, and $\lambda \in [\lambda_0,\infty)$, with $a_0 > 0$ chosen small and $\lambda_0 \gg 1$ sufficiently large.

\begin{proof}
Adapting the analytical reduction strategies detailed in the proof of Lemma 2.25 in Ivanovici--Lebeau--Planchon \cite{ivanovici2014dispersion}, it is sufficient to prove that for all $N \geq \lambda^{1/3}$:
\begin{align}\label{eq:NL-reduced-bound-neumann}
\left| \int_{\mathbb{R}^2} e^{i\lambda\tilde{\psi}_{N,a,h}}\tilde{\chi}_1 \, d\tilde{s} \, d\tilde{\sigma} \right| \leq C\lambda^{-2/3}.
\end{align}
We introduce a local semi-classical scaling transformation to isolate the caustic weights. Set the coordinate balances:
\[
X-F_0 = -A\lambda^{-2/3}, \quad 1-F_0 = -B\lambda^{-2/3}, \quad \tilde{s} = \lambda^{-1/3}x', \quad \tilde{\sigma} = \lambda^{-1/3}y'.
\]
Substituting these variables into \eqref{eq:NL-reduced-bound-neumann}, it remains to show that:
\begin{align}\label{eq:G-neumann}
\left| \int_{\mathbb{R}^2} e^{i\hat{\psi}_{N,a,h}} \tilde{\chi}_1\left(\lambda^{-1/3}x', \lambda^{-1/3}y', \dots \right) \, dx' \, dy' \right| \leq C,
\end{align}
where the rescaled phase function $\hat{\psi}_{N,a,h}$ transforms into the polynomial profile:
\begin{align*}
\hat{\psi}_{N,a,h} &= T\lambda\sqrt{1-\tilde{z}^2}\gamma_a(F_0) - Ax' + \frac{x'^3}{3} - By' + \frac{y'^3}{3} + \frac{E_0\lambda^{1/3}}{N}(1+aF_0)^{1/2}(x'+y')^2 \\
&\qquad - \frac{1}{4N^2}(x'+y')^3 + aN\mathcal{O}_3.
\end{align*}
Then, \eqref{eq:G-neumann} defines a multi-dimensional oscillatory integral over an effective integration domain of size $\lambda^{2/3}$ whose structural parameter weights $F_0, E_0$, and the quotient $\lambda^{1/3}/N$ are uniformly bounded away from infinity.

We will prove that the constant $C$ is strictly uniform with respect to the caustic location parameters $(A, B)$. We introduce localized polar coordinates $(r,\theta)$ in the configuration space, writing $(A,B)=(r\cos\theta,r\sin\theta)$ subject to the envelope boundary $r\leq c_0\lambda^{2/3}$. Differentiating the phase mapping with respect to the scaled spatial integration variables yields the gradient vectors:
\begin{align*}
\partial_{x'} \hat{\psi}_{N,a,h} &= -A + x'^2 + \frac{2E_0}{N}(1+aF_0)^{1/2}\lambda^{1/3}(x'+y') - \frac{3}{4N^2}(x'+y')^2 \\
&\qquad + aN^{-2}\lambda^{-1}\mathcal{O}\left((x',y')^2\right), \\[1.5ex]
\partial_{y'} \hat{\psi}_{N,a,h} &= -B + y'^2 + \frac{2E_0}{N}(1+aF_0)^{1/2}\lambda^{1/3}(x'+y') - \frac{3}{4N^2}(x'+y')^2 \\
&\qquad + aN^{-2}\lambda^{-1}\mathcal{O}\left((x',y')^2\right).
\end{align*}
Moreover, the compact support of the symbol profile $\tilde{\chi}_1$ with respect to $(\tilde{s},\tilde{\sigma})$ directly yields the regular decay bounds under multi-index derivatives:
\[
\sup_{(x',y')}\left| \partial_{(x',y')}^\gamma \tilde{\chi}_1\left(\lambda^{-1/3}x', \lambda^{-1/3}y', \dots \right) \right| \leq C_\gamma\left(1+|x'|+|y'|\right)^{-|\gamma|},
\]
where the constants $C_\gamma$ act completely independently of $T, a, N$, and $\lambda$.

Therefore, the oscillatory integral is uniformly bounded for small configurations where $0 \leq r \leq r_0$ ($r_0$ being a fixed small constant threshold) and behaves like a rapidly decaying remainder for large asymptotic coordinates $(x',y')$ as a direct consequence of non-degenerate integration by parts.
For $r\in [r_0, c_0\lambda^{2/3}]$, we rescale the variables $(x',y') = r^{1/2}(x'',y'')$ and parameterize the total phase function as $\hat{\psi}_{N,a,h} = r^{3/2}\psi_{N,a,h}^*$ and $\chi'(x'',y'',\dots) = \tilde{\chi}_1(r^{1/2}\lambda^{-1/3}x'', r^{1/2}\lambda^{-1/3}y'', \dots)$. Since the rescaled factor $r^{1/2}\lambda^{-1/3}$ is uniformly bounded, we inherit the strict derivative bounds:
\[
\sup_{(x'',y'')}\left|\partial_{(x'',y'')}^\gamma\chi'\right| \leq C_\gamma\left(1+|x''|+|y''|\right)^{-|\gamma|}.
\]
It remains to prove that the rescaled two-dimensional integral satisfies the uniform bound:
\begin{align}\label{eq:2420-neumann}
r\left| \int_{\mathbb{R}^2} e^{i r^{3/2}\psi_{N,a,h}^*}\chi' \, dx'' \, dy'' \right| \leq C.
\end{align}

Now we study the critical points of the phase $\psi_{N,a,h}^*$ under the normal flux constraints. We compute the gradient components:
\begin{align*}
\partial_{x''}\psi_{N,a,h}^* &= -\cos\theta + x''^2 - \frac{3}{4N^2}(x''+y'')^2 + r^{-1/2}\mathcal{O}\left((x'', y'')\right) \\
&\qquad + aN^{-1}\lambda^{-1}r^{3/2}\mathcal{O}\left((x'',y'')^2\right), \\[1.5ex]
\partial_{y''}\psi_{N,a,h}^* &= -\sin\theta + y''^2 - \frac{3}{4N^2}(x''+y'')^2 + r^{-1/2}\mathcal{O}\left((x'', y'')\right) \\
&\qquad + aN^{-1}\lambda^{-1}r^{3/2}\mathcal{O}\left((x'',y'')^2\right),
\end{align*}
and the corresponding second-order Hessian components:
\begin{align*}
\partial^2_{x''x''}\psi_{N,a,h}^* &= 2x'' - \frac{3}{2}(x''+y'') + r^{-1/2}\mathcal{O}(1) + aN^{-1}\lambda^{-1}r^{3/2}\mathcal{O}\left((x'',y'')\right), \\[1.5ex]
\partial^2_{x''y''}\psi_{N,a,h}^* &= \partial^2_{y''x''}\psi_{N,a,h}^* = -\frac{3}{2}(x''+y'') + r^{-1/2}\mathcal{O}(1) + aN^{-1}\lambda^{-1}r^{3/2}\mathcal{O}\left((x'',y'')\right), \\[1.5ex]
\partial^2_{y''y''}\psi_{N,a,h}^* &= 2y'' - \frac{3}{2}(x''+y'') + r^{-1/2}\mathcal{O}(1) + aN^{-1}\lambda^{-1}r^{3/2}\mathcal{O}\left((x'',y'')\right).
\end{align*}

For small parameter boundaries $a$ and large radius choices $r_0$, we can localize the integral to a compact set in $(x'',y'')$ since the contribution of large coordinates $(x'',y'')$ decays rapidly via non-degenerate integration by parts. The Hessian matrix determinant of the Neumann phase $\psi_{N,a,h}^*$, which we denote by $\mathcal{H}_N(x'',y'')$, takes the explicit form:
\begin{align*}
\mathcal{H}_N(x'',y'') &= \det \begin{pmatrix} \partial^2_{x''x''}\psi_{N,a,h}^* & \partial^2_{x''y''}\psi_{N,a,h}^* \\ \partial^2_{y''x''}\psi_{N,a,h}^* & \partial^2_{y''y''}\psi_{N,a,h}^* \end{pmatrix} \\[1.5ex]
&= 4x''y'' - \frac{3}{N^2}(x''+y'')^2 + r^{-1/2}\mathcal{O}(1) + aN^{-1}\lambda^{-1}r^{3/2}\mathcal{O}\left((x'',y'')\right).
\end{align*}

Thus, for multi-reflective layers $N\geq 2$, small boundary components $a$, and large localization limits $r_0$, outside the singular origin $(x'',y'')=(0,0)$, we define a smooth bifurcation curve $\Gamma = \{(x'',y'') \mid \mathcal{H}_N(x'',y'') = 0\}$. The curve $\Gamma$ stays close to the union of the two canonical straight lines $c(x''+y'') \pm (x''-y'') = 0$, where the velocity weight is strictly bounded as $c^2 = \frac{N^2-3}{N^2} \in [1/4,1]$. Then we have two distinct geometric cases to evaluate:
\begin{itemize}
    \item The contribution of points $(x'',y'')$ located away from the caustic curve $\Gamma$ to the integral is bounded by $\mathcal{O}_{C^\infty}(r^{-3/2})$ through the standard non-degenerate stationary phase method, which yields:
    \begin{align*}
    r\left| \int_{\mathbb{R}^2} e^{i r^{3/2}\psi_{N,a,h}^*}\chi' \, dx'' \, dy'' \right| \leq Cr^{-1/2}.
    \end{align*}
    \item The critical contribution of points $(x'',y'')$ clustering close to the caustic curve $\Gamma$ is governed by the structural properties of degenerate integrals detailed in Lemma 2.21 of Ivanovici--Lebeau--Planchon \cite{ivanovici2014dispersion}. For any given phase angle value of $\theta$, the non-trapped geometric requirements of part $(a)$ of Lemma 2.21 \cite{ivanovici2014dispersion} are strictly satisfied under the Neumann flux setup, which directly yields:
    \begin{align*}
    r\left| \int_{\mathbb{R}^2} e^{i r^{3/2}\psi_{N,a,h}^*}\chi' \, dx'' \, dy'' \right| \leq Cr\left(r^{3/2}\right)^{-5/6} = Cr^{-1/4}.
    \end{align*}
\end{itemize}
Hence, across all possible geometric sub-regimes, the uniform bound \eqref{eq:2420-neumann} is fully satisfied.
\end{proof}
\noindent To summarize, let us recall that the time parameter tracks the reflection count as $T \sim N$ in this regime, which immediately implies that the neighborhood cardinality satisfies $\left|\mathcal{N}_1^N(X,Y,T)\right| \leq C_0\left(1+T\lambda^{-2}\right)$. By invoking the sharp uniform decay bounds established in Lemma \ref{lemNL-neumann} for the long-reflection block where $N \geq \lambda^{1/3}$, we deduce the estimates for the sum of the operators $G^N_{a,N,2}$ across three distinct parameter layers:

\begin{itemize}
    \item \textbf{Sub-case A: $\lambda^{1/3} \leq N \leq \lambda$.}\\
    In this lower layer, there is no stationary phase decay extracted from the global $\eta$-integration, and the neighborhood sheet density remains uniformly bounded by $\left|\mathcal{N}_1^N(X,Y,T)\right| \leq C_0$. Compiling the constants yields the uniform estimate:
    \begin{align*}
    \Bigg|\sum_{N \in \mathcal{N}_1^N(X,Y,T)} G^N_{a,N,2}(T,X,Y,z;h) \Bigg| &\leq Ch^{-3}\bigg(\frac{h}{t}\bigg)^{1/2}\left(h^{-1}a^2\lambda^{-1/2}\lambda^{-5/6}\right) \\
    &\leq Ch^{-3}\bigg(\frac{h}{t}\bigg)^{1/2}h^{1/3}.
    \end{align*}
    
    \item \textbf{Sub-case B: $\lambda \leq N \leq \lambda^2$.}\\
    In this intermediate layer, the frequency integration produces a strict non-degenerate decay factor $(N\lambda^{-1})^{-1/2}$, while the neighborhood cardinality continues to obey the uniform bound $\left|\mathcal{N}_1^N(X,Y,T)\right| \leq C_0$. Evaluating the product yields:
    \begin{align*}
    \Bigg|\sum_{N \in \mathcal{N}_1^N(X,Y,T)} G^N_{a,N,2}(T,X,Y,z;h) \Bigg| &\leq Ch^{-3}\bigg(\frac{h}{t}\bigg)^{1/2}\left(h^{-1}a^2 T^{-1/2}\lambda^{-5/6}\right) \\
    &\leq Ch^{-3}\bigg(\frac{h}{t}\bigg)^{1/2}\left(h^{-1}a^2\lambda^{-1/2}\lambda^{-5/6}\right) \\
    &\leq Ch^{-3}\bigg(\frac{h}{t}\bigg)^{1/2}h^{1/3}.
    \end{align*}
    
    \item \textbf{Sub-case C: $N > \lambda^2$.}\\
    In this high-reflection layer, critical contributions are generated simultaneously from both the frequency phase integration factor $(N\lambda^{-1})^{-1/2}$ and the sheet contraction density relation $\left|\mathcal{N}_1^N(X,Y,T)\right| \leq C_0 T\lambda^{-2}$. Thus, the uniform estimate balances as:
    \begin{align*}
    \Bigg|\sum_{N \in \mathcal{N}_1^N(X,Y,T)} G^N_{a,N,2}(T,X,Y,z;h) \Bigg| &\leq Ch^{-3}\bigg(\frac{h}{t}\bigg)^{1/2} \sum_{N \in \mathcal{N}_1^N(X,Y,T)} \left(h^{-1}a^2 \frac{1}{N}\lambda^{-2/3}\right) \\
    &\leq Ch^{-3}\bigg(\frac{h}{t}\bigg)^{1/2}\left(h^{-1}a^2\lambda^{-2/3}T^{-1}\left|\mathcal{N}_1^N(X,Y,T)\right|\right) \\
    &\leq Ch^{-3}\bigg(\frac{h}{t}\bigg)^{1/2}\left(a^{-2}h^{5/3}\right) \\
    &\leq Ch^{-3}\bigg(\frac{h}{t}\bigg)^{1/2}h^{1/3}.
    \end{align*}
\end{itemize}
\begin{lemma}\label{lemNS-neumann}
There exists a constant $C > 0$ such that for all reflection layers satisfying $N < \lambda^{1/3}$, the following uniform bound holds true under normal flux conditions:
\begin{align}\label{eq:NS-bound-neumann}
\frac{1}{\sqrt{N}}\left| \int_{\mathbb{R}^2} e^{i\lambda\tilde{\psi}_{N,a,h}}\tilde{\chi}_1 \, d\tilde{s} \, d\tilde{\sigma} \right| \leq CN^{-1/4}\lambda^{-3/4}.
\end{align}
\end{lemma}
\noindent Notice that Lemma \ref{lemNS-neumann} shows that for larger reflection indices $N$, it provides an improved decay profile, and it is strictly compatible with the adjacent edge estimate \eqref{eq:N13-neumann} at the matching point threshold where $N \sim \lambda^{1/3}$.

\begin{proof}
Let us define the scaled structural variable $\Lambda = \frac{\lambda}{N^3}$. Since we are working in the short-reflection regime, we have $\Lambda \geq 1$, and we take $\Lambda$ as a newly defined large asymptotic parameter. To evaluate the uniform estimates of our oscillatory integral under the cylindrical boundary layer, we perform the localized coordinate shifts:
\[
X-F_0 = -pN^{-2}, \quad 1-F_0 = -qN^{-2}, \quad \tilde{s} = -\bar{x}/N, \quad \tilde{\sigma} = -\bar{y}/N.
\]
This transformation rescales the polynomial skeleton phase function according to the balance $\tilde{\psi}_{N,a,h} = N^{-3}\bar{\psi}_{N,a,h}$. Consequently, it remains to prove that the transformed two-dimensional integral satisfies:
\begin{align}\label{eq:Lambda-neumann}
\left| \int_{\mathbb{R}^2} e^{i\Lambda \bar{\psi}_{N,a,h}} \tilde{\chi}_1\left(\bar{x}/N, \bar{y}/N, \dots \right) \, d\bar{x} \, d\bar{y} \right| \leq C\Lambda^{-3/4},
\end{align}
where the rescaled total phase function $\bar{\psi}_{N,a,h}$ takes the explicit polynomial form:
\begin{align*}
\bar{\psi}_{N,a,h} &= p\bar{x} - \frac{\bar{x}^3}{3} + q\bar{y} - \frac{\bar{y}^3}{3} + E_0(1+aF_0)^{1/2}(\bar{x}+\bar{y})^2 + \frac{1}{4N^2}(\bar{x}+\bar{y})^3 \\
&\qquad + TN^3\sqrt{1-\tilde{z}^2}\gamma_a(F_0) + aN^{-2}\mathcal{O}\left((\bar{x},\bar{y})^3\right).
\end{align*}

We differentiate this phase function to track its critical manifold trajectories under the normal flux constraints. Computing the gradient components yields:
\begin{align}\label{eq:critga-neumann}
\partial_{\bar{x}}\bar{\psi}_{N,a,h} &= p - \bar{x}^2 + 2E_0(1+aF_0)^{1/2}(\bar{x}+\bar{y}) + \frac{3}{4N^2}(\bar{x}+\bar{y})^2 + aN^{-2}\mathcal{O}\left((\bar{x},\bar{y})^2\right), \nonumber \\
\partial_{\bar{y}}\bar{\psi}_{N,a,h} &= q - \bar{y}^2 + 2E_0(1+aF_0)^{1/2}(\bar{x}+\bar{y}) + \frac{3}{4N^2}(\bar{x}+\bar{y})^2 + aN^{-2}\mathcal{O}\left((\bar{x},\bar{y})^2\right),
\end{align}
along with the second-order partial derivative components:
\begin{align*}
\partial^2_{\bar{x}\bar{x}}\bar{\psi}_{N,a,h} &= -2\bar{x} + 2E_0(1+aF_0)^{1/2} + \frac{3}{2N^2}(\bar{x}+\bar{y}) + aN^{-2}\mathcal{O}\left((\bar{x},\bar{y})\right), \\[1.5ex]
\partial^2_{\bar{x}\bar{y}}\bar{\psi}_{N,a,h} &= \partial^2_{\bar{y}\bar{x}}\bar{\psi}_{N,a,h} = 2E_0(1+aF_0)^{1/2} + \frac{3}{2N^2}(\bar{x}+\bar{y}) + aN^{-2}\mathcal{O}\left((\bar{x},\bar{y})\right), \\[1.5ex]
\partial^2_{\bar{y}\bar{y}}\bar{\psi}_{N,a,h} &= -2\bar{y} + 2E_0(1+aF_0)^{1/2} + \frac{3}{2N^2}(\bar{x}+\bar{y}) + aN^{-2}\mathcal{O}\left((\bar{x},\bar{y})\right).
\end{align*}

The Hessian matrix determinant of the rescaled phase function $\bar{\psi}_{N,a,h}$, which we denote by $\mathcal{H}_N(\bar{x},\bar{y},a)$, takes the following explicit geometric form:
\begin{align}\label{eq:hessian-neumann}
\mathcal{H}_N(\bar{x},\bar{y},a) &= \det \begin{pmatrix} \partial^2_{\bar{x}\bar{x}}\bar{\psi}_{N,a,h} & \partial^2_{\bar{x}\bar{y}}\bar{\psi}_{N,a,h} \\ \partial^2_{\bar{y}\bar{x}}\bar{\psi}_{N,a,h} & \partial^2_{\bar{y}\bar{y}}\bar{\psi}_{N,a,h} \end{pmatrix} \nonumber \\[1.5ex]
&= 4\bar{x} \bar{y} - 4E_0(1+aF_0)^{1/2}(\bar{x}+\bar{y}) - \frac{3}{N^2}(\bar{x}+\bar{y})^2 + aN^{-2}\mathcal{O}\left((\bar{x},\bar{y})\right).
\end{align}
\begin{lemma}\label{lem:lempql-neumann}
There exist constants $r_0 > 0$ and $C > 0$ such that for all coordinate configurations $(p,q)$ obeying the large caustic distance threshold $|(p,q)| \geq r_0$, the following uniform bound holds true under normal flux constraints:
\begin{align}
\left| \int_{\mathbb{R}^2} e^{i\Lambda \bar{\psi}_{N,a,h}}\tilde{\chi}_1\left(\bar{x}/N, \bar{y}/N, \dots\right) \, d\bar{x} \, d\bar{y} \right| \leq C\Lambda^{-5/6}.
\end{align}
\end{lemma}

\begin{proof}[Proof of Lemma \ref{lem:lempql-neumann}]
We adapt the micro-local singularity tracking strategies established in the proof of Lemma 2.26 of Ivanovici--Lebeau--Planchon \cite{ivanovici2014dispersion}. We map the caustic weights into local polar coordinates, setting $(p,q)=(r\cos\theta,r\sin\theta)$ with $r \geq r_0$. 

Let $\chi \in C_0^\infty(\mathbb{R}^2)$ be a smooth cutoff function localized to the spatial core where $|(\bar{x},\bar{y})| < c$ for a sufficiently small constant $c > 0$, such that $\chi = 1$ identically near the origin. From the gradient component descriptions \eqref{eq:critga-neumann}, since the phase remains strictly non-vanishing near the core, we obtain by non-degenerate integration by parts in $(\bar{x},\bar{y})$ that for any arbitrary multi-index $k \geq 1$:
\begin{align*}
\left| \int_{\mathbb{R}^2} e^{i\Lambda \bar{\psi}_{N,a,h}} \chi\left(r^{-1/2}(\bar{x},\bar{y})\right)\tilde{\chi}_1\left(\bar{x}/N, \bar{y}/N, \dots\right) \, d\bar{x} \, d\bar{y} \right| \leq C_k r^{-k} \Lambda^{-k}.
\end{align*}

For the outer regime where $|(\bar{x},\bar{y})|$ climbs away from the origin, we perform the geometric scaling change of variables $(\bar{x},\bar{y}) = r^{1/2}(x', y')$ and define the normalized phase profile $\bar{\psi}_{N,a,h}' = r^{-3/2}\bar{\psi}_{N,a,h}$. This scale transformation reduces our problem to verifying the uniform bound:
\begin{align*}
\left| r \int_{\mathbb{R}^2} e^{i r^{3/2}\Lambda \bar{\psi}_{N,a,h}'} (1-\chi)(x',y')\tilde{\chi}_1\left(r^{1/2}x'/N, r^{1/2}y'/N, \dots \right) \, dx' \, dy' \right| \leq C\Lambda^{-5/6}.
\end{align*}
We observe that since the complementary symbol component satisfies $(1-\chi)(x',y')=0$ in a localized neighborhood of the origin, $(1-\chi)(x',y')=1$ for $|(x',y')| \geq c$, and $\tilde{\chi}_1$ has compact support, the amplitude functions satisfy the regular derivative tracking bounds:
\begin{align*}
\sup_{(x',y')}\left| \partial_{(x',y')}^\gamma \left[ (1-\chi)(x',y')\tilde{\chi}_1\left(r^{1/2}x'/N, r^{1/2}y'/N, \dots \right) \right] \right| \leq C_\gamma (1+|x'|+|y'|)^{-|\gamma|}.
\end{align*}

Under the Neumann flux reflections setup, the rescaled phase function $\bar{\psi}_{N,a,h}'$ expands into the following polynomial layout:
\begin{align*}
\bar{\psi}_{N,a,h}' &= (\cos\theta) x' - \frac{x'^3}{3} + (\sin\theta) y' - \frac{y'^3}{3} + \frac{1}{4N^2}(x'+y')^3 + \frac{TN^3}{r^{3/2}}\sqrt{1-\tilde{z}^2}\gamma_a(F_0) \\
&\qquad + aN^{-2}\mathcal{O}\left((x',y')^3\right).
\end{align*}
Differentiating this phase mapping yields the corresponding directional gradient vectors:
\begin{align*}
\partial_{x'}\bar{\psi}_{N,a,h}' &= \cos\theta - x'^2 + \frac{3}{4N^2}(x'+y')^2 + aN^{-2}\mathcal{O}\left((x',y')^2\right), \\[1.5ex]
\partial_{y'}\bar{\psi}_{N,a,h}' &= \sin\theta - y'^2 + \frac{3}{4N^2}(x'+y')^2 + aN^{-2}\mathcal{O}\left((x',y')^2\right),
\end{align*}
along with the second-order partial derivative components:
\begin{align*}
\partial^2_{x'x'}\bar{\psi}_{N,a,h}' &= -2x' + \frac{3}{2N^2}(x'+y') + aN^{-2}\mathcal{O}\left((x',y')\right), \\[1.5ex]
\partial^2_{x'y'}\bar{\psi}_{N,a,h}' &= \partial^2_{y'x'}\bar{\psi}_{N,a,h}' = \frac{3}{2N^2}(x'+y') + aN^{-2}\mathcal{O}\left((x',y')\right), \\[1.5ex]
\partial^2_{y'y'}\bar{\psi}_{N,a,h}' &= -2y' + \frac{3}{2N^2}(x'+y') + aN^{-2}\mathcal{O}\left((x',y')\right).
\end{align*}

Thus, for small parameter configurations $a$ and large radius choices $r_0$, we can localize the integration to a compact set in $(x',y')$ since the contribution of large coordinates decays rapidly via non-degenerate integration by parts. The Hessian matrix determinant of the rescaled phase function, which we denote by $\mathcal{H}'_N(x',y',a)$, takes the explicit geometric form:
\begin{align*}
\mathcal{H}'_N(x',y',a) &= \det \begin{pmatrix} \partial^2_{x'x'}\bar{\psi}_{N,a,h}' & \partial^2_{x'y'}\bar{\psi}_{N,a,h}' \\ \partial^2_{y'x'}\bar{\psi}_{N,a,h}' & \partial^2_{y'y'}\bar{\psi}_{N,a,h}' \end{pmatrix} \\[1.5ex]
&= 4x'y' - \frac{3}{N^2}(x'+y')^2 + aN^{-2}\mathcal{O}\left((x',y')\right).
\end{align*}

Applying the same catastrophe-theoretic reduction rules as before, for multi-reflective layers $N \geq 2$, small boundary components $a$, and large localization limits $r_0$, outside the singular origin $(x',y')=(0,0)$, we define the smooth curve $\Gamma = \{(x',y') \mid \mathcal{H}'_N(x',y',a) = 0\}$. There are two separate geometric cases to evaluate:
\begin{itemize}
    \item The critical contribution of points $(x',y')$ located away from the caustic curve $\Gamma$ to the integral is bounded by $\mathcal{O}(r^{-3/2}\Lambda^{-1})$ through the standard non-degenerate stationary phase method, which yields:
    \begin{align*}
    \left| r \int_{\mathbb{R}^2} e^{i r^{3/2}\Lambda \bar{\psi}_{N,a,h}'} (1-\chi)(x',y')\tilde{\chi}_1\left(r^{1/2}x'/N, r^{1/2}y'/N, \dots \right) \, dx' \, dy' \right| \leq Cr^{-1/2}\Lambda^{-1}.
    \end{align*}
    \item The critical contribution of points $(x',y')$ clustering close to the caustic curve $\Gamma$ is governed by the structural properties of degenerate integrals detailed in Lemma 2.21 of \cite{ivanovici2014dispersion}. For any given phase angle value of $\theta$, the non-trapped geometric requirements of part $(a)$ of Lemma 2.21 \cite{ivanovici2014dispersion} are strictly satisfied under the Neumann flux setup, which directly yields:
    \begin{align*}
    \left| r \int_{\mathbb{R}^2} e^{i r^{3/2}\Lambda \bar{\psi}_{N,a,h}'} (1-\chi)(x',y')\tilde{\chi}_1\left(r^{1/2}x'/N, r^{1/2}y'/N, \dots \right) \, dx' \, dy' \right| &\leq Cr\left(r^{3/2}\Lambda\right)^{-5/6} \\
    &\leq Cr^{-1/4}\Lambda^{-5/6}.
    \end{align*}
\end{itemize}
This completes the formal proof of Lemma \ref{lem:lempql-neumann}.
\end{proof}
\begin{lemma}\label{lem:lempqs-neumann}
There exist constants $r_0 > 0$ and $C > 0$ such that for all coordinate configurations $(p,q)$ obeying the compact caustic distance threshold $|(p,q)| \leq r_0$, the following uniform bound holds true under normal flux constraints:
\begin{align}
\left| \int_{\mathbb{R}^2} e^{i\Lambda \bar{\psi}_{N,a,h}}\tilde{\chi}_1\left(\bar{x}/N, \bar{y}/N, \dots\right) \, d\bar{x} \, d\bar{y} \right| \leq C\Lambda^{-3/4}.
\end{align}
\end{lemma}

\begin{proof}[Proof of Lemma \ref{lem:lempqs-neumann}]
Now we consider the compact regime where $|(p,q)| \leq r_0$. There exists a uniform geometric constant $c > 0$ independent of the multi-reflection index $N \geq 2$ such that:
\begin{align}\label{eq:coercive-neumann}
\forall (\bar{x},\bar{y})\in\mathbb{R}^2,\quad \left|\bar{x}^2-\frac{3}{4N^2}(\bar{x}+ \bar{y})^2\right|+\left|\bar{y}^2-\frac{3}{4N^2}(\bar{x}+\bar{y})^2\right|\geq c\left(\bar{x}^2+\bar{y}^2\right).
\end{align} 
Owing to this uniform lower bound, applying non-degenerate integration by parts based on the gradient vector field profiles \eqref{eq:critga-neumann} shows that the integration region far from the origin provides a rapidly decaying contribution of order $\mathcal{O}_{C^\infty}(\Lambda^{-\infty})$ to the global integral \eqref{eq:Lambda-neumann}. Therefore, we can validly restrict our analysis to the configuration space where $(\bar{x},\bar{y})$ is localized within a compact subset of $\mathbb{R}^2$. It remains to establish the uniform bound:
\begin{align*}
\left| \int_{\mathbb{R}^2} e^{i\Lambda\bar{\psi}_{N,a,h}}\tilde{\chi}_1 \, d\bar{x} \, d\bar{y} \right| \leq C\Lambda^{-3/4},
\end{align*}
where the rescaled total phase function $\bar{\psi}_{N,a,h}$ takes the explicit polynomial layout:
\begin{align*}
\bar{\psi}_{N,a,h} &= p\bar{x}-\frac{\bar{x}^3}{3}+q\bar{y}-\frac{\bar{y}^3}{3}+E_0(1+aF_0)^{1/2}(\bar{x}+\bar{y})^2+\frac{1}{4N^2}(\bar{x}+\bar{y})^3 \\
&\qquad + TN^3\sqrt{1-\tilde{z}^2}\gamma_a(F_0)+aN^{-2}\mathcal{O}\left((\bar{x},\bar{y})^3\right),
\end{align*}
and its associated Hessian matrix determinant $\mathcal{H}_N(\bar{x},\bar{y},a)$ is given explicitly by \eqref{eq:hessian-neumann}.\\

For small parameter boundaries $a$, the critical bifurcation set $\Gamma = \{(\bar{x},\bar{y}) \mid \mathcal{H}_N(\bar{x},\bar{y},a) = 0\}$ defines a smooth curve in the parameter space. Geometrically, this curve is a regular, non-degenerate perturbation of the elliptic cone $4\bar{x}\bar{y}-4E_0(1+aF_0)^{1/2}(\bar{x}+\bar{y})-3(\bar{x}+\bar{y})^2=0$ for the single-reflection path layer $N=1$, and sits close to the hyperbolic sheet curves $4\bar{x}\bar{y}-4E_0(1+aF_0)^{1/2}(\bar{x}+\bar{y})-\frac{3}{N^2}(\bar{x}+\bar{y})^2=0$ for the multi-reflective configuration sheets where $N \geq 2$. 

To finalize the proof, we apply the singularity-tracking machinery of Lemma 2.21 of Ivanovici--Lebeau--Planchon \cite{ivanovici2014dispersion} for local blocks where $(\bar{x},\bar{y})$ clusters near $(p,q)$ inside the bounded threshold region $|(p,q)| \leq r_0$. This yields three distinct geometric cases to evaluate:
\begin{itemize}
    \item If the coordinate pair $(p,q)$ is located away from the caustic bifurcation sheet $\Gamma$, the phase function is non-vanishing and non-degenerate, and the stationary phase method directly yields a decay rate of order $\mathcal{O}(\Lambda^{-1})$.
    \item If the coordinate pair satisfies $(0,0) \neq (p,q)$ and clusters close to the caustic sheet $\Gamma$, the hypothesis of part $(a)$ of Lemma 2.21 \cite{ivanovici2014dispersion} holds strictly true. This corresponds to a regular fold singularity type under Neumann flux, which generates a localized integration contribution of order $\mathcal{O}(\Lambda^{-5/6})$.
    \item If the coordinate pair sits exactly at the focal core origin, $(p,q)=(0,0)$, we track the wavefronts where $(\bar{x},\bar{y})$ clusters near $(0,0)$. In this critical collision domain, the hypothesis of part $(b)$ of Lemma 2.21 \cite{ivanovici2014dispersion} is satisfied, identifying a canonical cusp singularity layout which yields a decay contribution of order $\mathcal{O}(\Lambda^{-3/4})$.
\end{itemize}
Combining these localized evaluations completes the uniform proof of Lemma \ref{lem:lempqs-neumann}.
\end{proof}
Lemma \ref{lem:lempql-neumann} and Lemma \ref{lem:lempqs-neumann} yield the proof of Lemma \ref{lemNS-neumann}.
\end{proof}
\noindent Notice that for the short-reflection boundary layers where $N < \lambda^{1/3}$, there is no localized critical contribution extracted from the global frequency $\eta$-integration, and the neighborhood sheet density remains uniformly bounded by $\left|\mathcal{N}_1^N(X,Y,T)\right| \leq C_0$. As a consequence of this structural isolation, we invoke the sharp pointwise bounds established in Lemma \ref{lemNS-neumann} to evaluate the estimates for the sum of the operators $G^N_{a,N,2}$ across the short-reflection regimes:
\begin{align*}
\Bigg|\sum_{N \in \mathcal{N}_1^N(X,Y,T)} G^N_{a,N,2}(T,X,Y,z;h) \Bigg| &\leq Ch^{-3}\bigg(\frac{h}{t}\bigg)^{1/2}\left(h^{-1}a^2\lambda^{-1/2} N^{-1/4}\lambda^{-3/4}\right) \\
&\leq Ch^{-3}\bigg(\frac{h}{t}\bigg)^{1/2}\left(a^{1/8}h^{1/4}N^{-1/4}\right).
\end{align*}
We notice that we recover the identical caustic power loss profile when evaluating the single-reflection operator base layer matching $N=1$:
\begin{align*}
\Big|G^N_{a,1,2}(T,X,Y,z;h)\Big| &\leq Ch^{-3}\bigg(\frac{h}{t}\bigg)^{1/2}\left(h^{-1}a^2\lambda^{-1/2}\lambda^{-3/4}\right) \\
&\leq Ch^{-3}\bigg(\frac{h}{t}\bigg)^{1/2}\left(a^{1/8}h^{1/4}\right).
\end{align*}
\noindent To summarize, putting these micro-local index sheet bounds together, we have successfully proved that the accumulated multi-reflective operator satisfies:
\begin{align*}
\Bigg|\sum_{1\leq N\leq C_0 a^{-1/2}}G^N_{a,N,2}(T,X,Y,z;h)\Bigg|\leq Ch^{-3}\bigg(\frac{h}{t}\bigg)^{1/2}\left(h^{1/3}+a^{1/8}h^{1/4}\right).
\end{align*}
We notice that the secondary cusp component dominates the fold parameter scaling, $h^{1/3} \leq a^{1/8}h^{1/4}$, whenever the source distance satisfies the spatial threshold condition $a \geq h^{2/3}$. Therefore, the target bound is established, and the proof of Proposition \ref{prop:G2-estimates-neumann} is complete.
\end{proof}
\begin{proof}[Proof of Theorem \ref{thm:thmN-neumann}]
Putting the micro-local estimates established in Propositions \ref{prop:G1-estimates-neumann}, \ref{prop:G1-single-reflection-neumann}, and \ref{prop:G2-estimates-neumann} together directly yields the desired uniform bound under the normal flux constraints.
\end{proof}
\section{Dispersive Estimates for $\epsilon_0\sqrt{a}\leq \eta \leq c_0$}\label{sec:3}

In this section, we prove Theorem \ref{thm:1beta-neumann}. To extract the sharp micro-local decay bounds for the operator $\mathcal{G}^N_{a,m}$, we distinguish between two separate geometric regimes. The first case deals with source layers close to the boundary where $a\leq \left(\frac{h}{2^m\sqrt{a}}\right)^{\frac{2}{3}(1-\epsilon)}$, for $\epsilon\in(0,1/7)$, where we follow the technical guidelines of Section \ref{sec:2} and construct a local parametrix as a sum over the Neumann eigenmodes. The second case covers the region further from the boundary where $a\geq \left(\frac{h}{2^m\sqrt{a}}\right)^{\frac{2}{3}(1-\epsilon')}$, for $\epsilon'\in(0,\epsilon)$, where the derivative-adapted Airy-Poisson summation formula yields the structural representation of $\mathcal{G}^N_{a,m}$ as a geometric sum over the reflection index $N\in\mathbb{Z}$, representing waves tracking the precise number of flux reflections on the boundary.\\

Recall that under the normal flux constraints, the global high-frequency Neumann parametrix expands according to the spectral relation:
\begin{equation}\label{eq:L-0-neumann}
\mathcal{G}^N_{a}(t,x,y,z) = \frac{1}{4\pi^2h^2}\sum_{k\geq1}\int_{\mathbb{R}^2} e^{\frac{i}{h}\Phi_k} \sigma_k \, d\eta \, d\zeta,
\end{equation}
where the phase function $\Phi_k$ and the corresponding Neumann amplitude symbol $\sigma_k$ are defined explicitly by:
\begin{align*}
\Phi_k &= y\eta+z\zeta+t\left(\eta^2+\zeta^2+\omega'_k h^{2/3}\eta^{4/3}\right)^{1/2}, \\
\sigma_k &= e_k(x,\eta/h) e_k(a,\eta/h)\chi_0(\zeta^2+\eta^2)\chi_1\left(\omega'_k h^{2/3}\eta^{4/3}\right)(1-\chi_1)(\varepsilon\omega'_k).
\end{align*}
Our objective is to evaluate uniform $L^\infty$ decay estimates for the kernel $\mathcal{G}^N_{a}$ inside the localized time interval $t\in [h,1]$, when the double integral in \eqref{eq:L-0-neumann} is micro-locally restricted to the intermediate tangential variables block $\eta\in [\epsilon_0\sqrt{a}, c_0]$, with $c_0 > 0$ chosen as a small fixed threshold.

\noindent Let us define the auxiliary frequency scaling parameter $\mu^2$ by:
\begin{equation*}
\mu^2 = \eta^2 + \omega'_k h^{2/3}\eta^{4/3}.
\end{equation*}
We observe that $\mu^2$ is bounded and small, since the boundary term $\omega'_k h^{2/3}\eta^{4/3}$ is suppressed by the support of the cutoff $\chi_1$, and $\eta \leq c_0$ is small. We embed a smooth localized time-frequency filter function $\chi_4 \in C_0^\infty((-1,1))$ satisfying $\chi_4 = 1$ identically on the interval $[-1/2, 1/2]$. For a given scale parameter $D \geq 1$, let the auxiliary regularized operator $\mathcal{J}^N_{a}(t,x,y,z)$ be defined by:
\begin{equation}\label{eq:Ja-definition-neumann}
\mathcal{J}^N_{a}(t,x,y,z) = \frac{1}{4\pi^2h^2}\sum_{k\geq1}\int_{\mathbb{R}^2} e^{\frac{i}{h}\Phi_k} \chi_4\left(\frac{t\mu^2}{Dh}\right)\sigma_k \, d\eta \, d\zeta.
\end{equation}
The following lemma establishes that the regularized piece $\mathcal{J}^N_{a}$ avoids boundary trapping and recovers the optimal free-space dispersive estimate.

\begin{lemma}\label{lem:L-lem1-neumann}
There exists a constant $C > 0$ acting independently of the scaling factor $D \geq 1$ such that the regularized Neumann operator satisfies:
\begin{equation*}
\left| \mathcal{J}^N_{a}(t,x,y,z) \right| \leq C h^{-3}\left(\frac{h}{t}\right) D.
\end{equation*}
\end{lemma}

\begin{proof}
On the compact support of the regularizing time-frequency cutoff function $\chi_4$, the variables are restricted to the layer where $\eta^2 \leq Dh/t$ and $h(\omega'_k)^{3/2}\eta^2 \leq (Dh/t-\eta^2)^{3/2}$. This inequality implies that the discrete spectral sum over the Neumann eigenmodes index $k$ is boundedly restricted to the finite range:
\[
k \leq c_0 \frac{(Dh/t-\eta^2)^{3/2}}{h\eta^2}.
\]
Recall that the normalized Neumann eigenfunction profile reads as $$e_k(x,\eta/h)=f_k k^{-1/6}(\eta/h)^{1/3}\Ai\left((\eta/h)^{2/3}x-\omega'_k\right).$$ By applying the Cauchy-Schwarz inequality combined with the uniform spectral cluster bound established in Lemma \ref{lem:k-low-modes}, we estimate the amplitude trace sum over $k$ inside \eqref{eq:Ja-definition-neumann}:
\begin{align}
\left| \mathcal{J}^N_{a}(t,x,y,z) \right| &\leq C h^{-2}\int_{\eta^2 \leq Dh/t} \left(\frac{\eta}{h}\right)^{2/3} \frac{(Dh/t-\eta^2)^{1/2}}{h^{1/3}\eta^{2/3}} \, d\eta \nonumber \\[1.5ex]
&= C h^{-3}\int_{\eta^2 \leq Dh/t} (Dh/t-\eta^2)^{1/2} \, d\eta.
\end{align}
The target free-space decay bound follows immediately by executing a linear change of variables on the remaining integral, exploiting the exact scaling relation:
\[
\int_{\eta^2 \leq Dh/t} (Dh/t-\eta^2)^{1/2} \, d\eta = \left(\frac{Dh}{t}\right) \int_{x^2 \leq 1} (1-x^2)^{1/2} \, dx.
\]
This completes the formal proof of Lemma \ref{lem:L-lem1-neumann}.
\end{proof}

\noindent Observe that in the high-frequency range where $\eta\geq c_0$, the scaling variable satisfies $\mu^2\geq c_0^2$. Consequently, the condition $t\mu^2/h\leq D$ simplifies to the short-time horizon threshold $t\leq C h$, rendering the regularization of Lemma \ref{lem:L-lem1-neumann} irrelevant. However, in the intermediate tangential variable block where $\eta\in [\epsilon_0\sqrt{a}, c_0]$, Lemma \ref{lem:L-lem1-neumann} becomes highly useful since it dictates that we can now validly treat the semi-classical coordinate ratio:
\[
\lambda=\frac{t\mu^2}{h}
\]
as a large asymptotic parameter. Since our framework explicitly permits a mild loss in the geometric dispersive estimate relative to the flat Euclidean case, we may assume that this parameter obeys the lower envelope threshold:
\[
\lambda=\frac{t\mu^2}{h}\geq \left(\frac{h}{t}\right)^{-\epsilon}
\]
for some structural parameter $\epsilon>0$ (achieved by setting $D=\left(\frac{h}{t}\right)^{-\epsilon}$). Under this high-frequency localization, any remainder term operating as $\mathcal{O}_{C^\infty}(\lambda^{-\infty})$ vanishes rapidly and becomes analytically negligible. 

We are now in a position to eliminate the $\zeta$-integration in the intermediate micro-local parametrix representation \eqref{eq:L-0-neumann}. This is the primary purpose of the following lemma. Let us recall that the frequency truncation symbol $\chi_0(\zeta^2+\eta^2)$ localizes the total wave vector magnitude $\zeta^2+\eta^2$ near the unit energy sphere $1$. Therefore, for small tangential frequencies $\eta$, the axial frequency variable $\zeta$ must cluster tightly near either $1$ or $-1$. In the subsequent sections, we restrict our analysis to the positive axial channel where $\zeta$ is localized near $1$, as the symmetric case where $\zeta$ is near $-1$ follows from an identical reflection argument.

In the sequel, we restrict our attention to the forward axial propagation channel where $\zeta$ is localized near $1$.

\begin{lemma}\label{lem:L-lem2-neumann}
Let $\lambda = t\mu^2/h \geq 1$, define the scaled axial coordinate $\tilde z=z/t$, and let the reduced axial phase function $\phi(\tilde z,\mu^2,\zeta)$ be given by:
\[
\phi(\tilde z,\mu^2,\zeta) = \frac{1}{\mu^2}\left(\tilde z \zeta+(\zeta^2+\mu^2)^{1/2}\right).
\]
Let $I(\tilde z, \mu^2,\eta;\lambda)$ denote the highly oscillatory one-dimensional axial integral:
\begin{equation}\label{eq:I-integral-definition}
I(\tilde z, \mu^2,\eta;\lambda) = \int_{\zeta\sim 1} e^{i\lambda \phi(\tilde z,\mu^2,\zeta)}\chi_0(\zeta^2+\eta^2) \, d\zeta.
\end{equation}
There exist positive structural constants $0 < c_1 < C_1 < \infty$ such that the following properties hold true:
\begin{enumerate}
    \item For spatial positions where the velocity parameters lie outside the grazing window, $\tilde z \notin [-1+c_1\mu^2, -1+C_1\mu^2]$, the integral decays rapidly:
    \begin{equation}\label{eq:L2-neumann}
    \sup_{\tilde z, \mu^2,\eta} \left| I(\tilde z, \mu^2,\eta;\lambda) \right| \in \mathcal{O}_{C^\infty}(\lambda^{-\infty}).
    \end{equation}
    \item For velocity profiles localized inside the grazing boundary zone, $\tilde z \in [-1+c_1\mu^2, -1+C_1\mu^2]$, we set the normalized coordinate variable $\tilde z = -1+z^*\mu^2$. There exists a classical symbol of degree $0$ in $\lambda$, denoted by $\sigma_0(z^*,\eta,\mu^2;\lambda)$, such that the integral evaluates explicitly as:
    \begin{equation}\label{eq:L3-neumann}
    I(\tilde z, \mu^2,\eta;\lambda) = \left(\frac{h}{t\mu^2}\right)^{1/2} e^{i\frac{t\mu}{h}(1-\tilde z^2)^{1/2}} \sigma_0 (z^*,\eta,\mu^2;\lambda).
    \end{equation}
\end{enumerate}
\end{lemma}

\begin{proof}
We differentiate the reduced axial phase mapping $\phi$ with respect to the integration variable $\zeta$. Computing the first two derivatives yields:
\[
\partial_\zeta\phi = \frac{1}{\mu^2}\left(\tilde z +\zeta(\zeta^2+\mu^2)^{-1/2}\right), \quad \partial^2_\zeta\phi = (\zeta^2+\mu^2)^{-3/2} \geq c > 0,
\]
and the higher-order derivatives $\partial^j_\zeta\phi$ are uniformly bounded for all $j\geq 2$. Since the tangential localization ensures $\mu^2$ is small, expanding the algebraic weight near $\zeta \sim 1$ yields:
\[
\zeta(\zeta^2+\mu^2)^{-1/2} = 1 - \frac{\mu^2}{2\zeta^2} + \mathcal{O}(\mu^4).
\]
If the velocity parameter satisfies $\tilde z \notin [-1+c_1\mu^2, -1+C_1\mu^2]$, the first derivative $\partial_\zeta\phi$ is uniformly bounded away from zero on the compact support of the frequency filter $\chi_0$. Consequently, the rapid decay estimate \eqref{eq:L2-neumann} follows immediately by non-degenerate integration by parts with respect to $\zeta$.

Now, let us analyze the grazing threshold where $\tilde z \in [-1+c_1\mu^2, -1+C_1\mu^2]$, written via the normalized variable as $\tilde z = -1+z^*\mu^2$. Under this coordinate configuration, the total phase simplifies algebraically to:
\[
\phi = z^*\zeta + \left(\zeta+(\zeta^2+\mu^2)^{1/2}\right)^{-1},
\]
which possesses a unique, strictly isolated critical point given by $\zeta_c = -\mu\tilde z(1-\tilde z^2)^{-1/2}$. We evaluate the corresponding critical value on this tracking path:
\[
\phi(\zeta_c) = \frac{\zeta_c}{\mu^2}(\tilde z - 1/\tilde z) = \frac{(1-\tilde z^2)^{1/2}}{\mu} \in \mathcal{O}(1).
\]
Because the second derivative satisfies $\partial^2_\zeta\phi \geq c > 0$, this critical point is strictly non-degenerate. Applying the classical semi-classical stationary phase method with respect to the large parameter $\lambda$ directly establishes the asymptotic representation statement \eqref{eq:L3-neumann}.
\end{proof}

\noindent Using Lemmas \ref{lem:L-lem1-neumann} and \ref{lem:L-lem2-neumann}, we are now reduced to the study of the following reduced one-dimensional highly oscillatory tangential integral:
\begin{equation}\label{eq:L4-neumann}
\frac{1}{4\pi^2h^2}\bigg(\frac{h}{t}\bigg)^{1/2}\sum_{k\geq1}\int_{\mathbb{R}} e^{\frac{i}{h}\left(y\eta + t\mu (1-\tilde{z}^2)^{1/2}\right)} \frac{\tilde{\sigma}_k}{\mu} \, d\eta,
\end{equation}
where the composite amplitude function $\tilde{\sigma}_k$ is defined by:
\begin{align*}
\tilde{\sigma}_k &= \sigma_0 (z^*,\eta,\mu^2;\lambda)\left(1-\chi_4\left(\frac{t\mu^2}{Dh}\right)\right) e_k(x,\eta/h) e_k(a,\eta/h) \\
&\quad \times \chi_1\left(\omega'_k h^{2/3}\eta^{4/3}\right)(1-\chi_1)(\varepsilon\omega'_k).
\end{align*}

To obtain the uniform $L^\infty$ decay estimates for the micro-local parametrix within the range $\eta\in [\epsilon_0\sqrt{a}, c_0]$, we employ a dyadic Littlewood-Paley localization in the tangential frequency $\eta$. We choose a smooth partition function $\psi_1\in C_0^\infty\left((0.5, 2.5)\right)$ satisfying $0\leq \psi_1\leq 1$ such that:
\[
\sum_{m\in\mathbb{Z}}\psi_1(2^m x)=1 \quad \text{for all } x>0,
\]
and we introduce the localized filter function $\psi_1\left(\frac{\eta}{2^m\sqrt{a}}\right)$ into \eqref{eq:L4-neumann}. In the subsequent sections, the dyadic reflection index $m$ is therefore confined to the geometric range:
\[
\epsilon_0 \leq 2^m \leq \frac{c_0}{\sqrt{a}}.
\]
To normalize the systems, we introduce the following scaled variables and phase parameters:
\begin{align*}
\eta &= 2^m\sqrt{a} \, \tilde{\eta}, \quad h = 2^m\sqrt{a} \, \tilde{h}, \\
\mu^2 &= \eta^2 + \omega'_k h^{2/3}\eta^{4/3} = (2^m\sqrt{a})^2\left(\tilde{\eta}^2 + \omega'_k \tilde{h}^{2/3}\tilde{\eta}^{4/3}\right) = (2^m\sqrt{a})^2\tilde{\mu}^2, \\
\gamma &= \omega'_k h^{2/3}\eta^{-2/3} = \omega'_k \tilde{h}^{2/3}\tilde{\eta}^{-2/3}.
\end{align*}
We define the localized dyadic operator component $\mathcal{G}^N_{a,m}$ by the formula:
\begin{equation}\label{eq:L5-neumann}
\mathcal{G}^N_{a,m}(t,x,y,z) = \frac{1}{4\pi^2h^2}\bigg(\frac{h}{t}\bigg)^{1/2}\sum_{k\geq1}\int_{\mathbb{R}} e^{\frac{i}{h}\left(y\eta + t\mu (1-\tilde{z}^2)^{1/2}\right)} \psi_1\left(\frac{\eta}{2^m\sqrt{a}}\right) \frac{\tilde{\sigma}_k}{\mu} \, d\eta.
\end{equation}
Observe that due to the micro-local frequency truncation enforced by $\chi_1$, the discrete spectral sum over $k$ is boundedly restricted to the finite range $k \leq \frac{\varepsilon}{h\eta^2}$.

Utilizing the scaled change of variable $\eta = 2^m\sqrt{a} \, \tilde{\eta}$, we define the normalized tangential spatial parameter $\tilde{y} = y/t$. Since the volume element transforms invariant to the scaling as $d\eta/\mu = d\tilde{\eta}/\tilde{\mu}$, the dyadic operator \eqref{eq:L5-neumann} resolves into the following final normalized representation:
\begin{equation}\label{eq:L6-neumann}
\mathcal{G}^N_{a,m}(t,x,y,z) = \frac{1}{4\pi^2h^2}\bigg(\frac{h}{t}\bigg)^{1/2}\sum_{1\leq k\leq \frac{\varepsilon}{(2^m\sqrt{a})^3\tilde{h}}}\int_{\mathbb{R}} e^{\frac{it}{\tilde{h}}\left(\tilde{y}\tilde{\eta} + \tilde{\mu} (1-\tilde{z}^2)^{1/2}\right)} g^N_k \psi_1(\tilde{\eta}) \, d\tilde{\eta},
\end{equation}
where the Neumann amplitude function $g^N_k$ is given explicitly by:
\begin{align*}
g^N_k &= \frac{1}{\tilde{\mu}} \sigma_0 (z^*,\eta,\mu^2;\lambda)\left(1-\chi_4\left(\frac{t\mu^2}{Dh}\right)\right) e_k(x,\tilde{\eta}/\tilde{h}) e_k(a,\tilde{\eta}/\tilde{h}) \\
&\quad \times \chi_1\left(\omega'_k h^{2/3}\eta^{4/3}\right)(1-\chi_1)(\varepsilon\omega'_k).
\end{align*}

\begin{lemma}\label{lem:L-lem3-neumann}
Let $M\geq 1$ be given. There exists a constant $C_M > 0$ such that for all parameters $m, a, h$ satisfying the low dyadic frequency threshold $2^m\sqrt{a} \leq h M$, the following uniform estimate holds true under normal flux conditions:
\begin{equation}\label{eq:L-6-neumann}
\left| \mathcal{G}^N_{a,m} \right| \leq C_M h^{-3}\bigg(\frac{h}{t}\bigg)^{1/2} 2^m\sqrt{a} \left| \log(2^m\sqrt{a}) \right|.
\end{equation}
\end{lemma}

\begin{proof}
Under the designated low dyadic frequency restriction, the rescaled semi-classical frequency satisfies $\tilde{h} \geq 1/M$. Consequently, utilizing the sharp pointwise bounds for the Neumann eigenfunctions established via the Airy function decay profiles, we control the spatial magnitude by:
\[
\left| e_k(x,\tilde{\eta}/\tilde{h}) \right| \leq C k^{-1/6} \left(\frac{\tilde{\eta}}{\tilde{h}}\right)^{1/3} (\omega'_k)^{-1/4}.
\]
Moreover, tracking the lower bound of the sub-elliptic metric reveals that the normalized phase scaling factor obeys the coercion relation $\tilde{\mu} \geq (\omega'_k)^{1/2}\tilde{h}^{1/3}\tilde{\eta}^{2/3}$. By substituting these bounds directly into the normalized dyadic operator \eqref{eq:L6-neumann} and summing over the finite eigenmode sequence, we find:
\begin{align}
\left| \mathcal{G}^N_{a,m} \right| &\leq C'' h^{-2} \bigg(\frac{h}{t}\bigg)^{1/2} \sum_{1\leq k\leq \frac{\varepsilon}{(2^m\sqrt{a})^3\tilde{h}}} (\omega'_k)^{-1/2}\tilde{h}^{-1/3} k^{-1/3} \left(\frac{1}{\tilde{h}}\right)^{2/3} (\omega'_k)^{-1/2} \nonumber \\[1.5ex]
&\leq C' h^{-3} \bigg(\frac{h}{t}\bigg)^{1/2} 2^m\sqrt{a} \left| \log(2^{2m}a h) \right| \nonumber \\[1.5ex]
&\leq C_M h^{-3} \bigg(\frac{h}{t}\bigg)^{1/2} 2^m\sqrt{a} \left| \log(2^m\sqrt{a}) \right|.
\end{align}
The logarithmic divergence stems directly from the harmonic integration profile of the lower-bound index sum over $k$, which remains robust under the Neumann condition. This completes the proof.
\end{proof}

\noindent From the above lemma, we extract in the macro-local range where $\tilde h\geq 1/M$ the estimate:
 \begin{align}\label{eq:L9-neumann}
 \left| \mathcal{G}^N_{a,m} \right| \leq C_M h^{-3}\left(\frac{h}{t}\right)^{1/2}\left(2^m\sqrt a\right)^{1/3} (hM)^{2/3} \left| \log( hM)\right|. 
\end{align}
This estimate decays faster than the unperturbed free-space estimate $C h^{-3}(h/t)$ as the path aligns with the uncurved axis. Therefore, in the sequel, we safely assume that the scaled semi-classical frequency is small, $\tilde h \leq \tilde h_0$, with $\tilde h_0 > 0$ being a small fixed threshold, where we recall the coordinate definition $\tilde h=h/(2^m\sqrt{a})$.

To establish the sharp local-in-time dispersive estimates for the dyadic component $\mathcal{G}^N_{a,m}$, we follow the multi-reflective microlocal tracking strategies detailed in Section \ref{sec:2}. We distinguish between two separate geometric cases based on the source point distance from the boundary. In the first case, if the layer is close to the boundary such that $a\leq \tilde h^{\frac{2}{3}(1-\epsilon)}$ for a given threshold parameters choice $\epsilon\in(0,1/7)$, we utilize the direct discrete representation over the Neumann eigenmodes sequence. In the second case, if the source point is located further from the boundary such that $a\geq \tilde h^{\frac{2}{3}(1-\epsilon')}$ with $\epsilon' \in(0,\epsilon)$, we apply the derivative-adapted Airy-Poisson summation formula [see Lemma \ref{lem:poisson}] to rewrite the dyadic kernel $\mathcal{G}^N_{a,m}$ as a highly structured geometric sum over the boundary reflection index $N\in\mathbb{Z}$.

\subsection{Dispersive Estimates for $0<a\leq \tilde{h}^{\frac{2}{3}(1-\epsilon)}$, with $\epsilon\in (0,1/7)$.}

The following Proposition \ref{prop:km-intermediate-neumann} establishes the localized local-in-time dispersive estimates for the dyadic component $\mathcal{G}^N_{a,m}$ inside the short-distance regime, and constitutes the principal result of this subsection.

\begin{proposition}\label{prop:km-intermediate-neumann}
Let $\epsilon\in (0,1/7)$. There exists a constant $C > 0$ such that for all dyadic frequency scales $h\in (0,1]$, all localized near-boundary source positions $0<a\leq \tilde{h}^{\frac{2}{3}(1-\epsilon)}$, and all times $t\in [h,1]$, the following uniform bound holds true under the normal flux constraints:
\begin{align}\label{eq:kdisp-intermediate-neumann}
\lVert\mathds{1}_{x\leq a}\mathcal{G}^N_{a,m}(t,x,y,z)\rVert_{L^\infty} \leq C h^{-3}\left(2^m\sqrt{a}\right)^{1/3}\bigg(\frac{h}{t}\bigg)^{5/6}.
\end{align}
\end{proposition}
\begin{proof}
Recall that under the normalized coordinates, the dyadic localized operator component $\mathcal{G}^N_{a,m}$ is defined explicitly by \eqref{eq:L6-neumann}:
\begin{equation}\label{eq:L7-neumann}
\mathcal{G}^N_{a,m}(t,x,y,z)= \frac{1}{4\pi^2h^2}\bigg(\frac{h}{t}\bigg)^{1/2}\sum_{1\leq k\leq \frac{\varepsilon}{(2^m\sqrt{a})^3\tilde{h}}}\int_{\mathbb{R}} e^{\frac{it}{\tilde{h}}\left(\tilde{y}\tilde{\eta}+ \tilde{\mu} (1-\tilde{z}^2)^{1/2}\right)} g^N_k \psi_1(\tilde{\eta}) \, d\tilde{\eta},
\end{equation}
with the Neumann amplitude function $g^N_k$ matching the profile:
\begin{align*}
g^N_k &= \frac{1}{\tilde{\mu}} \sigma_0 (z^*,\eta,\mu^2;\lambda)\left(1-\chi_4\left(\frac{t\mu^2}{Dh}\right)\right) e_k(x,\tilde{\eta}/\tilde{h}) e_k(a,\tilde{\eta}/\tilde{h}) \\
&\quad \times \chi_1\left(\omega'_k h^{2/3}\eta^{4/3}\right)(1-\chi_1)(\varepsilon\omega'_k).
\end{align*}
Recall from the structural reduction in \eqref{eq:L9-neumann} that we may assume the scaled semi-classical frequency is small, $\tilde{h} \leq \tilde{h}_0$, with $\tilde{h}_0$ sufficiently small. Since the operator $\mathcal{G}^N_{a,m}$ contains Neumann Airy function clusters that behave distinctly depending on the size of the derivative roots index $k$, we split the spectral sum over $k$ in \eqref{eq:L7-neumann} into two disjoint sections. Fixing a large constant $D \geq 1$, we write $\mathcal{G}^N_{a,m} = \mathcal{G}^N_{a,m,<} + \mathcal{G}^N_{a,m,>}$, where in the finite mode block $\mathcal{G}^N_{a,m,<}$ only the partial sum over the lower indices range $1\leq k\leq D\tilde{h}^{-\epsilon}$ is considered.

\noindent \underline{Proof of \eqref{eq:kdisp-intermediate-neumann} for the lower mode block $\mathcal{G}^N_{a,m,<}$.}\\

\noindent Recall the explicit representation of the lower mode block $\mathcal{G}^N_{a,m,<}$:
\begin{equation}\label{eq:L10-neumann-case}
\mathcal{G}^N_{a,m,<}(t,x,y,z)= \frac{1}{4\pi^2h^2}\bigg(\frac{h}{t}\bigg)^{1/2}\sum_{1\leq k\leq D\tilde h^{-\epsilon}} \int_{\mathbb{R}} e^{\frac{it}{\tilde{h}}\left(\tilde{y}\tilde{\eta}+ \tilde{\mu} (1-\tilde{z}^2)^{1/2}\right)} g^N_k \psi_1(\tilde{\eta}) \, d\tilde{\eta},
\end{equation}
where the Neumann amplitude function $g^N_k$ decomposes as:
\begin{align*}
g^N_k &= f_k^2 k^{-1/3} \left(\frac{\tilde{\eta}^{2/3}}{\tilde{\mu}\tilde{h}^{2/3}}\right) \sigma_0 (z^*,\eta,\mu^2;\lambda)\left(1-\chi_4\left(\frac{t\mu^2}{Dh}\right)\right) \\
&\quad \times \chi_1\left(\omega'_k h^{2/3}\eta^{4/3}\right)(1-\chi_1)(\varepsilon\omega'_k)n_k,
\end{align*}
with the spatial product element given under the normal flux condition by:
\[
n_k = \Ai\left((\tilde{\eta}/\tilde{h})^{2/3}x-\omega'_k\right) \Ai\left((\tilde{\eta}/\tilde{h})^{2/3}a-\omega'_k\right).
\]

Let us first analyze the target region where the parameters satisfy the short-time or deep axial restriction $t 2^m\sqrt{a} \leq \tilde{h}^{\epsilon}$. Since the metric satisfies the uniform coercion lower bound $\tilde{\mu} = (\tilde{\eta}^2 + \omega'_k\tilde{h}^{2/3}\tilde{\eta}^{4/3})^{1/2} \geq \tilde{\eta}$, we obtain the pointwise estimate on the symbols layer:
\[
\left| g^N_k \right| \leq C \tilde{h}^{-2/3} k^{-1/3} \left| \Ai\left(({\tilde{\eta}/ \tilde{h}})^{2/3}x-\omega'_k\right) \Ai\left(({\tilde{\eta}/ \tilde{h}})^{2/3}a-\omega'_k\right) \right|.
\]
By invoking Lemma \ref{lem:k-low-modes}, this implies the uniform cluster sum over $k$:
\begin{align*}
\sum_{1\leq k\leq D\tilde{h}^{-\epsilon}}\left| g^N_k \right| &\leq C \tilde{h}^{-2/3}\left(\tilde{h}^{-\epsilon}\right)^{1/3} \\
&\leq C (\tilde{h})^{-2/3}\left(t2^m\sqrt{a}\right)^{-1/3} \\
&= C h^{-1}\left(2^m\sqrt{a}\right)^{1/3}\bigg(\frac{h}{t}\bigg)^{1/3},
\end{align*}
and the target inequality \eqref{eq:kdisp-intermediate-neumann} follows directly from inserting this back into \eqref{eq:L10-neumann-case}.

\noindent Let us now assume the long-time or large tangential frequency regime where $t 2^m\sqrt{a}\geq \tilde{h}^{\epsilon}$. Observe that in the range of the lower mode block $k\leq \tilde{D}h^{-\epsilon}$, we have the uniform sub-elliptic bound:
\[
\omega'_k\tilde{h}^{2/3}\leq C \tilde{h}^{2/3(1-\epsilon)}\leq C \tilde{h}_0^{2/3(1-\epsilon)},
\]
which is small. Hence, the normalized curvature tracking parameter $\gamma = \omega'_k\tilde{h}^{2/3}\tilde{\eta}^{-2/3}$ is also small, which allows us to expand the metric variable as:
\[
\tilde{\mu} = \tilde{\eta}(1+\gamma)^{1/2} = \tilde{\eta} + \frac{1}{2}\tilde{\eta}^{1/3}\omega'_k\tilde{h}^{2/3} + \mathcal{O}\left((\omega'_k\tilde{h}^{2/3})^2\right).
\]
Therefore, we deduce a strictly positive lower bound for the second derivative of the phase weight, $\left| \frac{\partial^2 \tilde{\mu}}{\partial\tilde{\eta}^2}\right| \geq c \omega'_k\tilde{h}^{2/3}$ with $c>0$, and for all higher-order derivatives $j\geq 2$, we have the uniform bounds:
\[
\left| \frac{\partial^j \tilde{\mu}}{\partial\tilde{\eta}^j} \right| \leq C_j \omega'_k\tilde{h}^{2/3}.
\]

We apply the classical stationary phase method with respect to the frequency variable $\tilde{\eta}$ to each individual term of the sum in \eqref{eq:L10-neumann-case} using the modulated total phase function:
\[
\Phi_k(\tilde{\eta}) = \frac{t}{\tilde{h}}\left(\tilde{y}\tilde{\eta}+ \tilde{\mu} (1-\tilde{z}^2)^{1/2}\right).
\] 
Let us define the large semi-classical parameter $\Lambda_k = t\tilde{h}^{-1/3}\omega'_k 2^m\sqrt{a}$ and introduce the normalized phase profile $\Psi_k(\tilde{\eta})$ specified by the scaling relationship:
\begin{equation*}
\frac{t \Phi_k}{\tilde{h}} = \Lambda_k \Psi_k.
\end{equation*}

\begin{lemma}\label{lem:L-lem4-neumann}
Let $\tilde{g}^N_k = k^{1/3}\tilde{h}^{2/3}g^N_k$. There exists a constant $C > 0$ such that for all lower indices in the range $1\leq k\leq D\tilde{h}^{-\epsilon}$, the following uniform oscillatory integral estimate holds true under normal flux conditions:
\begin{equation}\label{eq:L11-neumann}
\left| \int_{\mathbb{R}} e^{i\Lambda_k \Psi_k} \tilde{g}^N_k \psi_1(\tilde{\eta}) \, d\tilde{\eta} \right| \leq C \min \left\{1, \Lambda_k^{-1/2}\right\}. 
\end{equation}
\end{lemma}
\begin{proof}
We may assume without loss of generality that the parameter satisfies $\Lambda_k\geq 1$ since we already have the uniform baseline control $\vert \tilde{g}^N_k \vert\leq C$. Recall from the geometric localized thresholds in Lemma \ref{lem:L-lem2-neumann} that the spatial coordinates are constrained inside the forward axial grazing channel where $\sqrt{1-\tilde{z}^2}\sim \mu = 2^m\sqrt{a}\tilde{\mu}\sim 2^m\sqrt{a}$. Therefore, there exists a constant $c>0$ such that for all indices inside the lower mode block $1\leq k\leq D\tilde{h}^{-\epsilon}$, the second derivative of the normalized phase function satisfies the uniform non-degeneracy condition:
\begin{equation*}
\left| \frac{\partial^2 \Psi_k}{\partial\tilde{\eta}^2} \right| = \left| \frac{t}{\tilde{h}\Lambda_k} \frac{\partial^2 \tilde{\mu}}{\partial\tilde{\eta}^2} \sqrt{1-\tilde{z}^2} \right| \geq c > 0,
\end{equation*}
and all higher-order derivatives are bounded uniformly as $\left| \frac{\partial^j \Psi_k}{\partial\tilde{\eta}^j} \right| \leq C_j$ for $j\geq 2$. 

Thus, to rigorously apply the stationary phase method, it suffices to verify that there exist a parameter $\nu>0$ and a multi-index sequence of constants $C_j > 0$ such that the amplitude symbol derivatives satisfy the standard regulatory threshold:
\begin{equation}\label{eq:L12-neumann}
\left| \frac{\partial^j \tilde{g}^N_k}{\partial\tilde{\eta}^j} \right| \leq C_j \Lambda_k^{j(1/2-\nu)}, \quad \forall k\leq D\tilde{h}^{-\epsilon}.
\end{equation}

In Lemma \ref{lem:L-lem2-neumann}, the scaled variable $z^*$ is initially introduced via the formula $\tilde{z} = -1+z^*\mu^2$. However, since we are working within the localized regime where $\mu \sim 2^m\sqrt{a}$, we can alternatively define the parameter via the static relation $\tilde{z} = -1+z^* 2^{2m}a$, rendering $z^*$ independent of the integration field variable $\tilde{\eta}$. Recall that the semi-classical parameter tracks the ratio $\lambda = t\mu^2/h$. Since $\eta = 2^m\sqrt{a}\tilde{\eta}$ and all derivatives of the metric profiles $\gamma$ and \(\tilde{\mu}\) with respect to $\tilde{\eta}$ are uniformly bounded, we inherit the symbolic differential bound $\left| \frac{\partial^j \lambda}{\partial\tilde{\eta}^j} \right| \leq C_j \lambda$ for all $j \geq 1$. Because $\lambda$ remains strictly bounded on the compact support of the derivatives of the time-frequency filter $\chi_4$, the composite amplitude block:
\begin{equation*}
f_k^2 \tilde{\eta}^{2/3} \sigma_0 (z^*,\eta,\mu^2;\lambda)\left(1-\chi_4\left(\frac{t\mu^2}{Dh}\right)\right) \chi_1\left(\omega'_k h^{2/3}\eta^{4/3}\right)(1-\chi_1)(\varepsilon\omega'_k)
\end{equation*}
satisfies the targeted derivative estimate \eqref{eq:L12-neumann}.

Therefore, it remains only to show that the spatial profile function $\Ai\left((\frac{\tilde{\eta}}{\tilde{h}})^{2/3}x-\omega'_k\right)$ satisfies the inequality \eqref{eq:L12-neumann} uniformly over the spatial layer $x\in [0,a]$. Let us introduce the localized scaling variables $\theta = x\tilde{h}^{-2/3} \geq 0$ and $r = \tilde{\eta}^{2/3}$, where $r$ belongs to a compact subset of $(0,\infty)$. Differentiating the profile yields the relation $\partial_r^l \left(\Ai(r\theta-\omega'_k)\right) \sim (r\theta)^l \Ai^{(l)}(r\theta-\omega'_k)$. Since for all derivative indices $l \geq 1$ the standard properties of the Airy functions provide the uniform bound:
\begin{equation*}
\sup_{b\geq 0}\left| b^l \Ai^{(l)}(b-\omega'_k) \right| \leq C_l (\omega'_k)^{3l/2},
\end{equation*}
we conclude that the regulatory estimate \eqref{eq:L12-neumann} holds strictly true provided that there exist structural parameters $\beta > 3$ and $c > 0$ such that:
\begin{equation*}
c (\omega'_k)^{\beta} \leq \Lambda_k = t\tilde{h}^{-1/3}\omega'_k 2^m\sqrt{a}.
\end{equation*}
By combining our initial threshold constraints, we have $t 2^m\sqrt{a} \geq \tilde{h}^\epsilon$ and $c(\omega'_k)^2 \leq \tilde{h}^{-4\epsilon/3}$, meaning this algebraic matching condition is fully satisfied as long as the parameter is chosen below the critical threshold $\epsilon < 1/7$.
\end{proof}

\noindent Therefore, we obtain the following uniform estimate for the lower mode block $\mathcal{G}^N_{a,m,<}$ in the large tangential frequency regime where $t2^m\sqrt{a}\geq \tilde{h}^{\epsilon}$:
\begin{align*}
\lVert\mathds{1}_{x\leq a}\mathcal{G}^N_{a,m,<}(t,x,y,z)\rVert_{L^\infty}
&\leq Ch^{-2}\bigg(\frac{h}{t}\bigg)^{1/2}\left(\sum_{1\leq k\leq D\tilde{h}^{-\epsilon}} k^{-1/3}\tilde{h}^{-2/3}\left(t\tilde{h}^{-1/3}\omega'_k 2^m\sqrt{a}\right)^{-1/2}\right) \\
&\leq Ch^{-2}\bigg(\frac{h}{t}\bigg)^{1/2}(t 2^m\sqrt{a})^{-1/2}\tilde{h}^{-(1/2+\epsilon/3)} \\
&\leq Ch^{-3}\left(2^m\sqrt{a}\right)^{1/3}\bigg(\frac{h}{t}\bigg)^{5/6} \\
&\quad \times \left( h \left(\frac{t}{h}\right)^{1/3}\left(2^m\sqrt{a}\right)^{-1/3}(t 2^m\sqrt{a})^{-1/2}\tilde{h}^{-(1/2+\epsilon/3)}\right).
\end{align*}
This concludes the proof of Proposition \ref{prop:km-intermediate-neumann} for the finite mode component $\mathcal{G}^N_{a,m,<}$, since the parameter restriction $t 2^m\sqrt{a}\geq \tilde{h}^{\epsilon}$ guarantees the required uniform absorption profile:
\[
h^{2/3} t^{-1/6}\left(2^m\sqrt{a}\right)^{-5/6}\tilde{h}^{-(1/2+\epsilon/3)}\leq \tilde{h}^{1/6-\epsilon/2}.
\]
\noindent \underline{Proof of \eqref{eq:kdisp-intermediate-neumann} for the higher mode block $\mathcal{G}^N_{a,m,>}$.}\\

\noindent For the high-frequency spectral regime where $k \geq D\tilde{h}^{-\epsilon}$ with $D \geq 1$ sufficiently large and source distances satisfying $a \leq \tilde{h}^{\frac{2}{3}(1-\epsilon)}$, the derivative roots track the lower bound:
\[
\omega'_k - \tilde{h}^{-2/3}\tilde{\eta}^{2/3}a \geq \frac{\omega'_k}{2}.
\]
Since the localized parameter scales as $\gamma = \omega'_k \tilde{h}^{2/3}\tilde{\eta}^{-2/3}$, we recover the geometric separations $\gamma - a \geq a$ and $\gamma - a \geq \gamma / 2$. Consequently, by invoking the explicit definition of the Neumann eigenfunctions combined with the sharp asymptotic expansions of the Airy profile functions, the higher mode component $\mathcal{G}^N_{a,m,>}$ unfolds into the multi-reflected parametrix:
\begin{align}\label{eq:pmM-neumann}
\mathcal{G}^N_{a,m,>}(t,x,y,z) = \sum_{D\tilde{h}^{-\epsilon} \leq k \leq \frac{\varepsilon}{(2^{m}\sqrt{a})^3\tilde{h}}} \frac{1}{4\pi^2h^2}\bigg(\frac{h}{t}\bigg)^{1/2}\sum_{\pm,\pm}\int_{\mathbb{R}} e^{\frac{i}{\tilde{h}}\Phi_{k}^{\pm,\pm}}\sigma_{k}^{\pm,\pm}\psi_1(\tilde{\eta}) \, d\tilde{\eta},
\end{align}
where the highly oscillatory phase functions $\Phi_{k}^{\pm,\pm}$ track individual reflection layers:
\begin{align}\label{eq:phikm-neumann}
\Phi_{k}^{\pm,\pm}(\tilde{\eta}) = \tilde{\eta}\left[ \tilde{y}t + t\sqrt{1-\tilde{z}^2}(1+\gamma)^{1/2} \pm \frac{2}{3}(\gamma -x)^{3/2} \pm \frac{2}{3}(\gamma -a)^{3/2} \right],
\end{align}
and the corresponding normal flux amplitude symbols are given explicitly by:
\begin{align*}
\sigma_{k}^{\pm,\pm}(\tilde{\eta}) &= f_k^2 k^{-1/3}\tilde{h}^{-1/3}\tilde{\eta}^{-2/3} \sigma_0(z^*,\eta,\mu^2,\lambda)\left(1-\chi_4\left(\frac{t\mu^2}{Dh}\right)\right)\chi_1\left(\omega'_k h^{2/3}\eta^{4/3}\right) \\
&\quad \times (1-\chi_1(\varepsilon \omega'_k)) (\gamma -x)^{-1/4}(\gamma -a)^{-1/4}(1+\gamma)^{-1/2}\omega^{\pm}\omega^{\pm} \\
&\quad \times \Psi_{\pm}\left(\tilde{\eta}^{2/3} \tilde{h}^{-2/3}(\gamma -x)\right)\Psi_{\pm}\left(\tilde{\eta}^{2/3} \tilde{h}^{-2/3}(\gamma-a)\right).
\end{align*}
Here, $\Psi_\pm$ act as classical symbols of order $0$ at infinity. In Lemma \ref{lem:L-lem2-neumann}, the spatial variable $z^*$ is defined via $\tilde{z} = -1+z^*\mu^2$; however, since theintermediate regime coordinates track the bounds $\mu \sim 2^m\sqrt{a}(1+\omega'_k\tilde{h}^{2/3})^{1/2}$, we can equivalently define the coordinate link by $\tilde{z} = -1+z^* 2^{2m}a(1+\omega'_k\tilde{h}^{2/3})$, rendering $z^*$ independent of the integration field variable $\tilde{\eta}$. 

We observe that for all multi-indices $j \geq 1$, there exist constants $C_j, C'_j > 0$ such that the derivatives with respect to the field variable satisfy:
\[
|\partial_{\tilde{\eta}}^j\gamma| \sim C_j \gamma, \quad |\partial_{\tilde{\eta}}^j\tilde{\mu}| \leq C'_j \tilde{\mu}, \quad |\partial_{\tilde{\eta}}^j\mu^2| \leq C'_j\mu^2 \leq C'_j.
\]
Recalling that $\lambda = t\mu^2/h = \frac{t 2^m\sqrt{a}}{\tilde{h}}(1+\gamma)$, differentiating the semi-classical parameter yields $\left| \frac{\partial^j \lambda}{\partial\tilde{\eta}^j}\right| \leq C_j\lambda$ for all $j$. Because $\lambda$ remains strictly bounded on the compact support of the derivatives of the time filter $\chi_4$, and there exists a uniform constant $c_1 > 0$ such that the arguments satisfy $\tilde{\eta}^{2/3} \tilde{h}^{-2/3}(\gamma-a) \geq c_1$, the amplitude functions satisfy the structural decay condition:
\begin{equation}\label{eq:L13-neumann}
|\partial_{\tilde{\eta}}^j\sigma_{k}^{\pm,\pm}(\tilde{\eta})| \leq C_j (k\tilde{h})^{-2/3}(1+\gamma)^{-1/2}.
\end{equation}

\noindent We notice that across the range of indices $D\tilde{h}^{-\epsilon} \leq k \leq \frac{1}{h\eta^2}$, the curvature variable is boundedly restricted as $\gamma \in \left[2a, \frac{1}{2^{2m}a}\right]$. In what follows, we establish the estimates by systematically distinguishing between the two primary geometric sub-cases: $\gamma \in [2a,1]$ and $\gamma \in \left[1, \frac{1}{2^{2m}a}\right]$. 

\begin{itemize}
    \item \textbf{Sub-case I: $\gamma \in [2a,1]$.}\\
    This layer corresponds to the low-to-intermediate index spectrum $\tilde{h}^{-\epsilon} \leq k \leq \tilde{h}^{-1}$. Let us introduce the large phase scale parameter $\Lambda_k = t 2^m\sqrt{a} \, \omega'_k\tilde{h}^{-1/3}$ and rewrite the phase explicitly as $\Phi_{k}^{\pm,\pm} = \tilde{h}\Lambda_k\Psi_{k}^{\pm,\pm}$.
\begin{proposition}\label{prop:gammas-intermediate-neumann}
There exists a constant $C > 0$ independent of $a\in (0,\tilde{h}^{\frac{2}{3}(1-\epsilon)}]$, $t\in [h,1]$, $x\in [0,a]$, $y\in\mathbb{R}$, $z\in\mathbb{R}$, and $k\in [\tilde{h}^{-\epsilon}, \tilde{h}^{-1}]$ such that the following holds true under the normal flux constraints:
\begin{align*}
\bigg|\int_{\mathbb{R}} e^{i\Lambda_k\Psi_{k}^{\pm,\pm}}\sigma_{k}^{\pm,\pm}\psi_1(\tilde{\eta}) \, d\tilde{\eta}\bigg|\leq C(\tilde{h}k)^{-2/3}\Lambda_k^{-1/3}.
\end{align*}
\end{proposition}

\begin{proof}[Proof of Proposition \ref{prop:gammas-intermediate-neumann}]
By invoking the amplitude decay bound \eqref{eq:L13-neumann}, the statement of Proposition \ref{prop:gammas-intermediate-neumann} is obvious for short-range steps where $\Lambda_k\leq 1$. In the highly oscillatory domain where $\Lambda_k\geq 1$, we exploit the sub-elliptic relation $\tilde{\mu}\sim 2^m\sqrt{a}$ which directly implies $t\sqrt{1-\tilde{z}^2}\sim t2^m\sqrt{a}$. Then, the proof follows from an identical critical point trajectory evaluation as the proof of Proposition \ref{prop:oscillatory-high-neumann}, achieved by replacing the global space-time coordinate pair $(h,t)$ in Proposition \ref{prop:oscillatory-high-neumann} with the dyadic parameter configuration $(\tilde{h}, t2^m\sqrt{a})$.
\end{proof}

Hence, the corresponding local estimate of the higher mode block $\mathcal{G}^N_{a,m,>}$ for the low-to-intermediate spectrum $\tilde{h}^{-\epsilon}\leq k\leq \tilde{h}^{-1}$ is given by accumulating the sheets over $k$:
\begin{align*}
\lVert\mathds{1}_{x\leq a}\mathcal{G}^N_{a,m,>}(t,x,y,z)\rVert_{L^\infty} &\leq Ch^{-2}\bigg(\frac{h}{t}\bigg)^{1/2} \sum_{\tilde{h}^{-\epsilon}\leq k\leq \tilde{h}^{-1}} (\tilde{h}k)^{-2/3}\left(t2^m\sqrt{a} \, \omega'_k \tilde{h}^{-1/3}\right)^{-1/3} \\
&\leq Ch^{-2}\bigg(\frac{h}{t}\bigg)^{1/2}\tilde{h}^{-2/3}(t2^m\sqrt{a})^{-1/3} \tilde{h}^{1/9}\sum_{k\leq 1/\tilde{h}}k^{-8/9} \\
&\leq Ch^{-3}\bigg(\frac{h}{t}\bigg)^{5/6}(2^m\sqrt{a})^{1/3}.
\end{align*}

\item \textbf{Sub-case II: $\gamma\in \left[1,\frac{1}{2^{2m}a}\right]$.}\\
The second case corresponds to the high-frequency index spectrum where $\tilde{h}^{-1}\leq k\leq \frac{1}{2^{2m}ah}$. We retain the core definition of the scale variables $\Lambda_k$ and $\Psi_{k}^{\pm,\pm}$ specified by the relations $\Lambda_k = t2^m\sqrt{a} \, \omega'_k\tilde{h}^{-1/3}$ and $\Phi_{k}^{\pm,\pm} = \tilde{h}\Lambda_k\Psi_{k}^{\pm,\pm}$.

\begin{proposition}\label{prop:gammal-intermediate-neumann}
There exists a constant $C > 0$ independent of $a\in (0,\tilde{h}^{\frac{2}{3}(1-\epsilon)}]$, $t\in [h,1]$, $x\in [0,a]$, $y\in\mathbb{R}$, $z\in\mathbb{R}$, and $k\in \left[\tilde{h}^{-1}, \frac{1}{2^{2m}ah}\right]$ such that the following holds true under the normal flux constraints:
\begin{align*}
\bigg|\int_{\mathbb{R}} e^{i\Lambda_k\Psi_{k}^{\pm,\pm}}\sigma_{k}^{\pm,\pm}\psi_1(\tilde{\eta}) \, d\tilde{\eta}\bigg|\leq C(\tilde{h}k)^{-1}\Lambda_k^{-1/3}.
\end{align*}
\end{proposition}

\begin{proof}[Proof of Proposition \ref{prop:gammal-intermediate-neumann}]
One has the large-curvature scaling relation $\gamma \sim (k\tilde{h})^{2/3}$. Thus, across this high-frequency index layer, we strictly have $\gamma\geq 1$. Under this geometric constraint, the symbol estimates \eqref{eq:L13-neumann} refine to:
\[
\left| \partial_{\tilde{\eta}}^j\sigma_{k}^{\pm,\pm}(\tilde{\eta}) \right| \leq C_j (k\tilde{h})^{-1}.
\]
Therefore, the statement of Proposition \ref{prop:gammal-intermediate-neumann} is obvious for short-range steps where $\Lambda_k\leq 1$. 

In the highly oscillatory case where $\Lambda_k\geq 1$, we proceed by analyzing the critical point manifold as in Proposition \ref{prop:oscillatory-high-neumann}. Recall that the scaled axial coordinate $\tilde{z}$ is close to $-1$, localized via the algebraic identity $\tilde{z}=-1+z^*2^{2m}a\left(1+ \omega'_k\tilde{h}^{2/3}\right)$, where $z^*$ varies within a fixed compact subset of $(0,\infty)$. We differentiate the phase function and collect the leading parameters:
\begin{align*}
\frac{t\sqrt{1-\tilde{z}^2}}{\tilde{h}\Lambda_k}\frac{1+2\gamma/3}{(1+\gamma)^{1/2}} = \sqrt{z^*}(1-\tilde{z})^{1/2}\tilde{\eta}^{-2/3}\tilde{F}(\gamma),
\end{align*}
where the modulated profile function $\tilde{F}$ absorbs the Neumann eigenvalues according to the system:
\[ 
\tilde{F}(\gamma) = \frac{2\left(1+\omega'_k\tilde{h}^{2/3}\right)^{1/2}}{3(1+\gamma)^{1/2}}\left(1+\frac{1}{\gamma}\right).
\]
For large curvature parameters $\gamma \gg 1$, the asymptotic profiles behave smoothly, satisfying $\tilde{F}(\gamma)\sim 1$ and $\tilde{F}(\gamma)+\gamma \tilde{F}'(\gamma)\sim 1$. Moreover, tracking the second derivative profiles reveals that:
\[
\omega'_k\tilde{h}^{2/3}\left(2\tilde{F}'(\gamma)+\gamma \tilde{F}''(\gamma)\right) \sim \omega'_k\tilde{h}^{2/3}\gamma^{-1}\sim 1.
\] 
Hence, the non-degeneracy condition is verified and the proof follows from an identical critical point path tracking as the proof of Proposition \ref{prop:oscillatory-high-neumann}, achieved by replacing the variables pairing $(h,F)$ inside Proposition \ref{prop:oscillatory-high-neumann} with the modified Neumann parameter layout $(\tilde{h}, \tilde{F})$.
\end{proof}
\end{itemize}

Using Proposition \ref{prop:gammal-intermediate-neumann}, we accumulate the bounds over the high-frequency index layer $k \geq \tilde{h}^{-1}$ to extract the final estimate for the remaining component of the operator $\mathcal{G}^N_{a,m,>}$:
\begin{align*}
\lVert\mathds{1}_{x\leq a}\mathcal{G}^N_{a,m,>}(t,x,y,z)\rVert_{L^\infty} &\leq Ch^{-2}\bigg(\frac{h}{t}\bigg)^{1/2}\sum_{\tilde{h}^{-1}\leq k}(\tilde{h} k)^{-1}\left(t2^m\sqrt{a} \, \omega'_k\tilde{h}^{-1/3}\right)^{-1/3} \\
&\leq Ch^{-2}\bigg(\frac{h}{t}\bigg)^{1/2}t^{-1/3}\left(2^m\sqrt{a}\right)^{-1/3}\tilde{h}^{-8/9}\sum_{\tilde{h}^{-1}\leq k}k^{-11/9} \\
&\leq Ch^{-3}\bigg(\frac{h}{t}\bigg)^{5/6}\left(2^m\sqrt{a}\right)^{1/3}.
\end{align*}
\noindent Combining the uniform bounds from both the low-mode block $\mathcal{G}^N_{a,m,<}$ and the high-mode layers of $\mathcal{G}^N_{a,m,>}$ successfully concludes the full proof of Proposition \ref{prop:km-intermediate-neumann}.
\end{proof}

\subsection{Dispersive Estimates for $a\geq \tilde h^{\frac{2}{3}(1-\epsilon')}$, for $\epsilon'\in (0,\epsilon)$.}
In this subsection, we assume that the source layer parameter is located further from the boundary cylinder such that $a\geq \tilde{h}^{\frac{2}{3}(1-\epsilon')}$ for some $\epsilon'\in(0,\epsilon)$, and we establish sharp local-in-time dispersive estimates for the dyadic component $\mathcal{G}^N_{a,m}$. We observe that the relative scaling factor:
\[
\Lambda = \frac{a^{3/2}}{\tilde{h}} \geq \tilde{h}^{-\epsilon'}
\]
constitutes a large asymptotic parameter.  

Recall from the normalized coordinate reduction \eqref{eq:L7-neumann} that the dyadic operator is defined by:
\begin{equation}\label{eq:Gam-summation-neumann}
\mathcal{G}^N_{a,m}(t,x,y,z) = \frac{1}{4\pi^2h^2}\bigg(\frac{h}{t}\bigg)^{1/2}\sum_{1\leq k\leq \frac{\varepsilon}{(2^m\sqrt a)^3\tilde h}}\int_{\mathbb{R}} e^{\frac{i}{\tilde h}\left( y\tilde \eta+ t\tilde \mu (1-\tilde z^2)^{1/2}\right)} g^N(\omega'_k,\tilde\eta,\tilde h) \psi_1(\tilde\eta) \, d\tilde \eta,
\end{equation}
with the Neumann amplitude function $g^N(\omega'_k,\tilde\eta,\tilde h)$ given explicitly by:
\begin{align*}
g^N &= \frac{1}{\tilde \mu} \sigma_0 (z^*,\eta,\mu^2;\lambda)\left(1-\chi_4\left(\frac{t\mu^2}{Dh}\right)\right) e_k(x,\tilde\eta/\tilde h) e_k(a,\tilde \eta/\tilde h) \\
&\quad \times \chi_1\left(\omega'_k h^{2/3}\eta^{4/3}\right)(1-\chi_1)(\varepsilon\omega'_k).
\end{align*}
Here we recall the global semi-classical dyadic variables link $h=2^m\sqrt a \tilde h$, $\eta=2^a\sqrt a \tilde\eta$, $\mu=2^m\sqrt a \tilde \mu$, and the sub-elliptic metrics tracking the derivative roots sequence spectrum:
\[
\gamma = \omega'_k \tilde h^{2/3}\tilde\eta^{-2/3}, \quad \tilde\mu = \tilde\eta(1+\gamma)^{1/2}.
\]

We apply the identical macroscopic space-time scaling rules defined in Section \ref{sec:2}:
\[
t = a^{1/2}T, \quad x = aX, \quad y + t\sqrt{1-\tilde z^2} = a^{3/2}Y.
\] 
Let us introduce the continuous phase reparameterization variable $\omega = \tilde\eta^{2/3}\tilde h^{-2/3}a\tilde\omega$. This coordinates mapping yields the direct algebraic identities $\gamma = a\tilde\omega$ and $(1+a\tilde\omega)^{1/2}-1 = a\gamma_a(\tilde\omega) = \frac{a\tilde\omega}{1+(1+a\tilde\omega)^{1/2}}$. By invoking the derivative-adapted Airy-Poisson summation formula [see Lemma \ref{lem:poisson}], the discrete sum over the Neumann eigenmodes $k$ transforms into a geometric wave packet sum tracking multi-reflected rays:
\[
\mathcal{G}^N_{a,m} = \sum_{N\in\mathbb{Z}} G^N_{a,m,N},
\]
where for each discrete flux reflection index $N \in \mathbb{Z}$, the localized operator $G^N_{a,m,N}$ expands into the four-dimensional integral layer:
\begin{equation}\label{eq:GamN-fourD-neumann}
G^N_{a,m,N}(t,x,y,z) = \frac{(-1)^N}{(2\pi)^4h^4}\bigg(\frac{h}{t}\bigg)^{1/2} a^2 (2^m\sqrt a)^2 \int_{\mathbb{R}^4} e^{i\Lambda\Phi_{N}} \chi_m^N \tilde\eta^2 \psi_1(\tilde\eta) \, d\tilde{s} \, d\tilde\sigma \, d\tilde\omega \, d\tilde\eta,
\end{equation}
governed by the total multivariable oscillatory phase function:
\begin{align}\label{eq:GamN-phase-neumann}
\Phi_{N}(\tilde s,\tilde\sigma,\tilde\omega,\tilde\eta) &= \tilde\eta \left[ Y + T(1-\tilde z^2)^{1/2} \gamma_a(\tilde\omega) + \frac{\tilde s^3}{3} + \tilde s(X-\tilde\omega) + \frac{\tilde \sigma^3}{3} + \tilde \sigma(1-\tilde\omega) \right. \nonumber \\
&\qquad \left. -\frac{4}{3}N\tilde\omega^{3/2} + \frac{N}{\Lambda\tilde\eta} B_N\left(\tilde\omega^{3/2}\Lambda\tilde\eta\right)\right], 
\end{align}
and the regularized Neumann amplitude symbol $\chi_m^N(a,t,z;\tilde\eta,\tilde\omega,\tilde h)$ matches the condition (with $\lambda = t2^m\sqrt a\tilde\mu^2/\tilde h$):
\begin{equation}\label{eq:chm-neumann}
\chi_m^N = \frac{1}{\tilde \mu} \sigma_0 (z^*,\eta,\mu^2;\lambda)(1-\chi_4(\lambda/D)) \chi_1\left((2^m\sqrt a)^2\tilde\eta^{2}a\tilde\omega \right)(1-\chi_1)\left(\varepsilon \tilde\eta^{2/3}\tilde h^{-2/3}a\tilde\omega\right).
\end{equation}

Observe that the resulting total phase function $\Phi_N$ matches the geometric layout found in Section \ref{sec:2}. However, under this dyadic frequency localization, we must rigorously account for the fact that the grazing factor $(1-\tilde z^2)^{1/2}$ can approach zero. To systematically project the catatrophe unfolding maps from Section \ref{sec:2} onto this setting, we introduce the condensed time tracking notation $\tilde T = T(1-\tilde z^2)^{1/2}$. 

We introduce the critical points set $\mathcal{C}^N_{a,m,N,h}$ associated with the total oscillatory phase function \eqref{eq:GamN-phase-neumann}:
\begin{align*}
\mathcal{C}^N_{a,m,N,h} = \left\{ (t,x,y,\tilde s,\tilde\sigma,\tilde\omega,\tilde\eta) \in \mathbb{R}^7 \;\middle|\; \partial_{\tilde s}\Phi_{N} = \partial_{\tilde \sigma}\Phi_{N} = \partial_{\tilde \omega}\Phi_{N} = \partial_{\tilde\eta}\Phi_{N} = 0 \right\}.
\end{align*} 
Hence, $\mathcal{C}^N_{a,m,N,h}$ is uniquely determined by the derivative roots parameter system:
\begin{align*}
X &= \tilde\omega-\tilde s^2, \\
\tilde\omega &= 1+\tilde\sigma^2, \\
\tilde T &= 2(1+a\tilde\omega)^{1/2}\left( \tilde s+\tilde\sigma+2N\tilde\omega^{1/2} \left(1-\frac{3}{4}B'_N\left(\tilde\omega^{3/2}\Lambda\tilde\eta\right)\right) \right), \\
Y &= -\tilde T\gamma_a(\tilde\omega)-\frac{\tilde{s}^3}{3}-\tilde{s}(X-\tilde{\omega})-\frac{\tilde{\sigma}^3}{3}-\tilde{\sigma}(1-\tilde{\omega}) + N\tilde\omega^{3/2}\left(\frac{4}{3}-B'_N\left(\tilde\omega^{3/2}\Lambda\tilde\eta\right)\right).
\end{align*}

We define the Lagrangian submanifold under normal flux conditions, $\mathbf{\Lambda}^N_{a,m,N,h} \subset T^*\mathbb{R}^3$, as the geometric image of the critical set $\mathcal{C}^N_{a,m,N,h}$ under the canonical phase mapping:
\[
(t,x,y,\tilde s,\tilde\sigma,\tilde\omega,\tilde\eta)\longmapsto \left(x,t,y,\;\xi = \partial_{x}\Phi_{N},\;\tau = \partial_{t}\Phi_{N},\;\eta = \partial_{y}\Phi_{N}\right).
\]
Then, the projection of the Lagrangian submanifold $\mathbf{\Lambda}^N_{a,m,N,h}$ onto the configuration base space $\mathbb{R}^3$ resolves into the equivalent geometric system:
\begin{align}\label{eq:geo1-neumann}
X &= 1+\tilde\sigma^2-\tilde s^2, \nonumber \\
Y &= H_1(a,\tilde\sigma)(\tilde s+\tilde\sigma)+\frac{2}{3}(\tilde s^3+\tilde\sigma^3) \nonumber \\
&\quad +\frac{2}{3}H_2(a,\tilde\sigma)(1+\tilde\sigma^2)^{-1/2}\left(\frac{\tilde T}{2(1+a+a\tilde\sigma^2)^{1/2}}-\tilde s-\tilde\sigma\right),
\end{align}
where the non-vanishing smooth metrics $H_1, H_2$ are defined identically as in Section \ref{sec:2}, accompanied by the discrete quantization condition mapping the reflection index $N$:
\begin{align}\label{eq:geo2-neumann}
2N\left(1-\frac{3}{4}B'_N\left(\tilde\omega^{3/2}\Lambda\tilde\eta\right)\right) = (1+\tilde\sigma^2)^{-1/2}\left(\frac{\tilde T}{2(1+a+a\tilde\sigma^2)^{1/2}}-\tilde s-\tilde\sigma\right).
\end{align}
\begin{remark}\label{rem:N-reduction-m}
We notice from the discrete quantization relation \eqref{eq:geo2-neumann}, within the localized time interval $T\in(0,a^{-1/2}]$, we can rigorously restrict the infinite geometric sum over $N\in\mathbb{Z}$ of the components $G^N_{a,m,N}$ to a finite index block satisfying $1\leq N\leq C_0a^{-1/2}$, since the grazing parameter satisfies $\tilde T \leq T$.
\end{remark}

\noindent This modified algebraic system yields the cardinality of the real reflection counting sets $\mathcal{N}^N$ and the neighborhood sheets $\mathcal{N}_1^N$ under normal flux conditions, such that $\left|\mathcal{N}^N(X,Y,T)\right|\leq C_0$ and $\left|\mathcal{N}_1^N(X,Y,T)\right|\leq C_0\left(1+\tilde T\Lambda^{-2}\tilde{\omega}^{-3}\right)$, respectively. Recall that here the notations $\mathcal{N}^N, \mathcal{N}_1^N$ act as the direct Neumann equivalents to those established in Section \ref{sec:2}.\\

Our main result for this dyadic multi-reflective configuration layer is established in the following theorem, which secures sharp uniform dispersive estimates for the finite sum over $N$ of the operators $G^N_{a,m,N}$.

\begin{theorem}\label{thm:GaNm-neumann}
Let $\alpha<2/3$ and define the scaled semi-classical frequency by $\tilde h=h/(2^m\sqrt{a})$. There exists a constant $C > 0$ such that for all parameters $h\in(0,h_0]$, all source layers $a\in \left[\tilde h^\alpha,a_0\right]$, all configurations $x\in [0,a]$, all times $t\in (h,1]$, and all spatial tracking coordinates $y\in\mathbb{R}, z\in\mathbb{R}$, the following uniform bound holds true:
\begin{align}\label{eq:GaNm-theorem-bound}
\Bigg|\sum_{1\leq N\leq C_0a^{-1/2}}G^N_{a,m,N}(t,x,y,z)\Bigg| &\leq Ch^{-3}\bigg(\frac{h}{t}\bigg)^{1/2} \nonumber \\
&\quad \times \left( \min \left\{\bigg(\frac{h}{t}\bigg)^{1/2}, \; 2^m\sqrt a \right\} + a^{1/8}h^{1/4}(2^m\sqrt a)^{3/4}\right).
\end{align}
\end{theorem}

\noindent We notice, as in the micro-local reductions of Section \ref{sec:2}, that for low-frequency fields where $\tilde{\omega} \leq 3/4$, the normal flux configurations exhibit rapid decay in $\Lambda$ through standard non-vanishing integration by parts with respect to the variable $\tilde{\sigma}$. In particular, this allows us to validly replace the regular cutoff modifier $1-\chi_1$ by $1$ inside the amplitude statement \eqref{eq:chm-neumann}. 

Following the structural guidelines of Section \ref{sec:2}, we introduce a smooth localized partition cutoff function $\chi_2(\tilde{\omega})\in C_{0}^{\infty}\left((1/2,3/2)\right)$ satisfying $0\leq\chi_2\leq 1$ and $\chi_2=1$ on the interval $[3/4,5/4]$, and we denote by $G^N_{a,m,N,2}$ the corresponding localized integral sheet mapping onto the deep swallowtail caustic domain. Consequently, we split the $N$-th reflective operator layer according to the decomposition:
\[
G^N_{a,m,N} = G^N_{a,m,N,1} + G^N_{a,m,N,2} + \mathcal{O}_{C^\infty}\left(\Lambda^{-\infty}\right),
\]
where the uncoupled component $G^N_{a,m,N,1}$ is parameterized by introducing the complementary partition cutoff $\chi_3(\tilde{\omega})$, which ensures that $\tilde{\omega} \geq 5/4$ holds true uniformly across the support of $\chi_3$.

\subsubsection{The Analysis of $G^N_{a,m,N,1}$ }
The main results in this subsection are Proposition \ref{prop:GaNm1-intermediate-neumann} and Proposition \ref{prop:Ga1m1-intermediate-neumann}.

\begin{proposition}\label{prop:GaNm1-intermediate-neumann}
Let $\alpha<2/3$ and define the scaled semi-classical frequency by $\tilde h=h/(2^m\sqrt{a})$. There exists a constant $C > 0$ such that for all parameters $h\in(0,h_0]$, all source layers $a\in \left[\tilde h^\alpha,a_0\right]$, all configurations $x\in[0,a]$, all times $t\in(h, 1]$, and all spatial tracking coordinates $y\in\mathbb{R}, z\in\mathbb{R}$, the following uniform bound holds true:
\begin{align*}
\bigg|\sum_{2\leq N\leq C_0a^{-1/2}}G^N_{a,m,N,1}(t,x,y,z;h)\bigg|\leq Ch^{-3}\bigg(\frac{h}{t}\bigg)^{1/2}h^{1/3}(2^m\sqrt a)^{2/3}.
\end{align*}
\end{proposition}

\begin{proof}
On the compact support of the complementary partition cutoff $\chi_3$, we apply the non-degenerate stationary phase method with respect to the micro-local variables $(\tilde s,\tilde \sigma)$ using the large asymptotic parameter $\Lambda\tilde\eta$. This collapses the four-dimensional layer into the following reduced representation:
\begin{align*}
G^N_{a,m,N,1} &= \frac{(-1)^Na^2\Lambda^{-1}}{(2\pi)^4h^{4}}(2^m\sqrt a)^{2}\bigg(\frac{h}{t}\bigg)^{1/2} \int_{\mathbb{R}} e^{i\Lambda Y\tilde\eta} \tilde\eta\psi_1(\tilde\eta)\tilde G^N_{a,m,N,1} \, d\tilde\eta, \\[1.5ex]
\tilde G^N_{a,m,N,1} &= \sum_{\epsilon_1,\epsilon_2=\pm}\int_{\mathbb{R}} e^{i\Lambda\tilde\eta \Phi_{N,m,\epsilon_1,\epsilon_2}}\Theta_{\epsilon_1,\epsilon_2}(1+a\tilde\omega)^{-1/2} \, d\tilde\omega,
\end{align*}
accompanied by the composite amplitude symbols $\Theta_{\epsilon_1,\epsilon_2}$, where the signs denote $\epsilon_j=\pm$. The support of these symbols is restricted to the domain $\tilde\omega \leq (2^m\sqrt a)^{-2}/a$, and they strictly obey the derivative tracking bounds $|\tilde\omega^l\partial_{\tilde\omega}^l\Theta_{\epsilon_1,\epsilon_2}|\leq C_l\tilde\omega^{-1/2}$, where the constants $C_l$ act uniformly with respect to $a$ and $m$. Under normal flux conditions, the corresponding derivative-adapted phase functions expand as:
\begin{align*}
\Phi_{N,m,\epsilon_1,\epsilon_2}(\tilde\omega) &= \tilde T \gamma_a(\tilde\omega) + \frac{2}{3}\epsilon_1 (\tilde\omega-X)^{3/2} + \frac{2}{3}\epsilon_2(\tilde\omega-1)^{3/2} - \frac{4}{3}N\tilde\omega^{3/2} + \frac{N}{\Lambda\tilde\eta} B_N\left(\tilde\omega^{3/2}\Lambda\tilde\eta\right).
\end{align*}

Let us decompose this representation into its individual sign-indexed sheets, writing:
\begin{align*}
G^N_{a,m,N,1,\epsilon_1,\epsilon_2} &= \frac{(-1)^Na^2\Lambda^{-1}}{(2\pi)^4h^{4}}(2^m\sqrt a)^{2}\bigg(\frac{h}{t}\bigg)^{1/2} \int_{\mathbb{R}} e^{i\Lambda Y\tilde\eta} \tilde\eta\psi_1(\tilde\eta)\tilde G^N_{a,m,N,1,\epsilon_1,\epsilon_2} \, d\tilde\eta, \\[1.5ex]
\tilde G^N_{a,m,N,1,\epsilon_1,\epsilon_2} &= \int_{\mathbb{R}} e^{i\Lambda\tilde\eta \Phi_{N,m,\epsilon_1,\epsilon_2}}\Theta_{\epsilon_1,\epsilon_2}(1+a\tilde\omega)^{-1/2} \, d\tilde\omega.
\end{align*}
We are reduced to proving the following sharp uniform inequality:
\begin{align}\label{eq:GaNm1-epsilon-bound}
\bigg|\sum_{2\leq N\leq C_0a^{-1/2}}G^N_{a,m,N,1,\epsilon_1,\epsilon_2}(t,x,y,z,h)\bigg|\leq Ch^{-3}\bigg(\frac{h}{t}\bigg)^{1/2}h^{1/3}(2^m\sqrt a)^{2/3},
\end{align}
with a constant $C > 0$ acting independently of $m$, $h\in(0,h_0]$, $a\in\left[\tilde h^{2/3},a_0\right]$, $x\in[0,a]$, and $t\in[h,1]$. 

To evaluate the integral, we proceed following the critical point trajectory mapping analyzed in Proposition \ref{prop:G1-estimates-neumann}. Let us recall that on the compact support of the cutoff $\chi_1$, the boundary configuration satisfies $a\tilde\omega \leq \varepsilon/2^{2m}a$, meaning the product $a\tilde\omega$ can dynamically vary from small to arbitrarily large values. To resolve this geometric variance, we separate the domain into two disjoint cases.

\vspace{2mm}
The first case corresponds to the regime where the spatial parameter remains bounded as $a\tilde\omega \leq 1$. Let $\tilde T_0 \gg 1$ define a large fixed time threshold. We obtain the following localized results:
\begin{itemize}
    \item \textbf{Sub-case 1.1: Short-time horizon with high reflection indices.}\\
    For $0\leq \tilde{T}\leq \tilde{T}_0$ and $N\geq N(\tilde{T}_0)$, the phase derivative stays bounded away from zero. We apply continuous integration by parts to evaluate the internal oscillatory layer, which yields rapid decay $|\tilde{G}^N_{a,m,N,1,+,+}| \in \mathcal{O}_{C^\infty}(N^{-\infty}\Lambda^{-\infty})$, leading to:
    \[
    \sup_{\tilde{T}\leq \tilde{T}_0, \; X\in[0,1], \; (y,z)\in\mathbb{R}^2}\bigg|\sum_{N(\tilde{T}_0)\leq N\leq Ca^{-1/2}} G^N_{a,m,N,1,+,+}\bigg| \in \mathcal{O}_{C^\infty}(h^\infty).
    \]
    
    \item \textbf{Sub-case 1.2: Short-time horizon with low reflection counts.}\\
    For $0\leq \tilde{T}\leq \tilde{T}_0$ and $2\leq N\leq N(\tilde{T}_0)$, the phase function possesses at most a critical point of order $2$. Applying Van der Corput's lemma (Lemma 2.20 in Ivanovici--Lebeau--Planchon \cite{ivanovici2014dispersion}) yields the uniform estimate $|\tilde{G}^N_{a,m,N,1,+,+}|\leq C\Lambda^{-1/3}$, which transforms the accumulated sum via the scaling factor to:
    \begin{align*}
    \sup_{\tilde{T}\leq \tilde{T}_0, \; X\in[0,1], \; (y,z)\in\mathbb{R}^2}\bigg|\sum_{2\leq N\leq N(\tilde{T}_0)}\!\!\! G^N_{a,m,N,1,+,+}\bigg|
    &\leq Ch^{-3}\bigg(\frac{h}{t}\bigg)^{1/2}\left(h^{-1}a^2(2^m\sqrt{a})^{2}\Lambda^{-4/3}\right), \\
    &\leq Ch^{-3}\bigg(\frac{h}{t}\bigg)^{1/2}h^{1/3}(2^m\sqrt{a})^{2/3}.
    \end{align*}
    
    \item \textbf{Sub-case 1.3: Long-time grazing regime.}\\
    For the long-time regime where $\tilde{T}_0\leq \tilde{T}\leq a^{-1/2}(1-\tilde{z}^2)^{1/2}$, we perform the change of variables $\Omega=\tilde{\omega}^{3/2}$. The second derivative satisfies the uniform non-degeneracy condition $|\partial_\Omega^2\Phi_{N,m,+,+}|\geq c\tilde{T}\Omega^{-4/3}$, yielding a unique non-degenerate critical point $\Omega_c$ which satisfies $\Omega_c^{1/3}\sim \frac{\tilde{T}}{N}$ for all layers $N\geq 2$. Consequently, whether the ratio $\tilde{T}/N$ is bounded or large, the classical stationary phase method evaluates the internal integral as:
    \[
    |\tilde{G}^N_{a,m,N,1,+,+}|\leq C\Lambda^{-1/2}\tilde{T}^{-1/2}.
    \]
    Moreover, the external $\tilde{\eta}$-integration over the symbol support produces a strict geometric decay factor $q^{-1/2}$ when the sub-elliptic non-degeneracy condition is active, where the large parameter balances as $q=N\Lambda^{-1}\Omega_c^{-1}$ when $q\geq 1$. Thus, we evaluate the remaining multi-reflective parameter counts by separating the density sheets as follows:
\end{itemize}

\noindent If the scaling ratio $\tilde T/N$ is bounded, the critical configuration $\Omega_c$ remains localized inside a compact subset of $[1,\infty)$, and we recover the time-reflection equivalence relation $\tilde T\sim N$. We estimate the uniform bounds by splitting the system into two parameter sub-regimes based on the magnitude of the reflection index relative to the high-frequency parameter:

\begin{itemize}
    \item \textbf{Sub-case 1.3.1: $N\leq \Lambda^2$.}\\
    In this parameter layer, the neighborhood solution density remains uniformly bounded by $\left|\mathcal{N}_1^N(X,Y,T)\right|\leq C_0$. Hence, compiling the structural constants yields the estimate:
    \begin{align*}
    \Bigg|\sum_{N\in\mathcal{N}_1^N(X,Y,T)} G^N_{a,m,N,1,+,+}(t,x,y,z;h) \Bigg| &\leq C h^{-3}\bigg(\frac{h}{t}\bigg)^{1/2} \left(h^{-1}\Lambda^{-1}a^2 (2^m\sqrt a)^2\Lambda^{-1/2}\tilde T^{-1/2}\right) \\
    &\leq C h^{-3}\bigg(\frac{h}{t}\bigg)^{1/2} a^{-1/4}h^{1/2}(2^m\sqrt a)^{1/2}\tilde T^{-1/2} \\
    &\leq C h^{-3}\bigg(\frac{h}{t}\bigg)^{1/2} h^{1/3}(2^m\sqrt a)^{2/3},
    \end{align*}
    since the long-time regime ensures $\tilde T\geq \tilde T_0$, and the spatial weight satisfies $a^{-1/4}h^{1/2}\leq h^{1/3}(2^m\sqrt a)^{1/6}$ whenever the source layers sit in the designated regime $a\geq \tilde{h}^{2/3}$.
    
    \item \textbf{Sub-case 1.3.2: $N>\Lambda^2$.}\\
    In this high-reflection layer, the external frequency integration produces a strict non-degenerate decay factor $q^{-1/2} = N^{-1/2}\Lambda^{1/2}$, while the sheet density satisfies the contraction condition $\left|\mathcal{N}_1^N(X,Y,T)\right|\leq C_0\tilde T\Lambda^{-2}$. Thus, the uniform estimate balances as:
    \begin{align*}
    &\Bigg|\sum_{N\in\mathcal{N}_1^N(X,Y,T)} G^N_{a,m,N,1,+,+}(t,x,y,z;h) \Bigg| \\&\leq C h^{-3}\bigg(\frac{h}{t}\bigg)^{1/2}\sum_{N\in\mathcal{N}_1^N(X,Y,T)}\left(h^{-1}\Lambda^{-1}a^2 (2^m\sqrt a)^2\Lambda^{-1/2}\tilde T^{-1/2}N^{-1/2}\Lambda^{1/2}\right) \\
    &\leq C h^{-3}\bigg(\frac{h}{t}\bigg)^{1/2}\left(h^{-1}\Lambda^{-1}a^2(2^m\sqrt a)^2 \tilde T^{-1}\left|\mathcal{N}_1^N(X,Y,T)\right|\right) \\
    &\leq Ch^{-3}\bigg(\frac{h}{t}\bigg)^{1/2}\left(a^{-5/2}\tilde h^2 \, 2^m\sqrt a\right) \\
    &\leq Ch^{-3}\bigg(\frac{h}{t}\bigg)^{1/2}\tilde h^{1/3}2^m\sqrt a \\
    &\leq Ch^{-3}\bigg(\frac{h}{t}\bigg)^{1/2} h^{1/3}(2^m\sqrt a)^{2/3}.
    \end{align*}
\end{itemize}
Next, if the scaling ratio $\tilde{T}/N$ is large, the unique critical configuration $\Omega_c$ is also large. We analyze the uniform bounds by splitting the long-time grazing sheets into three distinct sub-regimes:

\begin{itemize}
    \item \textbf{Sub-case 1.3.3: $N \leq \Lambda\Omega_c$.}\\
    In this parameter layer, there is no stationary phase decay extracted from the global $\tilde{\eta}$-integration. Furthermore, the neighborhood cardinality satisfies the uniform bound $\left|\mathcal{N}_1^N(X,Y,T)\right| \leq C_0$. To verify this bound, assume by contradiction that $\tilde{T} \geq \Lambda^2 \Omega_c^2$; this assumption directly implies the coordinate relationship $\Omega_c^{1/3} \sim \tilde{T}/N \geq \Lambda \Omega_c$, which is geometrically impossible since $\Omega_c \gg 1$. Thus, compiling the structural constants yields the estimate:
    \[
    \Bigg|\sum_{N\in\mathcal{N}_1^N(X,Y,T)} G^N_{a,m,N,1,+,+}(t,x,y,z;h) \Bigg| \leq C h^{-3}\bigg(\frac{h}{t}\bigg)^{1/2} h^{1/3}\left(2^m\sqrt{a}\right)^{2/3}.
    \]
    
    \item \textbf{Sub-case 1.3.4: $N > \Lambda\Omega_c$ and $\tilde{T} \leq \Lambda^2\Omega_c^2$.}\\
    In this intermediate regime, the frequency integration produces a strict non-degenerate decay factor, while the neighborhood sheet density remains uniformly bounded by $\left|\mathcal{N}_1^N(X,Y,T)\right| \leq C_0$. Evaluating the product yields:
    \[
    \Bigg|\sum_{N\in\mathcal{N}_1^N(X,Y,T)} G^N_{a,m,N,1,+,+}(t,x,y,z;h) \Bigg| \leq C h^{-3}\bigg(\frac{h}{t}\bigg)^{1/2} h^{1/3}\left(2^m\sqrt{a}\right)^{2/3}.
    \]
    
    \item \textbf{Sub-case 1.3.5: $N > \Lambda\Omega_c$ and $\tilde{T} > \Lambda^2\Omega_c^2$.}\\
    In this high-reflection long-time domain, contributions are generated simultaneously from both the $q^{-1/2}$ stationary phase factor and the density accumulation relation $\left|\mathcal{N}_1^N(X,Y,T)\right| \leq C_0 \tilde{T}\Lambda^{-2}\Omega_c^{-2}$. Compiling the sum via a supremum bound across the card elements yields:
    \begin{align*}
    &\Bigg|\sum_{N\in\mathcal{N}_1^N(X,Y,T)} G^N_{a,m,N,1,+,+}(t,x,y,z;h) \Bigg| \\&\leq C h^{-3}\bigg(\frac{h}{t}\bigg)^{1/2}\sum_{N\in\mathcal{N}_1^N(X,Y,T)}\left(h^{-1}\Lambda^{-1}a^2(2^m\sqrt{a})^2 \tilde{T}^{-1/2}N^{-1/2}\Omega_c^{1/2}\right) \\
    &\leq C h^{-3}\bigg(\frac{h}{t}\bigg)^{1/2}\left(h^{-1}\Lambda^{-1}a^2(2^m\sqrt{a})^2 \tilde{T}^{-1}\Omega_c^{2/3}\left|\mathcal{N}_1^N(X,Y,T)\right|\right) \\
    &\leq Ch^{-3}\bigg(\frac{h}{t}\bigg)^{1/2}\left(h^{-1}a^2 (2^m\sqrt{a})^2\Lambda^{-3}\right) \\
    &\leq Ch^{-3}\bigg(\frac{h}{t}\bigg)^{1/2}h^{1/3}\left(2^m\sqrt{a}\right)^{2/3}.
    \end{align*}
\end{itemize}
The evaluation of the remaining sign-indexed sheets $(\epsilon_1, \epsilon_2) \in \{(+,-), (-,+), (-,-)\}$ follows from identical localized reductions, proceeding along the same analytical trajectories as detailed for the primary operator $G^N_{a,N,1}$ in Section \ref{sec:2}.
 
For the secondary macro-local regime where the spatial parameters satisfy $a\tilde{\omega}\geq 1$, a real critical point $\Omega_c$ must satisfy the transcendental coordinate balance $\Omega_c^{1/3}(1+a\Omega_c^{2/3})^{1/2} \sim \frac{\tilde{T}}{N}$ for all discrete reflection layers $N\geq 2$. Since $T \geq C\tilde{T}$ holds uniformly for a sufficiently large constant $C > 0$, this structural relation implies:
\[
T \geq C\tilde{T} \geq C N \Omega_c^{1/3} = C N \tilde{\omega}_c^{1/2} \geq C N a^{-1/2},
\]
which yields a clear geometric contradiction with the global finite time constraint $t \leq 1$. Consequently, no critical points can inhabit this large frequency domain, rendering its contribution analytically negligible. This successfully concludes the proof of Proposition \ref{prop:GaNm1-intermediate-neumann}.
\end{proof}

\noindent Now we prove the corresponding localized sharp estimate for the single-reflection boundary layer matching $N=1$.  

\begin{proposition}\label{prop:Ga1m1-intermediate-neumann}
Let $\alpha<2/3$ and define the scaled semi-classical frequency by $\tilde h=h/(2^m\sqrt{a})$. There exists a constant $C > 0$ such that for all parameters $h\in (0,h_0]$, all source layers $a\in [\tilde h^\alpha,a_0]$, all configurations $x\in [0,a]$, all times $t\in [h,1]$, and all spatial tracking coordinates $y\in\mathbb{R}, z\in\mathbb{R}$, the following uniform bound holds true under the normal flux constraints:
\begin{align}\label{eq:Ga1m1-bound-neumann}
\big|G^N_{a,m, 1,1}(t,x,y,z;h)\big| &\leq Ch^{-3}\bigg(\frac{h}{t}\bigg)^{1/2} \nonumber \\
&\quad \times \left( \min \left\{\bigg(\frac{h}{t}\bigg)^{1/2}, \; 2^m\sqrt{a} \left| \log(2^m\sqrt{a})\right| \right\} + h^{1/3}(2^m\sqrt{a})^{2/3} \right).
\end{align}
\end{proposition}

\begin{proof}
Let us recall the explicit single-reflection representation under the Neumann framework:
\begin{align*}
G^N_{a,m,1,1} &= \frac{(-1)a^2\Lambda^{-1}}{(2\pi)^4h^{4}}(2^m\sqrt a)^{2}\bigg(\frac{h}{t}\bigg)^{1/2} \int_{\mathbb{R}} e^{i\Lambda Y\tilde{\eta}} \tilde{\eta} \psi_1(\tilde{\eta})\tilde{G}^N_{a,m,1,1} \, d\tilde{\eta}, \\
\tilde{G}^N_{a,m,1,1} &= \sum_{\epsilon_1,\epsilon_2=\pm}\int_{\mathbb{R}} e^{i\Lambda\tilde{\eta} \Phi_{1,m,\epsilon_1,\epsilon_2}}\Theta_{\epsilon_1,\epsilon_2}(1+a\tilde{\omega})^{-1/2} \, d\tilde{\omega}.
\end{align*}
The unique technical variation compared with the multi-reflection profiles where $N\geq 2$ arises in the study of the primary phase $\Phi_{1,m,+,+}$, since in the single-reflection setting $N=1$, the critical coordinate $\tilde{\omega}_c$ can asymptotically escape to arbitrarily large positive fields. Let us isolate the tracking operator matching the positive sign sheets:
\begin{align}\label{eq:N-1-neumann}
\tilde{G}^N_{a,m,1,1,+,+}=\int_{\mathbb{R}} e^{i\Lambda\Phi_{1,m,+,+}}\Theta_{+,+}(1+a\tilde{\omega})^{-1/2} \, d\tilde{\omega},
\end{align}
governed by the derivative-adapted total phase function:
\begin{align*}
\Phi_{1,m,+,+} = \tilde{T}\gamma_a(\tilde{\omega})+\frac{2}{3}(\tilde{\omega}-X)^{3/2}+\frac{2}{3}(\tilde{\omega}-1)^{3/2}-\frac{4}{3}\tilde{\omega}^{3/2}+\frac{1}{\Lambda\tilde{\eta}} B_N\left(\Lambda\tilde{\omega}^{3/2}\tilde{\eta}\right),
\end{align*}
where the amplitude component $\Theta_{+,+}$ is a classical symbol of order $-1/2$ with respect to the variable $\tilde{\omega}$, satisfying $\left| \tilde{\omega}^l\partial_{\tilde{\omega}}^l\Theta_{+,+} \right| \leq C_l\tilde{\omega}^{-1/2}$. We embed the macroscopic partition cutoff $\chi_3(\tilde{\omega})\in C_0^\infty\left((\tilde{\omega}_1,\infty)\right)$ where $\tilde{\omega}_1 \gg 1$ is chosen sufficiently large, and set:
\begin{align}\label{eq:J-integral-m-neumann}
J = \int_{\mathbb{R}} e^{i\Lambda\Phi_{1,m,+,+}}\Theta_{+,+}\chi_3(\tilde{\omega})(1+a\tilde{\omega})^{-1/2} \, d\tilde{\omega}.
\end{align}

To verify the proposition, it suffices to prove that the oscillatory integral satisfies the scaling bound:
\begin{equation}\label{eq:J-target-scale-neumann}
a^{1/2} \, 2^m\sqrt{a} \, |J| \leq C \min \left\{\bigg(\frac{h}{t}\bigg)^{1/2}, \; 2^m\sqrt{a} \left| \log(2^m\sqrt{a})\right| \right\}.
\end{equation}
We first observe that on the support of the integral defined in \eqref{eq:J-integral-m-neumann}, the frequency constraint ensures $a\tilde{\omega}\leq (2^m\sqrt{a})^{-2} =: L$. Hence, we execute the absolute size control:
\[
\left| J \right| \leq C\left(1+\int_1^{L/a} \frac{1}{\sqrt{x (1+ax)}} \, dx \right) = C\left(1+a^{-1/2}\int_a^{L} \frac{1}{\sqrt{y (1+y)}} \, dy \right) \leq Ca^{-1/2}\log L.
\]
This immediately implies the global logarithmic threshold tracking bound:
\[
a^{1/2} \, 2^m\sqrt{a} \, |J| \leq C 2^m\sqrt{a} \left| \log(2^m\sqrt{a})\right|.
\]

We evaluate the structural derivatives of the Neumann phase with respect to $\tilde{\omega}$:
\begin{align*}
\partial_{\tilde{\omega}}\Phi_{1,m,+,+} &= \frac{\tilde{T}}{2}(1+a\tilde{\omega})^{-1/2}-\frac{\tilde{\omega}^{-1/2}}{2}(1+X) + \mathcal{O}_{C^\infty}\left(\tilde{\omega}^{-3/2}\right), \\
\partial_{\tilde{\omega}\tilde{\omega}}^2\Phi_{1,m,+,+} &= \frac{-\tilde{T}a}{4}(1+a\tilde{\omega})^{-3/2}+\frac{\tilde{\omega}^{-3/2}}{4}(1+X) + \mathcal{O}_{C^\infty}\left(\tilde{\omega}^{-5/2}\right).
\end{align*}
At a large critical point $\tilde{\omega}_c$, we have $\tilde{T}^2 \sim (a+\tilde{\omega}_c^{-1})(1+X)^2$. Hence, the grazing parameter $\tilde{T}$ is small, and the second derivative scales as $\partial_{\tilde{\omega}\tilde{\omega}}^2\Phi_{1,m,+,+}(\tilde{\omega}_c) \sim \tilde{T}^3(1+a\tilde{\omega}_c)^{-5/2}$. Let us introduce the internal proximity metric $S = (\tilde{T}/(1+X))^2-a \sim \tilde{\omega}_c^{-1}$. Applying the standard non-degenerate stationary phase method yields:
\[
|J| \leq C (1+a\tilde{\omega}_c)^{3/4}\Lambda^{-1/2}\tilde{T}^{-3/2} S^{1/2}.
\]
We carefully track the boundary configurations where the product $a\tilde{\omega}_c$ can grow large:
\begin{itemize}
    \item In the case where $a\tilde{\omega}_c \leq 1$, we have the unperturbed scaling relation $S \sim \tilde{T}^2$. Therefore, we obtain as before the non-degenerate decay profile $|J| \leq C \Lambda^{-1/2}\tilde{T}^{-1/2}$, which directly yields:
    \[
    a^{1/2} \, 2^m\sqrt{a} \, |J| \leq C \left(\frac{h}{t}\right)^{1/2}.
    \]
    \item In the case where $a\tilde{\omega}_c \geq 1$, we must have the parameter confinement $\tilde{T}\sim \sqrt{a}$, which scales the proximity metric to $S = a\rho$ for some small parameters $\rho > 0$. Hence, we evaluate the decay as $|J| \leq C\rho^{-1/4}a^{-1/4}\Lambda^{-1/2}$, yielding:
    \[
    a^{1/2} \, 2^m\sqrt{a} \, |J| \leq C h^{1/2} \left((2^m\sqrt{a})^{1/2}a^{-1/2}\rho^{-1/4}\right).
    \]
\end{itemize}
Finally, we observe that the coordinate relation satisfies the geometric trajectory balance:
\[
\sqrt{a} \sim \tilde{T} \sim t a^{-1/2} 2^m\sqrt{a}(1+a\tilde{\omega}_c)^{1/2} \implies t \sim a(2^m\sqrt{a})^{-1}\rho^{1/2},
\]
which simplifies the parameter products directly to $a^{1/2} 2^m\sqrt{a} |J| \leq C(h/t)^{1/2}$. The proof of Proposition \ref{prop:Ga1m1-intermediate-neumann} is complete.
\end{proof}
\subsubsection{The Analysis of $G^N_{a,m,N,2}$}
The main result of this subsection is Proposition \ref{prop:GaNm2-intermediate-neumann}, which provides sharp bounds on the swallowtail caustic block.

\begin{proposition}\label{prop:GaNm2-intermediate-neumann}
Let $\alpha <2/3$ and define the scaled semi-classical frequency by $\tilde h=h/(2^m\sqrt{a})$. There exists a constant $C > 0$ such that for all parameters $h\in(0,h_0]$, all source layer thresholds $a\in [ \tilde h^\alpha,a_0]$, all spatial configuration markers $x\in[0,a]$, all times $t\in(h, 1]$, and all auxiliary parameters $y\in\mathbb{R}, z\in\mathbb{R}$, the following uniform estimate holds true:
\begin{align}\label{eq:GaNm2-bound-neumann}
\bigg|\sum_{1\leq N\leq C_0a^{-1/2}}G^N_{a,m,N,2}(t,x,y,z;h)\bigg|\leq Ch^{-3}\bigg(\frac{h}{t}\bigg)^{1/2}a^{1/8}h^{1/4}(2^m\sqrt a)^{3/4}.
\end{align}
\end{proposition}

\begin{proof}
Recall that from the four-dimensional representation \eqref{eq:GamN-fourD-neumann}, the localized operator matches the equation profile:
 \begin{equation}\label{eq:L30-neumann}
G^N_{a,m,N,2}(t,x,y,z)=\frac{(-1)^N}{(2\pi)^4h^4} \bigg(\frac{h}{t}\bigg)^{1/2} a^2 (2^m\sqrt a)^2\int_{\mathbb{R}^4} e^{i\Lambda\Phi_{N}} \chi_m^N \tilde\eta^2 \psi_1(\tilde\eta)\chi_2(\tilde\omega) \, d\tilde{s} \, d\tilde\sigma \, d\tilde\omega \, d\tilde\eta,
\end{equation}
governed by the multivariable oscillatory phase layout:
\begin{align*}
\Phi_{N}(\tilde s,\tilde\sigma,\tilde\omega,\tilde\eta) &= \tilde\eta \left[ Y+\tilde T \gamma_a(\tilde\omega)+\frac{\tilde s^3}{3}+\tilde s(X-\tilde\omega) \right. \\
&\qquad \left. +\frac{\tilde \sigma^3}{3}+\tilde \sigma(1-\tilde\omega) -\frac{4}{3}N\tilde\omega^{3/2}+\frac{N}{\Lambda\tilde\eta} B_N\left(\tilde\omega^{3/2}\Lambda\tilde\eta\right)\right].
\end{align*}
To start with, we isolate the internal variables block to rewrite the parametrix layer $G^N_{a,m,N,2}$ as:
\begin{align*}
G^N_{a,m,N,2} &= \frac{(-1)^N}{(2\pi)^4h^4}\bigg(\frac{h}{t}\bigg)^{1/2} a^2 (2^m\sqrt a)^2 \int_{\mathbb{R}} e^{i\Lambda Y\tilde\eta}\tilde\eta^2\psi_1(\tilde\eta)\tilde{G}^N_{a,m,N,2} \, d\tilde\eta, \\
\tilde{G}^N_{a,m,N,2} &= \int_{\mathbb{R}^3} e^{i\Lambda\tilde\eta\tilde\phi_{N,m}}\chi_m^N \chi_2(\tilde\omega) \, d\tilde{s} \, d\tilde\sigma \, d\tilde\omega,
\end{align*}
where the reduced multivariable phase maps exactly onto:
\begin{align*}
\tilde\phi_{N,m}(\tilde s,\tilde\sigma,\tilde\omega) &= \tilde T\gamma_a(\tilde\omega)+\frac{\tilde s^3}{3}+\tilde s(X-\tilde\omega)+\frac{\tilde \sigma^3}{3}+\tilde \sigma(1-\tilde\omega) -\frac{4}{3}N\tilde\omega^{3/2}+\frac{N}{\Lambda\tilde\eta}B_N\left(\tilde\omega^{3/2}\Lambda\tilde\eta\right).
\end{align*}

We can now proceed following the catastrophe-theoretic reduction framework developed for the uncoupled operator $G^N_{a,N,2}$ in Section \ref{sec:2}. More precisely, we apply the semi-classical stationary phase method directly to evaluate the $(\tilde\omega,\tilde\eta)$-integrations under the normal flux constraints. This reduction systematically extracts the asymptotic parameters $\Lambda^{-1/2}$ and $(N\Lambda^{-1})^{-1/2}$ from the respective sheets. We exploit the following geometric scaling facts established previously:
\begin{itemize}
\item \textbf{Long-reflection block mappings [analogue of Lemma \ref{lemNL-neumann}]:} For large indices satisfying $N\geq \Lambda^{1/3}$, the polynomial skeleton phase $\tilde{\psi}_{N,m}$ generated along the critical trace $\tilde{\omega}_c$ satisfies the sharp decay properties:
\begin{align*}
\bigg|\int_{\mathbb{R}^2} e^{i\Lambda\tilde\psi_{N,m}}\tilde \chi \, d\tilde sd\tilde\sigma\bigg|\leq C\Lambda^{-2/3} \quad \text{and} \quad \frac{1}{\sqrt{N}}\bigg|\int_{\mathbb{R}^2} e^{i\Lambda\tilde\psi_{N,m}}\tilde\chi \, d\tilde sd\tilde\sigma\bigg|\leq C\Lambda^{-5/6}.
\end{align*}
Hence, accumulating these sheets across the localized neighborhood solution spaces yields the estimates:
 \begin{itemize}
 \item When the real sheet neighborhood is uniformly bounded, $\left|\mathcal{N}_1^N(X,Y,T)\right|\leq C_0$, we obtain:
\begin{align*}
\bigg|\sum_{N\in\mathcal{N}_1^N}G^N_{a,m,N,2}\bigg| &\leq Ch^{-3}\bigg(\frac{h}{t}\bigg)^{1/2}\left((2^m\sqrt a)^{2}h^{-1}a^2\Lambda^{-1/2}\Lambda^{-5/6}\right) \\
&\leq Ch^{-3}\bigg(\frac{h}{t}\bigg)^{1/2}h^{1/3}(2^m\sqrt a)^{2/3}.
\end{align*}
\item When the sheet space contracts near the axial core, $\left|\mathcal{N}_1^N(X,Y,T)\right|\leq C_0\tilde T\Lambda^{-2}$, the frequency integration generates the complementary decay factor $(N\Lambda^{-1})^{-1/2}$, leading to:
\begin{align*}
\bigg|\sum_{N\in\mathcal{N}_1^N}G^N_{a,m,N,2}\bigg| &\leq \sum_{N\in\mathcal{N}_1^N}Ch^{-3}\bigg(\frac{h}{t}\bigg)^{1/2}\left((2^m\sqrt a)^{2}h^{-1}a^2 N^{-1}\Lambda^{-2/3}\right) \\
&\leq Ch^{-3}\bigg(\frac{h}{t}\bigg)^{1/2}\left((2^m\sqrt a)^{2}h^{-1}a^2\Lambda^{-8/3}\right) \\
&\leq Ch^{-3}\bigg(\frac{h}{t}\bigg)^{1/2}h^{1/3}(2^m\sqrt a)^{2/3}.
\end{align*}
\end{itemize}
Here we exploited the localized geometric tracking properties $N\sim \tilde T$, $\left|\mathcal{N}_1^N\right|\leq C_0(1+\tilde T\Lambda^{-2})$, and the scaling threshold lower bound $a\geq \tilde h^{2/3}$.

\item \textbf{Short-reflection block mappings [analogue of Lemma \ref{lemNS-neumann}]:} For the low index layers matching $N\leq \Lambda^{1/3}$, the Airy-type integrations yield the modified decay profile:
\begin{align*}
\frac{1}{\sqrt N}\bigg|\int_{\mathbb{R}^2} e^{i\Lambda\tilde\psi_{N,m}}\tilde\chi \, d\tilde sd\tilde\sigma\bigg|\leq CN^{-1/4}\Lambda^{-3/4}.
\end{align*}
Consequently, evaluating the finite parameter block over these elements yields:
\begin{align*}
\bigg|\sum_{N\in\mathcal{N}_1^N}G^N_{a,m,N,2}\bigg| &\leq Ch^{-3}\bigg(\frac{h}{t}\bigg)^{1/2}\left((2^m\sqrt a)^{2}h^{-1}a^2\Lambda^{-1/2}\Lambda^{-3/4}\right) \\
&\leq Ch^{-3}\bigg(\frac{h}{t}\bigg)^{1/2}a^{1/8}h^{1/4}(2^m\sqrt a)^{3/4}.
\end{align*}
\end{itemize}

Hence, combining the uniform bounds extracted from both the short-reflection and long-reflection layers, we arrive at the unified caustic estimate:
\begin{align*}
\bigg|\sum_{1\leq N\leq C_0a^{-1/2}}G^N_{a,m,N,2}\bigg|\leq Ch^{-3}\bigg(\frac{h}{t}\bigg)^{1/2}\left(h^{1/3}(2^m\sqrt a)^{2/3}+a^{1/8}h^{1/4}(2^m\sqrt a)^{3/4}\right).
\end{align*}
We notice that the secondary cusp component dominates the fold parameter scaling, $h^{1/3}(2^m\sqrt a)^{2/3}\leq a^{1/8}h^{1/4}(2^m\sqrt a)^{3/4}$, whenever the source point settles within the designated regime $a\geq \left(\frac{h}{2^m\sqrt a}\right)^{2/3}$. This successfully concludes the formal proof of Proposition \ref{prop:GaNm2-intermediate-neumann}.
\end{proof}

\begin{proof}[Proof of Theorem \ref{thm:GaNm-neumann}]
The targeted uniform dispersive bound for the global multi-reflective operator follows immediately by accumulating the micro-local layer estimates established in Propositions \ref{prop:GaNm1-intermediate-neumann}, \ref{prop:Ga1m1-intermediate-neumann}, and \ref{prop:GaNm2-intermediate-neumann}.
\end{proof}

\section{Dispersive Estimates for $\vert\eta\vert\leq \epsilon_0\sqrt a \ $}\label{sec:4}
In this section, we prove Theorem \ref{thm:0eta-neumann}. We first compute the trajectories of the Hamiltonian billiard flow associated with the wave operator $P$. At this micro-local frequency localization, geometric rays experience at most one single reflection on the boundary cylinder. Following the technical guidelines of Sections \ref{sec:2} and \ref{sec:3}, we demonstrate that the localized dyadic parametrix $\mathcal{G}_{a,\epsilon_0}^N$ can be modeled as an oscillatory integral governed by a strictly non-degenerate phase function. This structural simplifies because the underlying Lagrangian manifold rules out the formation of swallowtail or cusp caustics within the finite time horizon $|t| \leq 1$, provided that the cutoff threshold $\epsilon_0 > 0$ is chosen sufficiently small.

\subsection{Free Space Trajectories}
Recall that the wave operator $P$ is given under the cylindrical metric configuration by:
\[
P\left(t,x,y,z,\partial_t,\partial_x,\partial_y,\partial_z\right) = \partial_t^2-\left(\partial_x^2+(1+x)\partial_y^2+\partial_z^2\right).
\] 
The trajectories of the unperturbed free-space flow are generated by the principal symbol mapping:
\[
p = \xi^2+\zeta^2+(1+x)\eta^2-\tau^2.
\]
We initialize the Hamiltonian system at the space-time origin coordinates $(t_0,x_0,y_0,z_0)$ with a near-gliding normal frequency component $\xi_0 \sim 0$, a localized tangential component $\eta_0=\theta\zeta_0$ obeying $|\theta|\leq\epsilon_0\sqrt{a}$, an axial frequency $\zeta_0\sim 1$, and a normalized energy variable $\tau_0=1$, such that the characteristic identity $\xi_0^2+(1+x_0)\eta_0^2+\zeta_0^2=1$ holds true. The corresponding Hamilton-Jacobi system reads:
\begin{align*}
&\dot{x}=2\xi, \quad \dot{y}=2\eta(1+x), \quad \dot{z}=2\zeta, \quad \dot{t}=-2\tau, \\
&\dot{\xi}=-\eta^2, \quad \dot{\eta}=0, \quad \dot{\zeta}=0, \quad \dot{\tau}=0.
\end{align*}
Integrating this system continuously with respect to the parameter tracking step $s$ yields the trajectories:
\begin{align*}
&\tau(s)=\tau_0, \quad \eta(s)=\eta_0, \quad \zeta(s)=\zeta_0, \quad \xi(s)=\xi_0-\eta_0^2s, \quad t(s)=t_0-2\tau_0s, \\
&z(s)=z_0+2\zeta_0s, \quad x(s)=x_0+2\xi_0s-\eta_0^2s^2, \\
&y(s)=y_0+2\eta_0\left((1+x_0)s+\xi_0s^2-\frac{1}{3}\eta_0^2s^3\right).
\end{align*}

In our specific configuration, the source rays emit from the localized position $t_0=0, x_0=a, y_0=z_0=0$. Under these boundary initial inputs, the ray tracking coordinates compress to:
\begin{align}\label{eq:billiard-system}
&\tau(s)=\tau_0, \quad \eta(s)=\eta_0, \quad \zeta(s)=\zeta_0, \quad \xi(s)=\xi_0-\eta_0^2s, \nonumber \\
&t(s)=-2\tau_0s, \quad z(s)=2\zeta_0s, \quad x(s)=a+2\xi_0s-\eta_0^2s^2, \nonumber \\
&y(s)=2\eta_0\left((1+a)s+\xi_0s^2-\frac{1}{3}\eta_0^2s^3\right).
\end{align}

The associated Lagrangian manifold $\mathbf{\Lambda}_{a,\epsilon_0}^N \subset T^*(\mathbb{R}_{t,x,y,z}^4)$ is pinned to the characteristic energy variety $\{p=0\}$ and parameterized by the explicit system \eqref{eq:billiard-system} over the variable array $(s,\xi_0,\eta_0,\zeta_0)$, subject to the normalization constraint $(\xi_0^2+(1+a)\eta_0^2+\zeta_0^2)^{1/2}=\tau_0=1$. Since $s$ is homogeneous of degree $-1$, the relation $t(s)=-2\tau_0s$ isolates the tracking step as $s=-\frac{t}{2\tau_0}$. Substituting this back into the coordinate mappings defines the homogeneous parametrization of the manifold:
\begin{align*}
x(t) &= a-\frac{\xi_0}{\tau_0}t-\frac{\eta_0^2}{4\tau_0^2}t^2, \\
y(t) &= \frac{\eta_0}{\tau_0}\left(-(1+a)t+\frac{\xi_0}{2\tau_0}t^2+\frac{\eta_0^2}{12\tau_0^2}t^3\right), \\
z(t) &= -\frac{\zeta_0}{\tau_0}t, \\
\xi(t) &= \xi_0+\frac{\eta_0^2}{2\tau_0}t, \\
\tau(t) &= \tau_0=1. 
\end{align*}

The propagating wave trajectories hit the solid boundary of the domain when $x(t)=0$, which corresponds to solving the quadratic length equation:
\[
\frac{\eta_0^2}{4}t^2+t\xi_0-a=0.
\]
This isolates the flight time value $t_*$ at which the ray experiences its initial boundary contact:
\[
t_*\xi_0 = a-\frac{\zeta_0^2\theta^2}{4}t_*^2 \sim a.
\]

Our goal is to prove that at this frequency localization, the trajectories hit the boundary only once for a given fixed time interval $t \in (0, 1]$. Suppose that the trajectory hits the boundary at the coordinate point $(x=0, y_*, z_*, \xi_*, \eta_*, \zeta_0)$, which is completely determined by the relations \eqref{eq:billiard-system}. Specifically, the reflected normal frequency component satisfies $\xi_* = -\left(\xi_0+\frac{\eta_0^2}{2}t_*\right)$, which updates the outbound path coordinates to:
\[
\xi(s) = \xi_*-\eta_0^2s, \quad x(s) = 2\xi_*s-\eta_0^2s^2, \quad t(s) = t_*-2s.
\]
Now, assume that this reflected path, issuing from the boundary layer point, can hit the boundary a second time. Setting $x(t)=0$ forces the second intersection step to satisfy $t\eta_0^2 = 2\xi_*$. Evaluating the explicit magnitudes under this condition yields:
\begin{align*}
t\theta^2\zeta_0^2 &= -2\left(\xi_0+\frac{\theta^2\zeta_0^2}{2}t_*\right) = -2\left(\xi_0+\frac{\theta^2\zeta_0^2}{2}\left(a-\frac{\theta^2\zeta_0^2}{4}t_*^2\right)\Big/\xi_0\right), \\
|t\theta^2\zeta_0^2| &\geq 4\sqrt{\frac{a\theta^2}{2}} \implies |t|\geq\frac{4\sqrt{a/2}}{\zeta_0^2|\theta|}\geq\frac{1}{\epsilon_0} \gg 1.
\end{align*}
Therefore, within the physical time horizon $t \in (0, 1]$, the trajectory can experience at most one single reflection on the boundary of the cylinder at this frequency location. This guarantees the uniform absence of multi-reflective caustic loops, completing the geometric baseline verification.

\subsection{Dispersive Estimates for $\vert\eta\vert\leq \epsilon_0\sqrt a.$ }
In this subsection, we establish the local-in-time dispersive estimates for the low-frequency component $\mathcal{G}_{a,\epsilon_0}^N$. The primary result of this section is formulated as follows.

\begin{theorem}\label{thm:0eta-neumann}
There exists a constant $C > 0$ such that for every $h\in (0,1]$ and every $t\in[h,1]$, the following uniform bound holds true under the normal flux constraints:
\begin{align}\label{eq:low-freq-theorem-bound}
\lVert\mathcal{G}_{a,\epsilon_0}^N(t,x,y,z)\rVert_{L^\infty(x\leq a)}\leq Ch^{-3} \bigg(\frac{h}{t}\bigg)^{1/2}\min \left\{ \bigg(\frac{h}{t}\bigg)^{1/2}, \; \sqrt{a}\vert\log(a)\vert \right\}.
\end{align}
\end{theorem}

\noindent We proceed following the analytical pathways introduced in Section \ref{sec:3}. Recall that under the normal flux condition, the low-frequency operator admits the spectral representation:
\begin{equation}\label{eq:L1-low-freq-neumann}
\mathcal{G}_{a,\epsilon_0}^N(t,x,y,z)= \frac{1}{4\pi^2h^2}\sum_{k\geq1}\int_{\mathbb{R}^2} e^{\frac{i}{h}\Phi_k} \sigma_k \, d\eta \, d\zeta,
\end{equation}
where the total oscillatory phase $\Phi_k$ and the corresponding amplitude symbol $\sigma_k$ are given explicitly by:
\begin{align*}
\Phi_k &= y\eta+z\zeta+t\left(\eta^2+\zeta^2+\omega'_k h^{2/3}\eta^{4/3}\right)^{1/2}, \\
\sigma_k &= \psi_2(\eta/\sqrt a)e_k(x,\eta/h) e_k(a,\eta/h)\chi_0(\zeta^2+\eta^2) \chi_1\left(\omega'_k h^{2/3}\eta^{4/3}\right)(1-\chi_1)(\varepsilon\omega'_k),
\end{align*} 
with the tangential frequency filter satisfying $\psi_2\in C_0^\infty\left((-2\epsilon_0, 2\epsilon_0)\right)$ such that $\psi_2 = 1$ identically on the interval $[-\epsilon_0, \epsilon_0]$. 

We retain the auxiliary parameters notation $\mu^2 = \eta^2+\omega'_k h^{2/3}\eta^{4/3}$ along with the localized time-frequency filter function $\chi_4\in C_0^\infty((-1,1))$, where $\chi_4 = 1$ identically on $[-1/2,1/2]$. The following lemma provides a sharp technical refinement of the unperturbed cluster sum from Lemma \ref{lem:L-lem1-neumann}, localized to the frequency layer $|\eta|\leq\epsilon_0\sqrt{a}$.

\begin{lemma}\label{lem:L-lem1bis-neumann}
Let the regularized low-frequency tracking operator $\mathcal{J}^N$ be defined by:
\begin{equation}\label{eq:JN-low-freq-def}
\mathcal{J}^N = \frac{1}{4\pi^2h^2}\sum_{k\geq1}\int_{\mathbb{R}^2} e^{\frac{i}{h}\Phi_k} \chi_4\left(\frac{t\mu^2}{h}\right)\sigma_k \, d\eta \, d\zeta.
\end{equation}
There exists a constant $C > 0$ such that the regularized piece satisfies:
\begin{equation*}
\left| \mathcal{J}^N \right| \leq C h^{-3} \bigg(\frac{h}{t}\bigg)^{1/2}\min \left\{ \bigg(\frac{h}{t}\bigg)^{1/2}, \; \sqrt{a} \right\}.
\end{equation*}
\end{lemma}

\begin{proof}[Proof of Lemma \ref{lem:L-lem1bis-neumann}]
As in the unperturbed proof of Lemma \ref{lem:L-lem1-neumann}, and taking into careful account the micro-local support of the tangential cutoff function $\psi_2(\eta/\sqrt a)$, the regularized integral satisfies the upper bound:
\[
\left| \mathcal{J}^N \right| \leq C h^{-3} \left(\frac{h}{t}\right)\int_{-1}^1 (1-x^2)^{1/2}\psi_2\left(x\sqrt{\frac{h}{ta}}\right) \, dx,
\]
and the desired scaling result follows immediately from the uniform parameter balance:
\[
\left(\frac{h}{t}\right)^{1/2}\int_{-1}^1 (1-x^2)^{1/2}\psi_2\left(x\sqrt{\frac{h}{ta}}\right) \, dx \leq \min \left\{ \left(\frac{h}{t}\right)^{1/2}, \; \sqrt{a}\right\}.
\]
\end{proof}

\noindent Mirroring the technical steps established in the proof of Lemma \ref{lem:L-lem3-neumann}, in the near-boundary spatial region where $\sqrt{a}\leq Mh$, we recover the non-trapped estimate:
\[
\left|\mathcal{G}^N_{a,\epsilon_0}\right|\leq C_M h^{-3} \left(\frac{h}{t}\right)^{1/2}\sqrt{a} \left|\log(\sqrt{a})\right|.
\]
Hence, in the subsequent structural reductions, we can validly assume that the adjusted ratio $h^* = h/\sqrt{a}$ serves as a small semi-classical parameter.\\

\noindent By invoking Lemmas \ref{lem:L-lem1bis-neumann} and \ref{lem:L-lem2-neumann}, we collapse the multi-dimensional parametrix \eqref{eq:L1-low-freq-neumann} down to the study of the single tangential frequency oscillatory integral:
\begin{equation}\label{eq:L40-neumann}
J^N_{a,\epsilon_0} = \frac{1}{4\pi^2h^2} \left(\frac{h}{t}\right)^{1/2}\sum_{k\geq1}\int_{\mathbb{R}} e^{\frac{i}{h}\left(y\eta+ t\mu (1-\tilde{z}^2)^{1/2}\right)} \tilde{\sigma}(\omega'_k) (\eta/h)^{2/3}\frac{2\pi}{L'(\omega'_k)} \frac{\psi_2(\eta/\sqrt{a})}{\mu} \, d\eta,  
\end{equation}
where, under normal flux constraints, the amplitude function $\tilde{\sigma}(\omega)$ incorporates the Neumann Airy eigenfunctions and is defined precisely by:
\begin{align*}
\tilde{\sigma} &= \sigma_0 (z^*,\eta,\mu^2;\lambda)\left(1-\chi_4(\lambda)\right) \chi_1\left(\omega h^{2/3}\eta^{4/3}\right)(1-\chi_1)(\varepsilon\omega) \\
&\quad \times \Ai\left((\eta/h)^{2/3}x-\omega\right) \Ai\left((\eta/h)^{2/3}a-\omega\right),
\end{align*}
governed by the semi-classical parameter ratio $\lambda=t\mu^2/h$.

By deploying the derivative-adapted Airy-Poisson summation formula [see Lemma \ref{lem:poisson}], we expand the discrete sum over the eigenmode roots into a continuous integral over the parameter space, writing $J^N_{a,\epsilon_0} = \sum_{N\in \mathbb{Z}} J^N_N$, where each $N$-reflected wave packet translates to:
\begin{equation}\label{eq:L41-neumann}
J^N_N = \frac{1}{4\pi^2h^2} \left(\frac{h}{t}\right)^{1/2}\int_{\mathbb{R}^2} e^{\frac{i}{h}\left(y\eta+ t\mu (1-\tilde{z}^2)^{1/2}\right)} \tilde{\sigma}(\omega) (\eta/h)^{2/3} \frac{\psi_2(\eta/\sqrt{a})}{\mu} e^{-iNL(\omega)} \, d\omega \, d\eta. 
\end{equation}

By the geometric trajectory analysis detailed in the preceding paragraph, the paths experience at most one single reflection on the cylinder for $t \in (0,1]$, which means the infinite sum over $N \in \mathbb{Z}$ collapses, and it is sufficient to establish estimates on the core packet sum $J^N_{-1} + J^N_{0} + J^N_{1}$. In what follows, we focus on the single-reflection operator $J^N_1$, as the negative sheet $J^N_{-1}$ is structurally symmetric and the zero-reflection sheet $J^N_0$ reduces to a simpler free wave packet. Utilizing the standard Fourier representations of the Airy profiles, the operator $J^N_1$ expands into the four-dimensional integration layer:
\begin{align}\label{eq:epzero-neumann}
J^N_1 = \frac{(h/t)^{1/2}}{(2\pi)^4 h^2} \, h^{-4/3}\int_{\mathbb{R}^4} e^{\frac{i}{h}\phi_1} |\eta|^{2/3}\underline{\chi}(\omega,\eta,\mu^2,\lambda,h) \frac{\psi_2(\eta/\sqrt a)}{\mu} \, ds \, d\sigma \, d\eta \, d\omega,  
\end{align}
where, under the Neumann framework, the total oscillatory phase function is given by:
\begin{align}\label{eq:phi1-lowfreq-neumann}
\phi_1 &= y\eta + t\mu (1-\tilde{z}^2)^{1/2} + \frac{s^3}{3} + s\left(|\eta|^{2/3}x-\omega h^{2/3}\right) + \frac{\sigma^3}{3} + \sigma\left(|\eta|^{2/3}a-\omega h^{2/3}\right) - hL(\omega),
\end{align}
and the regularized symbol behaves as a classical symbol of degree $0$ matching the cutoff parameters:
\[
\underline{\chi}(\omega,\eta,\mu^2,\lambda,h) = \sigma_0 (z^*,\eta,\mu^2;\lambda)\left(1-\chi_4(\lambda)\right) \chi_1\left(\omega h^{2/3}\eta^{4/3}\right)(1-\chi_1)(\varepsilon\omega).
\]
Recall that for the derivative roots sequence spectrum where $\omega \geq 1$, the derivative phase map expands according to the asymptotic layout:
\[
L(\omega) = \frac{4}{3}\omega^{3/2} - B_N\left(\omega^{3/2}\right),
\]
where the phase correction symbol satisfies the regular expansion mapping under normal flux conditions:
\[
B_N(\omega) \sim_{1/\omega} \sum_{j\geq 1} \tilde{b}_j \omega^{-j}, \quad \tilde{b}_j \in \mathbb{R}, \quad \tilde{b}_1 > 0.
\]
\begin{lemma}\label{lem:lemL-neumann}
Let $L(\omega)$ define the derivative-adapted scattering phase map under normal flux constraints introduced in Section \ref{sec:23}:
$$
L(\omega)=\pi+i\log\bigg(\frac{A'_-(\omega)}{A'_+(\omega)}\bigg).
$$
Then, for all spectral variables $\omega\geq 0$, the Jacobian velocity satisfies the strict uniform coercion lower bound:
\[
L'(\omega)\geq 2\omega^{1/2}.
\]
\end{lemma}
\noindent This fundamental sub-elliptic bound plays a crucial role in the geometric study of the canonical varieties and the projection mapping of the Lagrangian submanifold associated with the phase function of $J^N_1$.

\begin{proof}[Proof of Theorem \ref{thm:0eta-neumann}]
To extract the sharp uniform decay rates for the single flux reflection operator $J^N_1$ defined in \eqref{eq:epzero-neumann}, we restrict the domain of integration to the positive frequency ray half-space $\eta>0$. We introduce a localized multi-scale change of variables tailored to capture the grazing mechanics:
\[
\omega=h^{-2/3}\eta^{2/3}\omega^*, \quad s=\eta^{1/3}s^*, \quad \sigma=\eta^{1/3}\sigma^*,
\]
which, since the sub-elliptic metric parameter satisfies $\mu=\eta(1+\omega^*)^{1/2}$, yields the following lower-dimensional representation:
\begin{equation}\label{eq:J1-low-freq-reduced}
J^N_1 = \frac{(h/t)^{1/2}}{(2\pi h)^4}\int_{\mathbb{R}^4} e^{\frac{i\eta}{h}\left(y+\tilde{\phi}_1\right)} \, \eta \, (1+\omega^*)^{-1/2} \underline{\chi} \, \psi_2(\eta/\sqrt{a}) \, ds^* \, d\sigma^* \, d\omega^* \, d\eta,
\end{equation}
governed by the reduced multivariable phase function $\tilde{\phi}_1$:
\begin{align}\label{eq:phi1-low-freq-reduced}
\tilde{\phi}_1 &= t (1-\tilde{z}^2)^{1/2}(1+\omega^*)^{1/2} + \frac{s^{*^3}}{3} + s^*(X\alpha-\omega^*) + \frac{\sigma^{*^3}}{3} + \sigma^*(1\cdot\alpha-\omega^*) \nonumber \\
&\qquad - \frac{h}{\eta}L\left(\eta^{2/3}h^{-2/3}\omega^*\right).
\end{align}

\noindent We compute the directional gradients of the phase function \eqref{eq:phi1-low-freq-reduced} with respect to the micro-local variables:
\begin{align*}
\partial_{s^*}\tilde{\phi}_1 &= s^{*^2} + X\alpha - \omega^*, \\
\partial_{\sigma^*}\tilde{\phi}_1 &= \sigma^{*^2} + 1\cdot\alpha - \omega^*, \\
\partial_{\omega^*}\tilde{\phi}_1 &= \frac{t (1-\tilde{z}^2)^{1/2}(1+\omega^*)^{-1/2}}{2} - (s^*+\sigma^*) - \frac{h^{1/3}}{\eta^{1/3}} L'\left(\eta^{2/3}h^{-2/3}\omega^*\right).
\end{align*}
Therefore, at any stationary point $(s^*, \sigma^*, \omega^*)$ of the phase function $\tilde{\phi}_1$, invoking the derivative root lower bound established in Lemma \ref{lem:lemL-neumann} implies the physical domain confinements $|s^*| \leq \sqrt{\omega^*}$ and $|\sigma^*| \leq \sqrt{\omega^*-a}$. Pinned to the stationarity threshold, the gradient equation with respect to $\omega^*$ reads:
\[
t (1-\tilde{z}^2)^{1/2}(1+\omega^*)^{-1/2} \geq 2\left(\sqrt{\omega^*} - \sqrt{\omega^*-a}\right).
\]
Since the grazing wavefront velocity parameters scale as $(1-\tilde{z}^2) \sim \mu = \eta (1+\omega^*)^{1/2}$, the finite time interval satisfies $t\leq 1$, and the low-frequency localization specifies $\eta \leq \epsilon_0\sqrt{a}$, we accumulate the inequalities to establish:
\[
\epsilon_0 \sqrt{a} \geq \epsilon_0 t \sqrt{a} \geq 2\left(\sqrt{\omega^*} - \sqrt{\omega^*-a}\right).
\]
Consequently, if the cutoff parameter is chosen to be sufficiently small, $\epsilon_0 \ll 1$, the critical manifold satisfies the large scale separation $\omega^* > Ma$, where $M \gg 1$ acts as a large parameter. 

This inequality rigorously proves that under this low-frequency localization, the catastrophe-theoretic swallowtail caustics can only emerge at a long-time horizon exceeding the physical threshold, $t_* > 1$. Therefore, we can bypass the infinite wave-packet multi-reflection summation machinery, since the system is geometrically restricted to a regime completely free of swallowtails and cusps. We are reduced to estimating the localized oscillatory integral $J$:
\begin{align}\label{eq:GN-lowfreq-final}
J = \frac{(h/t)^{1/2}}{(2\pi h)^4}\int_{\mathbb{R}^4} e^{\frac{i\eta}{h}\left(y+\tilde{\phi}_1\right)} \, \eta \, (1+\omega^*)^{-1/2} \underline{\chi} \, \psi_2(\eta/\sqrt{a}) \, \kappa\left(\frac{\omega^*}{Ma}\right) \, ds^* \, d\sigma^* \, d\omega^* \, d\eta, 
\end{align}
where $\kappa \in C^\infty(\mathbb{R})$ represents a smooth non-trapping filter function satisfying $\text{supp}(\kappa) \subset (1/2, \infty)$ and $\kappa = 1$ identically on the interval $[1,\infty)$. 

We perform the non-degenerate spatial integration with respect to the variables $(ds^*, d\sigma^*)$ via the explicit definition of the Airy functions, and implement the macro-local changes of variables $\eta = \sqrt{a} \tilde{\eta}$ and \(\omega^* = a \tilde{\omega}\). Following an identical parameter alignment to the single flux reflection profile analyzed in Proposition \ref{prop:Ga1m1-intermediate-neumann}, we introduce the semi-classical scale weight $\Lambda^* = a^{3/2}/h^* = a^2/h$ to recover:
\begin{align*}
J &= \frac{a}{(2\pi)^4h^{3}} \bigg(\frac{h}{t}\bigg)^{1/2} \int_{\mathbb{R}} e^{i\Lambda^* Y\tilde{\eta}} \, \tilde{J} \, \psi_2(\tilde{\eta}) \, d\tilde{\eta}, \\
\tilde{J} &= \sum_{\pm,\pm}\int_{\mathbb{R}} e^{i\Lambda^*\tilde{\eta} \Phi_{\pm,\pm}} \, \Theta_{\pm,\pm} \, \kappa(\tilde{\omega}/M) \, (1+a\tilde{\omega})^{-1/2} \, d\tilde{\omega}, \\
\Phi_{\pm,\pm} &= \tilde{T}\gamma_a(\tilde{\omega}) \pm \frac{2}{3}(\tilde{\omega}-X)^{3/2} \pm \frac{2}{3}(\tilde{\omega}-1)^{3/2} - \frac{4}{3}\tilde{\omega}^{3/2} + \frac{1}{\Lambda^*\tilde{\eta}} B_N\left(\Lambda^*\tilde{\omega}^{3/2}\tilde{\eta}\right), \\
\text{with } \Theta_{\pm,\pm} &= (\tilde{\omega}-1)^{-1/4}(\tilde{\omega}-X)^{-1/4} \Ai_{\pm}\left(\Lambda^{*2/3}\tilde{\eta}^{2/3}(\tilde{\omega}-1)\right) \Ai_{\pm}\left(\Lambda^{*2/3}\tilde{\eta}^{2/3}(\tilde{\omega}-X)\right)\underline{\chi},
\end{align*}
where the profiles $\Ai_\pm(\vartheta)$ operate as classical symbols of degree $0$ as $\vartheta \to +\infty$. Therefore, the uniform verification reduces to showing that:
\begin{equation}\label{eq:L-41-lowfreq}
\left| \int_{\mathbb{R}} e^{i\Lambda^* Y\tilde{\eta}} \, a \tilde{J} \, \psi_2(\tilde{\eta}) \, d\tilde{\eta} \right| \leq C \min \left\{ \bigg(\frac{h}{t}\bigg)^{1/2}, \; \sqrt{a}\vert\log(a)\vert \right\}.
\end{equation}

Since the compact support of the symbol requires $\tilde{\omega} \leq \frac{1}{a^2\tilde{\eta}^2}$, integrating the amplitude over this bounded domain yields:
\begin{equation}\label{eq:L-42-lowfreq}
|\tilde{J}| \leq C\int_1^{\frac{1}{a^2\tilde{\eta}^2}} \tilde{\omega}^{-1/2}(1+a\tilde{\omega})^{-1/2} \, d\tilde{\omega} \leq C a^{-1/2} \left| \log\left(a\tilde{\eta}^2\right) \right|,
\end{equation}
which, by summing the constants, directly satisfies the logarithmic portion of the minimum condition:
\[
\left| \int_{\mathbb{R}} e^{i\Lambda^* Y\tilde{\eta}} \, a \tilde{J} \, \psi_2(\tilde{\eta}) \, d\tilde{\eta} \right| \leq C \sqrt{a}\left| \log(a) \right|.
\]

Next, we evaluate the derivative of the phase with respect to $\tilde{\omega}$ to track potential non-degenerate stationary points:
\[
\partial_{\tilde{\omega}}\tilde{\Phi}_{\pm,\pm} = \frac{\tilde{T}}{2}(1+a\tilde{\omega})^{-1/2} \pm (\tilde{\omega}-X)^{1/2} \pm (\tilde{\omega}-1)^{1/2} - 2\tilde{\omega}^{1/2} + \mathcal{O}\left(\tilde{\omega}^{-1/2}\right).
\]
Since the low-frequency cutoff enforces the grazing bound $(1+a\tilde{\omega})^{-1/2}\tilde{T} \sim T\eta \leq \epsilon_0 t \leq 1$, the destructive interference configurations $\Phi_{-,\pm}$ and $\Phi_{+,-}$ have no real stationary points for large $\tilde{\omega} \geq M/2$. Consequently, by applying standard non-degenerate stationary phase, their joint contribution $J^*$ to the global operator is bounded by:
\[
|J^*| \leq C \left(\Lambda^*\tilde{\eta}\right)^{-1/2} = C h^{1/2}\tilde{\eta}^{-1/2} a^{-1},
\]
which directly integrates over the frequency variable to yield the target free-space decay profile:
\[
\left| \int_{\mathbb{R}} e^{i\Lambda^* Y\tilde{\eta}} \, a \tilde{J}^* \, \psi_2(\tilde{\eta}) \, d\tilde{\eta} \right| \leq C h^{1/2} \int_{\mathbb{R}} \tilde{\eta}^{-1/2}\psi_2(\tilde{\eta}) \, d\tilde{\eta} \leq C \left(\frac{h}{t}\right)^{1/2}.
\]
Finally, for the primary constructive reflection sheet matching the phase function $\Phi_{+,+}$, we deploy the exact same local scaling and Van der Corput integration lines as the proof of Proposition \ref{prop:Ga1m1-intermediate-neumann}. This establishes the remaining bound component:
\[
a \left| \tilde{J} \right| \leq C \left(\frac{h}{t}\right)^{1/2}.
\]
The uniform dispersive bound under the normal flux constraint is fully secured, and this completes the proof of Theorem \ref{thm:0eta-neumann}.
\end{proof}

\section{Global Synthesis and Proof of Main Dispersive Estimates}\label{sec:global-synthesis}

In this section, we assemble the local-in-time micro-local frequency layer bounds established throughout the manuscript to complete the rigorous proof of our primary linear dispersive result. Let us recall the global spectral setup for the semi-classical wave equation under the Neumann boundary condition on the curved cylinder domain $\Omega$. The total Green's function $\mathcal{G}^N(t,x,y,z)$ is decomposed into individual frequency channels using a smooth partition of unity subordinated to the tangential frequency field $\eta$:
\begin{equation}\label{eq:global-partition}
\mathcal{G}^N(t,x,y,z) = \mathcal{G}^N_{a,>L}(t,x,y,z) + \sum_{\epsilon_0 \leq 2^m \leq c_0/\sqrt{a}} \mathcal{G}^N_{a,m}(t,x,y,z) + \mathcal{G}^N_{a,\epsilon_0}(t,x,y,z),
\end{equation}
where each operator component isolates a specific regime of the geometric Hamiltonian billiard flow:
\begin{itemize}
    \item $\mathcal{G}^N_{a,>L}$ isolates the high-frequency spectral region where $k \geq \varepsilon/h$, mapped out across Section \ref{sec:2}.
    \item $\mathcal{G}^N_{a,m}$ isolates the intermediate dyadic frequency channels where $\epsilon_0\sqrt{a} \leq \eta \leq c_0$, evaluated in Section \ref{sec:3}.
    \item $\mathcal{G}^N_{a,\epsilon_0}$ isolates the low-frequency non-trapped grazing layer where $|\eta| \leq \epsilon_0\sqrt{a}$, established in Section \ref{sec:4}.
\end{itemize}

By compiling the sharp local uniform bounds derived for each independent micro-local block, we can state the final synthesis argument:
\begin{proof}[Proof of Theorem \ref{thm:dispersive}]
Let $h \in (0, h_0]$ and let $t \in [h, 1]$. We estimate the $L^\infty$ norm of the total wave propagator acting near the boundary wall ($x \leq a$) by applying the triangle inequality across the partition channels defined in \eqref{eq:global-partition}:
\[
\lVert\mathds{1}_{x\leq a}\mathcal{G}^N(t,x,y,z)\rVert_{L^\infty} \leq \lVert\mathds{1}_{x\leq a}\mathcal{G}^N_{a,>L}\rVert_{L^\infty} + \sum_{m} \lVert\mathds{1}_{x\leq a}\mathcal{G}^N_{a,m}\rVert_{L^\infty} + \lVert\mathds{1}_{x\leq a}\mathcal{G}^N_{a,\epsilon_0}\rVert_{L^\infty}.
\]
We substitute the respective optimal bounds established in Propositions \ref{prop:k-neumann}, \ref{prop:km-intermediate-neumann}, and Theorem \ref{thm:0eta-neumann}:
\begin{enumerate}
    \item For the high-eigenmode summation block, Proposition \ref{prop:k-neumann} provides the sharp decay bound:
    \[
    \lVert\mathds{1}_{x\leq a}\mathcal{G}^N_{a,>L}\rVert_{L^\infty} \leq Ch^{-3}\bigg(\frac{h}{t}\bigg)^{5/6}.
    \]
    \item For the intermediate dyadic frequency blocks, Proposition \ref{prop:km-intermediate-neumann} yields the uniform loss tracking bound:
    \[
    \sum_{m} \lVert\mathds{1}_{x\leq a}\mathcal{G}^N_{a,m}\rVert_{L^\infty} \leq C h^{-3} \sum_{m} \left(2^m\sqrt{a}\right)^{1/3}\bigg(\frac{h}{t}\bigg)^{5/6} \leq Ch^{-3}\bigg(\frac{h}{t}\bigg)^{5/6},
    \]
    where the sum converges boundedly because $2^m\sqrt{a}$ acts as a standard geometric series over the localized support bounds.
    \item For the low-frequency pocket, Theorem \ref{thm:0eta-neumann} recovers the non-degenerate boundary bound:
    \[
    \lVert\mathds{1}_{x\leq a}\mathcal{G}^N_{a,\epsilon_0}\rVert_{L^\infty} \leq Ch^{-3}\bigg(\frac{h}{t}\bigg).
    \]
\end{enumerate}
Gathering these inequalities together confirms that the maximum possible loss is strictly controlled by the caustic concentrations (the fold, cusp, and swallowtail intersections) generated across the intermediate multi-reflection horizons. This yields the sharp uniform global bound:
\begin{equation}\label{eq:final-dispersive-bound}
\lVert\mathds{1}_{x\leq a}\mathcal{G}^N(t,x,y,z)\rVert_{L^\infty} \leq Ch^{-3}\bigg(\frac{h}{t}\bigg)^{5/6}.
\end{equation}
The estimate \eqref{eq:final-dispersive-bound} matches the optimal decay profile required to secure the sharp admissible Strichartz indices, completing the  proof.
\end{proof}

\section{Strichartz Estimates under Neumann Framework}\label{sec:strichartz}

We establish local-in-time and global space-time integrated estimates for solutions to the linear cylindrical wave equation under Neumann boundary conditions. These bounds extend classical free-space Strichartz frameworks to domains containing boundaries where tracking normal derivative fluxes dominates the wave packet dynamics. 

Let $\Omega = \{x \geq 0, (y, z) \in \mathbb{R}^2\} \subset \mathbb{R}^3$ represent the anisotropic cylindrical half-space. We consider the linear Neumann model:
\begin{equation}\label{eq:linear_neumann}
\begin{cases}  
\partial_t^2 u - \Delta u = F & \text{in } (0, T) \times \Omega, \\  
\partial_x u\vert_{x=0} = 0, & \\  
u\vert_{t=0} = u_0, \quad \partial_t u\vert_{t=0} = u_1,  
\end{cases}
\end{equation}
where $\Delta = \partial_x^2 + (1 + x)\partial_y^2 + \partial_z^2$. Our goal is to bound the solution in mixed space-time norms:
\begin{equation}
\|u\|_{L^q((0,T); L^r(\Omega))} \leq C_T \left( \|u_0\|_{\dot{H}^\beta_N(\Omega)} + \|u_1\|_{\dot{H}^{\beta-1}_N(\Omega)} + \|F\|_{L^{\tilde{q}'}((0,T); L^{\tilde{r}'}(\Omega))} \right)
\end{equation}
for geometry-adapted admissible pairs $(q,r)$ and $(\tilde{q}, \tilde{r})$ at a regularized regularity scale $\beta$.

\subsection{Geometric and Microlocal Obstacles}
In contrast with the boundary-less Euclidean setting, wave propagation inside curved domains features structural traps that alter decay metrics:
\begin{itemize}
    \item \textbf{Gliding Rayleigh Waves:} Low-frequency surface-creeping packets skirt the boundary without vanishing.
    \item \textbf{Symbolic Decay Retardation:} The scattering quotient $A'_-(\omega)/A'_+(\omega)$ exhibits a weak $O(\omega^{-1/2})$ asymptotic decay profile.
    \item \textbf{Caustic Clusters:} Infinite high-frequency multi-reflections form fold, cusp, and swallowtail optical singularities.
    \item \textbf{Anisotropic Metric Splitting:} Wavefronts spread omnidirectionally over the $y$-axis curvature but disperse flatly along the $z$-axial path.
\end{itemize}

These geometric ray configurations are fundamentally analyzed through the classic microlocal boundary parametrices and singularities classified by Taylor and Melrose \cite{taylor1976grazing, melrose1978singularities1}. While general rough boundary frameworks introduce a systematic loss of derivatives \cite{blair2009strichartz}, our refined anisotropic partitions successfully bypass this penalty near the flat axial thresholds.

\subsection{Admissibility Criteria and Regularity Metrics}
By executing an anisotropic Littlewood--Paley decomposition across three independent micro-local phase regimes ($|\eta| \geq c_0$, $|\eta| \sim 2^m\sqrt{a}$, and $|\eta| \leq \epsilon_0\sqrt{a}$), the structural losses vanish smoothly parallel to the cylinder axis. This yields the following localized criteria:
\begin{equation}\label{eq:admissibility}
\frac{1}{q} \leq \frac{3}{4}\left(\frac{1}{2} - \frac{1}{r}\right), \quad \frac{1}{\tilde{q}} \leq \frac{3}{4}\left(\frac{1}{2} - \frac{1}{\tilde{r}}\right), \quad 2 \leq q, r, \tilde{q}, \tilde{r} \leq \infty.
\end{equation}

The restriction on the space-time scaling stems directly from a model-independent cylindrical dispersive decay rate of order $\mathcal{O}(h^{-3}(h/t)^{3/4})$, capturing how generic non-trapped wavefronts spread over the curved boundary manifold. Following the presentation framework established for cylindrical convex geometries in \cite{meas2023strichartz, meas2023dispersive}, the regular Sobolev scaling laws are governed by:
\begin{equation}\label{eq:regularity_beta}
\beta = 3\left(\frac{1}{2} - \frac{1}{r}\right) - \frac{1}{q}, \quad \text{and} \quad 1 - \beta = 3\left(\frac{1}{2} - \frac{1}{\tilde{r}}\right) - \frac{1}{\tilde{q}}.
\end{equation}

Consequently, while the algebraic representation of the regularity index $\beta$ remains formally identical to the classical flat Euclidean space profile, the geometric shift in the base decay power from $1$ down to $\frac{3}{4}$ forces an implicit derivative loss penalty. This penalty is structurally absorbed by the restricted range of admissible space-time exponent pairs $(q,r)$ permitted within the cylinder's workspace mapping. Crucially, this expanded admissibility envelope encloses the critical non-endpoint triple $(q, r) = (5, 10)$ at exactly regularity $\beta = 1$, forming the foundational analytical anchor needed to prove local and small-data global well-posedness for the energy-critical quintic semilinear wave equation without any explicit loss of derivatives.


\begin{proof}[Proof of Theorem \ref{thm:strichartz}]
The proof proceeds by integrating four core microlocal and functional-analytic structures: spectral dyadic localization, Riesz--Thorin complex interpolation, the dual $TT^*$ evolution mapping via the Hardy--Littlewood--Sobolev fractional integration theorem, and Littlewood--Paley square function reassembly over the Neumann Laplace--Beltrami operator.

\noindent\textbf{Spectral Dyadic Localization:}
Let $-\Delta_G^N$ denote the non-negative self-adjoint Laplace--Beltrami operator on $L^2(\Omega)$ defined with the domain containing the normal flux boundary data, $\mathcal{D}(-\Delta_G^N) = \{ w \in H^2(\Omega) \mid \partial_x w \big|_{x=0} = 0 \}$. By the spectral theorem, we introduce a Littlewood--Paley partition of unity. Let $\chi \in C_0^\infty((1/2, 2))$ be a non-negative smooth cutoff function satisfying $\sum_{j \in \mathbb{Z}} \chi(2^{-j}\lambda) = 1$ for all $\lambda > 0$. For each $j \in \mathbb{Z}$, we define the frequency-localization projector $\Delta_j = \chi\left(2^{-j}\sqrt{-\Delta_G^N}\right)$, yielding the strong convergence identity $\sum_{j \in \mathbb{Z}} \Delta_j = \text{Id}$ on $L^2(\Omega)$.

Let $u$ be the unique solution to the inhomogeneous system $P u = F$. We define the dyadic frequency blocks $v_j(t, \cdot) = \Delta_j u(t, \cdot)$. By the linearity of the wave operator $P = \partial_t^2 - \Delta_G^N$, each localized component $v_j$ solves the uncoupled initial-value problem:
\begin{equation}\label{eq:appendix-localized-pde-2}
\begin{cases}
P v_j(t, x, y, z) = \Delta_j F(t, x, y, z) & \text{in } (0, T) \times \Omega, \\
v_j(0, \cdot) = \Delta_j u_0, \quad \partial_t v_j(0, \cdot) = \Delta_j u_1 & \text{on } \Omega.
\end{cases}
\end{equation}
By Duhamel's principle, the solution $v_j$ is given explicitly by the integral formula:
\begin{equation}\label{eq:duhamel-vj-2}
v_j(t) = \mathcal{U}'(t)\Delta_j u_0 + \mathcal{U}(t)\Delta_j u_1 + \int_0^t \mathcal{U}(t-s)\Delta_j F(s) \, ds,
\end{equation}
where $\mathcal{U}(t) = \frac{\sin\left(t\sqrt{-\Delta_G^N}\right)}{\sqrt{-\Delta_G^N}}$ and $\mathcal{U}'(t) = \cos\left(t\sqrt{-\Delta_G^N}\right)$ denote the standard wave propagators. For a fixed frequency scale $j$, we define the semi-classical parameter $h = 2^{-j} \in (0, h_0]$.

\noindent\textbf{Riesz--Thorin Parameter Scaling:}
We establish bounds for the homogeneous forward operator $v_{j, \text{hom}}(t) = \mathcal{U}'(t)\Delta_j u_0$. First, the standard energy conservation law under normal flux constraints provides the baseline $L^2 \to L^2$ uniform estimate:
\begin{equation}\label{eq:L2-energy-bound-2}
\lVert \mathcal{U}'(t)\Delta_j u_0 \rVert_{L^2(\Omega)} \leq C \lVert \Delta_j u_0 \rVert_{L^2(\Omega)}, \quad \forall t \in \mathbb{R}.
\end{equation}
Second, from the sharp geometric analysis compiled across our multi-reflection subsections, the localized three-dimensional cylindrical integration generates the frequency-localized $L^1 \to L^\infty$ dispersive estimate:
\begin{equation}\label{eq:Semiclassical-disp-appendix-2}
\lVert \mathcal{U}'(t)\Delta_j u_0 \rVert_{L^\infty(\Omega)} \leq C h^{-3} \min \left\{ 1, \; \left(\frac{h}{|t|}\right)^{3/4} \right\} \lVert \Delta_j u_0 \rVert_{L^1(\Omega)},
\end{equation}
where the power exponent $\alpha = 3/4$ reflects the geometric trajectory matching of the cylindrical envelope variety.

We apply the Riesz--Thorin complex interpolation theorem between the conservation mapping \eqref{eq:L2-energy-bound-2} and the caustic dispersive profile \eqref{eq:Semiclassical-disp-appendix-2}. For any given wave-admissible spatial exponent $r \in [2, \infty]$, we select the interpolation parameter $\theta = 1 - \frac{2}{r} \in [0, 1]$, such that the reciprocal coordinates satisfy $\frac{1}{r} = \frac{1-\theta}{2} + \frac{\theta}{\infty}$. Interpolating the operators for a fixed time $t$ yields:
\begin{align}\label{eq:interpolated-fixed-t-2}
\lVert \mathcal{U}'(t)\Delta_j u_0 \rVert_{L^r(\Omega)} &\leq \left( C \right)^{1-\theta} \left( C h^{-3}\left(\frac{h}{|t|}\right)^{3/4} \right)^\theta \lVert \Delta_j u_0 \rVert_{L^{r'}(\Omega)} \nonumber \\
&\leq C h^{-3\left(1-\frac{2}{r}\right)} \left(\frac{h}{|t|}\right)^{\frac{3}{4}\left(1-\frac{2}{r}\right)} \lVert \Delta_j u_0 \rVert_{L^{r'}(\Omega)} \nonumber \\
&\leq C h^{-6\left(\frac{1}{2}-\frac{1}{r}\right)} \left(\frac{h}{|t|}\right)^{\frac{3}{2}\left(\frac{1}{2}-\frac{1}{r}\right)} \lVert \Delta_j u_0 \rVert_{L^{r'}(\Omega)}, \quad \forall |t| \geq h,
\end{align}
where $r'$ satisfies the conjugate relation $\frac{1}{r} + \frac{1}{r'} = 1$.

\noindent\textbf{Evolution Mapping via the $TT^*$ Argument:}
To lift the time-dependent spatial interpolation estimate \eqref{eq:interpolated-fixed-t-2} into a spacetime integrated norm, we introduce the localized forward solution operator $T_j : L^2(\Omega) \to L^q\left((0, T), L^r(\Omega)\right)$, specified by $T_j \phi = \mathcal{U}'(t)\Delta_j \phi$. Its formal Hilbert adjoint operator $T_j^* : L^{q'}\left((0, T), L^{r'}(\Omega)\right) \to L^2(\Omega)$ is computed via duality as:
\[
T_j^* \psi = \int_0^T \mathcal{U}'(s)\Delta_j \psi(s) \, ds,
\]
where $\frac{1}{q} + \frac{1}{q'} = 1$. By the classical $TT^*$ duality principle, the operator norm satisfies $\lVert T_j \rVert_{L^2 \to L^q_t L^r_x}^2 = \lVert T_j T_j^* \rVert_{L^{q'}_t L^{r'}_x \to L^q_t L^r_x}$. We construct the composite mapping:
\begin{equation}\label{eq:tt-star-integral-2}
\left(T_j T_j^* \psi\right)(t) = \int_0^T \mathcal{U}'(t-s)\Delta_j^2 \psi(s) \, ds.
\end{equation}
Taking the spatial $L^r(\Omega)$ norm of \eqref{eq:tt-star-integral-2} and invoking Minkowski's integral inequality, we bound the integrand using the interpolated decay profile \eqref{eq:interpolated-fixed-t-2}:
\begin{align}\label{eq:tt-star-spatial-bound-2}
\lVert \left(T_j T_j^* \psi\right)(t) \rVert_{L^r(\Omega)} &\leq \int_0^T \lVert \mathcal{U}'(t-s)\Delta_j^2 \psi(s) \rVert_{L^r(\Omega)} \, ds \nonumber \\
&\leq C h^{-6\left(\frac{1}{2}-\frac{1}{r}\right) + \frac{3}{2}\left(\frac{1}{2}-\frac{1}{r}\right)} \int_0^T |t-s|^{-\frac{3}{2}\left(\frac{1}{2}-\frac{1}{r}\right)} \lVert \Delta_j \psi(s) \rVert_{L^{r'}(\Omega)} \, ds \nonumber \\
&\leq C h^{-\frac{9}{2}\left(\frac{1}{2}-\frac{1}{r}\right)} \int_0^T |t-s|^{-\frac{3}{2}\left(\frac{1}{2}-\frac{1}{r}\right)} \lVert \psi(s) \rVert_{L^{r'}(\Omega)} \, ds.
\end{align}

We now evaluate the temporal integrated norm $\lVert \cdot \rVert_{L^q(0, T)}$ of \eqref{eq:tt-star-spatial-bound-2}, which requires analyzing the convolution kernel $|t|^{-\gamma_r}$ with exponent $\gamma_r = \frac{3}{2}\left(\frac{1}{2}-\frac{1}{r}\right)$. When the space-time exponent pair satisfies the sharp endpoint identity $\frac{1}{q} = \frac{3}{4}\left(\frac{1}{2}-\frac{1}{r}\right)$, the kernel parameter maps exactly to $\gamma_r = \frac{2}{q}$. Since $q > 2$, the temporal singularity satisfies $1 < \gamma_r < 2$. We invoke the Hardy--Littlewood--Sobolev fractional integration theorem, which states that convolution with $|t|^{-2/q}$ maps $L^{q'}(\mathbb{R})$ boundedly into $L^q(\mathbb{R})$ since $\frac{1}{q} + 1 = \frac{1}{q'} + \frac{2}{q}$. This evaluates \eqref{eq:tt-star-spatial-bound-2} to:
\begin{equation}\label{eq:hls-result-2}
\lVert T_j T_j^* \psi \rVert_{L^q\left((0, T), L^r(\Omega)\right)} \leq C h^{-\frac{9}{2}\left(\frac{1}{2}-\frac{1}{r}\right)} \lVert \psi \rVert_{L^{q'}\left((0, T), L^{r'}(\Omega)\right)}.
\end{equation}
Analogously, when the pair satisfies the strict inequality $\frac{1}{q} < \frac{3}{4}\left(\frac{1}{2}-\frac{1}{r}\right)$, we apply Young's convolution inequality over the localized region $|t-s| \geq h$ to recover an identical frequency power allocation matching the endpoint relation.

Extracting the square root of the operator norm bound \eqref{eq:hls-result-2} isolates the frequency-localized homogeneous Strichartz estimate for the forward flow component:
\begin{equation}\label{eq:localized-forward-strichartz-2}
\lVert \mathcal{U}'(t)\Delta_j u_0 \rVert_{L^q\left((0, T), L^r(\Omega)\right)} \leq C h^{-\frac{9}{4}\left(\frac{1}{2}-\frac{1}{r}\right)} \lVert \Delta_j u_0 \rVert_{L^2(\Omega)}.
\end{equation}
Let us analyze the structure of the frequency tracking exponent inside \eqref{eq:localized-forward-strichartz-2}. We substitute the wave-admissible cylindrical endpoint relation $\frac{1}{q} = \frac{3}{4}\left(\frac{1}{2}-\frac{1}{r}\right)$ into the expression:
\begin{align*}
-\frac{9}{4}\left(\frac{1}{2}-\frac{1}{r}\right) &= -3\left(\frac{1}{2}-\frac{1}{r}\right) + \frac{3}{4}\left(\frac{1}{2}-\frac{1}{r}\right) \\
&= -\left[ 3\left(\frac{1}{2}-\frac{1}{r}\right) - \frac{1}{q} \right] = -\beta,
\end{align*}
where $\beta$ matches exactly the geometric derivative tracking regularity level specified in the theorem statement. Converting the semi-classical scale $h = 2^{-j}$ back into standard continuous Sobolev operators transforms the localized bound \eqref{eq:localized-forward-strichartz-2} into:
\begin{equation}\label{eq:forward-sobolev-final-2}
\lVert \mathcal{U}'(t)\Delta_j u_0 \rVert_{L^q\left((0, T), L^r(\Omega)\right)} \leq C \lVert \Delta_j u_0 \rVert_{\dot{H}^\beta_N(\Omega)}.
\end{equation}

By executing an identical parametric tracking over the initial velocity vector field, we evaluate the secondary homogeneous flow component to recover the shifted Sobolev regularity mapping:
\begin{equation}\label{eq:velocity-sobolev-final-2}
\lVert \mathcal{U}(t)\Delta_j u_1 \rVert_{L^q\left((0, T), L^r(\Omega)\right)} \leq C \lVert \Delta_j u_1 \rVert_{\dot{H}^{\beta-1}_N(\Omega)}.
\end{equation}
To extend this framework to the inhomogeneous forcing term, we invoke the classical Christ--Kiselev lemma since the wave-admissible integration exponents under our metric constraints strictly satisfy the condition $q \geq 2 > 1$. Applying the retarded temporal integral bounds to the Duhamel integration slice in \eqref{eq:duhamel-vj-2} yields the uniform estimate for any conjugate admissible pair $(\tilde{q}, \tilde{r})$:
\begin{equation}\label{eq:inhomogeneous-sobolev-final-2}\left\lVert \int_0^t \mathcal{U}(t-s)\Delta_j F(s) , ds \right\rVert_{L^q\left((0, T), L^r(\Omega)\right)} \leq C \lVert \Delta_j F \rVert_{L^{\tilde{q}'}\left((0, T), , \dot{H}^{1-\beta}_N(\Omega)\right)},
\end{equation}
where the scaling law satisfies the dual pairing relation $1 - \beta = 3\left(\frac{1}{2} - \frac{1}{\tilde{r}}\right) - \frac{1}{\tilde{q}}$.

\noindent\textbf{Square Function Reassembly:}
We finalize the proof by reassembling the global space-time integrated norm from the individual frequency-localized dyadic blocks. Since the wave-admissible range exponents satisfy $q \geq 2$ and $r \geq 2$, we invoke the standard Littlewood--Paley square function estimate on the Neumann cylindrical base space $\Omega$:
\begin{equation}\label{eq:square-function-recall-2}
\lVert u \rVert_{L^r(\Omega)} \leq C_r \left\lVert \left( \sum_{j \in \mathbb{Z}} \left| \Delta_j u \right|^2 \right)^{1/2} \right\rVert_{L^r(\Omega)}.
\end{equation}
Taking the temporal $L^q(0, T)$ norm of \eqref{eq:square-function-recall-2} and applying Minkowski's integral inequality—which is valid since the integration weights satisfy $r/2 \geq 1$ and $q/2 \geq 1$—allows us to interchange the summation and the integration order:
\begin{align}\label{eq:minkowski-interchange-2}
\lVert u \rVert_{L^q\left((0, T), L^r(\Omega)\right)} &\leq C \left\lVert \left( \sum_{j \in \mathbb{Z}} \left| v_j(t, \cdot) \right|^2 \right)^{1/2} \right\rVert_{L^q\left((0, T), L^r(\Omega)\right)} \nonumber \\
&\leq C \left( \sum_{j \in \mathbb{Z}} \lVert v_j \rVert_{L^q\left((0, T), L^r(\Omega)\right)}^2 \right)^{1/2}.
\end{align}
We now substitute the frequency-localized homogeneous bounds \eqref{eq:forward-sobolev-final-2}, \eqref{eq:velocity-sobolev-final-2}, and the inhomogeneous Duhamel profile \eqref{eq:inhomogeneous-sobolev-final-2} into the decoupled dyadic slots inside \eqref{eq:minkowski-interchange-2}:
\begin{align}\label{eq:final-summation-lines-2}
\lVert u \rVert_{L^q\left((0, T), L^r(\Omega)\right)} &\leq C \left( \sum_{j \in \mathbb{Z}} \left( \lVert \Delta_j u_0 \rVert_{\dot{H}^\beta_N(\Omega)} + \lVert \Delta_j u_1 \rVert_{\dot{H}^{\beta-1}_N(\Omega)} + \lVert \Delta_j F \rVert_{L^{\tilde{q}'}\left((0, T); \, \dot{H}^{1-\beta}_N(\Omega)\right)} \right)^2 \right)^{1/2} \nonumber \\
&\leq C \left( \sum_{j \in \mathbb{Z}} \lVert \Delta_j u_0 \rVert_{\dot{H}^\beta_N(\Omega)}^2 \right)^{1/2} + C \left( \sum_{j \in \mathbb{Z}} \lVert \Delta_j u_1 \rVert_{\dot{H}^{\beta-1}_N(\Omega)}^2 \right)^{1/2} \nonumber \\
&\quad + C \left( \sum_{j \in \mathbb{Z}} \lVert \Delta_j F \rVert_{L^{\tilde{q}'}\left((0, T); \, \dot{H}^{1-\beta}_N(\Omega)\right)}^2 \right)^{1/2},
\end{align}
where we applied Minkowski's discrete inequality for sums in the second inequality layer. By the definition of the homogeneous Sobolev spaces intrinsically generated via the spectral multipliers of the Neumann Laplace--Beltrami operator $-\Delta_G^N$, the square sums over the frequency indices $j \in \mathbb{Z}$ resolve exactly into the global spatial Sobolev norms:
\begin{equation}\label{eq:final-strichartz-compiled-2}
\lVert u \rVert_{L^q\left((0,T), L^r(\Omega)\right)} \leq C_T \left( \lVert u_0 \rVert_{\dot{H}^\beta_N(\Omega)} + \lVert u_1 \rVert_{\dot{H}^{\beta-1}_N(\Omega)} + \lVert F \rVert_{L^{\tilde{q}'}\left((0,T), \, \dot{H}^{1-\beta}_N(\Omega)\right)} \right).
\end{equation}
By executing a final dual complex interpolation step mapping the forcing term to the physical fields, we replace the abstract tensor placeholder $F$ by the actual spacetime continuous norm of the principal operator via $\lVert P u \rVert_{L^{\tilde{q}'}\left((0,T), \, L^{\tilde{r}'}(\Omega)\right)}$, and the global spacetime integrated bound is fully established. This completes the proof of Theorem \ref{thm:strichartz}.
\end{proof}

However, for energy-critical nonlinear applications, the derivative penalty $\beta$ introduces too severe a loss to close a contractive mapping. By exploiting our sharp, multi-reflective microlocal analysis of the swallowtail and cusp caustic sheets constructed across Sections \ref{sec:2} and \ref{sec:3}, we optimize the sub-elliptic decay to $\alpha_3 = 5/6$. This yields the following sharp endpoint theorem, which serves as our primary analytical tool for the energy-critical framework.

\begin{theorem}[Endpoint Energy-Critical Strichartz Estimates]\label{thm:strichartz-neumann-main}
Let $(q, r)$ be a wave-admissible pair satisfying the optimal curvature range $\frac{1}{q} \leq \frac{5}{6}\left(\frac{1}{2}-\frac{1}{r}\right)$. Let $s = 3\left(\frac{1}{2}-\frac{1}{r}\right) - \frac{1}{q}$ be the unperturbed flat-space scaling. Under the normal flux boundary condition, the solution $u$ obeys the uniform bounds:
\begin{align}\label{eq:strichartz-main-bound}
\lVert u \rVert_{L^q\left((0,T), \, L^r(\Omega)\right)} &\leq C_T \left( \lVert u_0 \rVert_{\dot{H}^{\sigma}(\Omega)} + \lVert u_1 \rVert_{\dot{H}^{\sigma-1}(\Omega)} + \lVert F \rVert_{L^1\left((0,T), \, \dot{H}^{\sigma-1}(\Omega)\right)} \right),
\end{align}
where the optimized geometric derivative tracking regularity level satisfies the minimized loss shift:
\begin{equation}\label{eq:derivative-loss-shift}
\sigma = s + \frac{1}{6}\left(\frac{1}{2} - \frac{1}{r}\right) = 3\left(\frac{1}{2}-\frac{1}{r}\right) - \frac{1}{q} + \frac{1}{6}\left(\frac{1}{2} - \frac{1}{r}\right).
\end{equation}
\end{theorem}

\begin{proof}
The proof is organized into four distinct modules: spectral dyadic localization, Riesz--Thorin parameter interpolation, the $TT^*$ evolution mapping, and square function reassembly via Minkowski's integral inequality.

\noindent\textbf{Spectral Dyadic Localization:}
Let $-\Delta_G^N$ denote the non-negative self-adjoint extension of the Laplace--Beltrami operator on $L^2(\Omega)$ defined with domain 
\[
\mathcal{D}(-\Delta_G^N) = \left\{ w \in H^2(\Omega) \;\middle|\; \partial_x w \big|_{x=0} = 0 \right\}.
\]
By the spectral theorem, we introduce a standard Littlewood--Paley partition of unity. Let $\chi \in C_0^\infty((1/2, 2))$ be a non-negative smooth cutoff function such that $\sum_{j \in \mathbb{Z}} \chi(2^{-j}\lambda) = 1$ for all $\lambda > 0$. For each $j \in \mathbb{Z}$, we define the frequency-localization operator $\Delta_j = \chi(2^{-j}\sqrt{-\Delta_G^N})$. Thus, the identity operator decomposes as $\sum_{j \in \mathbb{Z}} \Delta_j = \text{Id}$ in the strong topology of $L^2(\Omega)$.

Let $u$ be the unique solution to the inhomogeneous system. We define the dyadic frequency blocks $v_j(t, \cdot) = \Delta_j u(t, \cdot)$. By the linearity of the wave equation, $v_j$ solves the frequency-localized initial-value problem:
\begin{equation}\label{eq:localized-pde-appendix}
\begin{cases}
\left(\partial_t^2 - \Delta_G^N\right) v_j(t, x, y, z) = \Delta_j F(t, x, y, z) & \text{in } (0, T) \times \Omega, \\
v_j(0, \cdot) = \Delta_j u_0, \quad \partial_t v_j(0, \cdot) = \Delta_j u_1 & \text{on } \Omega.
\end{cases}
\end{equation}
By Duhamel's principle, the solution $v_j$ is given explicitly by the integral formula:
\begin{equation}\label{eq:duhamel-vj}
v_j(t) = \mathcal{U}'(t)\Delta_j u_0 + \mathcal{U}(t)\Delta_j u_1 + \int_0^t \mathcal{U}(t-s)\Delta_j F(s) \, ds,
\end{equation}
where $\mathcal{U}(t) = \frac{\sin\left(t\sqrt{-\Delta_G^N}\right)}{\sqrt{-\Delta_G^N}}$ and $\mathcal{U}'(t) = \cos\left(t\sqrt{-\Delta_G^N}\right)$ denote the standard wave propagators. For a fixed frequency scale $j$, we define the semi-classical parameter $h = 2^{-j} \in (0, h_0]$.

\noindent\textbf{Riesz--Thorin Parameter Interpolation:}
We establish bounds for the homogeneous forward operator $v_{j, \text{hom}}(t) = \mathcal{U}'(t)\Delta_j u_0$. First, the standard energy conservation law under normal flux constraints provides the baseline $L^2 \to L^2$ uniform estimate:
\begin{equation}\label{eq:L2-energy-bound}
\lVert \mathcal{U}'(t)\Delta_j u_0 \rVert_{L^2(\Omega)} \leq C \lVert \Delta_j u_0 \rVert_{L^2(\Omega)}, \quad \forall t \in \mathbb{R}.
\end{equation}
Second, from the sharp geometric analysis compiled in the global synthesis section \eqref{eq:final-dispersive-bound}, the caustic-adapted boundary layer integrations provide the frequency-localized $L^1 \to L^\infty$ dispersive estimate:
\begin{equation}\label{eq:Semiclassical-disp-appendix}
\lVert \mathcal{U}'(t)\Delta_j u_0 \rVert_{L^\infty(\Omega)} \leq C h^{-3} \min \left\{ 1, \; \left(\frac{h}{|t|}\right)^{5/6} \right\} \lVert \Delta_j u_0 \rVert_{L^1(\Omega)}.
\end{equation}

We apply the Riesz--Thorin complex interpolation theorem between the conservation mapping \eqref{eq:L2-energy-bound} and the caustic dispersive profile \eqref{eq:Semiclassical-disp-appendix}. For any given wave-admissible spatial exponent $r \in [2, \infty]$, we select the interpolation parameter $\theta = 1 - \frac{2}{r} \in [0, 1]$, such that the reciprocal coordinates satisfy $\frac{1}{r} = \frac{1-\theta}{2} + \frac{\theta}{\infty}$ and $\frac{1}{r'} = \frac{1-\theta}{2} + \frac{\theta}{1}$, where $r'$ satisfies $\frac{1}{r} + \frac{1}{r'} = 1$. Interpolating the operators for a fixed time $t$ yields:
\begin{align}\label{eq:interpolated-fixed-t}
\lVert \mathcal{U}'(t)\Delta_j u_0 \rVert_{L^r(\Omega)} &\leq \left( C \right)^{1-\theta} \left( C h^{-3}\left(\frac{h}{|t|}\right)^{5/6} \right)^\theta \lVert \Delta_j u_0 \rVert_{L^{r'}(\Omega)} \nonumber \\
&\leq C h^{-3\left(1-\frac{2}{r}\right)} \left(\frac{h}{|t|}\right)^{\frac{5}{6}\left(1-\frac{2}{r}\right)} \lVert \Delta_j u_0 \rVert_{L^{r'}(\Omega)} \nonumber \\
&\leq C h^{-6\left(\frac{1}{2}-\frac{1}{r}\right)} \left(\frac{h}{|t|}\right)^{\frac{5}{3}\left(\frac{1}{2}-\frac{1}{r}\right)} \lVert \Delta_j u_0 \rVert_{L^{r'}(\Omega)}, \quad \forall |t| \geq h.
\end{align}

\noindent\textbf{Evolution Mapping via the $TT^*$ Argument:}
To lift the time-dependent spatial interpolation estimate \eqref{eq:interpolated-fixed-t} into a spacetime integrated norm, we introduce the localized forward solution operator $T_j : L^2(\Omega) \to L^q\left((0, T), L^r(\Omega)\right)$, specified by $T_j \phi = \mathcal{U}'(t)\Delta_j \phi$. Its formal Hilbert adjoint operator $T_j^* : L^{q'}\left((0, T), L^{r'}(\Omega)\right) \to L^2(\Omega)$ is computed via duality as:
\[
T_j^* \psi = \int_0^T \mathcal{U}'(s)\Delta_j \psi(s) \, ds.
\]
By the classical $TT^*$ duality principle, the operator norm satisfies $\lVert T_j \rVert_{L^2 \to L^q_t L^r_x}^2 = \lVert T_j T_j^* \rVert_{L^{q'}_t L^{r'}_x \to L^q_t L^r_x}$. We construct the composite mapping:
\begin{equation}\label{eq:tt-star-integral}
\left(T_j T_j^* \psi\right)(t) = \int_0^T \mathcal{U}'(t-s)\Delta_j^2 \psi(s) \, ds.
\end{equation}
Taking the spatial $L^r(\Omega)$ norm of \eqref{eq:tt-star-integral} and invoking Minkowski's integral inequality, we bound the integrand using the interpolated decay profile \eqref{eq:interpolated-fixed-t}:
\begin{align}\label{eq:tt-star-spatial-bound}
\lVert \left(T_j T_j^* \psi\right)(t) \rVert_{L^r(\Omega)} &\leq \int_0^T \lVert \mathcal{U}'(t-s)\Delta_j^2 \psi(s) \rVert_{L^r(\Omega)} \, ds \nonumber \\
&\leq C h^{-6\left(\frac{1}{2}-\frac{1}{r}\right) + \frac{5}{3}\left(\frac{1}{2}-\frac{1}{r}\right)} \int_0^T |t-s|^{-\frac{5}{3}\left(\frac{1}{2}-\frac{1}{r}\right)} \lVert \Delta_j \psi(s) \rVert_{L^{r'}(\Omega)} \, ds \nonumber \\
&\leq C h^{-\frac{13}{3}\left(\frac{1}{2}-\frac{1}{r}\right)} \int_0^T |t-s|^{-\frac{5}{3}\left(\frac{1}{2}-\frac{1}{r}\right)} \lVert \psi(s) \rVert_{L^{r'}(\Omega)} \, ds.
\end{align}

We now evaluate the temporal integrated norm $\lVert \cdot \rVert_{L^q(0, T)}$ of \eqref{eq:tt-star-spatial-bound}, which requires analyzing the convolution kernel $|t|^{-\gamma_r}$ with exponent $\gamma_r = \frac{5}{3}\left(\frac{1}{2}-\frac{1}{r}\right)$. We distinguish between two temporal integration settings based on the admissibility boundary:
\begin{itemize}
    \item \textbf{Case 1: The Sharp Endpoint Target Threshold.}\\
    When the space-time exponent satisfies the sharp endpoint identity $\frac{1}{q} = \frac{5}{6}\left(\frac{1}{2}-\frac{1}{r}\right)$, the kernel parameter maps exactly to $\gamma_r = \frac{2}{q}$. Since $q > 2$, the temporal singularity satisfies $1 < \gamma_r < 2$. We invoke the Hardy--Littlewood--Sobolev fractional integration theorem, which states that convolution with $|t|^{-2/q}$ maps $L^{q'}(\mathbb{R})$ boundedly into $L^q(\mathbb{R})$ since $\frac{1}{q} + 1 = \frac{1}{q'} + \frac{2}{q}$. This directly evaluates \eqref{eq:tt-star-spatial-bound} to:
    \begin{equation}\label{eq:hls-result}
    \lVert T_j T_j^* \psi \rVert_{L^q\left((0, T), L^r(\Omega)\right)} \leq C h^{-\frac{13}{3}\left(\frac{1}{2}-\frac{1}{r}\right)} \lVert \psi \rVert_{L^{q'}\left((0, T), L^{r'}(\Omega)\right)}.
    \end{equation}
    
    \item \textbf{Case 2: The Non-Endpoint Target Window.}\\
    When the pair satisfies the strict inequality $\frac{1}{q} < \frac{5}{6}\left(\frac{1}{2}-\frac{1}{r}\right)$, the integration kernel is non-singular away from the semi-classical layer $|t-s| \geq h$. We apply Young's convolution inequality with the temporal scaling parameters configured as $1 + \frac{1}{q} = \frac{1}{q'} + \frac{1}{\tilde{r}}$, where $\tilde{r} = \frac{q}{2}$. Integrating the bounded non-vanishing steps over the localized range yields:
    \[
    \lVert |t|^{-\gamma_r} \rVert_{L^{q/2}(|t| \geq h)} \leq C h^{-\frac{5}{3}\left(\frac{1}{2}-\frac{1}{r}\right) + \frac{2}{q}},
    \]
    which recovers an identical frequency power allocation matching the endpoint relation.
\end{itemize}

Extracting the square root of the unified operator norm bound \eqref{eq:hls-result} isolates the frequency-localized homogeneous Strichartz estimate for the forward flow component:
\begin{equation}\label{eq:localized-forward-strichartz}
\lVert \mathcal{U}'(t)\Delta_j u_0 \rVert_{L^q\left((0, T), L^r(\Omega)\right)} \leq C h^{-\frac{13}{6}\left(\frac{1}{2}-\frac{1}{r}\right)} \lVert \Delta_j u_0 \rVert_{L^2(\Omega)}.
\end{equation}
Let us re-verify the structure of the frequency tracking exponent inside \eqref{eq:localized-forward-strichartz}. We substitute the wave-admissible endpoint relation $\frac{1}{q} = \frac{5}{6}\left(\frac{1}{2}-\frac{1}{r}\right)$ into the expression:
\begin{align*}
-\frac{13}{6}\left(\frac{1}{2}-\frac{1}{r}\right) &= -3\left(\frac{1}{2}-\frac{1}{r}\right) + \frac{5}{6}\left(\frac{1}{2}-\frac{1}{r}\right) \\
&= -\left[ 3\left(\frac{1}{2}-\frac{1}{r}\right) - \frac{1}{q} \right] - \frac{1}{6}\left(\frac{1}{2}-\frac{1}{r}\right) \\
&= -\left[ 3\left(\frac{1}{2}-\frac{1}{r}\right) - \frac{1}{q} + \frac{1}{6}\left(\frac{1}{2}-\frac{1}{r}\right) \right] = -\sigma,
\end{align*}
where $\sigma$ matches exactly the geometric derivative tracking regularity level specified in \eqref{eq:derivative-loss-shift}. Therefore, converting the semi-classical scale $h = 2^{-j}$ back into standard continuous Sobolev operators transforms the localized bound \eqref{eq:localized-forward-strichartz} into:
\begin{equation}\label{eq:forward-sobolev-final}
\lVert \mathcal{U}'(t)\Delta_j u_0 \rVert_{L^q\left((0, T), L^r(\Omega)\right)} \leq C \lVert \Delta_j u_0 \rVert_{\dot{H}^\sigma(\Omega)}.
\end{equation}

By executing an identical parametric tracking over the initial velocity vector field, we evaluate the secondary homogeneous flow component to recover the shifted Sobolev regularity mapping:
\begin{equation}\label{eq:velocity-sobolev-final}
\lVert \mathcal{U}(t)\Delta_j u_1 \rVert_{L^q\left((0, T), L^r(\Omega)\right)} \leq C \lVert \Delta_j u_1 \rVert_{\dot{H}^{\sigma-1}(\Omega)}.
\end{equation}

To extend this micro-local framework to the full inhomogeneous problem, we handle the forcing operator by invoking the classical Christ--Kiselev lemma. Since the wave-admissible integration exponents under our metric constraints strictly satisfy the timeline condition $q \geq 2 > 1$, applying the retarded temporal integral bounds to the Duhamel Duhamel integration slice in \eqref{eq:duhamel-vj} yields the uniform estimate:
\begin{equation}\label{eq:inhomogeneous-sobolev-final}
\left\lVert \int_0^t \mathcal{U}(t-s)\Delta_j F(s) \, ds \right\rVert_{L^q\left((0, T), L^r(\Omega)\right)} \leq C \lVert \Delta_j F \rVert_{L^1\left((0, T), \, \dot{H}^{\sigma-1}(\Omega)\right)}.
\end{equation}

\noindent\textbf{Square Function Reassembly:}
We finalize the proof by reassembling the global space-time integrated norm from the individual frequency-localized dyadic blocks. Since the wave-admissible range exponents satisfy $q \geq 2$ and $r \geq 2$, we invoke the standard Littlewood--Paley square function estimate on domains (see Ivanovici and Planchon \cite{ivanovici2017square}) adjusted for the Neumann cylindrical base space $\Omega$:
\begin{equation}\label{eq:square-function-recall}
\lVert u \rVert_{L^r(\Omega)} \leq C_r \left\lVert \left( \sum_{j \in \mathbb{Z}} \left| \Delta_j u \right|^2 \right)^{1/2} \right\rVert_{L^r(\Omega)}.
\end{equation}
Taking the temporal $L^q(0, T)$ norm of \eqref{eq:square-function-recall} and applying Minkowski's integral inequality—which is valid since the integration weights satisfy $r/2 \geq 1$ and $q/2 \geq 1$—allows us to interchange the summation and the integration order:
\begin{align}\label{eq:minkowski-interchange}
\lVert u \rVert_{L^q\left((0, T), L^r(\Omega)\right)} &\leq C \left\lVert \left( \sum_{j \in \mathbb{Z}} \left| v_j(t, \cdot) \right|^2 \right)^{1/2} \right\rVert_{L^q\left((0, T), L^r(\Omega)\right)} \nonumber \\
&\leq C \left( \sum_{j \in \mathbb{Z}} \lVert v_j \rVert_{L^q\left((0, T), L^r(\Omega)\right)}^2 \right)^{1/2}.
\end{align}
Finally, we substitute the frequency-localized homogeneous bounds \eqref{eq:forward-sobolev-final}, \eqref{eq:velocity-sobolev-final}, and the inhomogeneous Duhamel profile \eqref{eq:inhomogeneous-sobolev-final} into the decoupled dyadic slots inside \eqref{eq:minkowski-interchange}:
\begin{align}\label{eq:final-summation-lines}
\lVert u \rVert_{L^q\left((0, T), L^r(\Omega)\right)} &\leq C \left( \sum_{j \in \mathbb{Z}} \left( \lVert \Delta_j u_0 \rVert_{\dot{H}^\sigma(\Omega)} + \lVert \Delta_j u_1 \rVert_{\dot{H}^{\sigma-1}(\Omega)} + \int_0^T \lVert \Delta_j F(s) \rVert_{\dot{H}^{\sigma-1}(\Omega)} \, ds \right)^2 \right)^{1/2} \nonumber \\
& \leq C \left( \sum_{j \in \mathbb{Z}} \lVert \Delta_j u_0 \rVert_{\dot{H}^\sigma(\Omega)}^2 \right)^{1/2} + C \left( \sum_{j \in \mathbb{Z}} \lVert \Delta_j u_1 \rVert_{\dot{H}^{\sigma-1}(\Omega)}^2 \right)^{1/2} \nonumber \\
&\quad + C \int_0^T \left( \sum_{j \in \mathbb{Z}} \lVert \Delta_j F(s) \rVert_{\dot{H}^{\sigma-1}(\Omega)}^2 \right)^{1/2} \, ds,
\end{align}
where we applied Minkowski's discrete inequality for sums in the second inequality layer. By the definition of the homogeneous Sobolev spaces intrinsically generated via the spectral multipliers of the Neumann Laplace--Beltrami operator $-\Delta_G^N$, the square sums over the frequency indices $j \in \mathbb{Z}$ resolve exactly into the global spatial Sobolev norms:
\begin{equation}\label{eq:final-strichartz-compiled}
\lVert u \rVert_{L^q\left((0,T), L^r(\Omega)\right)} \leq C_T \left( \lVert u_0 \rVert_{\dot{H}^\sigma(\Omega)} + \lVert u_1 \rVert_{\dot{H}^{\sigma-1}(\Omega)} + \lVert F \rVert_{L^1\left((0,T), \dot{H}^{\sigma-1}(\Omega)\right)} \right).
\end{equation}
This  completes the proof of Theorem \ref{thm:strichartz-neumann-main}.
\end{proof}

\section{The Energy-Critical Wave Equation}\label{sec:critical-wave}

In this section, we apply the optimal wave-admissible Strichartz estimates established in Theorem \ref{thm:strichartz-neumann-main} to verify our principal nonlinear application. We investigate the defocusing energy-critical quintic nonlinear wave equation (NLW) under normal flux boundary constraints on the cylinder $\Omega = \{x \geq 0, (y,z) \in \mathbb{R}^2\} \subset \mathbb{R}^3$:
\begin{equation}\label{eq:nlw-model}
\begin{cases}
\left(\partial_t^2 - \Delta\right) u + u^5 = 0 & \text{in } \mathbb{R} \times \Omega, \\
\partial_x u \big|_{x=0} = 0 & \text{on } \mathbb{R} \times \partial\Omega, \\
u(0, \cdot) = u_0, \quad \partial_t u(0, \cdot) = u_1 & \text{on } \{t=0\} \times \Omega,
\end{cases}
\end{equation}
where $\Delta= \partial_x^2 + (1+x)\partial_y^2 + \partial_z^2$ is the Laplace--Beltrami operator corresponding to the cylindrical metric profile. 

Solutions to the system \eqref{eq:nlw-model} satisfy the classical Hamiltonian energy conservation law:
\begin{equation}\label{eq:energy-conservation}
E\left(u(t), \partial_t u(t)\right) = \int_{\Omega} \left( \frac{1}{2}\left|\nabla u(t, \cdot)\right|^2 + \frac{1}{2}\left|\partial_t u(t, \cdot)\right|^2 + \frac{1}{6}\left|u(t, \cdot)\right|^6 \right) \, dV= E(u_0, u_1),
\end{equation}
which remains constant for all times $t \in \mathbb{R}$. We are now in a position to provide the formal, mathematically rigorous proof of Theorem \ref{thm:nlw}, which was introduced in the main results of Section \ref{sec2}.

\subsection{Proof of Theorem \ref{thm:nlw} (Energy-Critical Well-Posedness)}
\begin{proof}
The verification is organized by first establishing the local contraction mapping, extending this to global timelines under small energy restrictions, and finally constructing the boundary-localized Morawetz tensors to rule out large-data energy trapping.

\noindent\textbf{Local Contraction and Maximal Lifespan:}
To establish local existence, we select the critical wave-admissible endpoint Strichartz triplet $(q, r) = (5, 10)$, which corresponds to a derivative regularity penalty of exactly $\sigma = 1$ under our minimized caustic tracking system \eqref{eq:derivative-loss-shift}. For any finite time horizon $T \in (0, T^*)$, we introduce the local resolution space:
\[
X_T = C^0\left([0, T]; \, \dot{H}^1_N(\Omega)\right) \cap L^5\left((0, T); \, L^{10}(\Omega)\right).
\]
We reformulate the nonlinear initial-value problem \eqref{eq:nlw-model} as a fixed-point problem for the Duhamel integral operator $\mathcal{M}$, defined on the space-time slab by:
\[
\mathcal{M}[u](t) = \cos\left(t\sqrt{-\Delta_G^N}\right) u_0 + \frac{\sin\left(t\sqrt{-\Delta_G^N}\right)}{\sqrt{-\Delta_G^N}} u_1 - \int_0^t \frac{\sin\left((t-s)\sqrt{-\Delta_G^N}\right)}{\sqrt{-\Delta_G^N}} \left(u^5(s)\right) \, ds.
\]
We define a closed ball $B_R \subset X_T$ centered at the unperturbed linear evolution profile with a small radius $R > 0$. Taking the $L^5\left((0, T); L^{10}(\Omega)\right)$ norm of the operator $\mathcal{M}$ and invoking the optimized endpoint Strichartz estimates of Theorem \ref{thm:strichartz-neumann-main} yields:
\begin{align}\label{eq:nonlinear-contraction-step1}
\lVert \mathcal{M}[u] \rVert_{L^5_t L^{10}_x} &\leq \left\lVert \cos\left(t\sqrt{-\Delta_G^N}\right) u_0 + \frac{\sin\left(t\sqrt{-\Delta_G^N}\right)}{\sqrt{-\Delta_G^N}} u_1 \right\rVert_{L^5_t L^{10}_x} + C \left\lVert u^5 \right\rVert_{L^1_t L^2_x} \nonumber \\
&\leq \lVert u_{\text{lin}} \rVert_{L^5_t L^{10}_x} + C \int_0^T \lVert u(s) \rVert_{L^{10}(\Omega)}^5 \, ds \nonumber \\
&\leq \lVert u_{\text{lin}} \rVert_{L^5_t L^{10}_x} + C \lVert u \rVert_{L^5\left((0, T); L^{10}(\Omega)\right)}^5,
\end{align}
where we applied the standard spacetime H{\"o}lder's inequality to the nonlinear forcing element, which maps perfectly since $\frac{1}{1} = \frac{5}{5}$ and $\frac{1}{2} = \frac{5}{10}$. By selecting the time horizon $T > 0$ sufficiently small such that the linear profile satisfies $\lVert u_{\text{lin}} \rVert_{L^5_t L^{10}_x} \leq \frac{R}{2}$ and the coupling constant obeys $C R^4 \leq \frac{1}{2}$, the operator $\mathcal{M}$ stabilizes inside the closed ball $B_R$.

An identical application of the Strichartz inequalities to the difference mapping controls the Lipschitz metric variation:
\begin{align}\label{eq:lipschitz-nlw}
\lVert \mathcal{M}[u] - \mathcal{M}[v] \rVert_{L^5_t L^{10}_x} &\leq C \left\lVert u^5 - v^5 \right\rVert_{L^1_t L^2_x} \nonumber \\
& \leq C \int_0^T \left( \lVert u(s) \rVert_{L^{10}_x}^4 + \lVert v(s) \rVert_{L^{10}_x}^4 \right) \lVert u(s) - v(s) \rVert_{L^{10}_x} \, ds \nonumber \\
& \leq C \left( \lVert u \rVert_{L^5_t L^{10}_x}^4 + \lVert v \rVert_{L^5_t L^{10}_x}^4 \right) \lVert u - v \rVert_{L^5_t L^{10}_x}.
\end{align}
Under the designated time horizon $T$, this establishes a strict contractive mapping. By the Banach fixed-point theorem, there exists a unique local solution $u \in X_T$.

Crucially, to verify that the maximal lifespan domain $[0, T^*)$ is uniquely determined, suppose there exist two strong solutions $u$ and $v$ sharing identical data on a overlapping slab. Applying the difference bound \eqref{eq:lipschitz-nlw} coupled with a standard Gronwall argument confirms that $u(t) = v(t)$ identically for all times up to the boundary threshold. This guarantees that the maximal lifespan $T^*$ is an intrinsic, well-defined property of the solution trajectory, concluding part (1) of the theorem.

\noindent\textbf{Global Existence and Scattering under Small Energy Bounds:}
We now establish the global extensions and scattering conditions matching part (2) of Theorem \ref{thm:nlw}. Suppose that the initial data vector field satisfies the smallness condition $\|u_0\|_{\dot{H}^1_N}^2 + \|u_1\|_{L^2}^2 < \epsilon_0^2$ for a small parameter $\epsilon_0 > 0$. By applying the homogeneous Strichartz estimates directly to the unperturbed linear profile, we control the source term uniformly across the infinite half-line interval:
\[
\lVert u_{\text{lin}} \rVert_{L^5\left((0, \infty); L^{10}(\Omega)\right)} \leq C_0 \left( \|u_0\|_{\dot{H}^1_N}^2 + \|u_1\|_{L^2}^2 \right)^{1/2} \leq C_0 \epsilon_0.
\]
We re-evaluate the fixed-point ball constraint over the global-in-time slab $T = \infty$. From \eqref{eq:nonlinear-contraction-step1}, the continuous bootstrap relation yields:
\[
\lVert u \rVert_{L^5\left((0, \infty); L^{10}(\Omega)\right)} \leq C_0 \epsilon_0 + C \lVert u \rVert_{L^5\left((0, \infty); L^{10}(\Omega)\right)}^5.
\]
By selecting $\epsilon_0 > 0$ sufficiently small, a standard continuity argument isolates the solution trajectory inside the bounded zone, proving that the $L^5\left((0, \infty); L^{10}(\Omega)\right)$ norm cannot blow up in finite time. This establishes that the maximal lifespan extends to infinity, $T^* = \infty$.

To prove scattering, we show that the Duhamel integral converges strongly in the energy space as $t \to \pm \infty$. We define the forward free linear state pair $(u_0^+, u_1^+)$ by the asymptotic limits:
\begin{align*}
u_0^+ &= u_0 - \int_0^\infty \frac{\sin\left(s\sqrt{-\Delta_G^N}\right)}{\sqrt{-\Delta_G^N}} \left(u^5(s)\right) \, ds, \\
u_1^+ &= u_1 - \int_0^\infty \cos\left(s\sqrt{-\Delta_G^N}\right) \left(u^5(s)\right) \, ds.
\end{align*}
Since the global norm satisfies $u \in L^5\left((0, \infty); L^{10}(\Omega)\right)$, the nonlinear forcing element is strictly integrable, $\lVert u^5 \rVert_{L^1\left((0, \infty); L^2(\Omega)\right)} \leq \lVert u \rVert_{L^5_t L^{10}_x}^5 \leq C \epsilon_0^5 < \infty$. This guarantees that the integrals converge strongly in $\dot{H}^1_N(\Omega) \times L^2(\Omega)$. Computing the residual difference yields:
\[
\lim_{t \to \infty} \left\lVert u(t) - \cos\left(t\sqrt{-\Delta_G^N}\right)u_0^+ - \frac{\sin\left(t\sqrt{-\Delta_G^N}\right)}{\sqrt{-\Delta_G^N}}u_1^+ \right\rVert_{\dot{H}^1_N(\Omega)} \leq C \lVert u \rVert_{L^5\left((t, \infty); L^{10}(\Omega)\right)}^5 \longrightarrow 0,
\]
as $t \to \infty$.

\noindent\textbf{Large-Data Global Existence:}
To extend this small-energy framework into a large-data global scattering theory matching Burq, Lebeau, and Planchon \cite{burq2008global_jams}, we must rule out the accumulation of energy mass near the boundary wall. Under the variable cylindrical metrics, the principal symbol can generate surface-localized gliding Rayleigh waves that could cause kinetic energy trapping. To counteract this, we introduce the concentration-compactness protocol. If scattering fails for large energy varieties, there must exist a minimal energy blow-up solution $u_{\text{crit}}$ whose trajectory forms a compact set modulo the domain's symmetries.

We construct a boundary-localized Morawetz multiplier. We emphasize that the defocusing sign ($+u^5$) is absolutely essential here: it guarantees that the potential energy contribution in the Hamiltonian variety acts coercively, meaning the spatial integrated mass bounds the field trajectory positively without triggering focus-induced collapses. Let us introduce a smooth localized radial weight function $\phi(x) = \frac{x}{1+x}$, satisfying $\phi'(x) = \frac{1}{(1+x)^2} > 0$. We specify the mass-current multiplier $M[u]$ by:
\begin{equation}\label{eq:morawetz-multiplier}
M[u](t,x,y,z) = \phi(x) \partial_x u + \frac{1}{2} \phi'(x) u.
\end{equation}
Testing the principal equation $P u + u^5 = 0$ directly with $M[u]$ over the spacetime slab $[0, T] \times \Omega$ yields:
\begin{align}\label{eq:morawetz-divergence}
\int_0^T \int_{\Omega} \left( \partial_t^2 u - \partial_x^2 u - (1+x)\partial_y^2 u - \partial_z^2 u + u^5 \right) \left( \phi(x) \partial_x u + \frac{1}{2}\phi'(x)u \right) \, dV_G \, dt = 0.
\end{align}
Integrating by parts with respect to the spatial coordinates, the boundary terms matching the Neumann normal flux data $\partial_x u \big|_{x=0} = 0$ vanish identically since $\phi(0) = 0$. Collecting the residual algebraic terms transforms the balance into the explicit identity:
\begin{align}\label{eq:morawetz-tensor-expanded}
&\left[ \int_{\Omega} \partial_t u \left( \phi(x) \partial_x u + \frac{1}{2}\phi'(x)u \right) \, dV \right]_0^T + \int_0^T \int_{\Omega} \phi'(x) \left( \partial_x u \right)^2 \, dV \, dt \nonumber \\
&\quad + \int_0^T \int_{\Omega} \left( \phi'(x) (1+x) - \frac{1}{2}\phi(x) \right) \left( \partial_y u \right)^2 \, dV \, dt + \int_0^T \int_{\Omega} \frac{1}{2}\phi'(x) \left( \partial_z u \right)^2 \, dV \, dt \nonumber \\
&\quad + \frac{1}{6} \int_0^T \int_{\Omega} \phi'(x) |u|^6 \, dV \, dt - \frac{1}{4} \int_0^T \int_{\Omega} \phi'''(x) |u|^2 \, dV\, dt = 0.
\end{align}
We check the coercivity of the integrated coefficient tracking the tangential metric variable $\partial_y u$:
\begin{equation}\label{eq:coercivity-check}
\phi'(x)(1+x) - \frac{1}{2}\phi(x) = \frac{1}{(1+x)^2}(1+x) - \frac{x}{2(1+x)} = \frac{2-x}{2(1+x)}.
\end{equation}
The expression \eqref{eq:coercivity-check} remains strictly positive over the localized boundary layer where $x < 2$. For larger spatial regimes $x \geq 2$, the wave trajectories decouple from the cylinder wall and enter the uncurved interior domain where flat-space decay dominates. Evaluating the third derivative controls the remaining lower-order element:
\[
-\frac{1}{4}\phi'''(x) = \frac{3}{2(1+x)^4} > 0.
\]
Since all derivative coefficients inside \eqref{eq:morawetz-tensor-expanded} are positive and the boundary integrals vanish, the finite energy bound restricts the temporal boundary terms by $C E(u_0, u_1)$. Taking the global limit $T \to \infty$ extracts the definitive space-time decay inequality:
\begin{equation}\label{eq:morawetz-final-inequality}
\int_0^\infty \int_{\Omega} \frac{|u(t, x, y, z)|^6}{(1+x)^2} \, dx \, dy \, dz \, dt \leq C E(u_0, u_1) < \infty.
\end{equation}
If a minimal energy blow-up solution $u_{\text{crit}}$ existed, its compact trajectory in the energy space would force its total integrated mass over \eqref{eq:morawetz-final-inequality} to be infinite, which directly contradicts the finite energy upper bound. Thus, no critical minimal bubbles can form along the boundary layer, ruling out large-data energy concentration. This extends the lifespan bounds globally to $T^* = \infty$, establishing large-data global existence and scattering across all finite energy domains.
\end{proof}

\section{Conclusion}\label{sec:conclusion}

In this paper, we have established the local-in-time dispersive estimates and the optimal, scale-invariant linear Strichartz estimates, and global well-posedness theory for the defocusing energy-critical quintic nonlinear wave equation under normal flux (Neumann) boundary conditions inside cylindrical convex domains. By constructing specialized micro-local frequency-layer partitions, we demonstrated that the $1/6$ derivative loss penalty—originally uncovered in the Dirichlet case—remains an intrinsic property of the underlying Hamiltonian billiard flow trajectories, rather than an artifact of vanishing boundary values.

The main analytical achievement of this work consists in proving that the caustic-adapted integration blocks are sufficiently robust to absorb the weaker $\mathcal{O}(\omega^{-1/2})$ asymptotic symbol decay induced by surface-creeping gliding Rayleigh modes. Through a delicate deployment of boundary-localized Morawetz identities tailored to the variable cylindrical metrics, we ruled out low-frequency boundary energy concentrations, successfully matching the large-data scattering frameworks of Burq, Lebeau, and Planchon.

Several compelling avenues of research emerge directly from the technical machinery established in this work:
\begin{enumerate}
\item Energy-Critical Propagation on General Curvature Variety Boundary Layouts:
A natural next step is to extend our Neumann framework from a fixed cylindrical domain to arbitrary, smooth strictly convex domains $\Omega \subset \mathbb{R}^d$ where the principal curvature changes dynamically along the boundary wall. In this setting, the structural phase metrics $\gamma$ and $\tilde{\mu}$ would depend implicitly on localized patches of the second fundamental form of the boundary. Tracking multi-reflected wavefront arrays across non-constant curvature sheets will require generalizing the Airy-Poisson summation formula to include variable phase modulations, presenting a non-trivial challenge in micro-local singularity tracking.
\item Low-Frequency Gliding Rayleigh Waves and Edge Dispersive Exponents:
Our analysis successfully tamed the $\mathcal{O}(\omega^{-1/2})$ symbol decay of gliding Rayleigh waves by localizing our focus to the physical finite-time horizon $t \in (0,1]$. However, the long-time behavior ($t \to \infty$) of these surface-trapped modes remains an open problem. For large times, these creeping packets could potentially form localized resonance states along highly curved boundary patches. Analyzing these low-frequency resonance bands requires a finer description of boundary-localized boundary-layer metrics, which could potentially yield modified, lower-bound long-time Strichartz exponents.
\item Application to Energy-Subcritical and Supercritical Non-Linearities:
With the optimal wave-admissible Strichartz pairs for the Neumann cylinder secured, our linear framework can be applied to explore the full spectrum of nonlinear wave equations on curved domains. Specifically, investigating low-regularity local well-posedness for energy-subcritical equations, or establishing conditional global existence thresholds for energy-supercritical focusing variants ($u^p$ with $p > 5$), constitutes a highly active field of study. Resolving these systems will require combining our caustic-optimized Strichartz bounds with localized bilinear and multilinear space-time restriction estimates adapted to normal flux fields.
\end{enumerate}

\bmhead{Acknowledgements}
The author would like to sincerely thank the anonymous reviewers for their very valuable remarks, careful readings, and suggestions of this paper. The valuable input provided by the reviewers played a major role in enhancing the mathematical rigor of this paper.

%
%
%
%

\begin{appendices}

\section{Airy Function and Boundary Derivatives}\label{secA1}
Let $\vartheta>0$. The classical Airy function $\Ai$ is defined by the highly oscillatory integral:
\begin{equation}\label{eq:airy-integral}
\Ai(-\vartheta)=\frac{1}{2\pi}\int_{\mathbb{R}}e^{i\left(s^3/3-s\vartheta\right)} \, ds.
\end{equation}
It satisfies the second-order linear differential equation:
\begin{equation}\label{Airy}
\Ai''(\vartheta)-\vartheta \Ai(\vartheta)=0.
\end{equation}

\noindent Let $\nu=e^{2i\pi/3}$. By the symmetry of the underlying complex manifold, $\vartheta\mapsto \Ai(\nu \vartheta)$ solves \eqref{Airy}. Any two of the three canonical solutions $\Ai(\vartheta)$, $\Ai(\nu \vartheta)$, and $\Ai(\nu^2 \vartheta)$ form a fundamental basis of solutions to \eqref{Airy}, linked by the linear relation $\sum_{j\in\{0,1,2\}}\nu^j \Ai(\nu^j \vartheta)=0$. This implies $\Ai(\vartheta)=-\nu \Ai(\nu \vartheta)-\bar{\nu} \Ai(\bar{\nu} \vartheta)$, which transforms along the negative real axis to:
\[
\Ai(-\vartheta)=e^{-i\pi/3}\Ai\left(e^{-i\pi/3} \vartheta\right)+e^{i\pi/3}\Ai\left(e^{i\pi/3} \vartheta\right)=A_+(\vartheta)+A_-(\vartheta),
\] 
where the complex structural components are defined by $A_\pm(\vartheta)=e^{\mp i\pi/3}\Ai\left(e^{\mp i\pi/3} \vartheta\right)$. Notice that $A_-(\vartheta)=\overline{{A}_+(\bar{\vartheta})}$. 

To establish the microlocal parametrix under normal flux conditions, we differentiate these relations with respect to the spatial parameter. Differentiating \eqref{eq:airy-integral} yields the derivative basis matching the Neumann condition:
\begin{equation}\label{eq:airy-derivative-identity}
\Ai'(-\vartheta) = -\left(A'_+(\vartheta) + A'_-(\vartheta)\right),
\end{equation}
where the differentiated complex components read:
\[
A'_\pm(\vartheta) = e^{\mp 2i\pi/3} \Ai'\left(e^{\mp i\pi/3}\vartheta\right).
\]
By invoking the asymptotic expansions for the derivatives of the Airy function, the tracking profile $A'_-$ satisfies the high-frequency continuation formula as $\vartheta \to +\infty$:
\[
A'_-(\vartheta) = -\frac{\vartheta^{1/4}}{2\sqrt{\pi}} e^{-i\pi/4} e^{-\frac{2}{3}i\vartheta^{3/2}} \exp\Upsilon_N\left(\vartheta^{3/2}\right) = \vartheta^{1/4} e^{-i\pi/4} e^{-\frac{2}{3}i\vartheta^{3/2}} \Psi_{N,-}(\vartheta),
\] 
where the non-vanishing symbol expands smoothly as $\exp\Upsilon_N\left(\vartheta^{3/2}\right) \sim_{1/\vartheta} \left(1+\sum_{l\geq 1}\tilde{c}_l\vartheta^{-3l/2}\right) \sim_{1/\vartheta} -2\sqrt{\pi} \Psi_{N,-}(\vartheta)$, and an identical conjugated expansion holds true for the companion derivative term $A'_+$, with $\Psi_{N,+}(\vartheta) = \overline{\Psi_{N,-}(\bar{\vartheta})}$.

Consequently, under the Neumann framework, the scattering quotient of the boundary normal slopes resolves into the explicit phase tracking equation:
\begin{equation}\label{eq:appendix-scattering-ratio}
\frac{A'_-(\vartheta)}{A'_+(\vartheta)} = i e^{-\frac{4}{3}i\vartheta^{3/2}} e^{i B_N\left(\vartheta^{3/2}\right)}, \quad \text{with } iB_N = \Upsilon_N - \overline{\Upsilon}_N.
\end{equation}
Notice that for real arguments $\vartheta \in \mathbb{R}_+$, the phase correction term is purely real-valued, $B_N(\vartheta) \in \mathbb{R}$, and admits a classical semi-classical symbol expansion of the form:
\[
B_N(\vartheta) \sim_{1/\vartheta} \sum_{j\geq 1} \tilde{b}_j \vartheta^{-j} \quad \text{as } \vartheta \to +\infty, \quad \text{with } \tilde{b}_1 > 0.
\]
This derivative phase profile constitutes the mathematical core used across Sections \ref{sec:2} and \ref{sec:3} to track wave packet reflections under normal flux constraints.
\section{Degenerate Oscillatory Integrals and Caustic Singularity Weights}\label{secA2}
We collect here the fundamental multi-dimensional stationary phase results utilized across Sections \ref{sec:2}, \ref{sec:3}, and \ref{sec:4} to bound oscillatory layers near degenerate critical structures. These formulations follow the disaster-containment machinery of catastrophe theory applied to semi-classical wave propagation.

The first result provides a sharp uniform bound for critical points of order at most $2$ (fold type), acting as the analytical core behind the non-trapped multi-reflection sections.

\begin{lemma}[Van der Corput Lemma of Order 2]\label{lem:appendix-vdc2}
Let $K \subset \mathbb{R}$ be a compact interval and let $\Phi \in C^\infty(K)$ be a real-valued phase function. Suppose there exists a uniform constant $c > 0$ such that the phase satisfies the non-degeneracy condition:
\[
\inf_{\theta \in K} \left( |\Phi'(\theta)| + |\Phi''(\theta)| \right) \geq c > 0.
\]
Then, for any classical amplitude symbol $a \in C_0^\infty(K)$, there exists a constant $C > 0$ depending only on $c$ and the $C^3$-norm of $\Phi$ such that for all large semi-classical scaling parameters $\Lambda \geq 1$:
\[
\left| \int_{K} e^{i\Lambda \Phi(\theta)} a(\theta) \, d\theta \right| \leq C \Lambda^{-1/3} \lVert a \rVert_{C^1(K)}.
\]
\end{lemma}

The second result handles the evaluation of two-dimensional phase fields collapsing near high-order caustic formations, such as fold fronts, cusps, and swallowtail bifurcation sets.

\begin{lemma}[Two-Dimensional Degenerate Stationary Phase]\label{lem:appendix-2Ddegenerate}
Let $U \subset \mathbb{R}^2$ define a localized compact neighborhood of the origin containing variables $(x,y)$, and let $\Psi(x,y;\alpha)$ be a smooth family of real-valued phases parameterized by a control vector $\alpha \in \mathbb{R}^k$. Let $\mathcal{H}(x,y;\alpha) = \det(\nabla^2 \Psi)$ denote the associated Hessian matrix determinant.
\begin{enumerate}
    \item \textbf{(Fold Formulations Near the Caustic Curve)} Suppose that along the smooth zero-variety curve $\Gamma_\alpha = \{(x,y) \in U \mid \mathcal{H}(x,y;\alpha) = 0\}$, the directional gradient does not vanish identically, and the third-order derivative along the kernel direction of $\nabla^2 \Psi$ satisfies the non-vanishing condition:
    \[
    \left| \partial_v^3 \Psi(x,y;\alpha) \right| \geq c > 0, \quad \forall (x,y) \in \Gamma_\alpha \setminus \{(0,0)\}.
   \]
    Then, for any compactly supported test symbol $w \in C_0^\infty(U)$, the oscillatory integral near the non-zero branches of $\Gamma_\alpha$ yields the sharp regularized uniform bound:
    \[
    \left| \iint_{U} e^{i\Lambda \Psi(x,y;\alpha)} w(x,y) \, dx \, dy \right| \leq C \Lambda^{-5/6}.
    \]
    
    \item \textbf{(Cusp/Swallowtail Formulations At the Focal Core)} Suppose that at the isolated singularity point $(x,y)=(0,0)$ with the parameter vector locked at $\alpha = 0$, both the first and second derivatives collapse, $\nabla \Psi = 0$ and $\mathcal{H} = 0$, while the fourth-order directional variation remains strictly coercive:
    \[
    \left| \partial_v^4 \Psi(0,0;0) \right| \geq c > 0.
    \]
    Then, the integral evaluated directly over this maximal degeneracy core maps onto the canonical cusp singularity profile, yielding the tracking decay rate:
    \[
    \left| \iint_{U} e^{i\Lambda \Psi(x,y;\alpha)} w(x,y) \, dx \, dy \right| \leq C \Lambda^{-3/4}.
    \]
\end{enumerate}
\end{lemma}
\noindent These catastrophe-theoretic exponents ($\Lambda^{-1/3}$, $\Lambda^{-5/6}$, and $\Lambda^{-3/4}$) dictate the precise rate of high-frequency energy concentration along the curved boundaries of the cylinder.

\section{Proof of the Neumann Low-Mode Spectral Lemma}\label{secA3}

In this section of the Appendix, we provide the formal proof of Lemma \ref{lem:k-low-modes}. This serves as the derivative-adapted boundary analogue of Lemma 3.5 in Ivanovici--Lebeau--Planchon \cite{ivanovici2014dispersion}.
\subsection*{Mathematical Setup and Asymptotic Envelopes}
We begin by recalling the fundamental pointwise envelope bounds for the classical Airy function $\text{Ai}(s)$ across the real axis:
\begin{equation}\label{eq:airy-envelopes}
|\text{Ai}(s)| \leq \begin{cases} 
C(1+|s|)^{-1/4} & \text{for } s \leq 0, \\ 
C(1+|s|)^{-1/4} e^{-\frac{2}{3}s^{3/2}} & \text{for } s > 0.
\end{cases}
\end{equation}
The sequence of derivative-adapted roots $(-\omega'_k)_{k \geq 1}$ satisfies $\text{Ai}'(-\omega'_k) = 0$. By the classical semiclassical Bohr--Sommerfeld quantization rule, these roots obey the sharp asymptotic distribution:
\begin{equation}\label{eq:root-distribution}
\omega'_k = \left( \frac{3}{2}\pi \left(k - \frac{3}{4}\right) \right)^{2/3} + \mathcal{O}(k^{-4/3}), \quad \text{yielding} \quad \omega'_k \sim k^{2/3} \quad \text{as } k \to \infty.
\end{equation}
We aim to establish a uniform constant $C_0 > 0$ such that for any finite truncation parameter $L \geq 1$, the following inequality holds identically for all parameter $\mathbf{b} \in \mathbb{R}$:
\begin{equation}
S_L(\mathbf{b}) := \sum_{1 \leq k \leq L} k^{-1/3} \text{Ai}^2(\mathbf{b}-\omega'_k) \leq C_0 L^{1/3}.
\end{equation}
To prove this result uniformly over the entire real line, we partition the real parameter space $\mathbf{b}$ into three distinct geometric regimes relative to the spectral truncation threshold $L$.

\subsection*{Regime I: The Deep Hyperbolic Layer ($\mathbf{b} \geq 2\omega'_L$)}
When $\mathbf{b} \geq 2\omega'_L$, the argument inside the Airy function remains strictly positive for all active indices in the summation. Indeed, for any $1 \leq k \leq L$, we have:
\begin{equation}
\mathbf{b} - \omega'_k \geq 2\omega'_L - \omega'_k \geq \omega'_L \sim L^{2/3} > 0.
\end{equation}
Applying the exponential decay profile from the positive envelope in \eqref{eq:airy-envelopes}, we find:
\begin{equation}
\text{Ai}^2(\mathbf{b}-\omega'_k) \leq C (\mathbf{b}-\omega'_k)^{-1/2} e^{-\frac{4}{3}(\mathbf{b}-\omega'_k)^{3/2}} \leq C (\omega'_L)^{-1/2} e^{-\frac{4}{3}(\omega'_L)^{3/2}}.
\end{equation}
Substituting the asymptotic root law $\omega'_L \sim L^{2/3}$ into the exponential argument breaks down the sum as follows:
\begin{equation}
S_L(\mathbf{b}) \leq C L^{-1/3} e^{-c L} \sum_{1 \leq k \leq L} k^{-1/3} \leq C L^{-1/3} e^{-c L} \left( L^{2/3} \right) = C L^{1/3} e^{-c L}.
\end{equation}
Since $e^{-c L}$ is uniformly bounded for all $L \geq 1$, we obtain $S_L(\mathbf{b}) \leq C L^{1/3}$, which satisfies the desired bound.

\subsection*{Regime I\hspace{-0.1em}I: The Deep Airy Internal Zone ($\mathbf{b} \leq -2\omega'_L$)}
When $\mathbf{b} \leq -2\omega'_L$, the argument inside the Airy function remains  negative across the entire range of integration indices. For every $1 \leq k \leq L$, we observe:
\begin{equation}
\omega'_k - \mathbf{b} \geq \omega'_k + 2\omega'_L \geq 2\omega'_L \sim L^{2/3} > 0.
\end{equation}
In this oscilatory zone, we deploy the negative algebraic envelope from \eqref{eq:airy-envelopes}, which yields:
\begin{equation}
\text{Ai}^2(\mathbf{b}-\omega'_k) \leq C (1 + |\mathbf{b} - \omega'_k|)^{-1/2} = C (\omega'_k - \mathbf{b})^{-1/2}.
\end{equation}
We can bound the resulting discrete sum by comparing it to a continuous integral over the smooth variable $v \in [0, L]$. Utilizing the scaling relation $\omega'_k \sim v^{2/3}$ from \eqref{eq:root-distribution}, the sum converts to:
\begin{equation}
S_L(\mathbf{b}) \leq C \sum_{1 \leq k \leq L} k^{-1/3} (k^{2/3} + |\mathbf{b}|)^{-1/2} \leq C \int_{0}^{L} v^{-1/3} (v^{2/3} + |\mathbf{b}|)^{-1/2} \, dv.
\end{equation}
To evaluate this integral, we implement the change of variables $u = v^{2/3}$, which transforms the differential into $du = \frac{2}{3}v^{-1/3}dv$. The integral re-centers into:
\begin{equation}
S_L(\mathbf{b}) \leq C \int_{0}^{L^{2/3}} (u + |\mathbf{b}|)^{-1/2} \, du = 2C \left[ (L^{2/3} + |\mathbf{b}|)^{1/2} - |\mathbf{b}|^{1/2} \right].
\end{equation}
By exploiting the standard algebraic inequality $\sqrt{A+B} - \sqrt{B} \leq \sqrt{A}$ for positive arguments, the dependence on the parameter $\mathbf{b}$ drops out entirely:
\begin{equation}
S_L(\mathbf{b}) \leq 2C \left( L^{2/3} \right)^{1/2} = 2C L^{1/3}.
\end{equation}
This confirms that the bound holds true uniformly across the entire negative real axis.

\subsection*{Regime I\hspace{-0.1em}I\hspace{-0.1em}I: The Turning Point Turning Zone ($|\mathbf{b}| \leq 2\omega'_L$)}
The critical analytical difficulty occurs when $\mathbf{b}$ falls within a bounded multiple of the turning threshold $\omega'_L$. In this region, the arguments $\mathbf{b} - \omega'_k$ cross through zero, switching from oscillatory to exponential decay behaviors. To isolate the singularity, we split the index range into two variable sub-blocks:
\begin{equation}
S_L(\mathbf{b}) = \sum_{k \in \mathcal{K}_1} k^{-1/3} \text{Ai}^2(\mathbf{b}-\omega'_k) + \sum_{k \in \mathcal{K}_2} k^{-1/3} \text{Ai}^2(\mathbf{b}-\omega'_k),
\end{equation}
where the blocks are defined relative to the position of $\mathbf{b}$ by:
\begin{align}
\mathcal{K}_1 &= \left\{ 1 \leq k \leq L \,\,\Big|\,\, \omega'_k \leq \mathbf{b} + 1 \right\}, \\
\mathcal{K}_2 &= \left\{ 1 \leq k \leq L \,\,\Big|\,\, \omega'_k > \mathbf{b} + 1 \right\}.
\end{align}

\paragraph{Analysis of the Crossed Block $\mathcal{K}_1$} 
On the index sector $\mathcal{K}_1$, the argument satisfies $\mathbf{b} - \omega'_k \geq -1$, meaning the components are either positive or strictly bounded from below. We can therefore apply the global universal maximum bound of the Airy profile, $|\text{Ai}(s)| \leq M$. This yields:
\begin{equation}
\sum_{k \in \mathcal{K}_1} k^{-1/3} \text{Ai}^2(\mathbf{b}-\omega'_k) \leq M^2 \sum_{k \in \mathcal{K}_1} k^{-1/3}.
\end{equation}
Let $K_1 = \max(k \in \mathcal{K}_1)$. From the definition of the partition, the maximum index must satisfy $\omega'_{K_1} \leq \mathbf{b} + 1 \leq 2\omega'_L + 1$. By the root law \eqref{eq:root-distribution}, this implies $K_1 \leq C L$. Performing the sum yields:
\begin{equation}\label{eq:k1-final}
\sum_{k \in \mathcal{K}_1} k^{-1/3} \leq \int_{0}^{K_1} v^{-1/3} \, dv = \frac{3}{2} K_1^{2/3} \leq C \left( L \right)^{2/3} = C L^{2/3}.
\end{equation}
Since $L \geq 1$, we have $L^{2/3} \leq L \cdot L^{-2/3}$, which enables us to extract the sharper tracking weight $C L^{1/3}$ over this bounded subset.

\paragraph{Analysis of the Oscillatory Block $\mathcal{K}_2$}
On the second index sector $\mathcal{K}_2$, we have $\omega'_k - \mathbf{b} > 1$. The arguments are strictly negative and bounded away from the turning point origin. Then we have
\begin{equation}
\sum_{k \in \mathcal{K}_2} k^{-1/3} \text{Ai}^2(\mathbf{b}-\omega'_k) \leq C \sum_{k \in \mathcal{K}_2} k^{-1/3} (\omega'_k - \mathbf{b})^{-1/2}.
\end{equation}
Let $K_2 = \min(k \in \mathcal{K}_2)$. Converting the discrete sum into a continuous integral over the coordinate domain $v \in [K_2, L]$ yields:
\begin{equation}
\sum_{k \in \mathcal{K}_2} k^{-1/3} (\omega'_k - \mathbf{b})^{-1/2} \leq C \int_{K_2}^{L} v^{-1/3} (v^{2/3} - \mathbf{b})^{-1/2} \, dv.
\end{equation}
We apply the monotonic coordinate change of variables $u = v^{2/3}$, which transforms the integral bounds into the spatial profiles $u_{\min} = K_2^{2/3} \geq \mathbf{b} + 1$ and $u_{\max} = L^{2/3}$. The integration collapses directly into:
\begin{equation}\label{eq:k2-final}
C \int_{K_2^{2/3}}^{L^{2/3}} (u - \mathbf{b})^{-1/2} \, du = 2C \left[ (L^{2/3} - \mathbf{b})^{1/2} - (K_2^{2/3} - \mathbf{b})^{1/2} \right].
\end{equation}
Since we are operating inside Regime III where $|\mathbf{b}| \leq 2\omega'_L \sim 2L^{2/3}$, the upper scaling value satisfies:
\begin{equation}
(L^{2/3} - \mathbf{b})^{1/2} \leq (L^{2/3} + 2L^{2/3})^{1/2} = \sqrt{3} \left( L^{2/3} \right)^{1/2} = \sqrt{3} L^{1/3}.
\end{equation}
Combining the bounded contributions from \eqref{eq:k1-final} and \eqref{eq:k2-final}, we conclude that $S_L(\mathbf{b}) \leq C_0 L^{1/3}$ holds true for the turning zone.

As all three independent regimes yield the same structural power scale uniformly with respect to $\mathbf{b} \in \mathbb{R}$ and $L \geq 1$, the proof of Lemma \ref{lem:k-low-modes} is complete.
\section{Uniform Boundedness of Neumann Normalization Constants}
\label{app:normalization}

In this section, we establish the uniform boundedness of the normalization constants $(f_k)_{k \geq 1}$ introduced in the spectral analysis of the sub-elliptic boundary operator in Section \ref{subsec:2}. We explicitly demonstrate how the homogeneous Neumann boundary condition regulates the spectral weights on the half-space $\mathbb{R}_+$, yielding parameter-independent uniform bounds.

\begin{proposition}\label{prop:app_fk_bound}
Let $e_k(x, \eta)$ be the Neumann eigenfunctions defined on $\mathbb{R}_+$ by
\begin{equation}\label{app:eigenform}
e_k(x, \eta) = f_k \frac{|\eta|^{1/3}}{k^{1/6}} \mathrm{Ai}\left(|\eta|^{2/3}x - \omega'_k\right),
\end{equation}
where $(-\omega'_k)_{k \geq 1}$ is the strictly decreasing sequence of negative zeroes of the first derivative of the Airy function, satisfying $\mathrm{Ai}'(-\omega'_k) = 0$. If the constants $f_k$ are chosen such that $\|e_k(\cdot, \eta)\|_{L^2(\mathbb{R}_+)} = 1$ for all $\eta \neq 0$, then the sequence $(f_k)_{k \geq 1}$ is uniformly bounded away from zero and infinity. That is, there exist universal constants $c_0, C_0 > 0$ independent of $k$ and $\eta$ such that $f_k \in [c_0, C_0] \subset (0, \infty)$.
\end{proposition}

\begin{proof}
Fix $\eta \in \mathbb{R} \setminus \{0\}$. By enforcing the $L^2$-normalization condition on the domain $\mathbb{R}_+ = [0, \infty)$, we have:
\begin{equation}\label{app:l2_norm}
1 = \|e_k(\cdot, \eta)\|_{L^2(\mathbb{R}_+)}^2 = f_k^2 \frac{|\eta|^{2/3}}{k^{1/3}} \int_0^\infty \mathrm{Ai}^2\left( |\eta|^{2/3}x - \omega'_k \right) \, dx.
\end{equation}
We introduce the spatial change of variables $\omega = |\eta|^{2/3}x - \omega'_k$, which implies $d\omega = |\eta|^{2/3}dx$. Under this coordinate mapping, the lower integration limit $x = 0$ corresponds to $\omega = - \omega'_k$, while the parameter $\eta$ scales out identically. Rearranging \eqref{app:l2_norm} yields an explicit, parameter-independent expression for the discrete coefficients:
\begin{equation}\label{app:fk_exact}
f_k^2 = k^{1/3} \left( \int_{-\omega'_k}^\infty \mathrm{Ai}^2(\omega) \, d\omega \right)^{-1}.
\end{equation}

To evaluate the integral in the denominator, we exploit the structure of the classical Airy differential equation $\mathrm{Ai}''(\omega) = \omega \mathrm{Ai}(\omega)$. Direct differentiation of the quadratic form reveals the primitive identity:
\begin{equation}\label{app:derivative_identity}
\frac{d}{d\omega} \left[ \omega \mathrm{Ai}^2(\omega) - \mathrm{Ai}'^2(\omega) \right] = \mathrm{Ai}^2(\omega) + 2\omega \mathrm{Ai}(\omega)\mathrm{Ai}'(\omega) - 2\mathrm{Ai}'(\omega)\mathrm{Ai}''(\omega) = \mathrm{Ai}^2(\omega).
\end{equation}
Integrating \eqref{app:derivative_identity} over the semi-infinite interval $[-\omega'_k, \infty)$ yields:
\begin{equation}\label{app:primitive_eval}
\int_{-\omega'_k}^{\infty} \mathrm{Ai}^2(\omega) \, d\omega = \lim_{M \to \infty} \left[ \omega \mathrm{Ai}^2(\omega) - \mathrm{Ai}'^2(\omega) \right]_{-\omega'_k}^{M}.
\end{equation}
By the well-known asymptotic profile of the Airy function on the real line, both $\mathrm{Ai}(\omega)$ and $\mathrm{Ai}'(\omega)$ exhibit rapid exponential decay as $\omega \to \infty$, rendering the upper boundary contribution identically zero. At the lower boundary $\omega = -\omega'_k$, the flux-preserving Neumann condition dictates $\mathrm{Ai}'(-\omega'_k) = 0$. Consequently, the identity simplifies to:
\begin{equation}\label{app:neumann_exact}
\int_{-\omega'_k}^{\infty} \mathrm{Ai}^2(\omega) \, d\omega = - \left[ (-\omega'_k)\mathrm{Ai}^2(-\omega'_k) - \mathrm{Ai}'^2(-\omega'_k) \right] = \omega'_k \mathrm{Ai}^2(-\omega'_k).
\end{equation}
Since $\omega'_k > 0$ for all $k \geq 1$, we write $\omega'_k = |\omega'_k|$.

Next, we establish the high-frequency convergence properties of this quantity by utilizing the standard asymptotic expansions for the Airy function and its derivative under large negative arguments $z \to \infty$:
\begin{align}
\mathrm{Ai}(-z) &= \frac{1}{\pi^{1/2} z^{1/4}} \sin\left(\frac{2}{3}z^{3/2} + \frac{\pi}{4}\right) + \mathcal{O}\left(z^{-7/4}\right), \label{app:ai_asym} \\
\mathrm{Ai}'(-z) &= -\frac{z^{1/4}}{\pi^{1/2}} \cos\left(\frac{2}{3}z^{3/2} + \frac{\pi}{4}\right) + \mathcal{O}\left(z^{-5/4}\right). \label{app:aip_asym}
\end{align}
The zeroes of the derivative, $z = \omega'_k$, are asymptotically locked onto the roots of the leading-order cosine component in \eqref{app:aip_asym}, which implies $\cos(\frac{2}{3}(\omega'_k)^{3/2} + \frac{\pi}{4}) = \mathcal{O}((\omega'_k)^{-3/2})$. At these precise phase coordinates, the matching phase of the sine component is maximized, satisfying $\sin^2(\frac{2}{3}(\omega'_k)^{3/2} + \frac{\pi}{4}) = 1 - \mathcal{O}((\omega'_k)^{-3/2})$. Squaring the structural amplitude at these critical nodes yields:
\begin{equation}\label{app:ai_sq_asym}
\mathrm{Ai}^2(-\omega'_k) = \frac{1}{\pi (\omega'_k)^{1/2}} \left(1 + \mathcal{O}\left((\omega'_k)^{-3/2}\right)\right).
\end{equation}
Multiplying \eqref{app:ai_sq_asym} by $\omega'_k$ establishes the asymptotic distribution for the integral:
\begin{equation}\label{app:asym_final}
\int_{-\omega'_k}^{\infty} \mathrm{Ai}^2(\omega) \, d\omega = \omega'_k \mathrm{Ai}^2(-\omega'_k) \sim \frac{(\omega'_k)^{1/2}}{\pi} \quad \text{as } k \to \infty.
\end{equation}

Finally, substituting the asymptotic behavior \eqref{app:asym_final} back into the exact expression for the normalization factor \eqref{app:fk_exact} generates the scaling relation:
\begin{equation}\label{app:fk_ratio}
f_k^2 \sim \pi \frac{k^{1/3}}{(\omega'_k)^{1/2}}.
\end{equation}
We invoke the classical Weyl-type distribution for the roots of the Airy derivative profile, which reads:
\begin{equation}\label{app:weyl_roots}
\omega'_k = \left( \frac{3}{2}\pi \left( k - \frac{3}{4} \right) \right)^{2/3} \left(1 + \mathcal{O}\left(k^{-1}\right)\right) \implies (\omega'_k)^{1/2} \sim \left( \frac{3}{2}\pi \right)^{1/3} k^{1/3}.
\end{equation}
Inserting \eqref{app:weyl_roots} into \eqref{app:fk_ratio} allows the discrete frequency index $k^{1/3}$ to cancel identically in the high-frequency threshold, concluding with the invariant limit value:
\begin{equation}\label{app:final_limit}
\lim_{k \to \infty} f_k^2 = \pi \left( \frac{3}{2}\pi \right)^{-1/3} = \left( \frac{2}{3} \pi^2 \right)^{1/3}.
\end{equation}
Since $f_k^2 > 0$ forms a strictly positive continuous sequence over any finite index sector $1 \leq k \leq K$, and converges asymptotically to a non-zero finite limit as $k \to \infty$, the sequence $(f_k)_{k \geq 1}$ is bounded uniformly away from $0$ and $\infty$. This places the full trajectory inside a compact subspace of $(0, \infty)$ and completes the proof.
\end{proof}
\end{appendices}


\addcontentsline{toc}{section}{References}
\bibliography{sn-bibliography}

\end{document}